\documentclass[PhD,english,french]{ulthese}
  \ifxetex\else \usepackage[utf8]{inputenc} \fi

  \usepackage[T1]{fontenc}
  \usepackage{graphicx}
  \usepackage{calrsfs}
  \usepackage{enumerate}
  \usepackage[shortlabels]{enumitem}
  \usepackage{longtable}
  \usepackage{amsmath,amssymb,amsfonts,amsthm}
  \usepackage{algorithm}
  \usepackage{comment}
  \usepackage{xcolor}
  \usepackage[lofdepth,lotdepth]{subfig}
  \usepackage{pgf,tikz}
  \usepackage{pgfplots}

  \DeclareMathAlphabet\mathpzc{T1}{pzc}{mb}{it}
  \DeclareMathAlphabet{\pazocal}{OMS}{cmsy}{m}{n}

  \DeclareMathOperator{\Proj}{Proj}
  \DeclareMathOperator{\inter}{int}
  \DeclareMathOperator{\supp}{supp}

  \def\norml{\left\|}
  \def\normr{\right\|}
  \DeclareMathOperator{\circl}{\circ\,}
  \DeclareMathOperator{\divv}{div}
  \DeclareMathOperator{\Divv}{\textbf{div}}
  \DeclareMathOperator{\gradd}{\nabla\!}
  \DeclareMathOperator{\grad}{\nabla\!}
  \DeclareMathOperator{\maxx}{p_{+}}
  \DeclareMathOperator{\maxxc}{p_{c,+}}
  \def\dmaxx{d\hspace{-0.1em}\maxx}
  
  \def\dqq{d\mathbf{q}}

  \DeclareMathOperator{\normalInt}{\mathbf{n}_\mathbf{o}}
  \DeclareMathOperator{\tanInt}{\mathbf{t}_\mathbf{o}}
  \DeclareMathOperator{\normalExt}{\mathbf{n}}
  \DeclareMathOperator{\tanExt}{\mathbf{t}}

  \DeclareMathOperator{\JacV}{J_\Omega}
  \DeclareMathOperator{\JacB}{J_\Gamma}
  
  \def\prodL2#1#2{\left(#1\right)_{#2}}
  
  \def\prodD#1#2{\langle#1\rangle_{#2}}

  \DeclareMathOperator{\hh}{\mathbf{h}}
  
  \DeclareMathOperator{\pp}{\mathbf{p}}
  \DeclareMathOperator{\qq}{\mathbf{q}}
  
  \DeclareMathOperator{\uu}{\mathbf{u}}
  \DeclareMathOperator{\vv}{\mathbf{v}}
  \DeclareMathOperator{\ww}{\mathbf{w}}
  \DeclareMathOperator{\xx}{\mathbf{x}}
  \DeclareMathOperator{\zz}{\mathbf{z}}
  \DeclareMathOperator{\gG}{\mathbf{g}}
  \DeclareMathOperator{\Aa}{\mathbb{C}}
  \DeclareMathOperator{\Bb}{\mathbf{B}}
  \DeclareMathOperator{\Kk}{\mathbf{K}}
  \DeclareMathOperator{\Ll}{\mathbf{L}}
  \DeclareMathOperator{\Hh}{\mathbf{H}}
  \DeclareMathOperator{\Ss}{\mathbf{S}}
  \DeclareMathOperator{\Tt}{\mathbf{T}}
  \DeclareMathOperator{\Vv}{\mathbf{V}}
  \DeclareMathOperator{\Cc}{\pmb{\pazocal{C}}}
  \DeclareMathOperator{\Ii}{\mathbf{I}}
  \DeclareMathOperator{\Id}{\mathrm{Id}}
  \DeclareMathOperator{\ff}{\mathbf{f}}
  \DeclareMathOperator{\tauu}{\boldsymbol{\tau}}
  \DeclareMathOperator{\epsilonn}{\boldsymbol{\epsilon}}
  \DeclareMathOperator{\varepsilonn}{\boldsymbol{\varepsilon}}
  \DeclareMathOperator{\sigmaa}{\boldsymbol{\sigma}}
  \DeclareMathOperator{\muu}{\boldsymbol{\mu}}
  \DeclareMathOperator{\nuu}{\boldsymbol{\nu}}
  
  \DeclareMathOperator{\thetaa}{\boldsymbol{\theta}}
  \DeclareMathOperator{\Thetaa}{\boldsymbol{\Theta}}
  \DeclareMathOperator{\xii}{\boldsymbol{\xi}}
  \DeclareMathOperator{\phii}{\boldsymbol{\phi}}
  \DeclareMathOperator{\psii}{\boldsymbol{\psi}}
  
  \DeclareMathOperator{\Ww}{\mathbf{W}}
  \DeclareMathOperator{\Xx}{\mathbf{X}}

  \newtheorem{theo}{Theorem}[section]
  \newtheorem{cor}{Corollary}[section]
  \newtheorem{lemma}{Lemma}[section]
  \newtheorem{hypothesis}{\textsc{Assumption}}
  \newtheorem{thrm}{Théorème}[section]
  \newtheorem{prpstn}{Proposition}[section]
  \newtheorem{crllr}{Corollaire}[section]
  \newtheorem{lmm}{Lemme}[section]
  \newtheorem{hyp}{\textsc{Hypothèse}}
  
  \theoremstyle{definition}
  \newtheorem{remark}{Remark}[section]
  \newtheorem{defn}{Definition}[section]
  \newtheorem*{notation}{Notation}
  \newtheorem{rmrk}{Remarque}[section]
  \newtheorem{dfntn}{Définition}[section]
  \newtheorem{exmpl}{Exemple}[section]
  
\newenvironment{prf} {\begin{proof}[Proof]} {\end{proof}}

\usepackage{color}
\usepackage[colorinlistoftodos,textwidth=3.5cm]{todonotes}

\def\BC#1{\textbf{\color{blue}#1}}

\definecolor{dartmouthgreen}{rgb}{0.05, 0.5, 0.06}

  \titre{Optimisation de formes pour les problèmes de contact en élasticité linéaire}

  \auteur{Bastien Chaudet-Dumas}
  \direction{Jean Deteix, directeur de recherche}
  \programme{Doctorat en Mathématiques}

\begin{document}

\frontmatter                    

\frontispice                    

\chapter*{Résumé}               
\label{chap:resume}             
\phantomsection\addcontentsline{toc}{chapter}{\nameref{chap:resume}} 

Cette thèse traite du problème d'optimisation de formes dans le contexte de la mécanique des solides en contact. Le modèle physique considéré est celui de solides linéaires élastiques en petites déformations, en contact (glissant ou avec frottement de Tresca) avec un corps rigide. Les formulations mathématiques étudiées sont deux versions régularisées de l'inéquation variationnelle décrivant le système d'origine: la formulation pénalisée et la formulation Lagrangien augmenté. Comme ces deux formulations présentent des non différentiabilités, nous proposons une approche par dérivées directionnelles pour obtenir les dérivées de forme associées. En particulier, pour chacune des formulations, nous exprimons des conditions suffisantes pour que la solution soit dérivable par rapport à la forme. Ceci nous permet de construire un algorithme d'optimisation topologique de type gradient, s'appuyant sur les dérivées obtenues et une représentation des formes par des ensembles de niveau (\textit{level-sets}). L'algorithme bénéficie en outre d'une technique de découpage de maillage qui permet d'obtenir une représentation explicite de la forme à chaque itération, et ainsi d'appliquer fortement les conditions aux limites sur la zone de contact. Après avoir détaillé les différentes étapes de la méthode, nous présentons des résultats numériques en deux et trois dimensions pour en tester la validité.
\chapter*{Abstract}             
\label{chap:abstract}           
\phantomsection\addcontentsline{toc}{chapter}{\nameref{chap:abstract}} 

This thesis deals with shape optimization for contact mechanics. More specifically, the linear elasticity model is considered under the small deformations hypothesis, and the elastic body is assumed to be in contact (sliding or with Tresca friction) with a rigid foundation. The mathematical formulations studied are two regularized versions of the original variational inequality: the penalty formulation and the augmented Lagrangian formulation. In order to get the shape derivatives associated to those two non-differentiable formulations, we suggest an approach based on directional derivatives. Especially, we derive sufficient conditions for the solution to be shape differentiable. This allows to develop a gradient-based topology optimization algorithm, built on these derivatives and a level-set representation of shapes. The algorithm also benefits from a mesh-cutting technique, which gives an explicit representation of the shape at each iteration, and enables to apply the boundary conditions strongly on the contact zone. The different steps of the method are detailed. Then, to validate the approach, some numerical results on two-dimensional and three-dimensional benchmarks are presented.
\cleardoublepage

\tableofcontents                
\cleardoublepage


\listoffigures                  
\cleardoublepage



\chapter*{Remerciements}        
\label{chap:remerciements}      
\phantomsection\addcontentsline{toc}{chapter}{\nameref{chap:remerciements}} 

Mes premiers remerciements vont à mon directeur de thèse Jean Deteix, pour son implication et sa disponibilité tout au long de ces années. Avoir un directeur qui accorde autant de temps à ses étudiants est un vrai privilège. Je le remercie également de m'avoir laissé libre des directions que je souhaitais donner à nos travaux de recherche. Ses connaissances en optimisation et son expertise en méthodes numériques sont une composante essentielle des résultats présentés dans ce manuscrit.

J'aimerais ensuite remercier les autres professeurs et les professionnels de recherche du GIREF. Nous disposons grâce à eux d'un cadre de travail extrêmement stimulant. En particulier, les algorithmes numériques mis en \oe uvre par les doctorants dans MEF++ ne seraient pas ce qu'ils sont sans le soutien technique d'Éric Chamberland et Thomas Briffard.

Merci à Partice Hauret (directeur de recherche, Michelin) de m'avoir fait découvrir le GIREF. J'ai eu la chance de bénéficier de ses conseils lors des quelques discussions enrichissantes dans les premières années. Sa présence parmi les membres du jury de cette thèse était pour moi une évidence. Je le remercie d'avoir accepté l'invitation.

Je remercie par ailleurs les professeurs Michel Delfour (Université de Montréal), Alberto Paganini (University of Leicester) et José Urquiza (Université Laval) d'avoir accepté de faire partie du jury et de prendre le temps d'examiner ce document. 

Je profite également de ces quelques lignes pour remercier mes collègues doctorants du GIREF ainsi que mes amis. À Québec comme à Montréal, les rencontres que j'ai faites dans cette belle province ont été profondément enrichissantes. Je remercie également les membres de ma famille, pour leur soutien constant, malgré les six mille kilomètres qui nous séparent, malgré le fait qu'ils soient loin d'être familiers avec le domaine de la recherche en mathématiques. Enfin et surtout, un immense merci à ma femme, d'avoir toujours été là pour moi, et d'avoir été si patiente.            
\chapter*{Avant-propos}         
\label{chap:avant-propos}       
\phantomsection\addcontentsline{toc}{chapter}{\nameref{chap:avant-propos}} 

L'optimisation de formes est une problématique qui a toujours été au c\oe ur du processus de conception de structures. En particulier, depuis le début des avancées du calcul scientifique, notamment en termes de puissance et de temps de calcul, la plupart des chaînes de conception se sont vues ajouter une étape d'optimisation de formes, automatisée par des algorithmes mathématiques. Dans de nombreux domaines d'application en ingénierie (automobile, aéronautique, biomécanique), les pièces conçues sont amenées à subir des contraintes en présence de contact (pneu sur la route, contact entre deux pièces mécaniques d'un moteur, prothèse de genou). Dans ces cas-là, le développement de méthodes permettant de traiter des problèmes de conception optimale de structures en présence de contact est nécessaire.

Par ailleurs, le modèle standard de la mécanique des solides élastiques en contact prend la forme d'une inéquation variationnelle (IV). Le problème d'optimisation de formes dans ce contexte s'écrit donc sous la forme plus générale d'un problème de contrôle optimal d'une IV, où le paramètre de contrôle est la forme. La difficulté de ces problèmes est que la solution de l'IV n'est en général pas différentiable (au sens classique) par rapport au paramètre de contrôle. Depuis les travaux de Mignot sur la différentiabilité des opérateurs de projection dans les années 1970, la théorie du contrôle optimal de telles inéquations est restée un domaine de recherche actif. Différentes approches ont été proposées pour parvenir à une expression des conditions d'optimalité associées à de tels systèmes: par régularisation, à l'aide de dérivées directionnelles, ou plus récemment en utilisant les outils du calcul sous-différentiel (de Clarke \cite{clarke1990optimization} ou de Mordukhovich \cite{mordukhovich2006variationalI,mordukhovich2006variationalII}).

Comme le souligne le paragraphe précédent, dans l'étude de tels problèmes, l'obstacle majeur est la non-différentiabilité. Si elle se contourne facilement pour obtenir des résultats quant à l'existence et l'unicité de la solution de l'IV (ou pour la mise en \oe uvre de méthodes de résolution numérique), ce n'est plus le cas lorsqu'on aborde des questions liées à l'écriture de conditions d'optimalité telles que l'analyse de sensibilité par rapport à la forme. Nous proposons ici un moyen d'appréhender cette difficulté.

L'objectif de cette thèse est de proposer une approche mathématique \textit{rigoureuse} pour le traitement des problèmes d'optimisation de formes dans le contexte de la mécanique des solides en contact. Ce travail se situe donc à l'intersection entre les deux domaines présentés dans les deux premiers paragraphes. Plus précisément, plutôt que de nous concentrer sur les applications indutrielles de ces problèmes, nous choisissons d'en traiter les aspects théoriques en adoptant le point de vue du contrôle optimal. Ce choix nous restreint naturellement aux phénomènes mécaniques dont les formulations possèdent de bonnes propriétés (existence, unicité et régularité de la solution), à savoir les problèmes de contact glissant ou avec frottement de Tresca, en élasticité linéaire. Pour ce qui est des formulations mathématiques associées, nous privilégions celles qui sont les plus régulières. En particulier, plutôt que de considérer une formulation \textit{purement lagrangienne}, qui consiste en une réécriture équivalente de l'IV sous la forme d'un problème de point-selle (\textit{inf-sup}), nous préférons étudier des formulations issues d'une \textit{régularisation} de l'IV de départ: la formulation \textit{pénalisée}, et la formulation \textit{lagrangien augmenté}, qui transforment l'IV en une équation variationnelle (éventuellement mixte). Nous faisons ce choix pour deux raisons. Premièrement, ces formulations sont les plus utilisées dans les applications industrielles, car le fait de transformer l'inégalité en égalité facilite le traitement numérique. De plus, cela permet de s'affranchir d'une partie des difficultés techniques lorsqu'on cherche à écrire les conditions d'optimalité. 

Ce manuscrit se compose de six chapitres, répartis dans quatre grandes parties. La première partie est une mise en contexte. Elle peut être vue comme une introduction \textit{détaillée} mais \textit{spécifique} à notre cadre d'étude. On y présente au chapitre \ref{chap:1.1} la méthode générale d'optimisation de formes à l'aide d'ensembles de niveau. Puis, dans le chapitre \ref{chap:1.2}, on introduit les concepts de base de la mécanique du contact, et en particulier les formulations que nous avons choisi d'étudier. La deuxième partie se concentre sur l'étude du problème d'optimisation de formes pour la formulation pénalisée, et comporte deux chapitres. Le chapitre \ref{chap:2.1} propose une approche pour déterminer les expressions des dérivées de forme dans ce contexte. Il a été accepté le 5 juillet 2019 dans \textit{SIAM Journal on Control and Optimization}. Le chapitre \ref{chap:2.2} peut être vu comme une annexe au chapitre \ref{chap:2.1}. Il fait le lien entre les résultats obtenus pour la formulation pénalisée et ceux obtenus pour une version régularisée (très populaire) de cette même formulation. La troisième partie ne contient que le chapitre \ref{chap:3.1}, qui s'intéresse à l'expression des dérivées de forme pour la formulation lagrangien augmenté. Ce chapitre  a été accepté dans la revue \textit{ESAIM: Control, Optimisation and Calculus of Variations} le 12 janvier 2020. La quatrième et dernière partie, composée du chapitre \ref{chap:4.1}, porte sur les aspects numériques de la méthode d'optimisation de formes proposée. 

Les chapitres \ref{chap:1.1}, \ref{chap:1.2}, \ref{chap:2.1} et \ref{chap:3.1} peuvent être lus indépendamment. Pour le chapitre \ref{chap:2.2}, il est préférable d'avoir lu le chapitre \ref{chap:2.1}. Quant au chapitre \ref{chap:4.1}, nous recommandons de le lire après avoir lu les chapitres \ref{chap:1.1} et \ref{chap:1.2}.             

\mainmatter                        

\chapter*{Introduction}         
\label{chap:introduction}       
\phantomsection\addcontentsline{toc}{chapter}{\nameref{chap:introduction}} 

Lors de la conception d'un objet en vue d'une utilisation dans un contexte spécifique, la question de savoir quelle forme géométrique donner à l'objet pour en améliorer les propriétés est fondamentale. On peut facilement imaginer de nombreuses situations où cette question est posée: quelle forme donner à une prothèse de genou pour minimiser son usure? Quelle forme donner à une aile d'avion pour maximiser la portance et minimiser la traînée? Quelle forme donner à une pièce mécanique pour qu'elle soit aussi légère que possible tout en résistant à des contraintes imposées? etc. L'ensemble de ce type de problèmes constitue le domaine de \textit{l'optimisation de formes} en mécanique, puisqu'il est induit que l'objet considéré est soumis à des contraintes mécaniques. Dans cette thèse, nous nous intéressons à de tels problèmes dans le cas particulier de structures mécaniques en contact.

Plus formellement, la définition d'un problème d'optimisation de formes en mécanique requiert trois éléments principaux: un \textit{modèle}, i.e.$\!$ une équation d'état qui traduit l'équilibre mécanique, une \textit{fonctionnelle} (ou \textit{critère}) qu'on cherche à minimiser, et un \textit{ensemble admissible} qui représente l'ensemble des formes parmi lesquelles on s'autorise à chercher notre forme optimale. Résoudre ce type de problème, c'est donc trouver, parmi les formes admissibles, celle qui minimise la fonctionnelle, sous la contrainte que l'équilibre mécanique soit vérifié. 
En général, il est difficile d'obtenir mieux que les conditions d'optimalité d'ordre 1. On se dirige donc vers des méthodes de type gradient.
Cependant, en fonction du choix de l'ensemble admissible et de la manière d'imposer la contrainte d'équilibre mécanique, on peut se ramener à des problèmes complètement différents, avec des complexités variables. L'approche retenue pour le traitement du problème d'optimisation est, bien entendu, influencée par ces choix.

En ce qui concerne le choix de l'ensemble admissible, nous mentionnons trois grandes approches. Premièrement, il y a l'approche dite \textit{paramétrique}, dans laquelle la forme est représentée par un nombre fini de paramètres (points de contrôle, dimensions relatives à conception, etc), voir \cite{braibant1984shape,samareh2001survey}. Dans ces cas-là, l'ensemble admissible est constitué du produit cartésien des ensembles de valeurs permises pour chacun de ces paramètres. Ensuite, il y a l'approche \textit{géométrique}, un peu plus générale, qui consiste à modifier la forme de l'objet en faisant varier ses frontières, sans pour autant changer sa topologie, voir par exemple \cite{pironneau1982optimal,cea1986conception,sokolowski1992introduction,haslinger2003introduction}. L'ensemble admissible est ici considérablement plus grand que dans l'approche précédente, puisqu'on n'a aucune restriction a priori sur la structure et la régularité des frontières du domaine. Notons qu'avec cette approche, d'un point de vue numérique, il est possible d'introduire des changements de topologie de manière relativement simple (voir par exemple \cite{allaire2004structural}), quoique peu rigoureuse d'un point de vue théorique. Ce type de méthodes est d'ailleurs fort répandu dans les contextes d'applications. Enfin, l'approche la moins restrcitive est l'approche \textit{topologique}, dans laquelle on permet de modifier non seulement la forme des frontières, mais aussi la topologie du domaine (par exemple le nombre de composantes connexes de son complémentaire), voir les différentes méthodes \cite{BenSig2003,allaire2004structural, allaire2012shape}. Dans ce dernier cas, on n'a presque aucune restriction sur l'ensemble admissible: on recherche la forme optimale parmi toutes les formes possibles, peu importe leurs propriétés géométriques ou topologiques. 

Pour ce qui est de l'imposition de la contrainte d'équilibre mécanique, il existe deux « paradigmes » (voir \cite{LiPet2004}): celui dans lequel il convient d'\textit{optimiser-puis-discrétiser} et l'autre, dans lequel on préfère \textit{discrétiser-puis-optimiser}. Dans le premier, on traduit l'équilibre mécanique sous la forme continue d'une équation aux dérivées partielles (EDP) ou le plus souvent, de la formulation variationnelle qui en découle. On se ramène alors à un problème d'optimisation en dimension infinie, dont on cherche à exprimer les conditions d'optimalité. L'étape suivante consiste alors à choisir un schéma de discrétisation pour l'EDP et les conditions d'optimalité. Bien que ce point de vue impose de travailler avec des espaces de fonctions, il présente l'avantage de donner une expression des gradients indépendante du schéma de discrétisation. À l'inverse, dans le cas \textit{discrétiser-puis-optimiser}, on commence par discrétiser l'EDP, ce qui nous donne une contrainte d'équilibre sous la forme d'un système algébrique, puis on écrit les conditions d'optimalité. Même si cette étape se traite plus simplement à l'aide des outils d'optimisation en dimension finie, on obtient d'une part des gradients qui dépendent du schéma de discrétisation, et d'autre part une taille de système qui peut devenir très grande si on raffine beaucoup le maillage du domaine de calcul. Dans le cas de l'optimisation topologique de structures, on mentionnera par exemple la différence de point de vue entre \cite{allaire2012shape}, où on optimise puis discrétise, et \cite{BenSig2003}, où on discrétise puis optimise.

En énumérant les différentes combinaisons possibles entre l'approche pour le choix de l'ensemble admissible et le point de vue sur l'ordre des étapes d'optimisation et de discrétisation, on obtient une bonne vue d'ensemble des différents groupes de méthodes pour résoudre les problèmes d'optimisation de formes. Dans cette thèse, nous choisissons de nous placer dans le cadre de l'optimisation géométrique dans le paradigme \textit{optimiser-puis-discrétiser}, en bénéficiant d'une méthode numérique permettant d'introduire des changements de topologie.

Comme nous l'avons mentionné, le contexte physique dans lequel nous posons le problème d'optimisation de formes est celui de la mécanique des solides en contact. Étant donné que les phénomènes de contact sont omniprésents en ingénierie, les champs d'application de ce type de problèmes dans l'industrie sont extrêmement nombreux et variés. Du côté de la recherche, l'étude théorique et la résolution numérique du problème de contact seul constituent des domaines très actifs, aussi bien en génie mécanique qu'en physique ou qu'en mathématiques.

Du point de vue théorique, les formulations mathématiques associées à ces problèmes présentent de nombreuses difficultés, provenant à la fois du comportement des matériaux considérés et des phénomènes de contact. En effet, même lorsqu'il n'y a pas de contact, si le problème traduisant l'équilibre mécanique est bien posé dans le cas d'un matériau linéaire élastique en petites déformations, l'existence d'une solution dans le cas non-linéaire n'est garantie que sous certaines conditions, voir par exemple \cite{ball1976convexity,Cia1988a}. Quant à la question de l'unicité, elle reste un problème ouvert. Pour s'assurer que le système initial a de bonnes propriétés, nous limitons donc notre étude au cas de matériaux linéaires élastiques en petites déformations. La prise en compte des phénomènes de contact se fait par l'ajout de conditions aux limites sur la partie du bord concernée: une condition de \textit{non-pénétration}, qui s'exprime comme une contrainte d'inégalité sur le déplacement, et une condition qui traduit les effets de \textit{frottement} éventuels. Outre la contrainte d'inégalité, la complexité de la nouvelle formulation mathématique dépend de la loi de frottement considérée. Nous citons les deux plus utilisées: celle de Tresca et celle de Coulomb. Dans la première, le problème prend la forme d'une inéquation variationnelle elliptique de deuxième espèce, dont le caractère bien posé et les propriétés de régularité sont bien connues depuis les années soixantes (voir entre autres \cite{browder1965nonlinear,hartman1966some,lewy1969regularity}). Dans la deuxième, qui constitue un modèle physique plus réaliste et donc plus complexe, la formulation s'écrit comme une inéquation quasi-variationnelle. On dispose quand même de résultats d'existence \cite{necas1980solution,jaruvsek1983contact,eck1998existence} et d'unicité \cite{renard2006uniqueness} de la solution, mais seulement si le coefficient de frottement est suffisamment petit.

Du point de vue numérique, à cause des non-linéarités et des non-différentiabilités induites par les conditions de contact, la résolution de ces problèmes est une tâche difficile. Bien qu'il semble y avoir consensus sur le choix de la discrétisation (méthode des éléments finis), on compte quatre grands types de méthodes pour le traitement de ces difficultés techniques.
Premièrement, la \textit{méthode de pénalisation} consiste à transformer l'IV en une équation variationnelle, voir par exemple \cite{KikSon1981,WriSimTay1985}. Ceci se fait en deux étapes: d'abord on remplace la contrainte d'inégalité par un terme qui pénalise le non-respect de cette contrainte à l'aide d'un paramètre appelé \textit{paramètre de pénalisation}, puis on utilise ce même paramètre pour régulariser les lois de frottement (qui sont non différentiables). On obtient alors une solution approchée, dont l'écart à la solution d'origine tend vers 0 lorsque le paramètre de pénalisation tend vers 0. Cependant, le choix d'un paramètre trop petit détériore fortement le conditionnement du système algébrique obtenu après discrétisation, ce qui peut mettre en péril la résolution. Même si cette méthode est facile à mettre en \oe uvre et peu coûteuse, il y a néanmoins un compromis à faire entre précision et robustesse pour le choix du paramètre.
Il y a par ailleurs la \textit{méthode des multiplicateurs de Lagrange}. Le principe est cette fois de réécrire l'IV sous une formulation mixte équivalente, faisant intervenir deux multiplicateurs associés à la contrainte d'inégalité et à la condition de frottement, respectivement. Comme il n'y a aucune forme de régularisation dans cette approche, la solution ne dépend d'aucun paramètre. On discrétise ensuite la formulation, ce qui conduit à la résolution d'un problème de type point-selle. La condition sur le premier multiplicateur prend la forme d'une inégalité, qu'on peut par exemple traiter avec une stratégie de contraintes actives (voir le livre \cite{bertsekas2014constrained}). On a donc une méthode d'une grande précision, au prix d'un coût de calcul relativement important, puisqu'on a ajouté les deux multiplicateurs à la liste des inconnues variationnelles, ce qui a évidemment pour effet d'augmenter la taille et la complexité du système algébrique à résoudre, voir \cite{diop2019phd}.
Nous mentionnons également la \textit{méthode de lagrangien augmenté}, à mi-chemin entre les deux méthodes précédentes, voir \cite{WriSimTay1985,LanTay1986}. L'idée est de considérer une version régularisée de la formulation mixte avec mutliplicateurs de Lagrange. Plus précisément, sous certaines conditions de régularité, cette formulation mixte avec inégalités peut se réécrire en une formulation mixte non-linéaire sans inégalités, faisant intervenir un paramètre de régularisation. On peut alors découpler le problème en résolvant avec un algorithme itératif de type point fixe, dont chacune des itérations consiste à résoudre une formulation semblable à celle obtenue par la méthode de pénalisation. Contrairement à la méthode de pénalisation, cette approche est consistante, c'est-à-dire que la solution de la nouvelle formulation coïncide avec la solution du problème d'origine. Elle ne dépend donc pas du choix du paramètre de régularisation.
Enfin, nous terminons par la \textit{méthode de Nitsche}, qui a connu un regain de popularité avec ses applications récentes aux problèmes de contact, voir la série de travaux \cite{chouly2013nitsche,Ren2013,chouly2014adaptation,ChoHilRen2015}. Du point de vue de la formulation, cette méthode peut être vue comme une généralisation de la méthode de lagrangien augmenté. En effet, sous les mêmes conditions de régularité, elle propose de réécrire la formulation mixte avec multiplicateurs de Lagrange en une formulation non-linéaire non-mixte régularisée dans laquelle apparaissent le déplacement et les contraintes sur la zone de contact (écrites comme une fonction du déplacement). On se retrouve donc avec une formulation consistante, faisant intervenir des paramètres de régularisation, mais avec pour seule inconnue le déplacement.

Dans le contexte de l'optimisation de formes, le choix du modèle mécanique et de la méthode de résolution est crucial, car il va déterminer la formulation mathématique à partir de laquelle on va travailler pour exprimer les conditions d'optimalité.  Nous choisissons dans ce travail le modèle de l'élasticité linéaire en petites déformations, et considérons des phénomènes de contact sans frottement ou avec frottement de Tresca. Ce choix est motivé par le caractère bien posé des formulations associées. En effet, les conditions pour que le problème mécanique ait de bonnes propriétés sont trop restrictives dans le cas des solides en grandes déformations ou encore dans celui du contact avec frottement de Coulomb, ce qui limite considérablement les possibilités d'obtenir des résultats théoriques intéressants du point de vue de l'optimisation de formes. En ce qui concerne la méthode de résolution, nous privilégions les méthodes de pénalisation et régularisation (lagrangien augmenté), pour leur facilité d'implémentation et leur robustesse. C'est d'ailleurs pour ces raisons qu'elles sont les plus fréquemment utilisées dans les applications numériques. Soulignons enfin que ces méthodes donnent des formulations variationnelles sans inégalités.

Compte tenu de la formulation mathématique des problèmes de contact, on remarque que le problème d'optimisation de formes en présence de contact peut être vu comme un problème de contrôle optimal d'une IV, il appartient donc à la classe plus large des problèmes de \textit{programmation mathématique avec contrainte d'équilibre} (MPEC en anglais, pour \textit{Mathematical Programming with Equilibrium Constraints}), voir par exemple l'ouvrage \cite{outrata2013nonsmooth}. La difficulté commune à ces problèmes est d'arriver à exprimer des conditions d'optimalité alors que précisément l'opérateur qui au paramètre de contrôle associe la solution de l'IV n'est pas différentiable, puisqu'il est semblable à un opérateur de projection. Les premiers travaux apportant des pistes pour contourner cette difficulté sont ceux de Zarantello \cite{zarantonello1971projections} et Mignot \cite{mignot1976controle}, dans lesquels on définit une notion plus faible de différentiabilité, à savoir la \textit{dérivabilité conique}, qui est un cas particulier de dérivabilité directionnelle. On peut ensuite utiliser cette notion pour obtenir des conditions d'optimalité plus générales, voir par exemple \cite{mignot1984optimal}, parfois appelées \textit{strong stationarity conditions}, cf \cite{hintermuller2009mathematical}. Ces méthodes ont plus tard été reprises en optimisation de formes de certaines IV (problème de l'obstacle, contact glissant, contact avec frottement de Tresca), voir \cite{sokolowski1992introduction}. Dans le même esprit, mais en travaillant avec des sous-gradients (de Clarke \cite{clarke1990optimization} ou de Mordukhovich \cite{mordukhovich2006variationalI,mordukhovich2006variationalII}) plutôt qu'avec des dérivées coniques, on obtient des conditions d'optimalité différentes, appelées respectivement \textit{C-stationarity} ou \textit{M-stationarity conditions}. Cette autre approche, assez utilisée en contrôle optimal du problème de l'obstacle en dimension infinie \cite{wachsmuth2014strong,wachsmuth2016towards,rauls2018generalized}, semble difficile à généraliser au cas du problème de contact, que ce soit dans le contexte du contrôle d'optimal classique ou de l'optimisation de formes. En dimension finie en revanche, nous citons la série de papiers \cite{beremlijski2002shape,beremlijski2009shape,beremlijski2014shape} qui se concentre sur l'optimisation de formes du problème de contact avec frottement de Coulomb selon le paradigme \textit{discrétiser-puis-optimiser}. Bien que ces approches par dérivées généralisées soient très élégantes, l'utilisation en pratique des conditions d'optimalité obtenues peut s'avérer difficile, surtout en dimension infinie, car il est difficile de caractériser les sous-différentiels intervenant dans les équations, voir par exemple \cite{jaruvsek2007sharp}. Une approche complètement différente et plus répandue de contourner le problème de non-différentiabilité consiste à régulariser l'IV considérée pour rendre l'opérateur solution différentiable, par exemple en remplaçant l'IV par une équation variationnelle par la méthode de pénalisation, puis en régularisant les fonctions non-lisses qui apparaissent dans la formulation pénalisée. Pour le cas des problèmes de contact, nous mentionnons dans cette direction \cite{amassad2002optimal} dans le contexte du contrôle optimal classique, et dans le contexte de l'optimisation de formes, nous citons entre autres \cite{kim2000meshless,maury2017shape} selon l'approche \textit{optimiser-puis-discrétiser} et \cite{haslinger1986shape} selon l'approche \textit{discrétiser-puis-optimiser}. Contrairement à la précédente, cette approche permet d'avoir des dérivées classiques qu'on sait bien utiliser en pratique (introduction d'un état adjoint, algorithme de type gradient, etc). Par contre, la forme optimale obtenue dépendra nécessairement des différents paramètres (pénalisation et régularisation) introduits.

Nous proposons dans cette thèse une approche intermédiaire nouvelle, entre l'approche « directe » de \cite{sokolowski1992introduction} où on considère l'IV d'origine et l'approche « complètement régularisée » de \cite{kim2000meshless,maury2017shape} où on considère une version pénalisée puis à nouveau régularisée, assez éloignée de la formulation d'origine. En effet, les formulations du problème de contact auxquelles nous nous intéressons (pénalisée et lagrangien augmenté) sont certes plus régulières que l'IV d'origine, mais elles ne sont pas pour autant différentiables. Plutôt que de lisser les fonctions non-lisses comme dans l'approche « complètement régularisée », nous préférons écrire les conditions d'optimalité associées à ces formulations, sans étape de régularisation supplémentaire, en travaillant avec des dérivées directionnelles. 

On retiendra donc que le travail présenté dans ce manuscrit traite du problème d'optimisation de formes pour des solides élastiques en contact avec frottement de Tresca. Les formulations considérées pour le problème mécanique sont les formulation pénalisée et lagrangien augmenté. De plus, dans le but d'utiliser un algorithme de type gradient pour la résolution, on cherche à écrire, en dimension infinie, les conditions d'optimalité associées à ce problème d'optimisation, en travaillant à partir de dérivées directionnelles. La présentation qui suit est structurée en quatre parties.

La première partie introduit les différents concepts et notions qui sont à la base de notre méthodologie, que ce soit dans le domaine de l'optimisation de formes ou dans celui de la mécanique du contact. Dans le premier chapitre, nous proposons une introduction à l'optimisation de formes du point de vue spécifique de l'approche par level-set. En particulier, nous expliquons comment représenter une forme à l'aide d'une fonction level-set, et en quoi une telle représentation est adaptée à l'optimisation topologique. Nous donnons ensuite les définitions et les principaux résultats de l'analyse de sensibilité par rapport à la forme, en s'attardant notamment sur la notion de \textit{dérivée de forme}, qui présente un intérêt à la fois théorique et pratique dans notre contexte. Puis nous présentons dans les grandes lignes l'algorithme d'optimisation de formes que nous avons développé, ce qui permet de mettre en évidence la manière dont l'analyse de sensibilité par rapport à la forme peut être couplée à une méthode de level-set pour concevoir un algorithme de type gradient. Le deuxième chapitre est consacré à la mécanique du contact. Nous commençons par y présenter la formulation d'origine des problèmes de contact en élasticité linéaire, sans frottement puis avec frottement de Tresca. À partir de là, nous dérivons la formulation mixte équivalente en suivant l'approche classique par multiplicateurs de Lagrange. Puis nous détaillons les étapes de régularisation qui permettent d'obtenir les deux formulations qui nous intéressent ici, à savoir la formulation pénalisée et la formulation lagrangien augmenté. Enfin, nous exposons la méthode de Newton semi-lisse en dimension infinie, qui nous permettra de résoudre numériquement les formulations régularisées considérées.

La deuxième partie se concentre sur l'obtention de dérivées de formes pour la formulation pénalisée. Le troisième chapitre prend la forme d'un article, il présente les résultats théoriques obtenus pour cette formulation et les illustre par un exemple numérique 2d. Comme nous cherchons à optimiser toute la frontière (incluant la zone de contact), nous sommes confrontés au problème de non-différentiabilité. Après avoir introduit la formulation pénalisée, nous proposons donc une méthode s'appuyant sur des dérivées directionnelles pour déduire une expression des dérivées de formes, même si le problème n'est pas Fréchet-différentiable. Puis nous exhibons des conditions suffisantes pour que la solution soit dérivable par rapport à la forme. Sous ces conditions, nous montrons que les méthodes habituelles s'appliquent et que l'algorithme de type gradient proposé est valide. Le quatrième chapitre est un complément du troisième. Il fait le lien entre notre approche intermédiaire et l'approche complètement régularisée proposée dans \cite{maury2017shape}. En particulier, on montre qu'à partir des expressions des dérivées de formes dans \cite{maury2017shape}, on retombe sur les expressions du chapitre précédent si on fait tendre le paramètre de l'étape supplémentaire de régularisation vers $+\infty$.

La troisième partie s'intéresse quant à elle aux dérivées de formes pour la formulation lagrangien augmenté. Elle se compose uniquement du chapitre \ref{chap:3.1}, qui prend lui aussi la forme d'un article. L'approche suivie est semblable à celle du chapitre \ref{chap:2.1}, avec les adaptations nécessaires pour que tout reste valide pour cette nouvelle formulation, mais nous nous limitons ici au cas du contact glissant. Comme, en pratique, la méthode de lagrangien augmenté se présente sous la forme d'un algorithme itératif, la solution obtenue sera celle correspondant à la dernière itération. On cherche donc à exprimer des dérivées de formes pour la formulation résolue à chaque itération plutôt que la formulation d'origine. Ensuite, sous certaines hypothèses très spécifiques, on montre que la dérivée de forme de la solution obtenue à l'itération $k$ converge vers la dérivée de forme conique de la solution du problème de contact d'origine lorsque $k\to+\infty$.

Enfin, la quatrième partie ne contient que le chapitre \ref{chap:4.1}, qui traite des aspects numériques de l'algorithme d'optimisation de formes proposé. Dans une première section, nous revenons en détails sur les différentes étapes de l'algorithme, et plus particulièrement sur la technique de découpage de maillage mise en \oe uvre, qui permet de disposer d'un domaine de calcul maillé à chaque itération. Dans une deuxième section, nous présentons divers exemples numériques. Afin de valider notre méthode, nous essayons d'abord de résoudre des cas académiques classiques en élasticité linéaire sans contact, en deux et trois dimensions. Ceci permet dans le même temps de comparer notre méthode de découpage avec la méthode de remaillage plus sophistiquée proposée dans \cite{dapogny2013shape}. Nous nous intéressons finalement à des cas-tests en élasticité avec contact. En nous inspirant d'exemples dans la littérature, à la fois côté contact et côté optimisation de formes, nous proposons un cas-test (2d et 3d) qui nous semble bien adapté pour tester les dérivées de forme sur la zone de contact. Nous mettons ainsi en évidence les différences entre les formes optimales obtenues pour chaque formulation (pénalisée ou lagrangien augmenté).             
\part{Mise en contexte}

\chapter*{Notations}         
\label{chap:notations}       
\phantomsection\addcontentsline{toc}{chapter}{\nameref{chap:notations}} 

\paragraph{Notations générales}
$\:$

\begin{longtable}{p{2.35cm}p{14cm}}
$\complement A$ & complémentaire de l'ensemble $A$ \\
$\overline{A}$ & adhérence de l'ensemble $A$ \\
$\chi_A$ & fonction caractéristique de l'ensemble $A$\\
$I_A$ & fonction indicatrice de l'ensemble $A$\\
$|A|$ & mesure de l'ensemble $A$ \\
$|z|$ & norme euclidienne du vecteur $z$ \\
$d_A$ & distance à l'ensemble $A\subset\mathbb{R}^d$, $d_A(x):=\inf \{|x-y| \: | \: y\in A \}$ \\
$\Id$ & application identité \\
$p.p.$ ou $a.e.$ & presque partout ou \textit{almost everywhere} (par rapport à la mesure de Lebesgue)\\
$q.p.$ ou $q.e.$ & quasi partout ou \textit{quasi everywhere} \\
$X$, $Y$, $Z$ & espaces de Banach \\
$X^*$ & dual de l'espace $X$ \\
$\norml \cdot \normr_X$ & norme sur $X$\\
$\prodD{\cdot,\cdot}{X^*,X}$ & produit de dualité entre $X$ et $X^*$ \\
$\mathcal{L}(X,Y)$ & espace des applications linéaires de $X$ dans $Y$ \\
$\mathcal{L}_c(X,Y)$ & espace des applications linéaires continues de $X$ dans $Y$ \\
$H$ & espace de Hilbert \\
$\prodL2{\cdot,\cdot}{H}$ & produit scalaire sur $H$ \\
$\Proj_K$ & projection sur l'ensemble convexe fermé $K$ \\
$\Omega \subset \mathbb{R}^d$ & ensemble ouvert dans $\mathbb{R}^d$, où $d=2$ ou $3$ \\
$\partial\Omega$ & frontière de $\Omega$ \\
$\normalInt$ & normale unitaire sortante à $\Omega$ \\
$\kappa$ & courbure moyenne le long de $\partial\Omega$ \\
$\divv_\Gamma$ & divergence tangentielle \\ 
$\frac{\partial}{\partial\normalInt}$ ou $\partial_{\normalInt}$ & dérivée normale \\
$u(\Omega)$ & fonction scalaire définie sur $\Omega$ \\
$\uu(\Omega)$ & fonction vectorielle définie sur $\Omega$ \\
\end{longtable}

\paragraph{Espaces de fonctions}
$\:$

\begin{longtable}{p{2.35cm}p{14cm}}
$\pazocal{C}^k(\Omega)$ & espace des fonctions scalaires dont les dérivées d'ordre $\leq k$ sont continues \\
$\Cc^k(\Omega)$ & espace des fonctions vectorielles dont les dérivées d'ordre $\leq k$ sont continues \\
$L^p(\Omega)$ & espace des fonctions scalaires $u$ telles que $|u|^p$ est intégrable sur $\Omega$ \\
$\Ll^p(\Omega)$ & espace des fonctions vectorielles $\uu$ telles que $|\uu|^p$ est intégrable sur $\Omega$ \\
$\prodL2{\cdot,\cdot}{\Omega}$ & produit scalaire $L^2(\Omega)$ ou $\Ll^2(\Omega)$ \\
$\norml \cdot \normr_{0,\Omega}$ & norme $L^2(\Omega)$ ou $\Ll^2(\Omega)$ \\
$W^{m,p}(\Omega)$ & espace des fonctions $\left\{ u\in L^p(\Omega) \ | \ D^\alpha u \in L^p(\Omega), \: \forall |\alpha|\leq m \right\}$ \\
$\Ww^{m,p}(\Omega)$ & espace des fonctions $\left\{ \uu\in \Ll^p(\Omega) \ | \ D^\alpha \uu \in \Ll^p(\Omega), \: \forall |\alpha|\leq m \right\}$ \\
$H^m(\Omega)$ & espace de Sobolev $W^{m,2}(\Omega)$ \\
$\Hh^m(\Omega)$ & espace de Sobolev $\Ww^{m,2}(\Omega)$ \\
$\norml \cdot \normr_{m,\Omega}$ & norme $H^m(\Omega)$ ou $\Hh^m(\Omega)$ \\
\end{longtable}

\paragraph{Notations relatives au problème de contact}
$\:$

\begin{longtable}{p{2.35cm}p{14cm}}
$\Omega$ & corps déformable\\
$\Omega_{rig}$ & corps rigide\\
$\Gamma_C$ & partie de $\partial\Omega$ où s'applique une condition de contact\\
$\Gamma_D$ & partie de $\partial\Omega$ où s'applique une condition de Dirichlet homogène\\
$\Gamma_N$ & partie de $\partial\Omega$ où s'applique une condition de Neumann non-homogène\\
$\Gamma$ & partie de $\partial\Omega$ libre de contrainte\\
$\normalInt$ & normale unitaire sortante à $\partial\Omega$ \\
$\normalExt$ & normale unitaire entrante à $\partial\Omega_{rig}$\\
$v_{\normalInt}$ & produit scalaire $v\cdot\normalInt$ entre le vecteur $v\in\mathbb{R}^d$ et $\normalInt$ \\
$v_{\normalExt}$ & produit scalaire $v\cdot\normalExt$ entre le vecteur $v\in\mathbb{R}^d$ et $\normalExt$ \\
$v_{\tanInt}$ & composante tangentielle $v - v_{\normalInt}\normalInt$ du vecteur $v\in\mathbb{R}^d$ au sens de $\normalInt$ \\
$v_{\tanExt}$ & composante tangentielle $v - v_{\normalExt}\normalExt$ du vecteur $v\in\mathbb{R}^d$ au sens de $\normalExt$ \\
$\gG_{\normalInt}$ & gap dans la direction de $\normalInt$: fonction distance orientée à $\partial\Omega$\\
$\gG_{\normalExt}$ & gap dans la direction de $\normalExt$: fonction distance orientée à $\partial\Omega_{rig}$\\
$\Xx$ & espace de fonctions $\Hh^1_{\Gamma_D}(\Omega):=\{ \vv\in \Hh^1(\Omega) \: | \: \vv = 0 \: \mbox{ p.p. sur } \Gamma_D\}$ \\
$\Lambda_{\normalExt}$ & opérateur de trace normale sur $\Gamma_C$, défini de $\Xx$ dans $H^\frac{1}{2}(\Gamma_C)$ \\
$\Lambda_{\tanExt}$ & opérateur de trace tangentielle sur $\Gamma_C$, défini de $\Xx$ dans $\Hh^\frac{1}{2}(\Gamma_C)$ \\
\end{longtable}

\paragraph{Notations d'analyse convexe}
$\:$

\begin{longtable}{p{2.35cm}p{14cm}}
$f^*$ & fonction convexe conjuguée de la fonction $f$\\
$\partial f$ & sous-différentiel de la fonction $f$ \\
\end{longtable}

\paragraph{Notations d'optimisation de formes}
$\:$

\begin{longtable}{p{2.35cm}p{14cm}}
$J(\Omega)$ & critère ou fonctionnelle de forme \\
$dJ(\Omega)[\thetaa]$ & dérivée de forme de $J$ dans la direction $\thetaa$ \\
$\dot{u}(\Omega)[\thetaa]$ & dérivée matérielle de $u$ en $\Omega$ dans la direction $\thetaa$\\
$du(\Omega)[\thetaa]$ &  dérivée de forme de $u$ en $\Omega$ dans la direction $\thetaa$\\
\end{longtable}

\chapter{Optimisation de formes par la méthode de level-set}
\label{chap:1.1}

\section*{Introduction}

Le problème d'optimisation de formes peut s'énoncer sous la forme générale suivante: trouver la forme qui minimise un critère $J$ donné parmi toutes les formes d'un ensemble admissible $\pazocal{U}_{ad}$. En termes de formulation mathématique, cela se traduit par:
\begin{equation}
    \inf_{\Omega \in \pazocal{U}_{ad}} J(\Omega) \:.
\end{equation}
Ce genre de problème suscite de l'intérêt depuis des siècles. Nous pouvons notamment citer le problème isopérimétrique (énoncé historiquement par la reine Didon dans la Grèce Antique), ou la recherche de surfaces minimales de façon plus générale, qui a occupé de nombreux scientifiques au dix-huitième siècle (Euler, Lagrange, Plateau, etc). 

Depuis quelques dizaines d'années, avec les avancées du calcul scientifique et les besoins grandissants de l'industrie dans ce domaine, la popularité de l'optimisation de formes est montée en flèche. Ce sujet de recherche a été énormément étudié, tant sur ses aspects théoriques que numériques et de nombreux logiciels commerciaux de simulation numérique proposent désormais des modules d'optimisation de formes. Les champs d'application sont divers, allant de la mécanique des fluides (conception d'avions ou de voitures) à l'électromagnétisme en passant par la mécanique des solides (conception de structures mécaniques, génie civil). Dans le cas d'applications en mécanique des solides, on parle en général d'optimisation de structures, ou \textit{structural optimization} en anglais. Dans cette thèse, nous nous concentrerons sur ce contexte spécifique. 

Le but de l'optimisation de structures est de trouver la forme qui optimise certaines propriétés mécaniques d'une structure soumise à un ensemble de contraintes données. Par exemple, on peut vouloir trouver la forme d'une pièce mécanique, soumise à un chargement donné, de sorte qu'elle soit la plus légère possible tout en étant la plus rigide possible. On peut reformuler ce problème de la façon suivante: trouver la forme qui minimise le volume de la pièce et sa compliance. Dans cet exemple simple, on voit que le critère $J$ à minimiser est défini comme l'intégrale sur le domaine $\Omega$ d'une fonction qui dépend de $\Omega$ et de $u(\Omega)$, la solution de l'équation d'état caractérisant l'équilibre mécanique. En conséquence, $J$ dépend de $\Omega$ à la fois directement (via le domaine d'intégration) et indirectement (via $u(\Omega)$). Au vu de cette remarque, dans notre contexte, nous pouvons spécifier le problème général d'optimisation de formes:
\begin{equation}
    \begin{array}{cl}
         \displaystyle \inf_{\Omega \in \pazocal{U}_{ad}} J(\Omega) =& \inf \ \pazocal{J}(\Omega,u(\Omega))  \\
         & \mbox{sujet à } \left\{
         \begin{array}{l}
              \Omega \in \pazocal{U}_{ad}\:, \\
              u(\Omega) \mbox{ solution de l'équation d'état . }
         \end{array}
         \right.
    \end{array}
    \label{StructOPT}
\end{equation}

Nous faisons maintenant un bref tour d'horizon des grandes approches proposées au cours des cinquante dernière années pour résoudre numériquement \eqref{StructOPT}. Ces approches peuvent être catégorisées en trois principales familles, en fonction du choix qui est fait pour représenter la forme.
\begin{itemize}
    \item \textit{Représentation explicite.} Ces méthodes définissent le bord du domaine à partir d'un nombre fini de paramètres. Cette paramétrisation peut prendre la forme d'un ensemble de points de contrôle (courbes/surfaces de Béziers, NURBS, voir \cite{braibant1984shape}) ou bien des coordonnées d'un ensemble de n\oe uds représentant le bord d'un objet maillé (voir \cite{haslinger1986shape} ou \cite{pironneau1982optimal}).
    \item \textit{Représentation implicite.} La forme est représentée par l'intermédiaire d'une fonction lisse définie sur un domaine de calcul plus grand. Parmi ces méthodes, nous citons les plus répandues: la méthode de level-set (voir \cite{allaire2004structural}) et la méthode de champ de phase (voir \cite{bourdin2003design}). Dans la première (qui sera détaillée dans la prochaine section), la frontière de la forme est définie comme l'ensemble des zéros d'une fonction scalaire, ce qui donne un modèle avec interface nette. À l'inverse, la deuxième propose un modèle d'interface diffuse, où la forme est implicitement définie via une fonction scalaire qui s'interprète comme la phase du matériau (solide-liquide-vide). Dans ces deux méthodes, on peut faire évoluer la frontière de la forme en résolvant une équation d'advection (Hamilton-Jacobi pour la levet-set et Allen-Cahn ou Cahn-Hilliard pour le champ de phase) pour la fonction auxiliaire. 
    \item \textit{Méthodes de densité.} L'idée de ces méthodes est d'agrandir l'ensemble admissible des formes. En particulier, on ne considère plus seulement les formes « classiques », mais aussi les formes issues de structures micro-perforées qui, lorsque la taille des perforations tend vers $0$, peuvent être considérées comme un mélange de vide et de matériau. Le cadre théorique associé à ce modèle est celui de l'homogénéisation (voir le livre \cite{allaire2012shape}). Dans le même esprit, nous mentionnons la méthode SIMP (Solid Isotropic Material with Penalisation, voir \cite{BenSig2003}) dans laquelle on modélise la structure à partir d'une densité de mélange vide-matériau (pas nécessairement issue d'une micro-structure sous-jacente) dans chaque élément d'un maillage de calcul donné.
\end{itemize}
Dans la plupart des cas, on résout numériquement \eqref{StructOPT} par des méthodes de descente de gradient, ce qui signifie qu'on doit calculer la sensibilité de $J$ par rapport à la forme. Pour les méthodes de type SIMP et les méthodes s'appuyant sur une représentation explicite de la forme, on peut faire un calcul de sensibilité « classique » et appliquer un algorithme d'optimisation (éventuellement sous contrainte) en dimension finie. Cependant, les méthodes pour lesquelles la forme est représentée implicitement nécessitent le calcul de cette sensibilité pour le modèle continu.

Dans ce travail, nous avons choisi de représenter les formes par la méthode de level-set, puis d'appliquer une descente de gradient qui s'appuie sur des dérivées de forme. Dans les sections suivantes, nous présentons plus en détails la méthode de level-set, puis nous introduisons les concepts relatifs à l'analyse de sensibilité par rapport à la forme en dimension inifinie.

\section{Représentation de la forme par une level-set}

La méthode de level-set a d'abord été introduite par Osher et Sethian \cite{OshSet1988} pour proposer une représentation implicite de la frontière d'un ouvert régulier $\Omega\subset \mathbb{R}^d$, où $d=2$ ou $3$. Plus précisément, $\Omega$ est associé à l'ensemble des valeurs strictement négatives d'une fonction auxiliaire régulière $\phi$ définie partout sur $\mathbb{R}^d$ (ou en pratique, sur un domaine de calcul $D$ suffisamment grand pour contenir toutes les formes admissibles). De la même manière, la frontière $\partial\Omega$ est définie comme la ligne/surface de niveau zéro de $\phi$. En d'autres termes, $\phi$ vérifie:
\begin{equation}
    \left\{ \
    \begin{array}{ll}
         \phi(x) < 0 & \mbox{ si } x\in \Omega\:,  \\
         \phi(x) = 0 & \mbox{ si } x\in \partial\Omega \:, \\
         \phi(x) > 0 & \mbox{ si } x\in D\setminus \overline{\Omega}\:.
    \end{array}
    \right.
    \label{DefLS}
\end{equation}
Un exemple typique de fonction $\phi$ est la fonction distance signée à la frontière $\mathrm{d}_{\partial\Omega}$. Nous renvoyons à \cite{DelZol2001} pour une étude détaillée autour de la fonction distance et de son rôle central en optimisation de formes. Rappelons que, si $\phi$ est de classe $\pazocal{C}^1$, le vecteur normal unitaire sortant à $\Omega$ en tout point $x\in\partial\Omega$ est donné par: 
\begin{equation}
    \normalInt(x) = \frac{\grad\phi(x)}{|\grad\phi(x)|} \:.
    \label{DefNormalLS}
\end{equation}
On peut noter que cette définition suggère un relèvement naturel du vecteur normal à tout $D$. Par ailleurs, si $\phi$ est de plus $\pazocal{C}^2$, alors on dispose également d'une formule explicite pour la courbure moyenne en tout point $x\in\partial\Omega$:
\begin{equation}
    \kappa(x) = \divv \left(\normalInt(x)\right) = \divv \left(\frac{\grad\phi(x)}{|\grad\phi(x)|}\right) \:.
    \label{DefCourbLS}
\end{equation}

Un des principaux atouts de cette méthode est qu'elle permet de gérer l'évolution d'un domaine, c'est-à-dire, le cas d'un domaine $\Omega(t)$ qui se déplace sur un intervalle de temps $[0,T]$. Il suffit pour cela de considérer une fonction $\phi$ qui dépend à la fois du temps et de l'espace, i.e.$\!$ pour tout $t\in[0,T]$, $\phi(t,\cdot)$ est la level-set associée à $\Omega(t)$. Supposons qu'on parte d'un domaine $\Omega(0)=\Omega_0$, qu'on peut représenter à l'aide de la fonction $\phi_0$. Dans ce contexte, l'évolution de $\Omega$ en fonction de $t$ se traduit par une équation aux dérivées partielles (EDP) sur $\phi$, associée d'une condition initiale.
Dans ce qui suit, nous essayons de donner l'intuition de la marche à suivre (formellement) pour obtenir cette EDP, en fermant les yeux sur les (nombreuses) difficultés techniques rencontrées.

Supposons que le mouvement de $\Omega(t)$ est régi par un champ de vecteurs régulier $\thetaa:[0,T]\times \mathbb{R}^d\to \mathbb{R}^d$. Soit $\pazocal{O}\subset[0,T]\times \mathbb{R}^d$ un voisinage ouvert de $(t,\partial\Omega(t))$ pour $t$ dans un petit intervalle $I_t$. Si $\Omega(t)$ évolue de façon régulière, alors pour tout $t_0\in I_t$ et $x_0\in\partial\Omega(t_0)$, il existe un intervalle $(t_0-\delta,t_0+\delta)$ et une courbe $t\mapsto x(t)$ définie sur cet intervalle de sorte que $x(t)\in\partial\Omega(t)$ pour tout $t\in(t_0-\delta,t_0+\delta)$ et $x(t_0)=x_0$. Par définition de la level-set, ceci entraîne que:
$$
    \phi(t,x(t))=0\:, \quad \forall t \in (t_0-\delta,t_0+\delta)\:.
$$
Comme l'expression précédente est valide pour tout $(t_0,x_0)\in \pazocal{O}$, si on suppose que tout est suffisamment régulier, on peut la dériver en $t=0$ et utiliser la règle de dérivation composée pour obtenir:
\begin{equation}
    \frac{\partial\phi}{\partial t}(t,x) + \thetaa(t,x)\cdot \grad\phi(t,x) = 0\:,
    \label{LSAdvEquation0}
\end{equation}
puisque la vitesse le long de la courbe $x(t)$ est précisément $x'(t)=\thetaa(t,x(t))$. De plus, si le champ de vecteurs $\thetaa$ est dirigé dans la direction de la normale $\normalInt$, disons $\thetaa=\theta \normalInt$, alors on peut réécrire \eqref{LSAdvEquation0} en:
\begin{equation}
    \frac{\partial\phi}{\partial t}(t,x) + \theta(t,x) \cdot|\grad\phi(t,x)| = 0\:. 
    \label{HJEquation0}
\end{equation}
Cette discussion motive donc l'introduction de deux équations aux dérivées partielles posées sur $[0,T]\times \mathbb{R}^d$, munies de conditions initiales: l'\textit{équation d'advection de la level-set} \eqref{LSAdvEquation} et une \textit{équation de Hamilton-Jacobi} \eqref{HJEquation}: 
\begin{equation}
    \begin{aligned}
    &\frac{\partial\phi}{\partial t}(t,x) + \thetaa(t,x)\cdot \grad\phi(t,x) = 0\:, \\
    &\phi(0,x) = \phi_0(x)\:.
    \label{LSAdvEquation}
    \end{aligned}
\end{equation}\begin{equation}
    \begin{aligned}
    &\frac{\partial\phi}{\partial t}(t,x) + \theta(t,x) \cdot|\grad\phi(t,x)| = 0\:, \\
    &\phi(0,x) = \phi_0(x)\:,
    \label{HJEquation}
    \end{aligned}
\end{equation}
Comme mentionné plus haut, l'étude de ces équations est loin d'être une tâche facile. On sait néanmoins, voir par exemple \cite{rauch2012hyperbolic}, que pour $T$ suffisament petit, ces équations admettent une unique solution classique (régulière). Pour obtenir l'existence de solutions globales (en temps), il est nécessaire de travailler avec la notion plus générale de \textit{solution de viscosité}, nous renvoyons le lecteur à l'article de référence \cite{crandall1983viscosity}. Pour la résolution numérique, nous citons encore \cite{OshSet1988}.

En pratique, on travaille avec une discrétisation de $\phi$ sur un maillage, et on résout \eqref{LSAdvEquation} ou \eqref{HJEquation} numériquement (en général par différences finies). À un $\phi$ discret donné, on associe un domaine $\Omega$ sur lequel on cherche à résoudre un problème mécanique. Cette représentation implicite ne permet a priori pas d'avoir un maillage du bord du domaine $\partial\Omega$ car $\phi$ peut prendre la valeur $0$ à l'intérieur de composantes du maillage. La question qui se pose est donc celle de l'application des conditions aux limites pour le problème mécanique. Il existe deux principales approches permettant d'y répondre.
\begin{itemize}
    \item Dans l'approche \textit{sans reconstruction de maillage}, on applique les conditions aux limites faiblement à l'aide de méthodes qui ne requièrent pas de maillage de $\partial\Omega$, voir par exemple \cite{BelXiaPar2003}. L'avantage principal est d'économiser le remaillage du domaine $D$ pour limiter le coût de calcul.
    \item L'approche par \textit{découpage} ou \textit{remaillage} adopte le point de vue opposé: le but est ici de résoudre le plus précisément possible la mécanique en appliquant les conditions aux limites de façon usuelle, voir par exemple \cite{allaire2013mesh}. Ceci nécessite évidemment de disposer d'un maillage de $\partial\Omega$, qu'on obtient à partir de $\phi$, par exemple en découpant le maillage de $D$ autour de la ligne ou surface $\{ \phi = 0\}$. 
\end{itemize}
Puisque nous nous concentrons dans ce travail sur les problèmes de contact, nous choisissons l'approche par découpage pour privilégier la précision dans l'application des conditions aux limites, notamment sur la zone de contact.

Dans le contexte de l'optimisation de formes, la méthode de level-set est une alternative très intéressante qui permet d'avoir une représentation continue de la forme, et qui donc ne dépend pas d'un ensemble de n\oe uds ou de points de contrôles définis a priori. L'idée d'utiliser cette approche a d'abord été introduite dans \cite{allaire2004structural} pour l'optimisation de structures. Depuis, de nombreux auteurs s'en sont inspirés, dans beaucoup de domaines d'applications différents. Cependant, pour l'appliquer, il faut pouvoir trouver des directions de descente $\thetaa$ pour le modèle continu. Ceci soulève la question de l'analyse de sensibilité du critère $J$ par rapport à la forme, à laquelle nous nous intéressons dans la prochaine section.

\begin{rmrk}
	Évidemment, les régularités de $\Omega$, $\phi$, $\phi_0$ et $\thetaa$ (ou $\theta$)  sont liées. Dans le cas où on considère un ouvert $\Omega$ de classe $\pazocal{C}^1$, on choisira, comme dans \cite{allaire2016second}, $\thetaa$ (ou $\theta$)  et $\phi_0$ de classe $\pazocal{C}^1$ de sorte qu'il existe une unique solution $\phi\in \pazocal{C}^1([0,T]\times \mathbb{R}^d,\mathbb{R})$ à \eqref{LSAdvEquation} (ou \eqref{HJEquation}), pour $T$ suffisament petit.
\end{rmrk}

\begin{rmrk}
    Au cours de la résolution de l'équation d'évolution de $\phi$, il est tout à fait possible que la topologie de la forme représentée par $\{ \phi<0\}$ soit modifiée à cause de la diffusion numérique. Du point de vue de l'optimisation de formes, cela signifie que juste à partir d'un champ de vecteur $\thetaa$ donné par l'analyse de sensibilité par rapport à la frontière de la forme (optimisation géométrique), on peut générer des changements de topologie. On ne peut pas pour autant parler de méthode d'optimisation topologique, car du point de vue théorique, rien ne garantit que le changement de topologie observé en advectant la level-set fasse diminuer la valeur de la fonctionnelle à minimiser. La bonne notion pour s'intéresser à l'analyse de sensibilité par rapport à la topologie est la \textit{dérivée topologique}, voir \cite{NovSok2013}. 
\end{rmrk}

\section{Analyse de sensibilité par rapport à la forme}

\subsection{Différentiabilité dans des espaces de Banach}

Avant de parler de dérivabilité par rapport à la forme, nous rappelons quelques définitions de base de calcul différentiel dans des espaces de Banach. Dans cette section, $X$, $Y$ et $Z$ sont des espaces de Banach, $x\in X$, $t$ est un réel strictement positif et $F$ est une foncion de $X$ dans $Y$.

\begin{dfntn}
	On dit que $F$ est \textit{Gateaux semi-différentiable} en $x$ dans la direction de $h\in X$ si la quantité suivante existe
	$$
		dF(x;h):= \lim_{t\searrow 0} \frac{F(x+th)-F(x)}{t}\:.
	$$
	On appelle alors \textit{dérivée directionnelle} ou \textit{Gateaux semi-dérivée} de $F$ en $x$ dans la direction de $h$ cette quantité.
\end{dfntn}

\begin{dfntn}
	On dit que $F$ est \textit{Hadamard semi-différentiable} en $x$ dans la direction de $h\in X$ si la quantité suivante existe
	$$
		d_HF(x;h):= \underset{\tilde{h}\to h}{\underset{t\searrow 0,}{\lim}} \frac{F(x+t\tilde{h})-F(x)}{t}\:.
	$$
	On appelle alors \textit{Hadamard semi-dérivée} de $F$ en $x$ dans la direction de $h$ cette quantité.
\end{dfntn}

\begin{dfntn}
	On dit que $F$ est \textit{Gateaux différentiable} en $x$ si $F$ admet des dérivées directionnelles en $x$ dans toutes les directions $h\in X$ et si de plus l'application $h\mapsto dF(x;h)$ est linéaire continue de $X$ dans $Y$. 
\end{dfntn}

\begin{dfntn}
	On dit que $F$ est \textit{Hadamard différentiable} en $x$ si $F$ admet des Hadamard semi-dérivées en $x$ dans toutes les directions $h\in X$ et si de plus l'application $h\mapsto d_HF(x;h)$ est linéaire continue de $X$ dans $Y$. 
\end{dfntn}

\begin{dfntn}
	On dit que $F$ est \textit{Fréchet différentiable} en $x$ s'il existe une application linéaire $L(x) \in \mathcal{L}(X,Y)$ telle que:
	$$
		\lim_{h\to 0} \: \frac{1}{\norml h\normr_X} \norml F(x+h) - F(x) - L(x)h \normr_Y = 0\:.
	$$
	L'application $L(x)$ est alors appelée la \textit{différentielle} de $F$ en $x$.
\end{dfntn}

\begin{dfntn}
	Soit un ouvert $U\subset X$. On dit que $F$ est \textit{Gateaux différentiable sur $U$} (respectivement \textit{Hadamard} ou \textit{Fréchet différentiable sur $U$}) si $F$ est Gateaux différentiable (respectivement Hadamard ou Fréchet différentiable) en tout point $x\in U$. 
\end{dfntn}

On peut classer ces notions de différentiabilité de la plus faible à la plus forte: 
$$
	\mbox{Fréchet } \:\Rightarrow \: \mbox{ Hadamard }\: \Rightarrow \:\mbox{ Gateaux }.
$$


\paragraph{Continuité et dérivation en chaîne.} Une des propriétés principales de la semi-différentiabilité de Hadamard est qu'elle permet d'assurer que la règle de dérivation en chaîne demeure valide. 
\begin{prpstn}
	Soient $F:X \to Y$ et une fonction $G:Y\to Z$. On suppose que $F$ est Gateaux semi-différentiable en $x$ dans la direction de $h\in X$ et que $G$ est Hadamard semi-différentiable en $F(x)$ dans la direction de $dF(x;h)$, alors $G\circ F$ est Gateaux semi-différentiable en $x$ dans la direction de $h$ et de plus, on a:
	$$
		d(G\circ F)(x;h) = d_H G\left( F(x); dF(x;v)\right) \:.
	$$ 
\end{prpstn}
Dans le cas où $G$ est Gateaux différentiable en $F(x)$ mais n'est pas Hadamard semi-différentiable en $F(x)$ dans la direction de $dF(x;h)$, alors le résultat n'est plus nécessairement vrai, voir \cite[Chapitre 3, Exemple 3.9]{delfour2012introduction} pour un contre-exemple.

\begin{prpstn}
	Si $d_HF(x;0)$ existe, alors $F$ est continue en $x$.
\end{prpstn}
Encore une fois, si la fonction $F$ est Gateaux différentiable en $0$ mais que $d_HF(x;0)$ n'existe pas, $F$ n'est pas nécessairement continue, voir \cite[Chapitre 3, Exemple 3.8]{delfour2012introduction}.

On retiendra donc que la différentiabilité au sens de Gateaux n'assure ni la continuité de la fonction, ni la validité de la règle de dérivation en chaîne.

Pour une présentation plus complète des différentes notions de différentiabilité et de leurs subtilités, ainsi que pour les preuves des résultats énoncés, nous renvoyons le lecteur à \cite[Chapitre 3]{delfour2012introduction} pour le cas de la dimension finie et \cite[Section 9.2]{DelZol2001} pour celui de la dimension infinie.

\subsection{Principe et notion de dérivée}

Nous présentons ici la méthode de variation des frontières de Hadamard, qui propose un cadre théorique pour étudier la variation des formes, et en particulier l'analyse de sensibilité par rapport à ces dernières.
Les idées fondatrices ont été introduites pour la première fois dans \cite{hadamard1908memoire}. La construction d'espaces métriques de formes a été faite pour la première fois en 1973 par Micheletti \cite{micheletti1972metrica} dans le cadre de problèmes de valeurs propres du laplacien. Dans le contexte de l'optimisation de formes, Murat et Simon présentent une construction similaire en 1975-1976 dans \cite{murat1975etude,murat1976controle}, où ils introduisent une notion de dérivabilité s'appuyant sur la \textit{méthode de perturbation de l'identité}. Parallèlement à ces travaux, en 1977 dans \cite{zolesio1977optimal} et en 1979 dans \cite{zolesio1979identification}, Zolésio propose une approche différente pour définir la notion de dérivabilité à partir de la \textit{méthode des vitesses} qui se veut plus générale. Nous faisons dans la suite une brève présentation de ces deux approches dans le cas le plus général de la déformation d'un domaine $\Omega$ à bord Lipschitz.

\paragraph{Méthode des vitesses.}
Soit $\tau>0$, et soit un champ de vitesse $\Vv:[0,\tau]\times\mathbb{R}^d\to \mathbb{R}^d$. On s'intéresse alors à l'équation différentielle, pour tout $X\in \mathbb{R}^d$,
$$
	\frac{d\xx}{dt}(t,X) = \Vv\left(t,\xx(t,X)\right), \: t\in [0,\tau], \:\: \mbox{ et } \xx(0,X)=X \:.
$$
À partir de la solution $\xx$ (lorsqu'elle existe), on peut générer une famille de transformations géométriques $\{ \Tt_t:\mathbb{R}^d\to \mathbb{R}^d \: | \: t\in [0,\tau] \}$ définies par:
$$
	\Tt_t(X):=\xx(t,X), \:\: \forall (t,X)\in [0,\tau]\times \mathbb{R}^d\:.
$$
On utilise ensuite ces transformations géométriques pour perturber le domaine $\Omega$ de la façon suivante:
$$
	\Omega(t):= \Tt_t(\Omega) = \{ \Tt_t(X) \: | \: X\in \Omega \}\:,
$$
où $\Omega(t)$ désigne le domaine perturbé. On sait que pour $\Vv \in \pazocal{C}^0\left( [0,\tau];\Ww^{1,\infty}(\mathbb{R}^d)\right)$, alors $\Tt_t$ est un difféomorphisme (quitte à prendre $\tau$ plus petit), voir \cite{DelZol2001}. Ainsi, si on se fixe un champ de vecteur $\thetaa\in \Ww^{1,\infty}(\mathbb{R}^d)$, on peut alors définir la dérivée dans la direction de $\thetaa$ comme on le fait classiquement par la méthode des vitesses dans des espaces de Banach.
\begin{dfntn}
    On définit la \textit{dérivée de forme} du critère $J$ en $\Omega$ dans la direction $\thetaa$ par:
    \begin{equation}
        d_HJ(\Omega)[\thetaa] := \underset{\Vv(0,\cdot)=\thetaa}{\underset{\Vv\in \pazocal{C}^0([0,\tau],\Ww^{1,\infty}),}{\underset{t\searrow 0,}{\lim}}} \frac{J\left( \Tt_t(\Omega) \right) - J(\Omega)}{t}\:.
    \label{DefDJSpeed}
    \end{equation}
\end{dfntn}

\begin{dfntn}
    On dit que $J$ est \textit{dérivable par rapport à la forme} en $\Omega$ si $d_HJ(\Omega)[\thetaa]$ existe pour n'importe quel $\thetaa\in \Ww^{1,\infty}(\mathbb{R}^d)$, et si de plus l'application $\thetaa \mapsto d_HJ(\Omega)[\thetaa]$ est linéaire continue.
\end{dfntn}
Cette notion est l'analogue de la différentiabilité de Hadamard puisque que la définition de la dérivéee dans la direction $\thetaa$ est indépendante du choix de $\Vv$ permettant d'approcher $\Omega$ dans cette direction.

\paragraph{Méthode de pertubartion de l'identité.}
Cette méthode constitue un cas particulier de la précédente. En effet, le principe consiste ici à perturber le domaine $\Omega$ dans la direction d'un champ de vecteurs $\thetaa\in \Ww^ {1,\infty}(\mathbb{R}^d)$ de la façon suivante:
$$
    \Omega(t) = (\Id+t\thetaa)\Omega\:,
$$
où $t>0$ et $\Omega(t)$ désigne le domaine perturbé dans la direction de $\thetaa$, voir la figure \ref{fig:DeformationOmega} dans le cas $t=1$. Ceci revient à la méthode des vitesses avec le choix spécifique du champ de vitesses autonome $V(t,x)=\thetaa(x)$, $\forall t\in [0,\tau]$. 

\begin{figure}
	\begin{center}
    	\includegraphics[width=0.4\textwidth]{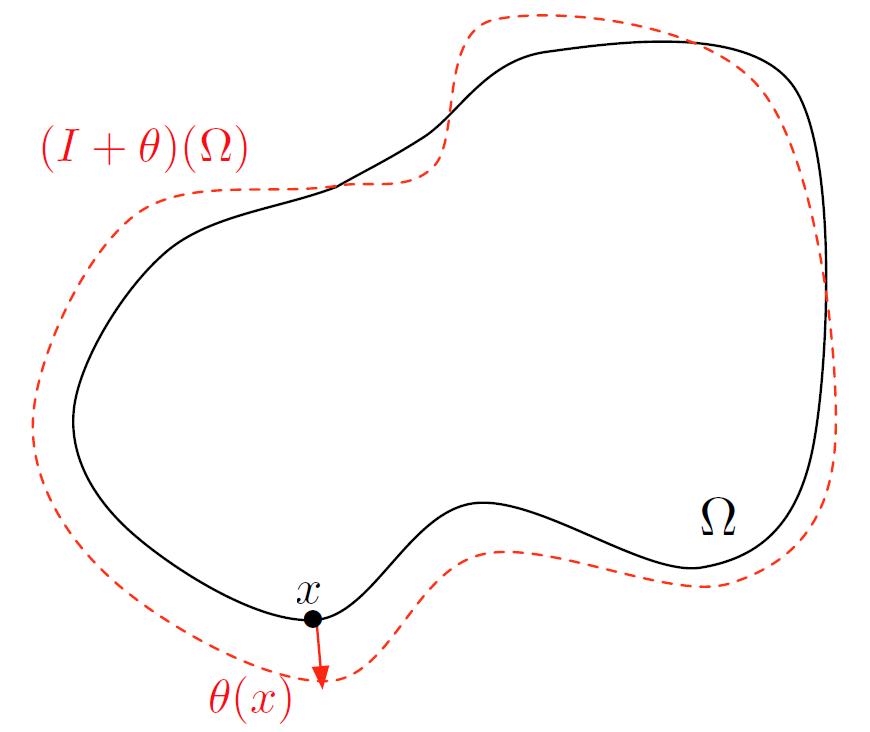}
    	\caption{Perturbation de $\Omega$ par $\thetaa$ (extrait de \cite{dapogny2013shape}).}
    	\label{fig:DeformationOmega}
  	\end{center}
\end{figure}  

On sait que pour $t$ suffisament petit, plus précisément pour $t$ tel que $t\norml \thetaa \normr_{\Ww^{1,\infty}}<1$, $(\Id+t\thetaa)$ est un difféomorphisme, voir par exemple \cite{henrot2006variation}. On peut donc définir la notion de différentiabilité par rapport à la forme par l'intermédiaire d'une notion classique de différentiabilité dans l'espace de Banach $\Ww^ {1,\infty}(\mathbb{R}^d)$. 
\begin{dfntn}
    On définit la \textit{dérivée de forme} du critère $J$ en $\Omega$ dans la direction $\thetaa$ par:
    \begin{equation}
        dJ(\Omega)[\thetaa] := \lim_{t\searrow 0} \frac{J\left( (\Id+t\thetaa)\Omega \right) - J(\Omega)}{t}\:.
    \label{DefDJ}
    \end{equation}
\end{dfntn}

\begin{dfntn}
    On dit que $J$ est \textit{dérivable par rapport à la forme} en $\Omega$ si $dJ(\Omega)[\thetaa]$ existe pour n'importe quel $\thetaa\in \Ww^{1,\infty}(\mathbb{R}^d)$, et si de plus l'application $\thetaa \mapsto dJ(\Omega)[\thetaa]$ est linéaire continue.
    \label{DefShapeDer}
\end{dfntn}
Cette définition correspond quant à elle à l'analogue de la différentiabilité de Gateaux.

\begin{rmrk}
    Soit $k\geq 1$. Les définitions précédentes tiennent encore si on prend des champs de vecteurs $\thetaa \in \Cc^ k_b(\mathbb{R}^d):=\Cc^k(\mathbb{R}^d)\cap \Ww^ {k,\infty}(\mathbb{R}^d)$, muni de la norme de $\Ww^{k,\infty}$. 
	On a de plus dans ces cas-là que $(\Id+t\thetaa)$ est un $\pazocal{C}^k$-difféomorphisme.
\end{rmrk}

La définition \ref{DefShapeDer} peut se reformuler en: $J$ est dérivable par rapport à la forme si l'application $\thetaa\mapsto J((\Id+\thetaa)\Omega)$ est Gateaux différentiable en $0$ dans $\Ww^{1,\infty}(\mathbb{R}^d)$. Cette observation suggère de proposer une troisième définition de dérivabilité par rapport à la forme qui est l'analogue de la différentiabilité de Fréchet.

\begin{dfntn}
    On dit que $J$ est \textit{Fréchet dérivable par rapport à la forme} s'il existe une application linéaire $\gradd J(\Omega): \Ww^ {1,\infty}(\mathbb{R}^d) \to \mathbb{R}$ telle que:
    $$
    	\lim_{\norml\thetaa \normr_{1,\infty}\to 0} \frac{1}{\norml\thetaa \normr_{1,\infty}}{|J((\Id+\thetaa)\Omega)-J(\Omega)-\prodD{\gradd J(\Omega),\thetaa}{} |} = 0\:.
    $$
\end{dfntn}

\begin{prpstn}
	Si $J$ est Fréchet dérivable par rapport à la forme, alors elle est Hadamard et Gateaux dérivable par rapport à la forme, et de plus, pour tout $\thetaa\in\Ww^{1,\infty}(\mathbb{R}^d)$:
	$$
		\prodD{\gradd J(\Omega),\thetaa}{} = d_HJ(\Omega)[\thetaa] = dJ(\Omega)[\thetaa]\:.
	$$  
\end{prpstn}

\begin{rmrk}
	Nous mentionnons également la méthode de dérivation proposée dans \cite{allaire2016second} dans le cas où la famille de domaines $\Omega(t)$ est définie par une fonction level-set $\phi(t,x)$ solution d'une équation de type \eqref{HJEquation}. En particulier, les auteurs montrent que lorsque $\Omega$ est de classe $\pazocal{C}^2$ et que la vitesse normale $\theta \in \pazocal{C}^1(\mathbb{R}^d)$, alors leur méthode donne une définition de la dérivée de forme (d'ordre 1) qui coïncide avec celles de Fréchet et de Hadamard.
\end{rmrk}

\begin{rmrk}
	Dans cette thèse, nous choisissons d'utiliser la définition qui est l'analogue de la différentiabilité de Gateaux. Néanmoins, il arrivera que nous démontrions la différentiabilité de Fréchet pour en déduire celle de Gateaux. Nous ne rencontrerons pas de problème pour la continuité puisque dans chaque cas où nous étudions la Gateaux différentiabilité d'une fonction, nous montrons au préalable sa continuité. Quant à la validité de la règle de dérivation en chaîne, nous nous en sortons car nous ne composerons les fonctions Gateaux différentiables qu'avec des fonctions au moins Hadamard différentiables.
\end{rmrk}

\begin{exmpl}
    Soient $\Omega$ localement Lipschitz, $J(\Omega)= |\Omega|$ et $\thetaa\in \Ww^{1,\infty}(\mathbb{R}^d)$, alors on a:
    $$
        J(\Omega(t)) = \int_{\Omega(t)}\!1 = \int_\Omega \det(I+t\grad\thetaa) \:,
    $$
    $$
        \mbox{d'où: } \qquad dJ(\Omega)[\thetaa] = \lim_{t\to 0} \frac{1}{t} \int_\Omega \left( \det(I+t\grad\thetaa)-1 \right) = \int_\Omega \divv \thetaa\: = \int_{\partial\Omega} \thetaa\cdot\normalInt.
    $$
    Il est clair que l'expression précédente est linéaire continue en $\thetaa$ dans les bons espaces, ce qui prouve que le volume est dérivable par rapport à la forme. On remarque que sa dérivée de forme ne dépend que de la composante de $\thetaa$ dans la direction de la normale $\normalInt$.
\end{exmpl}

Le fait que $dJ$ ne dépende que de $\thetaa\cdot\normalInt$ n'est pas spécifique à l'exemple précédent. En effet, pour des ouverts $\Omega$ plus réguliers, il existe un résultat très général: le théorème de structure de Hadamard-Zolésio. Bien que ce résultat ait été formalisé par Zolésio en 1979 dans sa thèse \cite{zolesio1979identification}, Dervieux et Palmerio avaient déjà commencé à l'étudier en 1976 \cite{dervieux1975formule}. Nous énonçons ici une version plus récente tirée de \cite[Théorème 2.27]{sokolowski1992introduction}.

\begin{thrm}
    Soient $D$ un ouvert de $\mathbb{R}^d$, $k\geq 1$, et $J$ un critère défini sur l'ensemble
    $$
        \pazocal{O}_k := \{\:  \Omega \subset D \ | \ \Omega \mbox{ ouvert et de classe } \pazocal{C}^k \:\}\:.
    $$
    Si $J$ est dérivable par rapport à la forme en tout $\Omega\in\pazocal{O}_k$ relativement à des perturbations $\thetaa\in\Cc^ k_b(\mathbb{R}^d)$, alors il existe une distribution scalaire $\mathfrak{g}\in \pazocal{D}^{-k}(\partial\Omega)$ d'ordre $-k$ telle que: pour tout $\Omega\in\pazocal{O}_k$, pour tout $\thetaa\in\Cc^k_b(\mathbb{R}^d)$,
    $$
        dJ(\Omega)[\thetaa] = \prodD{ \mathfrak{g}, \thetaa\cdot\normalInt}{\pazocal{D}^{-k}(\partial\Omega),\pazocal{D}^{k}(\partial\Omega)}\:.
    $$
    \label{ThmStruct}
\end{thrm}

\subsection{Optimisation de formes sous une contrainte d'équilibre}

Nous nous intéressons maintenant au cas particulier des problèmes d'optimisation de formes de type \eqref{StructOPT}. Comme mentionné précédemment, pour de tels problèmes, $J$ dépend de $\Omega$ à la fois directement et par l'intermédiaire d'une fonction $u(\Omega)$, solution d'une équation d'état posée sur $\Omega$. En particulier, $u(\Omega)$ n'est définie que sur $\Omega$, et elle vit dans un certain espace de Sobolev noté $W(\Omega)$. En perturbant le domaine dans une direction $\thetaa\in\Ww^{1,\infty}(\mathbb{R}^d)$ donnée,  on obtient $\Omega(t)$, et $u(\Omega(t))$, la solution de l'équation d'état  posée sur ce nouveau domaine. Pour étudier la dérivabilité de $u$ par rapport à la forme, il convient de s'intéresser au quotient dans \eqref{DefDJ}. Or, comme $u(\Omega(t))$ et $u(\Omega)$ ne sont pas définies au même endroit, ce quotient n'a de sens que pour des $x\in \Omega \cap \Omega(t)$. Il semble donc plus naturel de travailler avec la fonction composée $u(\Omega(t))\circ (\Id+t\thetaa)$ qui elle est bien définie sur le domaine de référence.

\begin{dfntn}
    La \textit{dérivée matérielle} ou \textit{dérivée lagrangienne} de $u$ en $\Omega$ dans la direction $\thetaa$ est la fonction de $W(\Omega)$ définie par:
    $$
        \dot{u}(\Omega)[\thetaa]:= \lim_{t\searrow 0} \frac{u(\Omega(t))\circ (\Id+t\thetaa)-u(\Omega)}{t}\:.
    $$
    S'il s'agit d'une limite faible (respectivement forte) dans $W(\Omega)$, on parle de dérivée matérielle faible (respectivement forte).
\end{dfntn}

\begin{dfntn}
    Si $\dot{u}(\Omega)[\thetaa]$ existe dans $W(\Omega)$ et si de plus $\grad u \thetaa \in W(\Omega)$, alors on peut définir la \textit{dérivée de forme} ou \textit{dérivée eulérienne} de $u$ en $\Omega$ dans la direction $\thetaa$ comme l'élément $du(\Omega)[\thetaa]\in W(\Omega)$ tel que:
    $$
        du(\Omega)[\thetaa] := \dot{u}(\Omega)[\thetaa] - \grad u \thetaa \:.
    $$
\end{dfntn}

\begin{dfntn}
    On dit que $u$ est \textit{dérivable par rapport à la forme en $\Omega$} si l'application $\thetaa\mapsto du(\Omega)[\thetaa]$ est linéaire continue de $\Ww^{1,\infty}(\mathbb{R}^d)$ dans $W(\Omega)$. 
\end{dfntn}

\begin{notation}
    Pour alléger les notations, lorsque le contexte est suffisament clair, on remplacera $u(\Omega)$, $\dot{u}(\Omega)[\thetaa]$ et $du(\Omega)[\thetaa]$ simplement par $u$, $\dot{u}$ et $du$.
\end{notation}

\begin{rmrk}
	Ici encore, les définitions précédentes tiennent encore si on remplace $\Ww^{1,\infty}$ par $\Cc^k_b$, pour tout $k\geq 1$.  
\end{rmrk}

\paragraph{Un peu de géométrie différentielle.}
Avant de se concentrer sur la dérivation d'un critère générique, nous rappelons brièvement deux définitions tirées de \cite{henrot2006variation} qui nous seront utiles.
\begin{dfntn}
	On suppose que $\Omega$ est $\pazocal{C}^1$. Soit un champ de vecteurs $\ww\in \Cc^1(\partial\Omega)$, on définit la \textit{divergence tangentielle} de $\ww$ par:
	$$
		\divv_\Gamma \ww := \left.\left( \divv \tilde{\ww} - (\gradd\tilde{\ww}\normalInt) \cdot \normalInt \right)\right|_{\partial\Omega}\:,
	$$
	où $\tilde{\ww}\in \Cc^1(\mathbb{R}^d)$ est un relèvement de $\ww$.
\end{dfntn}

\begin{dfntn}
	On suppose que $\Omega$ est $\pazocal{C}^2$. On définit alors la \textit{courbure moyenne} $\kappa$ le long de $\partial\Omega$ par:
	$$
		\kappa := \divv_\Gamma \normalInt\:.
	$$
\end{dfntn}

À partir de maintenant, pour être cohérent avec la représentation des formes par la méthode de level-set, nous considérons des ouverts $\Omega$ à bord $\pazocal{C}^1$, et donc des champs de vecteurs $\thetaa\in\Cc^1_b(\mathbb{R}^d)$. De plus, dans le cadre de l'optimisation de structures, nous nous limiterons à des critères $J$ s'exprimant sous la forme relativement générale:
\begin{equation}
  J(\Omega) = \pazocal{J}(\Omega,u(\Omega)):= \int_\Omega j(u(\Omega)) + \int_{\partial\Omega} k(u(\Omega))\:.
  \label{ExpressionJ}
\end{equation}
Les fonctions $j,k$ sont supposées $\pazocal{C}^1(\mathbb{R})$, et leurs dérivées par rapport à $u$, notées $j'$, $k'$, sont supposées Lipschitz. De plus, on impose qu'elles vérifient les conditions de croissance (\textit{growth conditions}) suivantes: pour tous $v$, $w \in \mathbb{R}$,
\begin{equation}
    |j(v)| \leq C\left(1+|v|^2\right) \qquad
    |k(v)| \leq C\left(1+|v|^2\right)
  \label{Condjk}
\end{equation}
\begin{equation}
    |j'(v)\cdot w| \leq C |v\cdot w| \qquad
    |k'(v)\cdot w| \leq C|v\cdot w| 
  \label{Condj'k'}
\end{equation}
pour des constantes $C>0$ données. Lorsque toutes ces hypothèses sont vérfiées, et lorsque $u$ est dérivable par rapport à la forme, alors on est capable de dériver l'expression \eqref{ExpressionJ} par composition.

\begin{prpstn}
    Si $\Omega$ est  $\pazocal{C}^1$, et si $u\in H^1(\Omega)$ admet une dérivée matérielle dans $H^1(\Omega)$ dans la direction $\thetaa\in\Cc^1_b(\mathbb{R}^d)$, alors $J$ admet une dérivée de forme dans la direction $\thetaa$ qui s'écrit:
    \begin{equation}
        dJ(\Omega)[\thetaa] = \int_{\Omega} j'(u)\cdot \dot{u} + j(u)\divv\thetaa + \int_{\partial\Omega} k'(u)\cdot \dot{u} + k(u)\divv_\Gamma\thetaa \:.
        \label{DJMDer}    
    \end{equation}
    Si de plus $\Omega$ est $\pazocal{C}^2$ et $u\in H^2(\Omega)$, alors la dérivée de forme $du\in H^1(\Omega)$ et on peut réécrire:
    \begin{equation}
        dJ(\Omega)[\thetaa] = \int_{\Omega} j'(u)\cdot du + \divv\left(j(u)\thetaa\right) + \int_{\partial\Omega} k'(u)\cdot du + (\kappa + \partial_{\normalInt})\left( k(u)\right) (\thetaa\cdot\normalInt)\:,
        \label{DJSDer}
    \end{equation}   
    où $\kappa$ et $\partial_{\normalInt}$ représentent respectivement la courbure moyenne et la dérivée normale le long de $\partial\Omega$.
\end{prpstn}

\begin{proof}
    Ce résultat s'obtient en adaptant légèrement les preuves du théorème 5.2.2 et de la proposition 5.4.18 \cite{henrot2006variation}.
\end{proof}

\begin{crllr}
    Sous les hypothèses de la proposition précédente, si $u$ est dérivable par rapport à la forme en $\Omega$, alors $J$ l'est aussi.
\end{crllr}

Nous allons maintenant voir à l'aide d'un exemple comment montrer que la solution d'une EDP est dérivable par rapport à la forme, et comment expliciter cette dérivée, afin d'obtenir une expression de $dJ$ qu'on pourra utiliser en pratique.

\subsection{L'exemple de l'élasticité linéaire}
\label{subsec:SDiffElas}

Puisque nous nous intéressons à l'optimisation de structures, l'équation d'état associée au problème \eqref{StructOPT} correspond dans notre cas à l'équation traduisant l'équilibre mécanique de la structure considérée. Avant d'étudier la dérivabilité de la solution de cette équation par rapport à la forme, nous rappelons brièvement le modèle de l'élasticité linéaire. Bien que relativement simple dans son écriture, ce modèle est bien adapté à la description du comportement de structures qui se déforment peu.

Pour mesurer la déformation d'un solide $\Omega$, on introduit le \textit{déplacement} $\uu(x)$, pour tout $x\in\overline{\Omega}$. Ce vecteur correspond à l'écart entre les positions du point considéré dans les configurations déformée et non déformée. À partir de cette quantité, on peut introduire le \textit{tenseur des déformations}:
$$
	\epsilonn(\uu) = \frac{1}{2}\left( \grad\uu + \grad\uu^T \right)\:.
$$  
Une autre objet central en mécanique des milieux continus est le \textit{tenseur des contraintes} de Cauchy $\sigmaa$. Il caractérise les efforts internes qui sont générés dans le matériau en présence de sollicitations. Dans le modèle de l'élasticité linéaire, on suppose que $\sigmaa$ dépend linéairement de $\epsilonn(\uu)$, de sorte que:
$$
    \sigmaa(\uu) = \Aa : \epsilonn(\uu)\:,
$$
où le tenseur $\Aa$ d'ordre 4 est appelé \textit{tenseur d'élasticité}. Si le matériau est de plus isotrope, alors la relation précédente prend la forme de la loi de Hooke:
$$
    \sigmaa(\uu) = 2\mu \epsilonn(\uu)+\lambda \divv\uu \,\Ii\:, 
$$
où $\lambda$ et $\mu$ sont les coefficients de Lamé du matériau, et $\Ii$ est le tenseur identité d'ordre 2. Souvent, on  exprime ces coefficients en fonction du module d'Young $E$ et du coefficient de Poisson $\nu$:
\begin{equation*}
    \lambda = \frac{E\nu}{(1+\nu)(1-2\nu)}\:, \hspace{1em} \mu = \frac{E}{2(1+\nu)}\:.
\end{equation*}
Considérons un corps $\Omega$ rempli d'un tel matériau. On suppose que $\Omega$ est encastré (déplacement nul) sur une partie de son bord $\Gamma_D$, qu'on lui applique une force de traction $\tauu$ sur une autre partie du bord $\Gamma_N$, et que le reste du bord $\Gamma=\partial\Omega \setminus(\Gamma_D\cup\Gamma_N)$ est libre de toute contrainte. Si de plus $\Omega$ est soumis à des forces volumiques externes $\ff$, alors les équations de l'équilibre mécanique s'écrivent:
\begin{equation}
    \left\{ \
    \begin{array}{rlr}
        - \Divv \sigmaa(\uu) &= \ff  &\mbox{ dans } \Omega, \\
      \uu                &= 0    &\mbox{ sur } \Gamma_D,  \\
        \sigmaa(\uu) \cdot \normalInt &= \tauu  &\mbox{ sur } \Gamma_N,   \\
        \sigmaa(\uu) \cdot \normalInt &= 0  &\mbox{ sur } \Gamma.
    \end{array}
    \right.
    \label{FFElas}
\end{equation} 
Introduisons maintenant le cadre mathématique approprié pour étudier ces équations. On suppose que $\Omega$ est à bord Lipschitz. On définit alors l'espace de Hilbert:
$$
    \Xx:= \Hh^1_{\Gamma_D}(\Omega) = \{\: \vv\in \Hh^1(\Omega) \ | \ \vv = 0 \mbox{ sur } \Gamma_D \:\}\:.
$$
Si $\Aa\in L^\infty(\Omega,\mathbb{T}^4)$, $\ff\in \Ll^2(\Omega)$ et $\tauu\in\Ll^2(\Gamma_N)$, alors la solution $\uu$ du système d'équations précédent vérifie:
\begin{equation}
    \int_\Omega \Aa:\epsilonn(\uu):\epsilonn(\vv) = \int_\Omega \ff\vv +\int_{\Gamma_N} \tauu \vv\:, \ \ \forall \vv \in \Xx.
    \label{FVElas0}
\end{equation}
En notant $a$ la forme bilinéaire associée au membre de gauche, et $L$ la forme linéaire associée au membre de droite, on peut réécrire \eqref{FVElas0} sous la forme plus compacte:
\begin{equation}
    a(\uu,\vv) = L(\vv)\:, \ \ \forall \vv \in \Xx.
    \label{FVElas}
\end{equation}
En plus des hypothèses précédentes, on suppose que $\Aa$ est symétrique et elliptique de constante d'ellipticité $\alpha_0$, i.e.$\!$ $\Aa$ vérifie
\begin{equation*}
	\begin{aligned}
		& \Aa_{ijkl} = \Aa_{klij} = \Aa_{jikl}\: \:\: \forall \: 1 \leq i, j, k, l \leq d\:, \\
		& \Aa:\varepsilonn:\varepsilonn \geq \alpha_0 \varepsilonn:\varepsilonn\: \:\: \forall \varepsilonn \mbox{ tenseur symétrique d'ordre 2 .} 
	\end{aligned}
\end{equation*}
Si de plus, $|\Gamma_D|>0$, alors on a l'inégalité de Korn, voir \cite{DuvLio1972}. La forme bilinéaire $a$ est donc coercive et continue sur $\Xx\times\Xx$, et comme $L$ est linéaire continue sur $\Xx$, le lemme de Lax-Milgram permet de conclure qu'il existe une unique solution $\uu\in\Xx$ au problème \eqref{FVElas}.

\begin{rmrk}
    Le résultat tient encore sous les hypothèses de régularité plus faibles: $\ff\in\Xx^*$ et $\tauu\in\Hh^{-\frac{1}{2}}(\Gamma_N)$. En effet, ces régularités suffisent à assurer la continuité de $L$. La seule précaution à prendre est de remplacer les intégrales à droite dans \eqref{FVElas0} par des crochets de dualité. 
\end{rmrk}

Il est possible de réécrire \eqref{FVElas} sous la forme (équivalente) d'un problème de minimisation:
\begin{equation}
    \inf_{\vv\in\Xx} \varphi(\vv) := \inf_{\vv\in\Xx} \ \frac{1}{2}a(\vv,\vv)-L(\vv)\:.
    \label{OPTElas}
\end{equation}
Cette formulation se prête mieux à l'interprétation physique: on cherche le champ de déplacements $\uu$ qui minimise l'énergie totale $\varphi$ du système mécanique.

\paragraph{Dérivabilité par rapport à la forme.}

On cherche ici à montrer que la solution $\uu$ de \eqref{FVElas} est dérivable par rapport à la forme. Il existe plusieurs façons de procéder. Nous en citons deux: la première est présentée dans \cite{sokolowski1992introduction} s'intéresse à la différentiabilité de Hadamard, et la deuxième dans \cite{henrot2006variation} s'intéresse à la différentiabilité de Fréchet. Avant de rentrer plus dans les détails, introduisons quelques notations.

\begin{notation}
    Soient $\thetaa\in\Cc^1_b(\mathbb{R}^d)$ et $t>0$, on note:
    \begin{equation*}
        \begin{aligned}
        	\Omega(\thetaa):= (\Id+\thetaa)\Omega\:, \qquad & \Omega(t):=(\Id+t\thetaa)\Omega\:, \\
            \uu_{\thetaa}:=\uu(\Omega(\thetaa))\:, \qquad & \uu^{\thetaa}:=\uu_{\thetaa}\circ(\Id+\thetaa)\:, \\
            \uu_t:=\uu(\Omega(t))\:, \qquad & \uu^t:=\uu_t\circ(\Id+t\thetaa)\:. \\
        \end{aligned}
    \end{equation*}
    Avec ces notations, on a:
    \begin{equation*}
        \begin{aligned}
        \grad \uu^{\thetaa} &= \left((\grad \uu_{\thetaa})\circ (\Id+\thetaa)\right) (\Ii+\grad \thetaa) \:, \\
        \grad \uu^t &= \left((\grad \uu_t)\circ (\Id+t\thetaa)\right) (\Ii+t\grad \thetaa) \:. 
        \end{aligned}
    \end{equation*}
\end{notation}

Dans l'approche proposée dans \cite{sokolowski1992introduction}, on commence par montrer l'existence de dérivées matérielles directionnelles. Pour ce faire, on fixe un $\thetaa$, puis on divise par $t$ la différence entre les formulations vérifiées par $\uu^t$ et $\uu$ pour faire apparaître le quotient $\frac{1}{t}(\uu^t-\uu)$. On passe ensuite à la limite $t\to 0$ ce qui donne la formulation vérifiée par $\dot{\uu}$. À partir de cette formulation, on peut conclure quant à l'existence et l'unicité de $\dot{\uu}$, ainsi qu'à sa linéarité/continuité par rapport à $\thetaa$. 

Dans l'approche proposée dans \cite{henrot2006variation}, on travaille avec $\Omega(\thetaa)$ plutôt qu'avec $\Omega(t)$. On écrit la formulation variationnelle vérifiée par $\uu^{\thetaa}$, puis, si tous les termes sont assez réguliers, on applique le théorème des fonctions implicites pour conclure que $\thetaa\mapsto \uu^{\thetaa}$ est (Fréchet) différentiable en $0$.

Dans tous les cas, pour pouvoir écrire les équations posées sur le domaine transporté, il faut que les données $\ff$, $\tauu$, $\Aa$ soient aussi définies en dehors de $\Omega$. De plus, comme on s'intéresse à des propriétés de dérivabilité, on a besoin d'avoir un peu plus de régularité sur les données pour que les preuves fonctionnent. En particulier, cela donne ici:

\begin{hyp}
    $\Omega$ est de classe $\pazocal{C}^1$, $\Aa\in \pazocal{C}^1_b(\mathbb{R}^d,\mathbb{T}^4),\:\, \ff \in \Hh^1(\mathbb{R}^d),\:\, \tauu \in \Hh^2(\mathbb{R}^d)$.
    \label{hyp:RegData}
\end{hyp}

Énonçons maintenant le résultat de dérivabilité pour le problème \eqref{FVElas}. Ensuite, puisque tous les termes sont réguliers, nous le démontrerons en suivant la deuxième approche.

\begin{prpstn}
    Sous l'hypothèse \ref{hyp:RegData}, la solution $\uu$ de \eqref{FVElas} est (fortement) dérivable par rapport à la forme de $\Cc^1_b(\mathbb{R}^d)$ dans $\Ll^2(\Omega)$.
\end{prpstn}

\begin{proof}

Tout d'abord, on écrit la formulation variationnelle pour le problème posé sur $\Omega(\thetaa)$: trouver $\uu_{\thetaa}\in \Xx_{\thetaa}:=\Hh^1_{\Gamma_D(\thetaa)}(\Omega(\thetaa))$ tel que, pour tout $\vv_{\thetaa} \in \Xx_{\thetaa}$,
\begin{equation} 
    \int_{\Omega(\thetaa)} \Aa : \epsilonn(\uu_{\thetaa}) : \epsilonn(\vv_{\thetaa})  = \int_{\Omega(\thetaa)} \ff \: \vv_{\thetaa} + \int_{\Gamma_N(\thetaa)} \tauu \: \vv_{\thetaa} \:.
  \label{FVTElas}
\end{equation}
On peut ramener \eqref{FVTElas} sur le domaine de référence $\Omega$ par changement de variable. Soient $\JacV(\thetaa) := $Jac$(\Id+\thetaa)$ et $\JacB(\thetaa):=$Jac$_{\Gamma(\thetaa)}(\Id+\thetaa)$ les Jacobien et Jacobien tangentiel de la transformation. Plus généralement, la composition avec l'opérateur $\circl (\Id+\thetaa)$ sera notée simplement $(\thetaa)$. Par exemple, $\Aa(\thetaa):=\Aa\circl(\Id+\thetaa)$.
Munis de ces notations, nous pouvons procéder au changement de variable:
\begin{equation*}
  \begin{aligned}
    \int_{\Omega} \Aa(\thetaa) & : \frac{1}{2} \left( (\gradd \uu_{\thetaa})(\thetaa) + ({\gradd \uu_{\thetaa}}^T)(\thetaa)  \right)  : \frac{1}{2} \left( (\gradd \vv_{\thetaa})(\thetaa) + ({\gradd \vv_{\thetaa}}^T)(\thetaa)  \right) \JacV(\thetaa) \\
    \: & \hspace{3em} = \int_{\Omega} \ff(\thetaa) \: \vv_{\thetaa}(\thetaa) \: \JacV(\thetaa) + \int_{\Gamma_{N}} \tauu(\thetaa) \: \vv_{\thetaa}(\thetaa) \: \JacB(\thetaa) \:.
     \end{aligned}
\end{equation*}
Comme $(\Id+\thetaa)$ est un $\pazocal{C}^1$-difféomorphisme, $\vv_{\thetaa} \mapsto \vv_{\thetaa}(\thetaa)$ est un isomorphisme de $\Xx_{\thetaa}$ dans $\Xx$. De fait, puisque \eqref{FVTElas} est valable pour tout $\vv_{\thetaa} \in \Xx_{\thetaa}$, il vient, pour tout $\vv \in \Xx$,
\begin{equation}
    \begin{aligned}
        \int_{\Omega} \Aa(\thetaa) & : \frac{1}{2} \left( \gradd\uu^{\thetaa} {(\Ii+\gradd\thetaa)}^{-1}  + {(\Ii+\gradd\thetaa^T)}^{-1} {\gradd \uu^{\thetaa}}^T \right) \\
        \: & : \frac{1}{2} \left( \gradd \vv{(\Ii+\gradd\thetaa)}^{-1} + {(\Ii+\gradd\thetaa^T)}^{-1}{\gradd \vv}^T  \right) \JacV(\thetaa) \\ 
        \: & \hspace{5em}- \int_{\Omega} \ff(\thetaa) \: \vv \: \JacV(\thetaa) - \int_{\Gamma_{N}} \tauu(\thetaa) \: \vv \: \JacB(\thetaa) = 0 \:.
        \end{aligned}
  \label{FVTRElas}
\end{equation}

On définit un opérateur linéaire continu $\pazocal{A}_{\thetaa} : \Xx \rightarrow \Xx^*$, et deux éléments $\pazocal{B}_{\thetaa}$, $\pazocal{C}_{\thetaa} \in \Xx^*$ tels que: pour tous $\ww$, $\vv\in \Xx$,
\begin{equation*}
    \begin{aligned}
        \prodD{\pazocal{A}_{\thetaa} \ww , \vv }{} := \int_{\Omega} \Aa(\thetaa) & : \frac{1}{2} \left( \gradd\ww {(\Ii+\gradd\thetaa)}^{-1}  + {(\Ii+\gradd\thetaa^T)}^{-1} {\gradd \ww}^T \right) \\
        \: & : \frac{1}{2} \left( \gradd \vv{(\Ii+\gradd\thetaa)}^{-1} + {(\Ii+\gradd\thetaa^T)}^{-1}{\gradd \vv}^T  \right) \JacV(\thetaa) \:, \\
        \prodD{\pazocal{B}_{\thetaa}, \vv}{} := \int_{\Omega} \ff(\thetaa)& \: \vv \: \JacV(\thetaa) \:, \\
        \prodD{\pazocal{C}_{\thetaa}, \vv}{} := \int_{\Gamma_{N}} \tauu(\thetaa&) \: \vv \: \JacB(\thetaa) \:.
        \end{aligned}
\end{equation*}

Maintenant, afin de prouver la différentiabilité de l'application $\Phi:\thetaa \in \Cc^1_b(\mathbb{R}^d) \mapsto \uu^{\thetaa} \in \Xx$ (i.e.$\!$ la dérivabilité matérielle de $\uu$ au sens de Fréchet), on introduit l'opérateur suivant:
 \begin{eqnarray*}
   \pazocal{F} : \: \:  \Cc^1_b(\mathbb{R}^d) \times \Xx & \longrightarrow & \Xx^* \\
   (\thetaa,\ww) & \longmapsto & \pazocal{A}_{\thetaa} \ww - \pazocal{B}_{\thetaa} - \pazocal{C}_{\thetaa}
 \end{eqnarray*}
 D'après les notations précédentes, la formulation \eqref{FVTRElas} peut être réécrite: trouver $\uu^{\thetaa} \in \Xx$ tel que,
 \begin{equation}
  \prodD{\pazocal{F} (\thetaa, \uu^{\thetaa}), \vv} \:= 0, \hspace{2em} \forall \vv \in \Xx 
  \label{FVTR1}
\end{equation}
Comme prévu, la différentiabilité de $\Phi$ peut être démontrée en appliquant le théorème des fonctions implicites à $\pazocal{F}$. Et donc, $\pazocal{F}$ doit satisfaire les hypothèses du théorème.

\subsubsection*{$\pazocal{F}$ est $\pazocal{C}^1$ par rapport à $\thetaa$.} Nous allons le vérifier pour chaque terme.
\subsubsection*{$\bullet\,\thetaa \in \Cc^1_b \mapsto \pazocal{A}_{\thetaa} \ww \in \Xx^*$.}
D'après les sections 5.3 et 5.4 de \cite{henrot2006variation}, pour $\thetaa \in \Cc^1_b$ les applications 
$$\thetaa \mapsto \JacV(\thetaa) \in L^\infty(\Omega),\: \thetaa \mapsto \JacB(\thetaa) \in \pazocal{C}^0(\partial\Omega)$$
$$\thetaa \mapsto (\Ii+\gradd\thetaa),\; (\Ii+\gradd\thetaa^T),\; {(\Ii+\gradd\thetaa)}^{-1},\; {(\Ii+\gradd\thetaa^T)}^{-1} \in L^\infty(\Omega, \pazocal{M}_{d\times d})$$ où $\pazocal{M}_{d \times d}$ représente l'ensemble des matrices $d \times d$ à coefficients réels, sont toutes de classe $\pazocal{C}^\infty$. 
De plus, pour $\Aa \in \pazocal{C}^1_b(\mathbb{R}^d, \mathbb{T}^4)$, $\thetaa \mapsto \Aa(\thetaa) \in L^\infty(\Omega, \mathbb{T}^4)$ est de classe $\pazocal{C}^1$. Par conséquent, pour tout $\ww \in \Xx$, l'application $\thetaa \in \pazocal{C}^1_b \mapsto \pazocal{A}_{\thetaa} \ww \in \Xx^*$ est de classe $\pazocal{C}^1$.
\subsubsection*{$\bullet\,\thetaa \in \Cc^1_b \mapsto \pazocal{B}_{\thetaa},\,\pazocal{C}_{\thetaa} \in \Xx^*$.}
Toujours d'après \cite{henrot2006variation} (Lemmes 5.3.3 et 5.3.9), pour $\ff \in \Hh^1(\mathbb{R}^d)$ et $\tauu \in \Hh^2(\mathbb{R}^d)$, les applications $\thetaa \in \Cc^1_b \mapsto \pazocal{B}_{\thetaa} \in \Xx^*$ et $\thetaa \in \Cc^1_b \mapsto \pazocal{C}_{\thetaa} \in \Xx^*$ sont de classe $\pazocal{C}^1$.

On a donc bien que pour tout $\ww \in \Xx$, $\thetaa \mapsto \pazocal{F}(\thetaa,\ww)$ est $\pazocal{C}^1$ sur $\Cc^1_b(\mathbb{R}^d)$, et a fortiori autour de zéro.

\subsubsection*{$D_{\ww} \pazocal{F}(0,\uu)$ est un isomorphisme.} 
Puisque tout est linéaire, il est facile d'exprimer la différentielle de $\pazocal{F}$ par rapport à $\ww$: pour tous $\ww$, $\vv \in \Xx$,
\begin{equation*}
  \prodD{D_{\ww} \pazocal{F}(0,\uu) \ww, \vv} \:=\: a(\ww,\vv)\:.
\end{equation*}
Il est donc clair que $D_{\ww} \pazocal{F}(0,\uu)$ un isomorphisme de $\Xx$ dans $\Xx^*$.

Toutes les hypothèses du théorème des fonctions implicites sont donc satisfaites. On en déduit l'existence d'une fonction $\thetaa \mapsto \zz(\thetaa) \in \Xx$ de classe $\pazocal{C}^1$ autour de zéro, et telle que $\pazocal{F}(\thetaa,\zz(\thetaa)) = 0$ dans $\Xx^*$. Par unicité de la solution à \eqref{FVTRElas}, on en conclut que $\zz(\thetaa) = \uu^{\thetaa}$. Par conséquent, $\thetaa \mapsto \uu^{\thetaa}$ est Fréchet différentiable en zéro par rapport à $\thetaa$. Ceci implique que $\thetaa \mapsto \dot{\uu}(\Omega)[\thetaa]$ est linéaire continue de $\Cc^1_b(\mathbb{R}^d)$ dans $\Xx$, et la définition de la dérivée de forme permet de conclure.

\end{proof}

Contrairement à l'approche proposée dans \cite{sokolowski1992introduction}, cette démonstration ne donne pas la formulation variationnelle vérifiée par la dérivée matérielle. Cependant, maintenant que nous savons que $\uu$ est dérivable par rapport à la forme, nous pouvons dériver directement \eqref{FVElas} terme à terme.

\paragraph{Formulation variationnelle vérifiée par $\dot{\uu}$.}
On se fixe une direction $\thetaa$, puis pour $t>0$, on s'intéresse à la formulation vérfiée par $\uu^t$, qui n'est rien d'autre que la formulation \eqref{FVTRElas} où on a remplacé les dépendances en $\thetaa$ par des dépendances en $t$. Il suffit ensuite de calculer la dérivée $\frac{\partial}{\partial t} |_{t=0}$ pour cette formulation.

En combinant les résultats de \cite{sokolowski1992introduction} (voir p.135 à 141) et de \cite{henrot2006variation} (théorème 5.2.2 et corollaire 5.2.5), on est en mesure de dériver chacun des termes de \eqref{FVTRElas}. Avant cela, nous avons besoin d'introduire la forme bilinéaire $a'$ et la forme linéaire $\epsilonn'$ telles que, pour tous $\ww$, $\vv \in \Xx$,
\begin{equation*}
    \begin{aligned} 
        &a'(\ww,\vv) := \int_\Omega \big\{ \Aa:\epsilonn'(\ww):\epsilonn(\vv) + \Aa:\epsilon(\ww):\epsilonn'(\vv) 
        \\ &\hspace{0.25\textwidth} 
        + (\divv \thetaa \: \Aa + \gradd \Aa \: \thetaa):\epsilonn(\ww):\epsilonn(\vv) \big\} \:, \\
        &\epsilonn'(\vv) := -\frac{1}{2}\left( \gradd \vv \gradd \thetaa + {\gradd \thetaa}^T {\gradd \vv}^T \right).
    \end{aligned}
\end{equation*}
Avec ces notations, les dérivées de chaque terme donnent:
\begin{itemize}[leftmargin=*]
  \item $\left. \dfrac{\partial}{\partial t} \prodD{\pazocal{A}_{t\thetaa}\uu^t,\vv}{} \right|_{t=0} = a(\dot{\uu},\vv)+a'(\uu,\vv)$ .
  
  \item $\displaystyle \left. \frac{\partial}{\partial t} \prodD{\pazocal{B}_{t\thetaa},\vv}{} \right|_{t=0} = \int_\Omega (\divv \thetaa \: \ff + \gradd \ff \: \thetaa) \: \vv$ .
  
  \item $\displaystyle \left. \frac{\partial}{\partial t} \prodD{\pazocal{C}_{t\thetaa},\vv}{} \right|_{t=0} = \int_{\Gamma_N} (\divv_\Gamma \thetaa \: \tauu+ \gradd \tauu \: \thetaa) \: \vv$ .
\end{itemize}
On obtient donc la formulation vérifiée par $\dot{\uu}$:
\begin{equation}
        a(\dot{\uu},\vv) = L[\thetaa](\vv) \:, \hspace{1em} \forall \vv \in \Xx ,
    \label{FVElasMDer}
\end{equation}
où la forme linéaire $L[\thetaa]$ est définie par, pour tout $\vv\in\Xx$:
\begin{equation*}
    L[\thetaa](\vv) := \int_\Omega (\divv \thetaa \: \ff + \gradd \ff  \thetaa) \vv + \int_{\Gamma_N} (\divv_\Gamma \thetaa \: \tauu+ \gradd \tauu \thetaa)\vv - \:a'(\uu,\vv) \:.
\end{equation*}

\paragraph{Calcul de la dérivée d'un critère générique}

On considère toujours un critère $J$ de la forme \eqref{ExpressionJ}, avec évidemment $j$, $k:\mathbb{R}^d\to \mathbb{R}$ car l'inconnue variationnelle $\uu$ est vectorielle ici. Soit $\thetaa$ une direction fixée. Puisque $\uu$ est dérivable par rapport à la forme, elle l'est en particulier dans la direction $\thetaa$, et on a d'après \eqref{DJMDer}:
$$
    dJ(\Omega)[\thetaa] = \int_{\Omega} j'(\uu)\cdot \dot{\uu} + j(\uu)\divv\thetaa + \int_{\partial\Omega} k'(\uu)\cdot \dot{\uu} + k(\uu)\divv_\Gamma\thetaa \:.
$$
Afin de pouvoir utiliser cette expression pour trouver des directions de descente, on voudrait que la dépendance en $\thetaa$ soit complètement explicite. Autrement dit, comme on ne connaît pas la dépendance de $\dot{\uu}$ en $\thetaa$, on aimerait se débarrasser des termes dans lesquels cette fonction apparaît. Pour cela, on utilise une méthode bien connue provenant du contrôle optimal: la \textit{méthode de l'adjoint}. Nous renvoyons le lecteur à \cite{hinze2008optimization} pour une introduction au contrôle optimal et à la notion d'état adjoint.
Dans le cas présent, compte tenu de la formulation \eqref{FVElasMDer} et des termes dont nous souhaitons nous débarrasser, on introduit l'état adjoint $\pp\in\Xx$ (même espace que $\dot{\uu}$) solution de:
\begin{equation}
    a(\pp,\ww) = -\int_\Omega j'(\uu)\cdot\ww - \int_{\Gamma_N} k'(\uu)\cdot\ww\:, \qquad \forall \ww \in\Xx\:.
    \label{FVElasAdj}
\end{equation}
D'après les hypothèses sur $j$ et $k$, on sait que le membre de droite constitue une forme linéaire continue sur $\Xx$. Le lemme de Lax-Milgram permet donc de conclure que $\pp$ existe et est unique. 

\begin{rmrk}
    Il est important de noter que $\pp$ ne dépend pas de $\thetaa$. Le calcul supplémentaire de $\pp$ est donc intéressant s'il nous permet de nous affranchir de $\dot{\uu}$.
\end{rmrk}

On peut prendre $\ww=\dot{\uu}$ comme fonction-test dans \eqref{FVElasAdj}, et réécrire les deux termes dont on souhaite se débarasser:
$$
    \int_{\Omega} j'(\uu)\cdot \dot{\uu}  + \int_{\partial\Omega} k'(\uu)\cdot \dot{\uu} = -a(\pp,\dot{\uu})\:.
$$
Puis, en utilisant la symétrie de $a$ et en prenant $\vv=\pp$ comme fonction-test dans \eqref{FVElasMDer}, il vient:
$$
    \int_{\Omega} j'(\uu)\cdot \dot{\uu}  + \int_{\partial\Omega} k'(\uu)\cdot \dot{\uu} = -L[\thetaa](\pp)\:.
$$
On en déduit la nouvelle expression de $dJ$, dans laquelle toutes les dépendances en $\thetaa$ sont explicites:
\begin{equation}
    dJ(\Omega)[\thetaa] = -L[\thetaa](\pp) + \int_{\Omega} j(\uu)\divv\thetaa + \int_{\partial\Omega} k(\uu)\divv_\Gamma\thetaa \:.
    \label{DJ0Vol}
\end{equation}
Si de plus on suppose que $\Omega$ est de classe $\pazocal{C}^2$ et que $\uu$, $\pp\in\Hh^2(\Omega)\cap\Xx$, alors on est capable d'obtenir une formule de $dJ$ sous la forme d'une intégrale de surface qui ne dépend que de la composante normale de $\thetaa$ (voir par exemple \cite{sokolowski1992introduction} pour le détail des calculs):
\begin{equation}
    dJ(\Omega)[\thetaa] = \int_{\partial\Omega} \mathfrak{g}(\thetaa\cdot\normalInt) \:,
    \label{DJ0Surf}
\end{equation}
où la fonction $\mathfrak{g}\in L^1(\partial\Omega)$ dépend de $\uu$, $\pp$ et des données:
\begin{equation}
    \mathfrak{g} = j(\uu)+\Aa:\epsilonn(\uu):\epsilonn(\pp)-\ff\pp + \chi_{\Gamma_N}(\kappa+\partial_{\normalInt})\left( k(\uu) - \tauu\pp\right)\:.
    \label{ExpStructdJ}
\end{equation}

\begin{rmrk}
    Dans le cas présent, où on dispose d'une régularité additionnelle, on est capable de donner une expression explicite de la distribution $\mathfrak{g}$ du théorème de structure.
\end{rmrk}

Maintenant que nous disposons de tous les outils permettant de construire une direction de descente, nous pouvons présenter dans les grandes lignes les étapes de l'algorithme type d'optimisation de formes par la méthode de level-set.

\section{Les étapes de l'algorithme}
\label{sec:algoODF}

Puisqu'on dispose des dérivées du critère à minimiser, on choisit un algorithme de gradient. Partant d'un domaine $\Omega^0$, on génère une suite de domaines $\{ \Omega^l\}_l\subset D$, où $D$ représente un domaine de calcul suffisament grand, de sorte $\{ J(\Omega^l)\}_l$ soit décroissante. Nous décrivons brièvement les quatre étapes d'une itération $l$ de l'algorithme, en nous appuyant à chaque fois sur l'exemple de l'élasticité linéaire afin d'illustrer notre propos.

\paragraph{Étape 1.} Résolution de l'équation d'état posée sur $\Omega^l$.

Dans le cas de l'élasticité linéaire, cela revient à déterminer $\uu(\Omega^l)$, la solution de \eqref{FVElas} posée sur $\Omega^l$.

\paragraph{Étape 2.} Résolution de la formulation adjointe posée sur $\Omega^l$.

À partir des données du problème et du critère $J$ considéré, on est capable de définir la formulation adjointe, puis de déterminer sa solution. Dans l'exemple, on cherche la solution $\pp(\Omega^l)$ de \eqref{FVElasAdj} posée sur $\Omega^l$.

\paragraph{Étape 3.} Calcul d'une direction de descente $\thetaa^l$.

À partir de l'expression explicite dont on dispose pour $dJ(\Omega^l)[\cdot]$ grâce à $\uu(\Omega^l)$ et $\pp(\Omega^l)$, voir \eqref{DJ0Surf} et \eqref{ExpStructdJ}, on peut choisir une direction de descente $\thetaa^l$, c'est-à-dire un champ de vecteurs vérifiant $dJ(\Omega^l)[\thetaa^l]<0$ (voir détails plus loin). Le théorème de structure (Théorème \ref{ThmStruct}) implique que seule la composante normale du champ de vecteurs a une influence sur la valeur de la dérivée de $J$ dans cette direction. Ceci suggère de chercher des $\thetaa^l$ dirigés selon la normale, i.e.$\!$ $\thetaa^l=\theta^l \normalInt^l$.

\paragraph{Étape 4.} Évolution du domaine $\Omega^l \to \Omega^{l+1}$.

On déforme le domaine $\Omega^l$ à partir de la direction de descente trouvée précédemment. En pratique plutôt que de calculer $\Omega^{l+1}=(\Id+t^l\thetaa^l)\Omega^l$ pour un certain pas $t^l>0$, on utilise la méthode de level-set pour advecter la frontière de $\Omega^l$. Plus précisément, si $\phi^l$ désigne la fonction level-set associée à $\Omega^l$, on transporte $\phi^l$ sur un intervalle de temps $[0,T^l]$, $T^l>0$.
La résolution de \eqref{LSAdvEquation} avec $\thetaa=\thetaa^l$, ou de \eqref{HJEquation} avec $\theta = \theta^l$, nous donne $\phi^{l+1}$, la fonction level-set associée à $\Omega^{l+1}$.
L'inconvénient d'une telle approche est qu'elle exige la connaissance de $\phi^l$ et $\thetaa^l$ (ou $\theta^l$) sur tout $D$, et qu'elle produit aussi $\phi^{l+1}$ sur tout $D$.

\begin{rmrk}
    Évidemment, un simple algorithme de gradient permet au mieux de trouver un minimum local. Cependant, en optimisation de formes, il est assez rare d'avoir existence d'une forme optimale. Nous renvoyons le lecteur à \cite[Chapitre 4]{henrot2006variation} ou à \cite[Section 6.2]{allaire2007conception} pour des discussions à ce sujet. De plus, dans les applications pratiques, et notamment dans l'industrie, il est fréquent de vouloir trouver un minimum local autour d'un domaine $\Omega^0$ donné.
\end{rmrk}

\paragraph{Relèvement d'une direction de descente.}

Supposons que nous sommes dans le cas où tout est suffisament régulier pour que la dérivée de forme de $J$ dans une direction $\thetaa$ donnée s'exprime sous une forme du type de \eqref{DJ0Surf}:
$$
    dJ(\Omega)[\thetaa] = \int_{\partial\Omega} \mathfrak{g}(\thetaa\cdot\normalInt)\:.
$$
La première intuition pour trouver une direction de descente est de prendre $\thetaa$ tel que $\thetaa = -\mathfrak{g}\normalInt$ sur $\partial\Omega$. Comme nous avons besoin d'un champ $\thetaa$ défini partout dans $D$ pour pouvoir advecter la level-set, il faut procéder à un relèvement. Si on choisit par exemple le relèvement harmonique, on obtient $\thetaa$ tel que:
\begin{equation}
    \left\{ \
    \begin{array}{cc}
         -\Delta\!\thetaa = 0 &  \mbox{ dans } D \:,\\
         \thetaa = -\mathfrak{g}\normalInt & \mbox{ sur } \partial\Omega\:.
    \end{array}
    \right.
    \label{RelevementDir}
\end{equation}

Bien que ce choix assure une direction de descente, il n'est pas très satisfaisant du point de vue de la régularité. D'une part, comme nous l'avons vu (cf \eqref{ExpStructdJ}), la fonction $\mathfrak{g}$ n'est en général pas très régulière, et donc $\thetaa$ ne le sera pas non plus sur le bord (qui est précisément l'endroit qui nous intéresse).  D'autre part, le fait d'imposer une condition de Dirichlet à l'intérieur du domaine $D$ peut générer des singularités dans la direction normale à l'interface $\partial\Omega$ (qui est précisément la direction qui nous intéresse). Plus précisément, pour que \eqref{RelevementDir} admette une solution, il faut déjà que $\mathfrak{g}\in H^{\frac{1}{2}}(\partial\Omega)$. Dans ce cas, on aura alors un unique $\thetaa \in \Hh^1(\Omega)\cap \Hh^1(D\setminus\Omega)$ solution, voir par exemple \cite{brezis2010functional}. On sait d'après la condition de Dirichlet que $\thetaa$ sera continu à l'interface, mais on ne peut rien dire de sa dérivée normale à cet endroit.

Un autre point de vue consiste à choisir un espace de Hilbert $\Hh$, une forme bilinéaire $b$ définie positive sur $\Hh$, puis définir $\thetaa$ comme l'unique solution de:
$$
    b(\thetaa,\vv) = -dJ(\Omega)[\vv]\:, \ \ \forall \vv\in \Hh\:.
$$
On obtient alors $dJ(\Omega)[\thetaa]=-b(\thetaa,\thetaa)$, et la valeur de $\thetaa$ sur $\partial\Omega$ n'est cette fois pas connue.
Un tel choix nous permet d'éviter les difficultés mentionnées plus haut, et d'obtenir une direction de descente $\thetaa$ plus régulière. Nous renvoyons à \cite{de2006velocity} pour une discussion plus détaillée à ce sujet.

\begin{rmrk}
    En pratique, on prend souvent $\Hh=\Hh^1(D)$ et $b$ comme le produit scalaire $\Hh^1$, éventuellement pondéré par des coefficients strictement positifs:
    $$
        \beta_0 \int_D \grad\thetaa \grad \vv + \beta_1 \int_D \thetaa \vv = -dJ(\Omega)[\vv] \:,
    $$
    où $\beta_0$ et $\beta_1$ sont des réels strictement positifs. Ce qui se réécrit sous la formulation forte:
\begin{equation}
    \left\{ \
    \begin{array}{cc}
         -\beta_0 \Delta\!\thetaa + \beta_1\thetaa = 0 &  \mbox{ dans } D \:,\\
         \dfrac{\partial\thetaa}{\partial\normalInt} = -\mathfrak{g}\normalInt & \mbox{ sur } \partial\Omega\:.
    \end{array}
    \right.
    \label{RelevementNeu}
\end{equation}      
    Cette fois, on sait d'après le lemme de Lax-Milgram qu'il existe une unique solution $\thetaa \in \Hh^1(D)$ dès que $\mathfrak{g}\in H^{-\frac{1}{2}}(\partial\Omega)$.
\end{rmrk}

\chapter{Mécanique des solides en contact}
\label{chap:1.2}

\section*{Introduction}

Nous avons présenté dans l'exemple du chapitre précédent le modèle de l'élasticité linéaire, qui permet de décrire la déformation d'un corps lorsque ce dernier est soumis à des efforts externes (surfaciques ou volumiques). Ici, on considère non plus un seul corps mais deux corps, chacun soumis à des efforts externes. On souhaite prendre en compte la possibilité que les deux corps se touchent dans la configuration déformée. Autrement dit, on veut traduire le fait que les déformations que subissent ces deux corps les font entrer en \textit{contact} l'un avec l'autre. Nous nous limiterons dans cette thèse au cas d'un corps déformable $\Omega$ (composé d'un matériau isotrope linéaire élastique) en contact avec un corps rigide (qui ne se déforme pas). On parle dans ce cas de contact \textit{rigide-déformable}. Comme souvent pour le contact en élasticité linéaire, nous supposerons de plus que nous sommes dans le contexte des \textit{petits déplacements}.

Lorsqu'on étudie le contact entre deux solides, on veut s'assurer de prendre en compte deux phénomènes: d'abord le fait que les deux objets ne puissent pas s'interpénétrer, et ensuite le fait que le contact puisse générer du frottement entre les deux solides. Si on pense à l'exemple d'un pied en contact avec un sol incliné, la force de réaction exercée par le sol sur le pied fait en sorte que le pied ne traverse pas le sol (pas d'interpénétration). Et de plus, si le pied peut rester immobile au lieu de glisser vers le bas, c'est grâce à l'effet des forces de frottement.

Les problèmes de contact présentent des difficultés tant sur le plan théorique que numérique à cause des non-linéarités qu'ils induisent. Plus précisément, même dans le cas le plus simple du contact rigide-déformable glissant (i.e.$\!$ sans frottement), on ne connaît pas a priori la partie de $\partial\Omega$ qui sera en contact avec le corps rigide. En particulier, on a besoin de connaître la déformation de $\Omega$ pour pouvoir déterminer les efforts de réaction du corps rigide. Or ces efforts vont eux aussi avoir un impact sur la déformation de $\Omega$.

Dans ce chapitre, nous commençons par présenter les modèles qui décrivent les conditions de contact, puis nous faisons une brève analyse de la formulation mathématique obtenue (dans le cas glissant, et dans le cas frottant). Ensuite, nous introduisons différentes formulations approchées qui permettent de contourner les irrégularités de la formulation initiale; l'objectif étant de faciliter à la fois la résolution numérique, mais aussi l'application à l'optimisation de formes.

\section{Concepts de base et formulation}

Comme mentionné, on considère le cas du contact rigide-déformable en élasticité linéaire. On note $\Omega$ le corps déformable, et $\Omega_{rig}$ le corps rigide. Pour éviter les difficultés techniques, nous supposons suffisamment de régularité sur ces deux domaines.

\begin{hyp}
    $\Omega$ est un ouvert borné à bord $\pazocal{C}^1$, et $\Omega_{rig}$ est un ouvert borné à bord $\pazocal{C}^3$.
    \label{hyp:RegBord}
\end{hyp}

La frontière $\partial\Omega$ de $\Omega$ est composée de parties sur lesquelles portent des conditions de Neumann et de Dirichlet (homogène), $\Gamma_N$ et $\Gamma_D$ respectivement, et de plus on introduit $\Gamma_C$: la zone de contact potentiel. C'est sur cette partie de la frontière qu'on impose les conditions aux limites de contact. Comme dans le cas de l'élasticité sans contact, on note $\Gamma=\partial\Omega\setminus(\Gamma_D\cup\Gamma_N\cup\Gamma_C)$ le reste du bord, qui est supposé libre de contrainte (i.e. $\!$ condition de Neumann homogène). Pour des questions de régularité, on suppose que $\overline{\Gamma_C}\cap\overline{\Gamma_D}=\emptyset$. Comme précédemment, $\normalInt$ désigne la normale sortante à $\partial\Omega$, et on introduit également $\normalExt$, la normale entrante à $\partial\Omega_{rig}$. La solution $\uu$ de ce problème vérifie donc le système d'élasticité \eqref{FFElas}, associé des conditions aux limites appropriées sur $\Gamma_C$. Voyons maintenant comment s'expriment ces conditions aux limites.

\subsection{Contact glissant}

Dans le modèle de contact glissant, on néglige les phénomènes de frottement pour se concentrer uniquement sur la \textit{condition de non-pénétration}, aussi appelée \textit{contrainte unilatérale}. Bien que ce problème ait déjà été considéré au dix-neuvième siècle par Hertz, c'est Signorini qui en donne la première formulation mathématique dans les années 1930. La preuve d'existence et unicité de la solution de cette formulation n'est obtenue qu'une trentaine d'années plus tard par un étudiant de Signorini, Fichera, dans \cite{fichera1963sul}.

\begin{figure}
\begin{center}
\begin{tikzpicture}

\draw [black] plot [smooth, tension=0.8] 
coordinates {(3,0) (4,1) (3,2.5) (1,3) (-0.5,1.5) (0,0.25) (1,0)};

\draw[black] (-1,0) -- (5,0);

\draw[black] (2.95,2.45) -- (3.05,2.55);
\draw[black] (.95,2.9) -- (0.90,3.05);
\draw[black] (-0.45,1.45) -- (-0.6,1.5);
\draw[black] (3.9,1) -- (4.1,1);
\draw[black] (-0.05,0.2) -- (.05,0.35);

\node[] at (4,2.25) {$\Gamma_D$};
\node[] at (-.75,2.25) {$\Gamma_N$};
\node[] at (1.5,0.3) {$\Gamma_C$};
\node[] at (1.65,1.75) {$\Omega$};
\draw[->] (3.75, 0.5) -- (4.1,0.15) ;
\node[] at (3.5, 0.75) {$x$};
\node[] at (4.8, 0.6) {$\normalInt(x)$};
\draw[red, densely dashed] (3.75, 0.5) -- (3.75,0) ;
\node[red] at (3.75,0) {\tiny{$\bullet$}};
\node[red] at (3.1,0.3) {$\gG_{\normalExt}(x)$};
\draw[->] (3.75,0) -- (3.75, -.5) ;
\node[] at (4.5, -0.5) {$\normalExt(x)$};
\node[] at (0.75,-0.75) {$\Omega_{rig}$};

\end{tikzpicture}
\end{center}
  \label{SchContact}
  \caption{Schéma du contact entre les corps $\Omega$ et $\Omega_{rig}$.}
\end{figure}
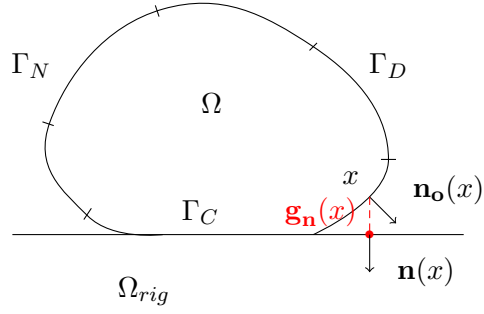

\paragraph{Fonction gap.}
Une quantité cruciale dans l'étude des problèmes de contact est la distance entre $\Omega$ et $\Omega_{rig}$. On appelle \textit{gap} (écart en français) cette quantité. On peut faire le choix de calculer cette distance dans la direction de la normale au corps rigide $\normalExt$ ou bien dans celle de la normale au corps déformable $\normalInt$. Nous choisissons ici $\normalExt$. Nous reviendrons plus tard sur l'impact de ce choix. Dans la configuration non déformée, le gap est donné par une fonction $\gG_{\normalExt}$ qui dépend uniquement des géométries respectives de $\Omega$ et $\Omega_{rig}$. En revanche, dans la configuration déformée, comme chaque point $x$ subit un déplacement $\uu(x)$, le gap devient $\gG_{\normalExt}-\uu_{\normalExt}$ si on considère de petits déplacements, voir \cite[Chapitre 2]{KikOde1988}.
Puisque $\Omega$ et $\Omega_{rig}$ sont suffisament réguliers (cf hypothèse \ref{hyp:RegBord}), on peut exprimer ces quantités géométriques à partir de la fonction distance orientée à $\partial\Omega_{rig}$, notée $b_{\Omega_{rig}}$. On suit le formalisme de \cite{DelZol2001}, ce qui donne pour la définition de cette fonction:
$$
	\gG_{\normalExt}(x) = b_{\Omega_{rig}}(x):= d_{\Omega_{rig}}(x) - d_{\mathbb{R}^d\setminus \Omega_{rig}}(x),  \: \forall x\in \mathbb{R}^d\:.
$$
On introduit ensuite la projection sur $\partial\Omega_{rig}$, notée $\Proj_{\partial\Omega_{rig}}$:
$$
	\Proj_{\partial\Omega_{rig}}(x)= x - b_{\Omega_{rig}}(x)\gradd b_{\Omega_{rig}}(x), \: \forall x\in \mathbb{R}^d\:.
$$
Puisque $\partial\Omega_{rig}$ est de classe $\pazocal{C}^3$ et compact, on sait d'après \cite{delfour1994shape} qu'il existe $h_0$ suffisamment petit tel que, pour tout $h\leq h_0$, l'ouvert
\begin{equation*}
    \partial\Omega_{rig}^h := \{ x\in\mathbb{R}^d \ | \ \ |b_{\Omega_{rig}}(x)| < h \}  
\end{equation*}
constitue un voisinage de $\partial\Omega_{rig}$ sur lequel $b_{\Omega_{rig}}$ est de classe $\pazocal{C}^3$. Ceci implique en particulier qu'on peut étendre $\normalExt$ de manière $\pazocal{C}^2$ à $\partial\Omega_{rig}^h$ de la façon suivante: pour $x\in \partial\Omega_{rig}^h\setminus \partial\Omega_{rig}$,
$$
	\normalExt(x):= \normalExt\left( \Proj_{\partial\Omega_{rig}}(x)\right)\:.
$$
Ainsi, on a que $\gG_{\normalExt}\normalExt\in \pazocal{C}^2(\partial\Omega_{rig}^h,\mathbb{R}^d)$, et on peut donc étendre cette fonction à tout $\mathbb{R}^d$ de sorte que l'extension $\gG$ vérifie: $\gG\in\pazocal{C}^2(\mathbb{R}^d)$ et $\gG=0$ sur $\complement(\partial\Omega_{rig}^{h'})$ pour $h'>h$. Enfin, comme $\overline{\Gamma_C}\cap\overline{\Gamma_D}=\emptyset$, il suit que $\overline{\Gamma_D}\subset\complement(\partial\Omega_{rig}^{h'})$ pour $h$, $h'$ suffisamment petits. On en déduit donc que $\gG\in\Xx$.

\paragraph{Condition de non-pénétration.}
En utilisant les notations précédentes, on peut naturellement traduire la contrainte unilatérale de la manière suivante:
\begin{equation}
    \uu_{\normalExt} \leq \gG_{\normalExt}\:, \ \ \mbox{p.p. sur } \Gamma_C\:.
    \label{CondUnil}
\end{equation}
On va donc chercher la solution de notre problème parmi les champs de déplacement qui satisfont cette condition, qu'on appelle l'ensemble des \textit{déplacements admissibles}
$$
	\Kk:=\{ \vv\in\Xx \ | \ \vv_{\normalExt} \leq \gG_{\normalExt} \mbox{ p.p. sur } \Gamma_C\}\:.
$$

\begin{rmrk}
    Il est clair que cet ensemble est convexe, fermé et non-vide ($\gG\in\Kk$).
\end{rmrk}

\paragraph{Formulations.}
Comme dans le cas sans contact, le déplacement $\uu$ solution doit minimiser l'énergie mécanique totale du système. La différence est qu'ici on ajoute la contrainte que $\uu$ appartienne à l'ensemble des déplacements admissibles. Le problème se traduit donc dans ce contexte en un problème de minimisation sous contrainte:
\begin{equation}
    \inf_{\vv\in\Kk} \ \varphi(\vv)\:,
    \label{OPTGliss}
\end{equation}
où $\varphi$ représente la fonctionnelle d'énergie introduite dans \eqref{OPTElas}. Ce problème rentre dans le cadre de l'optimisation dans des espaces de Hilbert, et tous les résultats classiques d'analyse convexe s'appliquent. En particulier, on a le résultat suivant.
\begin{prpstn}
    Lorsque l'hypothèse \ref{hyp:RegBord} est vérifiée, $\Aa$ est elliptique, $|\Gamma_D|>0$, $\ff\in \Ll^2(\Omega)$ et $\tauu\in\Ll^2(\Gamma_N)$, alors le problème \eqref{OPTGliss} admet une unique solution. De plus, $\uu\in\Kk$ est solution de \eqref{OPTGliss} si et seulement si elle vérifie l'inéquation variationelle (de première espèce):
    \begin{equation}
	    a(\uu,\vv-\uu) \:\geq\: L(\vv-\uu), \:\:\: \forall \vv \in \Kk\: .
        \label{IVGliss}
	\end{equation}
	\label{PropExisUniqGliss}
\end{prpstn}
\begin{proof}
    Puisque $a$ est coercive, on a que $\varphi$ est propre (i.e.$\!$ $\varphi\not\equiv +\infty$), coercive, continue et strictement convexe sur $\Xx$ qui est un espace de Banach réflexif. Comme de plus, $\Kk$ est convexe, fermé, non-vide, on sait d'après \cite[Chapitre 2]{ekeland1999convex} que \eqref{OPTGliss} admet une unique solution. Toujours d'après la même référence, comme $\varphi$ est dérivable au sens de Gateaux sur $\Xx$, avec des dérivées continues, on obtient que la solution $\uu$ est entièrement caractérisée par la condition d'optimalité \eqref{IVGliss}.
\end{proof}
\begin{rmrk}
    Comme dans le cas sans contact, les résultats précédents tiennent encore si on impose moins de régularité sur les données: $\ff\in\Xx^*$ et $\tauu\in\Hh^{-\frac{1}{2}}(\Gamma_C)$.
\end{rmrk}
En partant de la formulation faible \eqref{IVGliss}, on peut revenir à une formulation forte du type de \eqref{FFElas} (et vice-versa) à condition que les données soient suffisamment régulières, voir par exemple \cite{DuvLio1972}.
\begin{prpstn}
    Sous les hypothèses de la proposition \ref{PropExisUniqGliss}, on a que $\uu$ est solution de \eqref{IVGliss} si et seulement si $\uu$ est solution de:
\begin{equation}
    \left\{ \
    \begin{array}{rlr}
        - \Divv \sigmaa(\uu) &= \ff  &\mbox{ dans } \Omega, \\
      \uu                &= 0    &\mbox{ sur } \Gamma_D,  \\
        \sigmaa(\uu) \cdot \normalInt &= \tauu  &\mbox{ sur } \Gamma_N,   \\
        \sigmaa(\uu) \cdot \normalInt &= 0  &\mbox{ sur } \Gamma, \\
        \uu_{\normalExt} \leq \gG_{\normalExt}, \ \sigmaa_{\normalInt\!\normalExt}(\uu) &\leq 0, \ \sigmaa_{\normalInt\!\normalExt}(\uu)(\uu_{\normalExt}-\gG_{\normalExt})=0 & \mbox{ sur } \Gamma_C, \\
        \sigmaa_{\normalInt\!\tanExt}(\uu) &= 0 & \mbox{ sur } \Gamma_C,
     \end{array}
    \right.
    \label{FFGliss}
\end{equation} 
où on a introduit: $\sigmaa_{\normalInt\!\normalExt}(\uu)= \sigmaa(\uu) \cdot \normalInt \cdot \normalExt$ et $\sigmaa_{\normalInt\!\tanExt}(\uu)=\sigmaa(\uu) \cdot \normalInt - \sigmaa_{\normalInt\!\normalExt}(\uu)\normalExt$.
\end{prpstn}
Nous nous intéressons maintenant de plus près à l'expression des conditions aux limites de contact glissant (deux dernières lignes de \eqref{FFGliss}). Les trois conditions de l'avant-dernière ligne de \eqref{FFGliss} décrivent les phénomènes liés au contact dans la direction normale à la surface de contact, i.e.$\!$ dans la direction de $\normalExt$. La condition de la dernière ligne décrit quant à elle les phénomènes qui interviennent tangentiellement à la surface de contact.
\begin{itemize}
	\item $\uu_{\normalExt} \leq \gG_{\normalExt}$: On a déjà vu que cette relation traduisait la condition de non-pénétration.
	\item $\sigmaa_{\normalInt\!\normalExt}(\uu)\leq 0$: La contrainte normale sur la surface de contact est dirigée dans le sens opposé à $\normalExt$. Ceci est intuitif car cette contrainte traduit la réaction du corps rigide sur le corps déformable.
	\item $\sigmaa_{\normalInt\!\normalExt}(\uu)(\uu_{\normalExt}-\gG_{\normalExt})=0$: Si on a une réaction non nulle du corps rigide, alors nécessairement on est en contact, c'est-à-dire $\uu_{\normalExt}=\gG_{\normalExt}$. À l'inverse, si on n'est pas en contact ($\uu_{\normalExt}<\gG_{\normalExt}$), alors nécessairement aucun effort n'est transmis au corps déformable.
	\item $\sigmaa_{\normalInt\!\tanExt}(\uu) = 0$: Comme on est dans un modèle de contact glissant, le contact entre les deux corps ne génère aucun effort tangentiel.
\end{itemize}

\paragraph{Choix de la direction pour le calcul du gap.}
Comme mentionné plus haut, contrairement à la plupart des auteurs qui traitent le cas de l'élasticité linéaire, nous faisons le choix de calculer le gap dans la direction de $\normalExt$. En effet, si on note $\gG_{\normalInt}$ le gap dans la direction de $\normalInt$ pour la configuration non déformée, alors les conditions aux limites de contact obtenues habituellement sont:
\begin{equation}
    \left\{ \
    \begin{array}{rlr}
        \uu_{\normalInt} \leq \gG_{\normalInt}, \ \sigmaa_{\normalInt\!\normalInt}(\uu) &\leq 0, \ \sigmaa_{\normalInt\!\normalInt}(\uu)(\uu_{\normalInt}-\gG_{\normalInt})=0 & \mbox{ sur } \Gamma_C, \\
        \sigmaa_{\normalInt\!\tanInt}(\uu) &= 0 & \mbox{ sur } \Gamma_C.
     \end{array}
    \right.
    \label{CLGlissBis}
\end{equation}
Bien que ces conditions soient à première vue différentes de celles dans \eqref{FFGliss}, elles sont en réalité équivalentes sous l'hypothèse des petits déplacements. En effet, cette dernière consiste à dire que lorsque les déplacements sont suffisament petits, on peut d'une part linéariser $\normalExt$ et $\gG_{\normalExt}$ autour de la configuration de référence (non déformée). Et d'autre part, au premier ordre, on peut considérer que $\normalInt$ et $\normalExt$ sont interchangeables le long de $\Gamma_C$, et idem pour $\gG_{\normalInt}$ et $\gG_{\normalExt}$. Nous renvoyons le lecteur à \cite[Section 2.2]{KikOde1988} pour une explication plus complète et rigoureuse.
En particulier, cela implique effectivement que les conditions aux limites de \eqref{FFGliss} et \eqref{CLGlissBis} sont identiques, et donc qu'elles mènent à la même solution.

\begin{rmrk}
	Le choix de travailler avec $\normalExt$ plutôt qu'avec $\normalInt$ s'avère très pratique du point de vue de l'optimisation de formes. Plus précisément, il est plus facile de dériver $\normalExt$ et $\gG_{\normalExt}$ par rapport à la forme que $\normalInt$ et $\gG_{\normalInt}$, comme nous le verrons au chapitre \ref{chap:2.1}. 
\end{rmrk}

\subsection{Contact frottant}

Dans ce travail, nous nous limitons au modèle de frottement de Tresca, introduit pour la première fois dans \cite{DuvLio1972}. Bien que ce modèle soit relativement simple, il permet de prendre en compte des contraintes tangentielles sur la surface de contact. Du point de vue pratique, on peut l'utiliser dans une boucle de point fixe pour \textit{simuler} une loi de Coulomb (plus réaliste), voir par exemple \cite{necas1980solution,licht1991remarks}. Du point de vue théorique, il est très populaire car même s'il présente des difficultés techniques supplémentaires par rapport à un modèle de contact glissant, il dispose de toutes les bonnes propriétés (sous forme d'une inéquation variationnelle, existence et unicité de la solution, etc). 

\paragraph{Conditions aux limites.}
Commençons par introduire le coefficient de frottement $\mathfrak{F} > 0$ tel que $\mathfrak{F} : \Gamma_C \rightarrow \mathbb{R}$ est uniformément Lipschitz. L'idée du modèle de Tresca est de remplacer le seuil de Coulomb $|\sigmaa_{\normalInt \!\normalExt}(\uu)|$ par une fonction positive $s\in L^2(\Gamma_C)$ fixée (indépendante de $\uu$). Les conditions aux limites imposées aux contraintes normales sont les mêmes que dans le cas du contact glissant. Et en ce qui concerne les contraintes tangentielles, on remplace donc la dernière équation de \eqref{FFGliss} par:
\begin{equation}
    \left\{ \
    \begin{array}{rlr}
    	|\sigmaa_{\normalInt\!\tanExt}(\uu)| \:&<\: \mathfrak{F} s & \: \mbox{ sur } \{ x \in \Gamma_C \: | \: \uu_{\tanExt}(x) = 0 \}\:, \\
    	\sigmaa_{\normalInt\!\tanExt}(\uu) \:&=\: -\mathfrak{F} s  \frac{\uu_{\tanExt}}{|\uu_{\tanExt}|} & \:\mbox{ sur } \{ x \in \Gamma_C \: | \: \uu_{\tanExt}(x) \neq 0 \}\:.
     \end{array}
    \right.
    \label{CLTresca}
\end{equation}

\begin{rmrk}
	On trouve parfois ces conditions de frottement écrites sous la forme:
	\begin{equation}
    		\left\{ \
    		\begin{array}{rll}
    			|\sigmaa_{\normalInt\!\tanExt}(\uu)| \:&<\: \mathfrak{F} s & \: \Longrightarrow \ \uu_{\tanExt}=0 \:, \\
    			|\sigmaa_{\normalInt\!\tanExt}(\uu)| \:&=\: \mathfrak{F} s & \: \Longrightarrow \ \exists \alpha\geq 0, \: \uu_{\tanExt} = -\alpha \sigmaa_{\normalInt\!\tanExt}(\uu) \:.
     	\end{array}
    		\right.
    		\label{CLTrescaBis}
	\end{equation}
	Avec cette formulation, on comprend que si les contraintes tangentielles au point considéré sont en dessous d'un certain seuil, alors le contact est \textit{adhérent}. À l'inverse, si ce seuil est atteint, alors c'est qu'il y a un glissement qui s'opère dans le sens opposé aux contraintes.	
\end{rmrk}

\paragraph{Formulations.}
On introduit la fonctionnelle convexe, non-linéaire et non-différentiable $j_T : \Xx \to \mathbb{R}$ définie par:
\begin{equation*}
	j_T(\vv) := \int_{\Gamma_C} \mathfrak{F} s |\vv_{\tanExt}| \: .
\end{equation*}
Du point de vue de l'optimisation, prendre en compte les phénomènes de frottement revient à ajouter le terme correspondant à la dissipation due au frottement dans l'expression de l'énergie mécanique totale. Ici, cela signifie qu'on cherche la solution de:
\begin{equation}
    \inf_{\vv\in\Kk} \ \varphi(\vv) + j_T(\vv)\:.
    \label{OPTTresca}
\end{equation}
Il s'agit d'un problème d'optimisation convexe non-lisse avec contrainte, qui rentre encore dans le cadre d'étude classique de l'analyse convexe. On dispose donc de tous les outils nécessaires pour montrer que ce problème a de bonnes propriétés.
\begin{prpstn}
	Sous les hypothèses de la proposition \ref{PropExisUniqGliss}, lorsque $\mathfrak{F}$ est strictement positive et uniformément Lipschitz, et $s\in L^2(\Gamma_C)$, le problème \eqref{OPTTresca} admet une unique solution. De plus, $\uu\in\Kk$ est solution de \eqref{OPTTresca} si et seulement si elle vérifie l'inéquation variationelle (de deuxième espèce):
    \begin{equation}
	    a(\uu,\vv-\uu) + j_T(\vv) - j_T(\uu) \:\geq\: L(\vv-\uu), \:\:\: \forall \vv \in \Kk\: .
        \label{IVTresca}
	\end{equation}
	\label{PropExisUniqTresca}
\end{prpstn}
\begin{proof}
	Compte tenu des propriétés de $j_T$, il est clair que $\varphi + j_T$ est propre, coercive, continue et strictement convexe sur $\Xx$. Donc les arguments du cas glissant s'appliquent encore, et on obtient l'existence et l'unicité de la solution.
	De plus, d'après \cite[Section 1.5]{oden1980theory}, comme $\varphi$ est Gateaux-différentiable, on a que \eqref{IVTresca} constitue une formulation équivalente de \eqref{OPTTresca}. 
\end{proof}

Comme pour le cas glissant, on a encore équivalence entre cette formulation faible et la formulation forte, voir \cite{DuvLio1972}.

\subsection{Approche par multiplicateurs de Lagrange}

Nous terminons cette section en proposant une dernière formulation du problème de contact considéré. Celle-ci s'appuie sur la théorie des multiplicateurs de Lagrange en analyse convexe, voir par exemple \cite{ito2008lagrange}. L'idée est de réécrire le problème d'optimisation \eqref{OPTTresca} en une équation variationnelle mixte, en faisant intervenir des variables \textit{duales} ou \textit{multiplicateurs de Lagrange}.

Dans les prochaines pages, nous allons faire appel à quelques notions de calcul sous-différentiel, ainsi qu'à des résultats plus généraux d'analyse convexe. Nous renvoyons le lecteur à \cite[Chapitre 1]{ekeland1999convex} pour une présentation détaillée de ces concepts.

\begin{thrm}
  Si $\uu \in \Kk$ est solution de \eqref{OPTTresca}, alors il existe un unique couple de variables duales $(\lambda,\muu) \in H^{-\frac{1}{2}}(\Gamma_C) \times \Hh^{-\frac{1}{2}}(\Gamma_C)$ tel que:
  \begin{subequations} \label{LMFTresca:all}
    \begin{align}
    a(\uu,\vv) - L(\vv) + \prodD{\lambda, \vv_{\normalExt}}{\Gamma_C} + \prodD{\muu, \vv_{\tanExt}}{\Gamma_C} &= 0\:, \hspace{1em} \forall \vv \in \Xx\:, \label{LMFTresca:1}\\
    \prodD{\lambda, \zeta }{\Gamma_C}  &\geq 0\:, \hspace{1em} \forall \:\zeta \in H^\frac{1}{2}(\Gamma_C)\:, \:\: \zeta \geq 0\:, \label{LMFTresca:2}\\
    \prodD{\lambda, \uu_{\normalExt}-\gG_{\normalExt}}{\Gamma_C} &= 0\:, \label{LMFTresca:3}\\
    \prodD{ \mathfrak{F} s, |\nuu| }{\Gamma_C} - \prodD{ \muu, \nuu }{\Gamma_C} &\geq 0\:, \hspace{1em} \forall \nuu \in \Hh^{\frac{1}{2}}(\Gamma_C)\:, \label{LMFTresca:4}\\
    \prodD{ \mathfrak{F} s, |\uu_{\tanExt}| }{\Gamma_C} - \prodD{ \muu, \uu_{\tanExt}}{\Gamma_C} &= 0\:. \label{LMFTresca:5}
    \end{align}
  \end{subequations}
  \label{ThmExistLM0}
\end{thrm}
\begin{proof}
	On commence par introduire les opérateurs de trace normale $\Lambda_{\normalExt}$ et tangentielle $\Lambda_{\tanExt}$ sur $\Gamma_C$. On a $\Lambda_{\normalExt}:\Xx\to H^{\frac{1}{2}}(\Gamma_C)$ et $\Lambda_{\tanExt}:\Xx\to \Hh^{\frac{1}{2}}(\Gamma_C)$ tels que, pour tout $\vv\in\Xx$,
	$$
		\Lambda_{\normalExt}\vv = \vv_{\normalExt}\:, \hspace{3em} \Lambda_{\tanExt}\vv = \vv_{\tanExt}\:.
	$$
	On sait que ces opérateurs sont linéaires, continus et surjectifs.
	On définit également le convexe $\mathpzc{K}\subset L^2(\Gamma_C)$ par
	$$
		\mathpzc{K}:=	\{ \zeta \in L^2(\Gamma_C) \ | \ \zeta \leq \gG_{\normalExt} \mbox{ p.p. sur } \Gamma_C\}\:,
	$$
	ainsi que la fonctionnelle $h_T:\Ll^2(\Gamma_C)\to\mathbb{R}$ telle que, pour tout $\nuu\in \Ll^2(\Gamma_C)$,
	$$
		h_T(\nuu) := \int_{\Gamma_C} \mathfrak{F} s |\nuu| \: .
	$$
	En utilisant ces notations, on peut réécrire \eqref{OPTTresca} sous la forme d'un problème d'optimisation non lisse et non contraint:
	$$
		\inf_{\vv\in\Kk} \ \varphi(\vv) + j_T(\vv) \ = \  \inf_{\vv\in\Xx} \ \varphi(\vv) + I_\mathpzc{K}(\Lambda_{\normalExt}\vv) + h_T(\Lambda_{\tanExt}\vv) \:,
	$$
	où $I_\mathpzc{K}$ représente la fonction indicatrice de $\mathpzc{K}$. Puisque $\varphi+I_\mathpzc{K}\circl\Lambda_{\normalExt}+ h_T\circl\Lambda_{\tanExt}$ est convexe, $\uu$ est solution du problème de minimisation précédent si et seulement si:
	\begin{equation*}
	\begin{aligned}
		0 \ &\in \ \partial\left( \varphi+I_\mathpzc{K}\circl\Lambda_{\normalExt}+ h_T\circl\Lambda_{\tanExt} \right)(\uu) \\
		\Longleftrightarrow \quad 0 \ &\in \ \varphi'(\uu) + \partial\left( I_\mathpzc{K}\circl\Lambda_{\normalExt} \right)(\uu) + \partial\left( h_T\circl\Lambda_{\tanExt} \right)(\uu) \\
		\Longleftrightarrow \quad  0 \ &\in \ \varphi'(\uu) + \Lambda_{\normalExt}^*\:\partial I_\mathpzc{K}(\Lambda_{\normalExt}\uu) + \Lambda_{\tanExt}^*\:\partial h_T(\Lambda_{\tanExt}\uu)\:.
	\end{aligned}
	\end{equation*}
	Or la dernière inclusion signifie exactement qu'il existe $\lambda\in \partial I_\mathpzc{K}(\uu_{\normalExt})$ et $\muu\in \partial h_T(\uu_{\tanExt})$ tels que:
	$$
		a(\uu,\vv) - L(\vv) + \prodD{\lambda,\vv_{\normalExt}}{\Gamma_C} + \prodD{\muu,\vv_{\tanExt}}{\Gamma_C} = 0\:, \quad \forall \vv\in\Xx\:.
	$$
	Il ne reste plus qu'à caractériser les sous-différentiels $\partial I_\mathpzc{K}(\uu_{\normalExt})$ et $\partial h_T(\uu_{\tanExt})$. Par définition, on a $\lambda\in\partial I_\mathpzc{K}(\uu_{\normalExt})$ ssi  $\lambda\in\ H^{-\frac{1}{2}}(\Gamma_C)$ et:
	$$
		\prodD{\lambda, \zeta - \uu_{\normalExt}}{\Gamma_C} \: \leq \: I_\mathpzc{K}(\zeta) - I_\mathpzc{K}(\uu_{\normalExt})\:, \quad \forall \zeta\in H^{\frac{1}{2}}(\Gamma_C)\:.  
	$$
	On note au passage que si $\uu_{\normalExt}\notin \mathpzc{K}$, alors le sous-différentiel est vide. Ici, ce n'est pas le cas, puisqu'on suppose que $\uu$ est solution de \eqref{OPTTresca}.\\
	De même, on a $\muu\in\partial h_T(\uu_{\tanExt})$ ssi $\muu\in\Hh^{-\frac{1}{2}}(\Gamma_C)$ et:
	$$
		\prodD{\muu, \nuu - \uu_{\tanExt}}{\Gamma_C} \: \leq \: h_T(\nuu) - h_T(\uu_{\tanExt})\:, \quad \forall \nuu\in \Hh^{\frac{1}{2}}(\Gamma_C)\:. 
	$$
	En retravaillant les deux dernières relations à partir des définitions de $I_\mathpzc{K}$ et $h_T$, on parvient à retrouver les conditions \eqref{LMFTresca:2}, \eqref{LMFTresca:3}, \eqref{LMFTresca:4}, \eqref{LMFTresca:5}.
	
	Quant à l'unicité du couple $(\lambda,\muu)$, elle découle de \eqref{LMFTresca:1}.
\end{proof}

Lorsqu'on passe de la formulation faible \eqref{LMFTresca:1} à la formulation forte, on observe que les multiplicateurs $\lambda$ et $\muu$ correspondent en fait à $-\sigmaa_{\normalInt\!\normalExt}(\uu)$ et $-\sigmaa_{\normalInt\!\tanExt}(\uu)$, respectivement. Notons que ceci est cohérent avec la régularité de ces multiplicateurs puisque d'après \eqref{FF0:1}, $\sigmaa(\uu)\in \Hh(\divv;\Omega)$ lorsque $\ff\in \Ll^2(\Gamma_C)$.

\begin{rmrk}
	Un autre point de vue équivalent consiste à voir \eqref{LMFTresca:all} comme les conditions d'optimalité du problème de point selle suivant:
	\begin{equation}
		\inf_{\vv\in\Xx} \sup_{(\eta,\xii)\in H_+\times \Bb_s} \mathcal{L}(\vv,\eta,\xii)\:,
	\label{PtSelleTresca}
	\end{equation}
	où nous avons utilisé les notations:
	\begin{equation*}
	\begin{aligned}	
		H_+ &:=\{ \eta\in H^{-\frac{1}{2}}(\Gamma_C) \ | \ \prodD{\eta,\zeta}{\Gamma_C} \geq 0, \ \forall \zeta\geq 0 \}\:, \\
		\Bb_s &:= \{ \xii\in \Hh^{-\frac{1}{2}}(\Gamma_C) \ | \ \prodD{|\xii|-\mathfrak{F}s,\nuu}{\Gamma_C} \leq 0, \ \forall \nuu \geq 0 \}\:, \\
		\mathcal{L}(\vv,\eta,\xii) &:= \varphi(\vv) + \prodD{\eta,\vv_{\normalExt}-\gG_{\normalExt}}{\Gamma_C} + \prodD{\xii, \vv_{\tanExt}}{\Gamma_C} \:.
	\end{aligned}
	\end{equation*}
	Une conséquence directe du théorème \ref{ThmExistLM0} est donc que si $\uu$ est la solution de \eqref{OPTTresca}, alors le triplet $(\uu,\lambda,\muu)$ est l'unique solution de \eqref{PtSelleTresca}.
\end{rmrk}

\begin{rmrk}
	Lorsque le seuil de Tresca $s\in L^2(\Gamma_C)$, on peut montrer (voir \cite{ito2008lagrange}) qu'on a un gain de régularité pour le multiplicateur associé à la contrainte tangentielle: $\muu\in\Ll^2(\Gamma_C)$. De plus, dans ce cas, les crochets de dualité $\Hh^{-\frac{1}{2}}(\Gamma_C)-\Hh^{\frac{1}{2}}(\Gamma_C)$ peuvent être remplacés par des produits scalaires $\Ll^2(\Gamma_C)$. En revanche, on ne peut généralement pas espérer mieux qu'une régularité $H^{-\frac{1}{2}}(\Gamma_C)$ pour $\lambda$. Nous en discuterons plus en détails dans la section suivante.
	\label{RmkRegMu}
\end{rmrk}

\section{Résolution à l'aide de formulations régularisées}

Bien que les problèmes de contact soient bien connus et que le cadre théorique pour leur étude soit bien défini, leur résolution numérique reste encore aujourd'hui un domaine de recherche actif. En effet, les non-linéarités et les non-différentiabilités inhérentes à ces problèmes, que nous avons mis en évidence dans la section précédente, rendent inefficaces ou inapplicables une grande partie des méthodes numériques classiques. 

L'approche la plus répandue pour surmonter ces difficultés consiste à considérer des formulations légèrement différentes mais plus régulières, donc plus faciles à résoudre, et dont la solution est « suffisamment proche » (en un sens qu'on précisera) de la solution $\uu$ de la formulation d'origine. Parmi les méthodes qui suivent cette approche, nous choisissons dans cette thèse de nous intéresser à la méthode de \textit{pénalisation} et à la méthode de \textit{lagrangien augmenté}, car elles sont les deux plus utilisées dans le domaine de l'industrie.

\subsection{Formulation pénalisée}

La méthode de pénalisation a d'abord été introduite pour traiter le cas général des problèmes d'optimisation sous contrainte, voir par exemple \cite{Lio1969}. Le principe consiste à remplacer la contrainte par l'ajout d'un terme de pénalité dans la fonctionnelle à minimiser. On transforme ainsi le problème d'optimisation sous contrainte en un problème d'optimisation sans contrainte. Encore mieux, si ce terme de pénalité est suffisamment régulier et que la fonctionnelle initiale est différentiable, alors la condition d'optimalité associée à ce nouveau problème aura une forme relativement simple.

Bien que cette méthode soit déjà utilisée dans le domaine de l'optimisation sous contrainte depuis les années 1950-1960, ses premières applications aux problèmes de contact datent des années 1980, nous citons entre autres \cite{KikSon1981,WriSimTay1985,SimBer1986,KikOde1988}. Elle a connu depuis une immense popularité, notamment dans l'industrie, grâce à sa facilité d'implémentation. Plus récemment, on peut également citer \cite{chouly2013convergence} pour une étude détaillée de la discrétisation du problème de Tresca pénalisé, par la méthode des éléments finis.

Du point de vue de l'optimisation, le problème \eqref{OPTTresca} présente deux principales diffcultés: d'abord la contrainte $\uu\in \Kk$, et ensuite la non-différentiabilité de $j_T$. Ici, la méthode de pénalisation va donc nous permettre à la fois de relaxer la contrainte, et de régulariser la fonction non-lisse $j_T$.

\paragraph{Mise en équations.}
Soit $\varepsilon > 0$ le \textit{paramètre de pénalisation}. On commence d'abord par traiter la contrainte $\uu\in \Kk$. Au lieu de considérer le problème de minimisation de $\varphi+j_T$ sous contrainte, on considère  celui de $\varphi+j_T+j_\varepsilon$ sans contrainte, où $j_\varepsilon$ est un terme qui pénalise le non-respect de la contrainte. Pour pouvoir appliquer les résultats classiques des méthodes de pénalisation, on doit choisir $j_\varepsilon=\frac{1}{\varepsilon}j$, où $j$ désigne une fonction convexe, propre, semi-continue inférieurement (s.c.i.), et vérifiant
\begin{equation*}
  j(\vv)=0 \: \Longleftrightarrow \:\vv \in \Kk \:, \hspace{1em} \mbox{et} \hspace{1em} j(\vv) \geq 0 \:\:\: \forall \vv\in \Xx\:.
\end{equation*}
Ici, comme dans \cite{KikSon1981}, nous choisissons $j(\vv):=\frac{1}{2}\norml \, \maxx\left(\vv_{\normalExt}-\gG_{\normalExt}\right)\normr_{0,\Gamma_C}^2$, pour tout $\vv\in \Xx$, où $\maxx$ représente la projection sur $\mathbb{R}_+$, aussi appelée fonction partie positive: $\maxx:=\max(0,\cdot)$. Il est clair qu'un tel $j$ satisfait les conditions énoncées plus haut, et que $j$ est de plus différentiable. La fonction $j_\varepsilon$ ainsi obtenue peut s'interpréter comme une régularisation de la fonction indicatrice $I_{\Kk}$.

Ensuite, on peut également utiliser le paramètre de pénalisation comme paramètre de régularisation pour définir $j_{T,\varepsilon}$, une approximation régulière de $j_T$:
\begin{equation*}
    j_{T,\varepsilon}(\vv) := \int_{\Gamma_C} Q_\varepsilon(\vv_{\tanExt})\:,
\end{equation*}
où la fonction $Q_\varepsilon: \mathbb{R}^{d-1} \to \mathbb{R}$ est une régularisation de $\mathfrak{F}s|\cdot|$ telle que, pour tout $z\in\mathbb{R}^{d-1}$,
\begin{equation*}
    Q_\varepsilon(z) := \left\{
    \begin{array}{lr}
         \dfrac{1}{2\varepsilon}\left( |z|^2 + (\varepsilon\mathfrak{F}s)^2\right) & \mbox{ si } |z|\leq \varepsilon\mathfrak{F}s ,\\
         \mathfrak{F}s|z| & \mbox{ sinon.}
    \end{array}
    \right.
\end{equation*}
Contrairement à $\mathfrak{F}s|\cdot|$, on voit effectivement que $Q_\varepsilon$ est Fréchet différentiable. De plus, sa dérivée est donnée par:
$$
	Q_\varepsilon'(z) = \frac{1}{\varepsilon}\qq(\varepsilon\mathfrak{F}s,z)\:, \ \forall z\in \mathbb{R}^{d-1}\:,
$$
où nous avons introduit la fonction $\qq$ définie sur $\mathbb{R}_+\times\mathbb{R}^{d-1}$ par: $\forall (\alpha,z) \in \mathbb{R}_+\times\mathbb{R}^{d-1}$,
\begin{equation*}
    \qq(\alpha,z) := \left\{
    \begin{array}{lr}
         z & \mbox{ si } |z|\leq \alpha ,\\
         \alpha\dfrac{z}{|z|} & \mbox{ sinon.}
    \end{array}
    \right.
\end{equation*}
\begin{rmrk}
	Pour tout $\alpha\in\mathbb{R}_+$, la fonction $\qq(\alpha,\cdot)$ correspond à la projection sur la boule $\mathcal{B}(0,\alpha)$ dans $\mathbb{R}^{d-1}$.
\end{rmrk}

On peut maintenant introduire la formulation pénalisée du problème de contact avec frottement de Tresca:
\begin{equation}
  \underset{\vv \in \Xx}{\inf} \ \varphi(\vv) + j_{T,\varepsilon}(\vv) + j_\varepsilon(\vv) \: .
  \label{OPTPena}
\end{equation}
Comme annoncé, il s'agit d'une version régularisée de \eqref{OPTTresca}, qui prend la forme d'un problème d'optimisation lisse sans contrainte. Grâce aux propriétés de $j_\varepsilon$ et $j_{T,\varepsilon}$, on peut montrer que ce problème est bien posé et qu'il peut se réécrire en une formulation variationnelle.
\begin{prpstn}
	Sous les hypothèses de la proposition \ref{PropExisUniqTresca}, le problème \eqref{OPTPena} admet une unique solution. De plus, $\uu_\varepsilon\in\Xx$ est solution de \eqref{OPTPena} si et seulement si elle vérifie la formulation variationelle:
    \begin{equation}
	    a(\uu_\varepsilon,\vv) + \frac{1}{\varepsilon} \prodL2{\maxx\left(\uu_{\varepsilon,\normalExt} -\gG_{\normalExt}\right), \vv_{\normalExt}}{\Gamma_C}  + \frac{1}{\varepsilon} \prodL2{\qq(\varepsilon\mathfrak{F}s,\uu_{\varepsilon,\tanExt}), \vv_{\tanExt}}{\Gamma_C} = L(\vv), \:\:\: \forall \vv \in \Xx\: .
        \label{FVPena}
	\end{equation}	
	\label{PropExisUniqPena}	
\end{prpstn}
\begin{proof}
	Puisque $j_\varepsilon$ et $j_{T,\varepsilon}$ sont convexes, continues et positives,  la fonctionnelle $\varphi + j_\varepsilon + j_{T,\varepsilon}$ est strictement convexe, coercive et continue sur $\Xx$. Elle admet donc un unique minimum $\uu_\varepsilon$ sur $\Xx$. Par ailleurs, comme cette fonctionnelle est Fréchet différentiable et $\Xx$ est un espace vectoriel, on sait que $\uu_\varepsilon$ est entièrement caractérisée par la condition d'optimalité:
	$$	
		\prodD{\varphi'(\uu_\varepsilon) + j_\varepsilon'(\uu_\varepsilon) + j_{T,\varepsilon}'(\uu_\varepsilon), \vv}{\Xx^*,\Xx} = 0\:, \quad \forall\vv \in \Xx\:,
	$$
	qui correspond précisément à \eqref{FVPena}.
\end{proof}

\begin{rmrk}
	La version pénalisée du problème de contact prend donc la forme d'une formulation variationnelle. Cette équation est non-linéaire et non-différentiable, mais elle est néanmoins plus facile à résoudre numériquement qu'une inéquation variationnelle. À la fin de cette section, nous présenterons une méthode numérique permettant de traiter ce type de problèmes. 
\end{rmrk}

\paragraph{Rôle du paramètre de pénalisation.}
Le paramètre $\varepsilon$ intervient à la fois dans $j_\varepsilon$ et dans $j_{T,\varepsilon}$. Nous avons vu que dans les deux cas, il pouvait s'interpréter comme un paramètre de régularisation, de sorte que lorsque $\varepsilon\to 0$, les fonctions $j_\varepsilon$ et $j_{T,\varepsilon}$ « se rapprochent » respectivement de $I_{\Kk}$ et $j_T$. Évidemment, lorsqu'on introduit une formulation approchée, on veut que la solution  de cette formulation soit une bonne approximation de la solution d'origine. Dans notre cas, cela revient à dire qu'on veut que $\uu_\varepsilon \to \uu$ quand $\varepsilon\to 0$. Pour les choix de régularisation que nous avons fait, on a bien un tel résultat de convergence, voir par exemple \cite{chouly2013convergence}.
\begin{thrm}
	Sous les hypothèses de la proposition \ref{PropExisUniqPena}, on a $\uu_\varepsilon\to\uu$ fortement dans $\Xx$ lorsque $\varepsilon\to 0$.
\end{thrm}

\subsection{Formulation lagrangien Augmenté}

Cette approche consiste aussi en une régularisation, mais contrairement à la précédente, cette régularisation n'a pas d'impact sur la solution. En d'autres termes, là où la formulation pénalisée donne une solution $\uu_\varepsilon \neq \uu$ et dépendant du paramètre $\varepsilon$, la formulation lagrangien augmenté donne la solution $\uu$, indépendamment du paramètre de régularisation choisi. Le principe consiste à régulariser le lagrangien $\mathcal{L}$ de \eqref{PtSelleTresca} sans en changer le point selle. On obtient alors une version plus régulière de \eqref{LMFTresca:all}, qu'on peut résoudre par une méthode de type point fixe.

Historiquement, cette approche remonte à la fin des années 1960 avec \cite{Hes1969} et \cite{Pow1969} pour des problèmes d'optimisation non-linéaire sous contrainte d'égalité. Elle n'a été utilisée dans le cas des problèmes de contact que plus tard, dans les années 1980-1990, on peut citer notamment \cite{WriSimTay1985,LanTay1986} pour le contact glissant, et \cite{AlaCur1991,SimLau1992,LauSim1993,HeeCur1993} pour le contact frottant.
Pour le cadre théorique de la méthode en optimisation convexe non-lisse dans des espaces de Hilbert, nous citons \cite{ItoKun2000a}. Pour les applications à la résolution numérique de problèmes aux limites, nous renvoyons le lecteur à l'ouvrage de référence \cite{ForGlo1983a}.

\paragraph{Régularité de $\lambda$.}
Nous reprenons ici la discussion dans \cite[Section 4.4]{stadler2004infinite}. Pour qu'on puisse introduire la régularisation désirée, il faut que $(\lambda,\muu)\in L^2(\Gamma_C)\times \Ll^2(\Gamma_C)$. Comme nous l'avons mentionné dans la remarque \ref{RmkRegMu}, même si nous pouvons obtenir cette régularité pour $\muu$ en imposant $s\in L^2(\Gamma_C)$, rien ne garantit en général une telle régularité pour $\lambda$. Commençons donc par énoncer cette hypothèse nécessaire à l'application de la méthode, puis nous nous intéresserons à un cas particulier dans lequel elle est vérifiée.
\begin{hyp}
	 Le multiplicateur $\lambda$ est dans $L^2(\Gamma_C)$.
	\label{hyp:RegLambda}
\end{hyp}

Toujours dans \cite[Section 4.4]{stadler2004infinite}, l'auteur propose une condition suffisante pour que cette hypothèse soit vérifiée. 
\begin{prpstn}
 Soit $(\uu,\lambda,\muu)$ solution de \eqref{LMFTresca:all} avec $\ff\in \Ll^2(\Gamma_C)$, $\tauu\in\Hh^{\frac{1}{2}}(\Gamma_C)$. Si $\overline{\{ \uu_{\normalExt}-\gG_{\normalExt}=0\}} \subset \Gamma_C$, alors l'hypothèse \ref{hyp:RegLambda} est vérifiée.
\end{prpstn}

\begin{proof}
	Dans le cas où $\ff\in \Ll^2(\Gamma_C)$, $\tauu\in\Hh^{\frac{1}{2}}(\Gamma_C)$, et lorsque $\Omega$, $\Gamma_C$ sont suffisament réguliers, on sait d'après \cite{kinderlehrer1981remarks} que $\lambda=-\sigmaa_{\normalInt\!\normalExt}(\uu)\in L^2_{loc}(\Gamma_C)$.
Puis, puisque $\lambda$ est donc défini presque partout, on peut réécrire les conditions \eqref{LMFTresca:2} et \eqref{LMFTresca:3} en des conditions p.p. sur $\Gamma_C$, ce qui donne en particulier $\lambda=0$ p.p. sur $\Gamma_C\setminus \{ \uu_{\normalExt}-\gG_{\normalExt}=0\}$. Par conséquent, sous l'hypothèse $\{ \uu_{\normalExt}-\gG_{\normalExt}=0\}$ est strictement contenu dans $\Gamma_C$, il vient que $\lambda\in L^2_{loc}(\Gamma_C)$ implique $\lambda\in L^2(\Gamma_C)$.
\end{proof}

\begin{rmrk}
	La condition $\overline{\{ \uu_{\normalExt}-\gG_{\normalExt}=0\}} \subset \Gamma_C$ est en théorie difficile (voire impossible) à vérifier puisqu'elle dépend de $\uu$ qui n'est pas connue a priori. En pratique, dans de nombreux cas, la géométrie du problème et le choix de la zone de contact potentiel $\Gamma_C$ permettent de prédire qu'on aura bien cette condition. 
\end{rmrk}

\paragraph{Mise en équations.}
Soient $\gamma_1$, $\gamma_2>0$. On introduit le lagrangien augmenté $\mathcal{L}_{\gamma}$ défini par: $\forall (\vv,\eta,\xii) \in \Xx\times L^2(\Gamma_C)\times\Ll^2(\Gamma_C)$,
\begin{equation}
\begin{aligned}
	\mathcal{L}_\gamma(\vv,\eta,\xii) := \varphi(\vv) &+ \frac{1}{2\gamma_1} \left( \norml \maxx \left( \eta + \gamma_1(\vv_{\normalExt}-\gG_{\normalExt})\right) \normr_{0,\Gamma_C}^2 - \norml \eta \normr_{0,\Gamma_C}^2 \right) \\
	&+ \frac{1}{2\gamma_2} \left( \norml \qq(\mathfrak{F}s,\xii + \gamma_2\vv_{\tanExt})\normr_{0,\Gamma_C}^2 - \norml \xii \normr_{0,\Gamma_C}^2 \right) \:,
\end{aligned}
\label{DefLagAug}
\end{equation}
On peut maintenant énoncer le problème de point selle associé à ce nouveau lagrangien:
\begin{equation}
	\inf_{\vv\in\Xx} \sup_{(\eta,\xii)\in L^2\times \Ll^2} \mathcal{L}_\gamma(\vv,\eta,\xii)\:.
	\label{PtSelleLagAug}
\end{equation}

\begin{rmrk}
	Contrairement à \eqref{PtSelleTresca}, le problème de point selle associé au lagrangien augmenté est posé sur tout l'espace vectoriel $\Xx\times L^2(\Gamma_C)\times\Ll^2(\Gamma_C)$ sur lequel $\mathcal{L}_\gamma$ est défini.
\end{rmrk}

\begin{thrm}
	On suppose que l'hypothèse \ref{hyp:RegLambda} est vérifiée. Alors, pour tous $\gamma_1$, $\gamma_2>0$, le problème \eqref{PtSelleLagAug} admet pour unique solution le triplet $(\uu,\lambda,\muu)$ solution de \eqref{PtSelleTresca}.
\end{thrm}

\begin{proof}
	Nous renvoyons par exemple à \cite{ito2008lagrange} ou \cite{cohen2004optimisation} pour une présentation dans le cas plus général. L'idée est de commencer par montrer que l'expression \eqref{DefLagAug} est obtenue à partir d'une régularisation de type Moreau-Yosida appliquée à $\mathcal{L}$, voir le chapitre \ref{chap:3.1} pour le cas glissant. Ensuite, grâce aux propriétés de cette régularisation, on peut montrer que la solution de \eqref{PtSelleTresca} vérifie les conditions d'optimalité associées à \eqref{PtSelleLagAug}, et que \eqref{PtSelleLagAug} admet une unique solution entièrement caractérisée par ces conditions d'optimalité. Ce qui donne le résultat désiré.  
\end{proof}

Ce problème nous donne bien la même solution que \eqref{PtSelleTresca}, et ce quel que soit le choix des paramètres $\gamma_1$, $\gamma_2$. Maintenant, nous pouvons donc caractériser cette solution à partir des conditions d'optimalité associées à \eqref{PtSelleLagAug}, qui s'écriront comme des équations.

\begin{prpstn}
	Soient $\gamma_1$, $\gamma_2>0$. Le triplet $(\uu,\lambda,\muu)\in \Xx\times L^2(\Gamma_C)\times\Ll^2(\Gamma_C)$ est solution de \eqref{PtSelleLagAug} si et seulement si il vérifie les conditions d'optimalité:
  \begin{subequations} \label{LMFLagAug:all}
    \begin{align}
    a(\uu,\vv) - L(\vv) + \prodL2{\lambda, \vv_{\normalExt}}{\Gamma_C} + \prodL2{\muu, \vv_{\tanExt}}{\Gamma_C} &= 0\:, \hspace{1em} \forall \vv \in \Xx\:, \label{LMFLagAug:1}\\
    \prodL2{\lambda, \zeta }{\Gamma_C} - \prodL2{ \maxx\left( \lambda + \gamma_1(\uu_{\normalExt}-\gG_{\normalExt}) \right), \zeta}{\Gamma_C} &= 0\:, \hspace{1em} \forall \:\zeta \in L^2(\Gamma_C)\:, \label{LMFLagAug:2}\\
    \prodL2{ \muu, \nuu }{\Gamma_C} - \prodL2{ \qq(\mathfrak{F}s,\muu+\gamma_2\uu_{\tanExt}, \nuu }{\Gamma_C} &= 0\:, \hspace{1em} \forall \nuu \in \Ll^2(\Gamma_C)\:. \label{LMFLagAug:3}
    \end{align}
  \end{subequations}
\end{prpstn}

\begin{proof}
	Avant toute chose, on rappelle que si $B$ est un sous-ensemble convexe fermé non-vide d'un espace de Hilbert $H$, alors l'opérateur de projection sur $B$, noté $\Proj_B$, est bien défini, et de plus l'application $\frac{1}{2}\norml \Proj_B(\cdot) \normr_H^2 : H\to \mathbb{R}$ est Fréchet différentiable. En outre, la dérivée directionnelle de cette application en un point $x$ dans la direction $h$ est donnée par $(\Proj_B(x),h)_H$. D'après ces propriétés, $\mathcal{L}_\gamma$ est bien différentiable car $\maxx$ et $\qq(\mathfrak{F}s,\cdot)$ sont des opérateurs de projection, et l'écriture des conditions d'optimalité associées à \eqref{PtSelleLagAug} donne directement \eqref{LMFLagAug:all}.
\end{proof}

\begin{rmrk}
	Comme toutes les fonctions sont dans $L^2(\Gamma_C)/\Ll^2(\Gamma_C)$, on peut réécrire les conditions \eqref{LMFLagAug:2} et \eqref{LMFLagAug:3} comme des égalités presque partout sur $\Gamma_C$:
  \begin{subequations} \label{CondLagAugPP:all}
    \begin{align}
    \lambda &= \maxx\left( \lambda + \gamma_1(\uu_{\normalExt}-\gG_{\normalExt}) \right)\:, \hspace{1em} \mbox{ p.p. sur } \Gamma_C\:, \label{CondLagAugPP:1}\\
    \muu &= \qq(\mathfrak{F}s,\muu + \gamma_2\uu_{\tanExt}) \:, \hspace{1em} \mbox{ p.p. sur } \Gamma_C\:. \label{CondLagAugPP:2}
    \end{align}
  \end{subequations}
  Lorsque $(\uu,\lambda,\muu)$ est solution de \eqref{LMFTresca:all} et $(\lambda,\muu)\in L^2(\Gamma_C)\times\Ll^2(\Gamma_C)$, alors, à partir des conditions $\uu\in\Kk$, \eqref{LMFTresca:2}, \eqref{LMFTresca:3} on peut montrer que $(\uu,\lambda)$ vérifie \eqref{CondLagAugPP:1} pour tout $\gamma_1>0$. De la même manière, à partir de \eqref{LMFTresca:4} et \eqref{LMFTresca:5} on peut montrer que $(\uu,\muu)$ vérifie \eqref{CondLagAugPP:2} pour tout $\gamma_2>0$. Nous renvoyons à \cite{stadler2004infinite} pour le détail des calculs.
\end{rmrk}

\paragraph{Méthode itérative.}
Dans ce qui précède, nous avons proposé une formulation régularisée équivalente dans les cas où notre problème d'origine est suffisamment régulier. Cette formulation régularisée s'écrit comme une formulation variationnelle non-linéaire mixte. Une des approches pour résoudre ce genre de problèmes consiste à découpler les variables en utilisant un algorithme itératif de type point fixe. C'est celle que nous suivons ici avec la \textit{méthode de lagrangien augmenté}, dont nous rappelons les différentes étapes ci-dessous.

\begin{algorithm}
\caption*{\textbf{Algorithme:} Méthode de lagrangien augmenté}
\begin{enumerate}
  \item Choisir $(\lambda^0,\muu^0)\in L^2(\Gamma_C)\times\Ll^2(\Gamma_C)$ et initialiser $k=0$.
  \item Choisir $\gamma_1^{k+1}$, $\gamma_2^{k+1}>0$, puis trouver $\uu^{k+1}\in\Xx$ la solution de, $\forall \vv \in \Xx$,
    \begin{equation} 
    \begin{aligned} 
        a(\uu^{k+1},\vv) &+ \prodL2{\maxx\left( \lambda^k + \gamma_1^{k+1} (\uu^{k+1}_{\normalExt}-\gG_{\normalExt}) \right), \vv_{\normalExt}}{\Gamma_C} \\
        &+ \prodL2{ \qq\left( \mathfrak{F} s ,\muu^k + \gamma_2^{k+1} \uu^{k+1}_{\tanExt} )\right), \vv_{\tanExt}}{\Gamma_C} = L(\vv) \:.
        \label{IterLagAug}
    \end{aligned}
    \end{equation}
    \item Mettre à jour les multiplicateurs selon la règle: 
    \begin{subequations}\label{MAJMultLagAug:all}
      \begin{align}
      \lambda^{k+1} &= \maxx\left( \lambda^k + \gamma_1^{k+1} (\uu^{k+1}_{\normalExt}-\gG_{\normalExt}) \right) \:\: \mbox{ p.p. sur } \Gamma_C \: , \label{MAJMultLagAug:1}\\
        \muu^{k+1} &= \qq\left( \mathfrak{F} s ,\muu^k + \gamma_2^{k+1} \uu^{k+1}_{\tanExt} )\right)  \:\: \mbox{ p.p. sur } \Gamma_C \: \label{MAJMultLagAug:2}.
      \end{align}
    \end{subequations}
    \item Tant que le critère de convergence n'est pas respecté, mettre à jour $k=k+1$ et retourner à l'étape 2.
\end{enumerate}
\end{algorithm}

\begin{rmrk}
	Contrairement à la méthode d'Uzawa (voir \cite{uzawa1958iterative}), $\lambda^{k+1}$ et $\muu^{k+1}$ ne sont pas donnés explicitement en fonction de la solution de l'itération précédente. La méthode de lagrangien augmenté est d'ailleurs parfois vue comme une version implicite de la méthode d'Uzawa, cf \cite{ItoKun2000a}. Du point de vue de la résolution, la conséquence de ce choix est que le problèmes aux limites \eqref{IterLagAug} à résoudre à chaque itération est non-linéaire. Bien que la résolution d'une itération soit plus compliquée que dans le cas complètement explicite, cette méthode permet d'obtenir un résultat de convergence indépendant du choix des paramètres $\gamma_1$, $\gamma_2$, comme nous allons le voir. 
\end{rmrk}

\begin{rmrk}
	La formulation \eqref{IterLagAug} est exactement du même type que la formulation pénalisée \eqref{FVPena}: à savoir non-linéaire et non-différentiable. Elles pourront donc être traitées par la même méthode numérique (voir sous-section suivante).
\end{rmrk}

Il est bien connu que la méthode de lagrangien augmenté converge vers la solution de \eqref{PtSelleLagAug}, et ce quel que soit le choix des paramètres de régularisation. Nous énonçons le résultat suivant (voir par exemple \cite{stadler2004infinite} pour une preuve détaillée dans le cas particulier du contact avec frottement de Tresca).

\begin{thrm}
  Pour n'importe quel choix de paramètres $0<\gamma_1^1\leq\gamma_1^2\leq \cdots$, et $0<\gamma_2^1\leq \gamma_2^2\leq \cdots$, les itérées $\uu^k$ convergent vers $\uu$ fortement dans $\Xx$. De plus, les multiplicateurs $(\lambda^k,\muu^k)$ convergent vers $(\lambda,\muu)$ faiblement dans $L^2(\Gamma_C)\times\Ll^2(\Gamma_C)$.
  \label{ThmCvLagAug}
\end{thrm}

\begin{rmrk}
	Si on choisit $(\lambda^0,\muu^0)\in H^{\frac{1}{2}}(\Gamma_C)\times\Hh^\frac{1}{2}(\Gamma_C)$, alors compte tenu des formules \eqref{MAJMultLagAug:all}, on aura également $(\lambda^k,\muu^k)\in H^{\frac{1}{2}}(\Gamma_C)\times\Hh^\frac{1}{2}(\Gamma_C)$ pour tout $k>0$.
\end{rmrk}

\subsection{Résolution numérique}

Compte tenu des formulations approchées auxquelles nous nous intéressons, nous avons besoin d'une méthode numérique capable de traiter les formulations variationnelles (elliptiques) non-linéaires et non-différentiables. Ceci exclut donc la méthode de Newton  « classique ». Cependant, de nombreux travaux proposent des généralisations de la méthode de Newton dans les cas non-différentiables. On parle alors de \textit{méthode de Newton généralisée} ou \textit{méthode de Newton semi-lisse}. Pour le cas de la dimension finie, nous citons par exemple les papiers \cite{qi1993nonsmooth,qi1993convergence,alart1997methode} et l'ouvrage \cite{facchinei2007finite}. Pour le cas de la dimension infinie, nous renvoyons à \cite{chen2000smoothing,kummer2000generalized,hintermuller2002primal,ito2003semi,ulbrich2002semismooth}
ainsi qu'à \cite[Chapitre 8]{ito2008lagrange} pour la méthode générale, et à \cite{stadler2004semismooth,kunisch2005generalized,stadler2004infinite} pour des applications aux problèmes de contact.

\paragraph{Généralités.}
Nous suivons ici l'approche proposée dans \cite{stadler2004infinite} et \cite[Chapitre 8]{ito2008lagrange}. On commence par rappeler les définitions et principaux résultats. Le principe de base consiste à introduire une notion plus faible de différentiabilité sous laquelle on sera capable d'obtenir un résultat de convergence pour la méthode itérative proposée.

Dans cette section, $X$, $Y$ et $Z$ sont des espaces de Banach. On définit tout d'abord la notion de différentiabilité au sens de Newton telle qu'introduite dans \cite{hintermuller2002primal}.

\begin{dfntn}
	On dit que la fonction  $F:X\to Y$ est \textit{différentiable au sens de Newton} ou \textit{Newton différentiable} sur l'ouvert $U\subset X$ s'il existe une application $G:U\to \mathcal{L}(X,Y)$ telle que:
	$$
		\lim_{h\to 0}\ \frac{1}{\norml h \normr_X} \norml F(x+h) - F(x) - G(x+h)h \normr_Y \ = \ 0\:, 
	$$
	pour tout $x\in U$. L'application $G$ est alors appelée \textit{dérivée généralisée} de $F$.
\end{dfntn}

\begin{rmrk}
	La définition précédente requiert l'existence d'une dérivée généralisée, mais il n'est pas nécessaire qu'elle soit unique.
\end{rmrk}

On donne également une règle de dérivation en chaîne pour les dérivées généralisées dans un cas qui nous sera utile: celui de la composée d'une application affine et d'une fonction non-lisse (voir \cite[Théorèmes 2.10 et 2.11]{stadler2004infinite}).

\begin{thrm}
	Soient $L:X\to Y$ une application affine de la forme $Lx := Bx+b$ avec $B\in \mathcal{L}(X,Y)$, $F:Y\to Z$ une fonction Newton différentiable sur un ouvert $U$, et $G$ une dérivée généralisée de $F$. Si $L^{-1}(U)$ est non vide, alors $F\circ L$ est Newton différentiable sur $L^{-1}(U)$, et de plus $G(Lx)B$ est une dérivée généralisée de $F\circ L$ en $x$ pour tout $x\in L^{-1}(U)$.
	\label{ThmNChainRule1}
\end{thrm} 

\begin{thrm}
	Soient $L:Y\to Z$ une application affine de la forme $Ly := By+b$ avec $B\in \mathcal{L}(Y,Z)$, $F:X\to Y$ une fonction Newton différentiable sur un ouvert $U$, et $G$ une dérivée généralisée de $F$ telle que $\left\{ \norml G(x)\normr \ |\ x\in U \right\}$ est borné. Alors $L\circ F$ est Newton différentiable sur $U$, et de plus $BG(x)$ est une dérivée généralisée de $L\circ F$ en $x$ pour tout $x\in U$.
	\label{ThmNChainRule2}
\end{thrm} 

Supposons qu'on cherche à résoudre l'équation $F(x)=0$, avec $F$ non-lisse. On peut utiliser la notion de différentiabilité introduite plus haut pour proposer une méthode de type Newton. L'énoncé suivant présente une telle méthode, et donne un résultat de convergence.

\begin{thrm}
	On suppose que $x^*\in X$ est une solution de $F(x)=0$, que $F$ est Newton différentiable sur $U$ voisinage de $x^*$. Soit $G$ une dérivée généralisée de $F$ sur $U$ telle que $G(x)$ est inversible pour tout $x\in U$. Si l'ensemble $\left\{ \norml G(x)^{-1}\normr \ |\ x\in U \right\}$ est borné, alors la suite d'itérées $\{ x^l \}_l$ définie par
	$$
		x^{l+1} = x^l - G(x^l)^{-1}F(x^l)
	$$
	converge super-linéairement vers $x^*$, à condition que $\norml x^0-x^*\normr_X$ soit suffisamment petit.
	\label{ThmCvgSSN}
\end{thrm}

Nous renvoyons le lecteur à \cite[Chapitre 8]{ito2008lagrange} (entre autres) pour une preuve de ce résultat.

\paragraph{Cas du contact.} Nous nous appuyons ici sur la méthode de Newton semi-lisse proposée dans \cite[Section 3.8]{stadler2004infinite} pour résoudre une version régularisée du problème de Tresca. En particulier, l'auteur prouve la différentiabilité au sens de Newton pour $\maxx$ et $\qq(\mathfrak{F}s,\cdot)$ (voir définitions p.43). Avant d'énoncer ce résultat, on introduit les fonctions $G_+:\mathbb{R}\to\mathcal{L}(\mathbb{R})$ et $G_s:\mathbb{R}^{d-1}\to\mathcal{L}(\mathbb{R}^{d-1})$ telles que: pour tout $x\in\mathbb{R}$, pour tout $z\in\mathbb{R}^{d-1}$,
$$
	G_+(x):= \left\{
    \begin{array}{lr}
         0 & \mbox{ si } x\leq 0 ,\\
         1 & \mbox{ sinon,}
    \end{array}
    \right.
    \ \mbox{ et } \ 
    G_s(z) := \left\{
    \begin{array}{lr}
         \Id & \mbox{ si } |z|\leq \mathfrak{F}s ,\\
         \frac{\mathfrak{F}s}{|z|}\big(\Id  -  \frac{1}{|z|^2} z \otimes  z\big) & \mbox{ sinon,}
    \end{array}
    \right.
$$
où $z\otimes z$ représente le produit extérieur de $z$ par lui-même, parfois écrit $z^T z$ au sens matriciel.

\begin{notation}
	Dans la suite, par abus de notation, $\maxx$ pourra désigner à la fois la fonction de $\mathbb{R}$ dans $\mathbb{R}$ et l'opérateur correspondant de $L^q(\Gamma_C)$ dans $L^p(\Gamma_C)$. Il en va de même pour $\qq$, $G_+$ et $G_s$.
\end{notation}

\begin{prpstn}
	Soient $1\leq p < q < +\infty$. Alors $\maxx:L^q(\Gamma_C)\to L^p(\Gamma_C)$ et $\qq(\mathfrak{F}s,\cdot):\Ll^q(\Gamma_C)\to \Ll^p(\Gamma_C)$ sont Newton différentiables sur leurs ensembles de définition. De plus, $G_+:L^q(\Gamma_C)\to\mathcal{L}(L^q(\Gamma_C),L^p(\Gamma_C))$ et $G_s:\Ll^q(\Gamma_C)\to \mathcal{L}(\Ll^q(\Gamma_C),\Ll^p(\Gamma_C))$ sont des dérivées généralisées de $\maxx$ et $\qq(\mathfrak{F}s,\cdot)$, respectivement.
	\label{PropNDiffPQ}
\end{prpstn}

Avant d'utiliser ce résultat pour traiter les formulations qui nous intéressent, notons que pour tous $x\in\mathbb{R}$, $z\in\mathbb{R}^{d-1}$, les applications $G_+(x)$ et $G_s(z)$ vérifient:
\begin{itemize}
	\item $G_+(x)y\cdot y \geq 0$ pour tout $y\in\mathbb{R}$,
	\item $G_s(z)h\cdot h \geq 0$ pour tout $h\in\mathbb{R}^{d-1}$.
\end{itemize}

\paragraph{Application au contact pénalisé.}
On commence par réécrire \eqref{FVPena} sous la forme plus adéquate $F_\varepsilon(\uu_\varepsilon)=0$. Pour cela, on a besoin d'introduire l'opérateur $A\in \mathcal{L}_c(\Xx,\Xx^*)$, et $\ell \in \Xx^*$, tels que pour tous $\vv$, $\ww\in\Xx$,
$$
	\prodD{A\vv,\ww}{\Xx^*,\Xx}=a(\vv,\ww)\:, \hspace{2em} \prodD{\ell,\ww}{\Xx^*,\Xx}=L(\ww)\:.
$$
Grâce à ces notations, on peut définir $F_\varepsilon:\Xx\to\Xx^*$ par: $\forall \vv\in \Xx$,
$$
	F_\varepsilon(\vv) := A\vv + \frac{1}{\varepsilon}\:\Lambda_{\normalExt}^*\maxx(\Lambda_{\normalExt}\vv-\gG_{\normalExt}) + \frac{1}{\varepsilon}\:\Lambda_{\tanExt}^* \qq(\varepsilon\mathfrak{F}s,\Lambda_{\tanExt}\vv) - \ell\:.,
$$
où $\Lambda_{\normalExt}$ et $\Lambda_{\tanExt}$ représentent respectivement les traces normale et tangentielle sur $\Gamma_C$.
Par définition de $F_\varepsilon$, il est clair que $\uu_\varepsilon$ est solution de \eqref{FVPena} ssi $F_\varepsilon(\uu_\varepsilon)=0$. On peut donc appliquer la méthode présentée plus haut. En effet, comme $\Lambda_{\normalExt}\vv-\gG_{\normalExt} \in H^{\frac{1}{2}}(\Gamma_C)\hookrightarrow L^q(\Gamma_C)$ et $\Lambda_{\tanExt}\vv \in \Hh^{\frac{1}{2}}(\Gamma_C)\hookrightarrow \Ll^q(\Gamma_C)$ avec $q>2$, on déduit de la Proposition \ref{PropNDiffPQ} que $F_\varepsilon$ est Newton différentiable sur $\Xx$. Et de plus, à partir des dérivées généralisées de $\maxx$ et $\qq(\mathfrak{F}s,\cdot)$ et des Théorèmes \ref{ThmNChainRule1} et \ref{ThmNChainRule2}, on est capable de déterminer une dérivée généralisée $G_\varepsilon:\Xx\to\mathcal{L}(\Xx,\Xx^*)$ de $F_\varepsilon$ qui s'écrit: $\forall \vv\in \Xx$,
\begin{equation}
	G_\varepsilon(\vv) = A + \frac{1}{\varepsilon}\:\Lambda_{\normalExt}^*G_+(\Lambda_{\normalExt}\vv-\gG_{\normalExt})\Lambda_{\normalExt} + \frac{1}{\varepsilon}\:\Lambda_{\tanExt}^* G_{\varepsilon s}(\Lambda_{\tanExt}\vv)\Lambda_{\tanExt}\:.
	\label{ExpDerGenFPena}
\end{equation}

\begin{crllr}
	Soit $\uu_\varepsilon$ la solution de \eqref{FVPena}. La suite d'itérées $\{ \uu_\varepsilon^l \}_l$ définie par
	$$
		\uu_\varepsilon^{l+1} = \uu_\varepsilon^l - {G_\varepsilon(\uu_\varepsilon^l)}^{-1}F_\varepsilon(\uu_\varepsilon^l)
	$$
	converge super-linéairement vers $\uu_\varepsilon$, à condition que $\norml \uu_\varepsilon^0-\uu_\varepsilon\normr_{\Xx}$ soit suffisamment petit.
	\label{ThmCvgSSNPena}
\end{crllr}

\begin{proof}
	Il suffit de montrer que $\left\{ \norml G_\varepsilon(\vv)^{-1}\normr \ |\ \vv\in \Xx \right\}$ est borné, puis on pourra appliquer le Théorème \ref{ThmCvgSSN}. Pour cela, prenons $\vv\in\Xx$ et $\psii\in\Xx^*$, et considérons le problème: trouver $\uu_{\psii}\in\Xx$ solution de
	$$
		G_\varepsilon(\vv)\uu_{\psii} = \psii\:.
	$$
	On peut réécrire ce problème sous la forme équivalente: trouver $\uu_{\psii}\in\Xx$ solution de
	\begin{equation*}
	\begin{aligned}
		a(\uu_{\psii},\ww) &+ \frac{1}{\varepsilon}\prodL2{G_+(\vv_{\normalExt}-\gG_{\normalExt})\uu_{\psii,\normalExt},\ww_{\normalExt}}{\Gamma_C} \\
		&+ \frac{1}{\varepsilon}\prodL2{G_{\varepsilon s}(\vv_{\tanExt})\uu_{\psii,\tanExt},\ww_{\tanExt}}{\Gamma_C} = \prodD{\psii,\ww}{\Xx^*,\Xx}\:, \ \forall \ww\in\Xx\:.
	\end{aligned}
	\end{equation*}
	D'après les propriétés de $G_+$ et $G_{\varepsilon s}$, on sait que les deux derniers termes du membre de gauche constituent des formes bilinéaires continues et positives. Ainsi, tout le membre de gauche est une forme bilinéaire continue et coercive. Comme le membre de droite est clairement une forme linéaire continue sur $\Xx$, on déduit du lemme de Lax-Milgram l'existence d'un unique $\uu_{\psii}=G_\varepsilon(\vv)^{-1}\psii\in\Xx$ solution. De plus, en utilisant la coercivité de $a$, on obtient l'estimation:
	$$
		\norml G_\varepsilon(\vv)^{-1}\psii \normr_{\Xx} \: \leq \: \frac{1}{\alpha_0} \norml \psii \normr_{\Xx^*}\:,
	$$
	où $\alpha_0$ désigne la constante d'ellipticité du tenseur d'élasticité $\Aa$. Cette inégalité permet de conclure, puisque $\psii$ et $\vv$ ont été choisis arbitrairement.
\end{proof}

\paragraph{Application au contact par lagrangien augmenté.} Comme nous l'avons mentionné dans la section précédente, la formulation \eqref{IterLagAug} considérée à chaque itération de la méthode de lagrangien augmenté est presque identique à \eqref{FVPena}. On peut donc suivre le même raisonnement que précédemment, et introduire $F^k:\Xx\to\Xx^*$ telle que: $\forall \vv \in \Xx$,
$$
	F^k(\vv) := A\vv + \Lambda_{\normalExt}^*\maxx\left(\lambda^{k-1}+\gamma_1^k(\Lambda_{\normalExt}\vv-\gG_{\normalExt})\right) + \Lambda_{\tanExt}^* \qq\left(\mathfrak{F}s,\muu^{k-1}+\gamma_2^k\Lambda_{\tanExt}\vv\right) - \ell\:, 
$$
de sorte que $\uu^k\in\Xx$ est solution de \eqref{IterLagAug} si et seulement si $F^k(\uu^k)=0$. De plus, on a vu que lorsque $(\lambda^0,\muu^0)\in H^{\frac{1}{2}}(\Gamma_C)\times \Hh^{\frac{1}{2}}(\Gamma_C)$, alors $(\lambda^{k-1},\muu^{k-1})\in H^{\frac{1}{2}}(\Gamma_C)\times \Hh^{\frac{1}{2}}(\Gamma_C)$. Ainsi, comme pour la formulation pénalisée, la fonction $F^k$ est Newton différentiable, et on a une dérivée généralisée $G^k:\Xx\to \mathcal{L}(\Xx,\Xx^*)$ définie par: $\forall \vv\in \Xx$,
\begin{equation}
	G^k(\vv) = A + \Lambda_{\normalExt}^*G_+\left(\lambda^{k-1}+\gamma_1^k(\Lambda_{\normalExt}\vv-\gG_{\normalExt})\right) \Lambda_{\normalExt} + \Lambda_{\tanExt}^* G_{s}\left(\muu^{k-1}+\gamma_2^k\Lambda_{\tanExt}\vv\right) \Lambda_{\tanExt}\:.
	\label{ExpDerGenFLagAug}
\end{equation}
On a également convergence de la méthode de Newton semi-lisse dans ce cas-là.

\begin{crllr}
	Soit $\uu^k$ la solution de \eqref{IterLagAug}. La suite d'itérées $\{ \uu^{k,l} \}_l$ définie par
	$$
		\uu^{k,l+1} = \uu^{k,l} - {G^k(\uu^{k,l})}^{-1}F^k(\uu^{k,l})
	$$
	converge super-linéairement vers $\uu^k$, à condition que $\norml \uu^{k,0}-\uu^k\normr_{\Xx}$ soit suffisamment petit.
	\label{ThmCvgSSNLagAug}
\end{crllr}

\part{Formulation pénalisée des problèmes de contact}

\chapter{Shape derivatives for the penalty formulation of contact problems with Tresca friction}     
\label{chap:2.1}                   

\section*{Résumé}

  Dans cet article, on s'intéresse au problème d'optimisation de formes d'un solide linéaire élastique en contact frottant (Tresca) avec une fondation rigide. La solution de la formulation pénalisée associée à ce modèle n'est en général pas dérivable par rapport à la forme car cette-dernière fait intervenir des opérateurs de projection. Par ailleurs, pour pouvoir optimiser la zone de contact, il est nécessaire d'avoir une méthode de calcul pour le gap et pour sa dérivée de forme. Nous proposons ici une approche par dérivées directionnelles pour traiter cette non-différentiabilité. Nous montrons d'abord que la solution admet toujours des dérivées directionnelles de forme, puis nous donnons des conditions suffisantes pour qu'elle soit dérivable par rapport à la forme. Nous présentons enfin quelques résultats numériques obtenus avec un algorithme de type gradient s'appuyant sur ces dérivées.

\section*{Abstract}

\begin{otherlanguage*}{english}
  In this article, the shape optimization of a linear elastic body subject to frictional (Tresca) contact is investigated. Due to the projection operators involved in the formulation of the contact problem, the solution is not shape differentiable in general. Moreover, shape optimization of the contact zone requires the computation of the gap between the bodies in contact, as well as its shape derivative. Working with directional derivatives, sufficient conditions for shape differentiability are derived.  
  Then, some numerical results, obtained with a gradient descent algorithm based on those shape derivatives, are presented.
\end{otherlanguage*}


\section{Introduction}

Optimal design is becoming a key element in industrial conception and applications. As the interest to include shape optimization in the design cycle of structures broadens, we are confronted with increasingly complex mechanical context. Large deformations, plasticity, contact and such can lead to difficult mathematical formulation. The non-linearities and/or non-differentiabilities stemming from the mechanical model give rise to complex shape sensitivity analysis which often requires a specific and delicate treatment.

This article deals with bodies in frictional (Tresca model) contact with a rigid foundation. Therefore this model is concerned with the non-penetrability and the eventual friction of the bodies in contact. From the mathematical point of view, it takes the form of an elliptic variational inequality of the second kind, see for example \cite{DuvLio1972,boieri1987existence} for existence, uniqueness, and regularity results.

Our approach to solve shape optimization problems is based on a gradient descent and Hadamard's boundary variation method, which requires the shape derivative of the cost functional. Such approaches, following the pioneer work \cite{hadamard1908memoire}, have been widely studied for the past forty years, for example in \cite{murat1975etude,simon1980differentiation,pironneau1982optimal,sokolowski1992introduction,DelZol2001,henrot2006variation}, to name a few. Obviously this raises the question of the differentiability of the cost functional with respect to the domain, which naturally leads to shape sensitivity analysis of the associated variational inequality. More specifically, as in \cite{allaire2004structural}, we use a level-set representation of the shapes, which allows to deal with changes of topology during the optimization process. Regarding more general topology optimization methods, let us mention \textit{density methods}, in which the shape is represented by a local density of material inside a given fixed domain. Among the most popular, we cite \cite{BenSig2003} for the \textit{SIMP method} (Solid Isotropic Material with Penalisation) and  \cite{allaire2012shape} for the \textit{homogenization method}.

As projection operators are involved in the formulation, the solution map is non-differentiable with respect to the shape or any other control parameter. There exist three main approaches to treat this non-differentiability.
The first one was introduced in \cite{mignot1984optimal}, where the author proves differentiability in a weaker sense, namely conical differentiability, and derives optimality conditions using this notion. We mention \cite{sokolowski1988shape} for the application of this method to shape sensitivity analysis of contact problems with Tresca friction. 
Another approach consists in discretizing the formulation, then use the tools from subdifferential calculus, see the series of papers \cite{kovccvara1994optimization,beremlijski2002shape,haslinger2012shape,beremlijski2014shape}, in the context of shape optimization for elastic bodies in frictional contact with a plane. 
The third approach, which is very popular in mechanical engineering, is to consider the penalized contact problem, which takes the form of a variational equality, then regularize all non-smooth functions. This leads to an approximate formulation having a Fréchet differentiable solution map. 
%
%
%
Following this penalty/regularization approach, we mention \cite{kim2000meshless} for two-dimensional parametric shape optimization, where the authors consider contact with a general rigid foundation and get interested in the differentiation of the gap. We also mention the more recent work \cite{maury2017shape} for shape optimization using the level set method (see \cite{allaire2004structural}), where the authors compute shape derivatives for the continuous problem in two and three dimensions, but do not take into account a possible gap between the bodies in contact. 
The same approach can be found in the context of optimal control, see among others \cite{ito2000optimal} for the general framework, and \cite{amassad2002optimal} for the specific case of frictional (Tresca) contact mechanics.

Let us finally mention the substantial work of Haslinger et al., who proved existence of optimal shapes for contact problems in some specific cases, see \cite{haslinger1985existence,haslinger1985shape}. Moreover, in  \cite{haslinger1986shape}, they proved consistency of the penalty approach in this context. 

In this paper, we aim at expressing shape derivatives for the continuous penalty formulation of frictional contact problems of Tresca type. Our approach is similar to the penalty/regularization, but we do not regularize non-smooth functions involved in the formulation. Indeed, shape differentiability does not require Fréchet-differentiability of the solution map, which makes the regularization step unnecessary.
Especially, the goal is to get similar results to \cite{maury2017shape} whitout regularizing, and extend those results in two ways.
First, we add a gap in the formulation, which enables to completely optimize the contact zone, as in \cite{kim2000meshless}. 
Second, we work in the slightly more general case where the Tresca threshold is not necessarily constant. 
This way, the formulae obtained could also be used in the context of the numerical approximation of a regularized Coulomb friction law by a fixed-point of Tresca problems. We refer to \cite{DuvLio1972,oden1983nonlocal,cocu1984existence} for existence and uniqueness results for this regularized Coulomb problem, and to \cite{hueber2008primal} (among others) for its numerical resolution by means of a fixed-point algorithm. 

This work is structured as follows. Section 2 presents the problem, its formulation and some related notations. Section 3 is dedicated to shape optimization. Especially, we express sufficient conditions for the solution of the penalty formulation to be shape differentiable (Theorem \ref{ThmExistMDer} and Corollary \ref{CorExistSDer}). The shape optimization algorithm of gradient type, based on those shape derivatives, is briefly discussed. Finally, in section 5, some numerical results are exposed.

\section{Problem formulation}

\subsection{Geometrical setting}\label{sec:geo.2}
The body $\Omega \subset \mathbb{R}^d$, $d \in \{2,3\}$, is assumed to have $\pazocal{C}^1$ boundary, and to be in contact with a rigid foundation $\Omega_{rig}$, which has a $\pazocal{C}^3$ compact boundary $\partial\Omega_{rig}$, see Figure \ref{SchOpen.2}. Let $\Gamma_D$ be the part of the boundary where a homogenous Dirichlet conditions applies (blue part), $\Gamma_N$ the part where a non-homogenous Neumann condition $\tauu$ applies (orange part), $\Gamma_C$ the potential contact zone (green part), and $\Gamma$ the rest of the boundary, which is free of any constraint (i.e. homogenous Neumann boundary condition). Those four parts are mutually disjoint and moreover: $\overline{\Gamma_D} \cup \overline{\Gamma_N} \cup \overline{\Gamma_C} \cup \overline{\Gamma} = \partial \Omega$. In order to avoid technical difficulties, it is assumed that $\overline{\Gamma_C} \cap \overline{\Gamma_D} = \emptyset$.

The outward normal to $\Omega$ is denoted $\normalInt$. Similarly, the inward normal to $\Omega_{rig}$ is denoted $\normalExt$.

\subsection{Notations and function spaces}\label{sec:esp.2}
Throughout this article, for any $\pazocal{O}\subset \mathbb{R}^d$, $L^p(\pazocal{O})$ represents the usual set of $p$-th power measurable functions on $\pazocal{O}$, and $\left(L^p(\pazocal{O})\right)^d = \Ll^p(\pazocal{O})$. The scalar product defined on $L^2(\pazocal{O})$ or $\Ll^2(\pazocal{O})$ is denoted (without distinction) by  $\prodL2{\cdot,\cdot}{\pazocal{O}}$ and its norm $\|\cdot\|_{0,\pazocal{O}}$. 

The Sobolev spaces, denoted $W^{m,p}(\pazocal{O})$ with $p\in [1,+\infty]$, $p$ integer are defined as 
$$
W^{m,p}(\pazocal{O}) = \left\{u\in L^p(\pazocal{O})\: : \: D^{\alpha} u \in L^p(\pazocal{O})\ \forall |\alpha|\le m\right\},
$$
where $\alpha$ is a multi-index in $\mathbb{N}^d$ and $\Ww^{m,p}(\pazocal{O})= \left(W^{m,p}(\pazocal{O})\right)^d$. The spaces $W^{s,2}(\pazocal{O})$ and $\Ww^{s,2}(\pazocal{O})$, $s\in \mathbb{R}$, are denoted $H^s(\pazocal{O})$ and $\Hh^s(\pazocal{O})$ respectively. Their norm are denoted $\|\cdot\|_{s,\pazocal{O}}$. 

The subspace of functions in $H^s(\pazocal{O})$ and $\Hh^s(\pazocal{O})$ that vanish on a part of the boundary $\gamma\subset\partial \pazocal{O}$ are denoted $H^s_\gamma(\pazocal{O})$ and $\Hh^s_\gamma(\pazocal{O})$. In particular, we denote the vector space of admissible displacements $\Xx:=\Hh^1_{\Gamma_D}(\Omega)$, and $\Xx^*$ its dual.

In order to fit the notations of functions spaces, vector-valued functions are denoted in bold. For example, $w\in L^2(\Omega)$ while $\ww\in \Ll^2(\Omega)$.

For any $v$ vector in $\mathbb{R}^d$, the product with the normal $v \cdot \normalInt$ (respectively with the normal to the rigid foundation $v \cdot \normalExt$) is denoted $v_{\normalInt}$ (respectively $v_{\normalExt}$). Similarly, the tangential part of  $v$ is denoted $v_{\tanInt} = v - v_{\normalInt} \normalInt$ (respectively $v_{\tanExt} = v - v_{\normalExt} \normalExt$).

Finally, the space of second order tensors in $\mathbb{R}^d$, i.e. the space of linear maps from $\mathbb{R}^d$ to $\mathbb{R}^d$, is denoted $\mathbb{T}^2$. In the same way, the space of fourth order tensors is denoted $\mathbb{T}^4$.
\subsection{Mechanical model}\label{sec:meca.2}

In this work the material is assumed to verify the linear elasticity hypothesis (small deformations and Hooke's law, see for example \cite{Cia1988a}), associated with the small displacements assumption (see \cite{kikuchi1988contact}).
The physical displacement is denoted $\uu$, and belongs to $\Xx$.
The stress tensor is defined by $\sigmaa(\uu) = \Aa : \epsilonn(\uu)$, where $\boldsymbol{\epsilon}(\uu)=\frac{1}{2}(\gradd\uu+\gradd\uu^T)$ denotes the linearized strain tensor, and $\Aa$ is the elasticity tensor. This elasticity tensor is a fourth order tensor belonging to $L^\infty(\Omega,\mathbb{T}^4)$, and it is assumed to be elliptic (with constant $\alpha_0>0$). Regarding external forces, the body force $\ff \in \Ll^2(\Omega)$, and traction (or surface load) $\tauu \in \Ll^2(\Gamma_N)$. 
\begin{figure}
\begin{center}
\begin{tikzpicture}

\fill [blue!60] plot [smooth cycle] 
coordinates {(3.9,1.8) (4.055,1.64) (4.095,1.5) (4.055,1.34) (3.73,1.175) (3.76,1.5)};

\draw[black,fill=blue!15, thin,densely dotted,rotate=-15] (3.3,2.425) ellipse (0.05cm and .33cm);

\fill [orange!50] plot [smooth  cycle] 
coordinates {(0.2,1.22) (0.3,1.6) (0.05,1.22) (0.03,0.9) (0.15, 1.)};

\draw[black,fill=orange!15, thin,densely dotted,rotate=-23] (-0.3,1.22) ellipse (0.05cm and .42cm);

\fill [dartmouthgreen!60] plot [smooth  cycle] 
coordinates {(2.795,.50) (2.87,.1) (2.65,-.55) (1.5, -.55) (.5,-.3) (0.15,.5) (1.5, .6)};

\draw[black, fill=dartmouthgreen!25, thin,dotted, rotate=-0.5] (1.51,0.485) ellipse (1.4cm and .12cm);

\draw [black] plot [smooth cycle] 
coordinates {(0,1) (0.3, 1.6) (1,2) (2, 2.1) (3,2) (3.9,1.8) (4.055,1.64) (4.095,1.5) (4.055,1.34) (3.73,1.175) (3,1) (2.7,-.5) (1,-.5) (0.3, -.1)};

\draw[black, fill=dartmouthgreen!60,thin,dashed,rotate=-.5] (1.6,-.39) ellipse (.9cm and .1cm);

\draw[black, ultra thin] (-2,-0.75) -- (2.75,-0.75) -- (4.75,-.25) -- (0, -.25) -- (-2,-0.75);

\node[] at (2.9,2.3) {$\Gamma$};
\node[] at (4.5,1.25) {$\Gamma_D$};
\node[] at (-.25,1.5) {$\Gamma_N$};
\node[] at (3.3,0.) {$\Gamma_C$};
\node[] at (1.7,1.25) {$\Omega$};
\draw[->] (0.3, 0.1) -- (-0.1, -0.1) ;
\node[] at (0.5, 0.1) {$x$};
\node[] at (-0.3, 0.2) {$\normalInt(x)$};
\draw[red, densely dashed] (0.3, 0.1) -- (0.3, -0.6) ;
\node[red] at (0.3,-0.6) {\tiny{$\bullet$}};
\node[red] at (-0.3, -0.4) {$\gG_{\normalExt}(x)$};
\draw[->] (0.3 , -0.6) -- (0.3, -1.1) ;
\node[] at (.75, -0.97) {$\normalExt(x)$};

\end{tikzpicture}
\end{center}
  \caption{Elastic body in contact with a rigid foundation.}
  \label{SchOpen.2}
\end{figure}
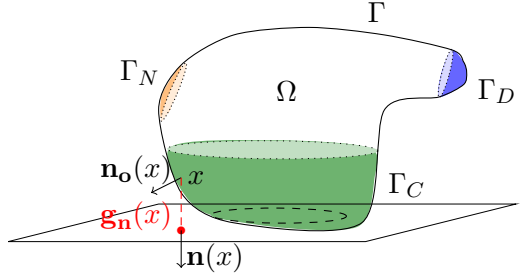

\subsection{Non-penetration condition} At each point $x$ of $\Gamma_C$, let us define the gap $\gG_{\normalExt}(x)$, as the oriented distance function to $\Omega_{rig}$ at $x$, see Figure \ref{SchOpen.2}. 
Due to the regularity of the rigid foundation, there exists $h$ sufficiently small such that
\begin{equation*}
    \partial\Omega_{rig}^h := \{ x\in\mathbb{R}^d \: : \: |\gG_{\normalExt}(x)| < h \}  \:, 
\end{equation*}
is a neighbourhood of $\partial\Omega_{rig}$ where $\gG_{\normalExt}$ is of class $\pazocal{C}^3$, see \cite{delfour1995boundary}. In particular, this ensures that $\normalExt$ is well defined on $\partial\Omega_{rig}^h$, and that $\normalExt\in \pazocal{C}^2(\partial\Omega_{rig}^h,\mathbb{R}^d)$. Moreover, in the context of small displacements, it can be assumed that the potential contact zone $\Gamma_C$ is such that $\Gamma_C \subset \partial\Omega_{rig}^h$. Hence there exists a neighbourhood of $\Gamma_C$ such that $\gG_{\normalExt}$ and $\normalExt$ are of class $\pazocal{C}^3$ and $\pazocal{C}^2$, respectively.

The non-penetration condition can be stated as follows: $\uu_{\normalExt}\leq \gG_{\normalExt}$ a.e.$\!$ on $\Gamma_C$. Thus, we introduce the closed convex set of admissible displacements that realize this condition, see \cite{eck2005unilateral}:
\begin{displaymath}
	\Kk:=\{\vv \in \Xx \: : \: \vv_{\normalExt}\leq \gG_{\normalExt} \:\:\mbox{a.e.$\!$ on}\: \Gamma_C\} .
\end{displaymath}
\subsection{Mathematical formulation of the problem}
Let us introduce the bilinear and linear forms $a : \Xx \times \Xx \rightarrow \mathbb{R}$ and $L : \Xx \rightarrow \mathbb{R}$, such that:
\begin{equation*}
  a(\uu,\vv) := \int_\Omega \Aa : \epsilonn(\uu) : \epsilonn(\vv) \:, \hspace{1.5em} L(\vv) := \int_\Omega \ff \vv + \int_{\Gamma_N} \tauu \vv \:.
\end{equation*}
According to the assumptions of the previous sections, one is able to show (see \cite{Cia1988a}) that $a$ is $\Xx$-elliptic with constant $\alpha_0$ (ellipticity of $\Aa$ and Korn's inequality), symmetric, continuous, and that $L$ is continuous (regularity of $\ff$ and $\tauu$). 

The unknown displacement $\uu$ of the frictionless contact problem is the minimizer of the total mechanical energy of the elastic body, which reads, in the case of pure sliding (unilateral) contact problems:
\begin{equation}
	\underset{\vv\in \Kk}{\inf} \:\: \varphi(\vv) \: := \underset{\vv\in \Kk}{\inf} \:\: \frac{1}{2}a(\vv,\vv)-L(\vv) \:.
 	\label{OPT0.2}
\end{equation}
It is clear that the space $\Xx$, equipped with the usual $\Hh^1$ norm, is a Hilbert space. Moreover, under the conditions of the previous section, since $\Kk$ is obviously non-empty and the energy functional is strictly convex, continuous and coercive, we are able to conclude (see e.g.$\!$ \cite[Chapter 1]{ekeland1999convex}) that $\uu$ solution of \eqref{OPT0.2} exists and is unique.

It is well known that \eqref{OPT0.2} may be rewritten as a variational inequality (of the first kind):
\begin{equation}
	a(\uu,\vv-\uu) \:\geq\: L(\vv-\uu), \:\:\: \forall \vv \in \Kk\: .
    \label{IV0.2}
\end{equation}
Moreover, as $\ff\in\Ll^2(\Omega)$ and $\tauu\in \Ll^2(\Gamma_N)$, it can be shown (see \cite{DuvLio1972}) that \eqref{OPT0.2} and \eqref{IV0.2} are also equivalent to the strong formulation: 
\begin{subequations} \label{FF0:all}
    \begin{align}
    - \Divv \sigmaa(\uu) &= \ff & \mbox{in } \Omega, \label{FF0:1}\\
      \uu                &= 0   & \mbox{on } \Gamma_D,  \label{FF0:2}\\
    \sigmaa(\uu) \cdot \normalInt &= \tauu & \mbox{on } \Gamma_N,   \label{FF0:3}\\
    \sigmaa(\uu) \cdot \normalInt &= 0 & \mbox{on } \Gamma,   \label{FF0:4}\\
     \uu_{\normalExt} \leq \gG_{\normalExt}, 
     \sigmaa_{\normalInt\!\normalExt}(\uu) &\leq 0, \sigmaa_{\normalInt\!\normalExt}(\uu)(\uu_{\normalExt}-\gG_{\normalExt})=0 & \mbox{on } \Gamma_C,   \label{FF0:5}\\
    \sigmaa_{\normalInt\!\tanExt}(\uu) &= 0 & \mbox{on } \Gamma_C,  \label{FF0:6}
    \end{align}
\end{subequations}
where $\sigmaa_{\normalInt\!\normalExt}(\uu) = \sigmaa(\uu)\cdot\normalInt\cdot\normalExt$ and $\sigmaa_{\normalInt\!\tanExt}(\uu) = \sigmaa(\uu)\cdot\normalInt-\sigmaa_{\normalInt\!\normalExt}(\uu)\normalExt$ are the normal and tangential constraints on $\Gamma_C$.
\begin{remark}
	Note that existence and uniqueness of the solution to \eqref{OPT0.2} $\uu\in \Xx$ also holds under weaker assumptions on the data, namely $\ff\in\Xx^*$ and $\tauu\in \Hh^{-\frac{1}{2}}(\Gamma_N)$ (under the appropriate modifications in the definition of $L$). Here, we choose the minimal regularity that ensures equivalence between \eqref{IV0.2} and \eqref{FF0:all}. Regarding regularity results, the reader is referred to \cite{kinderlehrer1981remarks}.
\end{remark}

\begin{remark}
  Conditions \eqref{FF0:5} and \eqref{FF0:6} may seem different from the usual
  \begin{subequations} \label{CL0:all}
    \begin{align}
       \uu_{\normalInt} \leq \gG_{\normalInt}, \hspace{1em}\sigmaa_{\normalInt\! \normalInt}(\uu) &\leq 0, \hspace{1em}\sigmaa_{\normalInt\! \normalInt}(\uu)(\uu_{\normalInt}-\gG_{\normalInt})=0 & \mbox{on } \Gamma_C, \label{CL0:1}\\
      \sigmaa_{\normalInt\! \tanInt}(\uu) &= 0 & \mbox{on } \Gamma_C,  \label{CL0:2}
    \end{align}
  \end{subequations}
  where $\gG_{\normalInt}(x)$ denotes the distance between $x\in \Gamma_C$ and the rigid foundation computed in the direction of the normal $\normalInt$ to $\Gamma_C$. 
  %
  %
  Actually, since we assume the deformable body undergoes small displacements relative to its reference configuration, both sets of conditions are equivalent. 
  %
  %
  %
  More specifically, from the small displacement hypothesis, the normal vector $\normalExt$ and the gap $\gG_{\normalExt}$ to the rigid foundation can be replaced by $\normalInt$ and $\gG_{\normalInt}$ 
  (we refer to \cite[Chapter 2]{kikuchi1988contact} for the details).

  Therefore, in our context, writing the formulation associated to the contact problem using $\normalInt$ or $\normalExt$ makes absolutely no difference. We choose the latter formulation because it proves itself very convenient when dealing with shape optimization, see Section \ref{sec:shapeopt.2}.
  \label{RemHypHPP}
\end{remark}

\subsection{Friction condition}
Let $\mathfrak{F} : \Gamma_C \rightarrow \mathbb{R}$, $\mathfrak{F} > 0$, be the friction coefficient. The basis of Tresca model is to replace the usual Coulomb threshold $|\sigmaa_{\normalInt \!\normalExt}(\uu)|$ by a fixed strictly positive function $s$, which leads to the following conditions on $\Gamma_C$:
\begin{equation}
  \left\{
	\hspace{0.5em}    
    \begin{aligned}
    	|\sigmaa_{\normalInt\!\tanExt}(\uu)| \:&<\: \mathfrak{F} s & \:\: \mbox{on } \{ x \in \Gamma_C \: : \: \uu_{\tanExt}(x) = 0 \}\:, \\
    	\sigmaa_{\normalInt\!\tanExt}(\uu) \:&=\: -\mathfrak{F} s  \frac{\uu_{\tanExt}}{|\uu_{\tanExt}|} & \:\: \mbox{on } \{ x \in \Gamma_C \: : \: \uu_{\tanExt}(x) \neq 0 \}\:,
    	\label{CL1}
    \end{aligned}
  \right.    	
\end{equation}
which represent respectively \textit{sticking} and \textit{sliding} points.
\begin{remark}
	Of course, replacing the Coulomb threshold by the fixed function $s$ leads to a simplified and approximate model of friction. Especially, in the Tresca model, there may exist points $x\in \Gamma_C$ such that $\sigmaa_{\normalInt\!\normalExt}(\uu)(x) = 0$ and $\uu_{\tanExt}(x)\neq 0$, in which case $\sigmaa_{\normalInt\!\tanExt}(\uu)(x)\neq 0$. In other words, friction can occur even if there is no contact.
\end{remark}
In order to avoid regularity issues, it is assumed that $\mathfrak{F}$ is uniformly Lipschitz continuous and $s \in L^2(\Gamma_C)$. Before stating the minimization problem in this case, let us introduce the non-linear functional $j_T : \Xx \rightarrow \mathbb{R} $ defined by:
\begin{equation*}
	j_T(\vv) := \int_{\Gamma_C} \mathfrak{F} s |\vv_{\tanExt}| \: .
\end{equation*}
With these notations, since considering the Tresca friction model means taking into account the frictional term $j_T$ in the energy functional, the associated minimization problem writes:
\begin{equation}
	\underset{\vv\in \Kk}{\inf} \:\: \varphi(\vv)+j_T(\vv) \:.
 	\label{OPT1}
\end{equation}
Since the additional term $j_T$ in the functional is convex, positive and  continuous 
one can deduce existence and uniqueness of the solution $\uu\in \Xx$, see for example \cite[Section 1.5]{oden1980theory}.
From this reference, one also gets that \eqref{OPT1} can be equivalently rewritten as a variational inequality (of the second kind):
\begin{equation}
	a(\uu,\vv-\uu) + j_T(\vv) - j_T(\uu) \:\geq\: L(\vv-\uu)\:, \:\:\: \forall \vv \in \Kk\: .
    \label{IV1}
\end{equation}
Again, from \cite{DuvLio1972} problems \eqref{OPT1} and \eqref{IV1} are equivalent to the strong formulation \eqref{FF0:all}, except for the last condition \eqref{FF0:6}, which is replaced by the two conditions \eqref{CL1}.

\begin{remark}
	Since $\uu\in \Xx$, the regularity of $\sigmaa(\uu)\cdot\normalInt$ is in general $\Hh^{-\frac{1}{2}}(\Gamma_C)$. However, even though no better regularity can be expected for the normal component $\sigmaa_{\normalInt\!\normalExt}$ in the general case, one has that the tangential component $\sigmaa_{\normalInt\!\tanExt}\in \Ll^2(\Gamma_C)$ when $s\in L^2(\Gamma_C)$. We refer to \cite[Chapter 4]{stadler2004infinite} for further details.
\end{remark}

\subsection{Penalty formulation}

The formulation that will be studied here originate from the classical penalty method, see \cite{Lio1969} and \cite{aubin2007approximation} for the general method,  \cite{kikuchi1981penalty} or \cite{kikuchi1988contact} for its application to unilateral contact problems, and \cite{chouly2013convergence} for its application to the Tresca friction problem. 
This formulation reads: find $\uu_\varepsilon$ in $\Xx$ such that, for all $\vv \in \Xx$,
\begin{equation}
  a(\uu_\varepsilon,\vv) + \frac{1}{\varepsilon} \prodL2{\maxx(\uu_{\varepsilon,\normalExt} -\gG_{\normalExt}), \vv_{\normalExt}}{\Gamma_C}  + \frac{1}{\varepsilon} \prodL2{\qq(\varepsilon\mathfrak{F}s,\uu_{\varepsilon,\tanExt}), \vv_{\tanExt}}{\Gamma_C} = L(\vv) \:,
  \label{FV}
\end{equation}
where $\maxx$ denotes the projection onto $\mathbb{R}_+$ in $\mathbb{R}$ (also called the positive part function) and, for any $\alpha\in\mathbb{R}_+$, $\qq(\alpha,\cdot)$ denotes the projection onto the ball $\mathcal{B}(0,\alpha)$ in $\mathbb{R}^{d-1}$. Those projections admit analytical expressions: for all $y\in\mathbb{R}$, $z\in\mathbb{R}^{d-1}$:
\begin{equation*}
    \maxx(y):=\max\{0,y\} \:,
\qquad
    \qq(\alpha,z) := \left\{
    \begin{array}{lr}
         z & \mbox{ if } |z|\leq \alpha ,\\
         \alpha\dfrac{z}{|z|} & \mbox{ else.}
    \end{array}
    \right.
\end{equation*}
It is well known (see for example \cite{kikuchi1981penalty,chouly2013convergence} or \cite[Section 6.5]{kikuchi1988contact}) that \eqref{FV} admits a unique solution $\uu_\varepsilon\in\Xx$. Moreover, from the same references, one gets that passing to the limit $\varepsilon \to 0$ leads to $\uu_\varepsilon \to \uu$ strongly in $\Xx$.

\begin{remark}
    Formulation \eqref{FV} is actually the optimality condition related to the unconstrained differentiable optimization problem derived from \eqref{OPT1}:
$$
   \underset{\vv \in \Xx}{\inf} \left\{ \varphi(\vv) + j_{T,\varepsilon}(\vv) + j_\varepsilon(\vv) \right\} \:,
$$
where $j_\varepsilon$ is a penalty term introduced to relax the constraint $\vv\in\Kk$, and $j_{T,\varepsilon}$ is a regularization of $j_T$.

Moreover, in this model, one gets from \eqref{FV} that the non-penetration conditions \eqref{FF0:5}-\eqref{FF0:6} and the friction condition \eqref{CL1} rewrite:
  \begin{subequations} \label{CL2:all}
    \begin{align}
       \sigmaa_{\normalInt\! \normalExt}(\uu_\varepsilon) &= -\frac{1}{\varepsilon} \maxx(\uu_{\varepsilon,\normalExt} -\gG_{\normalExt}) & \mbox{on } \Gamma_C, \label{CL2:1}\\
      \sigmaa_{\normalInt\! \tanExt}(\uu_\varepsilon) &= - \frac{1}{\varepsilon} \qq(\varepsilon\mathfrak{F}s,\uu_{\varepsilon,\tanExt}) & \mbox{on } \Gamma_C.  \label{CL2:2}
    \end{align}
  \end{subequations}
  From those expressions, one deduces the new definitions for the sets of points of particular interest:
  \begin{itemize}
  	\item points in contact: $\{ x \in \Gamma_C \: |\: \uu_{\varepsilon,\normalExt}\geq \gG_{\normalExt} \}$,
  	\item sticking points: $\{ x \in \Gamma_C \: |\:\: |\uu_{\varepsilon,\tanExt}|\leq \varepsilon\mathfrak{F}s \}$,
  	\item sliding points: $\{ x \in \Gamma_C \: |\:\: |\uu_{\varepsilon,\tanExt}|\geq \varepsilon\mathfrak{F}s \}$.
  \end{itemize}
\end{remark}

\section{Shape optimization}
\label{sec:shapeopt.2}

Given a cost functional $J(\Omega)$ depending explicitly on the domain $\Omega$, and also implicitly, through $y(\Omega)$ the solution of some variational problem on $\Omega$, the optimization of $J$ with respect to $\Omega$ or \textit{shape optimization problem} reads:
\begin{equation}
    \inf_{\Omega \in \pazocal{U}_{ad}} J(\Omega) \:,
    \label{ShapeOpt.2}
\end{equation}
where $\pazocal{U}_{ad}$ stands for the set of admissible domains. 

Here, since the physical problem considered is modeled by \eqref{FV}, one has $y(\Omega)=\uu_\varepsilon(\Omega)$ solution of \eqref{FV} defined on $\Omega$. Therefore, let us replace the notation of the functional $J$ by $J_\varepsilon$ to emphasize the dependence with respect to the penalty parameter.  
%
%
Let $D\subset \mathbb{R}^d$ be a fixed bounded smooth domain, and let $\hat{\Gamma}_D\subset\partial D$ be a part of its boundary which will be the "potential" Dirichlet boundary. This means that for any domain $\Omega\subset D$, the Dirichlet boundary associated to $\Omega$ will be defined as $\Gamma_D:=\partial\Omega \cap \hat{\Gamma}_D$. With these notations, we introduce the set $\pazocal{U}_{ad}$ of all admissible domains, which consists of all smooth open domains $\Omega$ such that the Dirichlet boundary $\Gamma_D\subset \partial D$ is of stritly positive measure, that is:
$$
   \pazocal{U}_{ad} := \{ \Omega \subset D \: : \: \Omega \mbox{ is of class $\pazocal{C}^1$ and } |\partial\Omega \cap \hat{\Gamma}_D| > 0 \}.
$$
\subsection{Derivatives}
\label{sec:deriv.2}
The shape optimization method followed in this work is the so-called \textit{perturbation of the identity}, as presented in~\cite{murat1975etude} and \cite{henrot2006variation}. Let us introduce $\Cc^1_b(\mathbb{R}^d) := {(\pazocal{C}^1(\mathbb{R}^d)\cap W^{1,\infty}(\mathbb{R}^d) )}^d$, equipped with the $d$-dimensional $W^{1,\infty}$ norm, denoted $\norml\cdot \normr_{1,\infty}$. In order to move the domain $\Omega$, let $\thetaa \in \Cc^1_b(\mathbb{R}^d)$ be a (small) geometric deformation vector field. The associated perturbed or transported domain in the direction $\thetaa$ will be defined as: $\Omega(t) := (\Id+t\thetaa)(\Omega)$ for any $t>0$. To make things clear
some basic notions of shape sensitivity analysis from \cite{sokolowski1992introduction} are briefly recalled. 

We denote again $y(\Omega)$ the solution, in some Sobolev space denoted $W(\Omega)$, of a variational formulation posed on $\Omega$. For any fixed $\thetaa$, for any small $t>0$, let $y(\Omega(t))$ be the solution of the same variational formulation posed on $\Omega(t)$. 
If the variational formulation is regular enough (e.g.$\!$ if it is linear), it can be proved (see~\cite[Chapter 3]{sokolowski1992introduction}) that $y(\Omega(t))\circl(\Id+t\thetaa)$ also belongs to $W(\Omega)$.
\begin{itemize}[leftmargin=*]
    \item The \textit{Lagrangian derivative} or \textit{material derivative} of $y(\Omega)$ in the direction $\thetaa$ is the element $\dot{y}(\Omega)[\thetaa] \in W(\Omega)$ defined by:
    $$
        \dot{y}(\Omega)[\thetaa] := \lim_{t\searrow 0} \:\frac{1}{t}\left( y(\Omega(t))\circl(\Id+t\thetaa) - y(\Omega)\right) \: .
        \label{DefMatDer}
    $$
    If the limit is computed weakly in $W(\Omega)$ (respectively strongly), we talk about \textit{weak} material derivative (respectively \textit{strong} material derivative).
    \item If the additional condition $\gradd y(\Omega)\thetaa \in W(\Omega)$ holds for all $\thetaa \in \Cc^1_b(\mathbb{R}^d)$, then one may define a directional derivative called the \textit{Eulerian derivative} or \textit{shape derivative} of $y(\Omega)$ in the direction $\thetaa$ as the element $dy(\Omega)[\thetaa]$ of $W(\Omega)$ such that:
    $$
        dy(\Omega)[\thetaa] := \dot{y}(\Omega)[\thetaa] - \gradd y(\Omega)\thetaa \: .
        \label{DefShaDer}
    $$
    \item The solution $y(\Omega)$ is said to be \textit{shape differentiable} if it admits a directional derivative for any admissible direction $\thetaa$, and if the map $\thetaa \mapsto dy(\Omega)[\thetaa]$ is linear continuous from $\Cc^1_b(\mathbb{R}^d)$ to $W(\Omega)$.
\end{itemize}
%
\begin{remark}
    Linearity and continuity of $\thetaa \mapsto \dot{y}(\Omega)[\thetaa]$ is actually equivalent to the Gâteaux differentiability of the map $\thetaa \mapsto y(\Omega(\thetaa))\circl(\Id+\thetaa)$. The reader is referred to \cite[Chapter 8]{DelZol2001} for a complete review on the different notions of derivatives.
\end{remark}
When there is no ambiguity, the material and shape derivatives of some function $y$ at $\Omega$ in the direction $\thetaa$ will be simply denoted $\dot{y}$ and $dy$, respectively. 

\subsection{Shape sensitivity analysis of the penalty formulation}

The goal of this section is to prove the differentiability of $\uu_\varepsilon$ with respect to the shape. 
As functions $\maxx$ and $\qq$ fail to be Fréchet differentiable it is not possible to rely on the implicit function theorem as in \cite[Chapter 5]{henrot2006variation}.
%
Nevertheless, these functions admit directional derivatives. %
Hence, working with the directional derivatives of $\maxx$ and $\qq$ and following the approach in \cite{sokolowski1992introduction}, we show existence of directional material/shape derivatives for $\uu_\varepsilon$. 
Then, under assumptions on some specific subsets of $\Gamma_C$ (this will be presented and referred to as Assumption \ref{A1.2}), shape differentiability of $\uu_\varepsilon$ is proved.

Since the domain is transported, the functions $\Aa$, $\ff$, $\tauu$, $\mathfrak{F}$, and $s$ have to be defined everywhere in $\mathbb{R}^d$. 
They also need to enjoy more regularity for usual differentiability results to hold. In particular 
we make the following regularity assumptions :
\begin{hypothesis}\label{A0}
    $\Aa\in \pazocal{C}^1_b(\mathbb{R}^d,\mathbb{T}^4),\:\, \ff \in \Hh^1(\mathbb{R}^d),\:\, \tauu \in \Hh^2(\mathbb{R}^d)$, $s\in L^2(\Gamma_C)$ and $\mathfrak{F}s \in H^2(\mathbb{R}^d)$.
\end{hypothesis}
%
\begin{notation} Through change of variables, we can transform expression on $\Omega(t)$ to expression on 
$\Omega$. 
Composition with the operator $\circl (\Id+t\thetaa)$ will be denoted by $(t)$, for instance, $\Aa(t):=\Aa\circl(\Id+t\thetaa)$. 
The normal and tangential component associated to $\normalExt(t)$ of a vector $v$ is denoted $v_{\normalExt(t)}$ and $v_{\tanExt(t)}$ respectively. 
For integral expressions the Jacobian and tangential Jacobian of the transformation gives $\JacV(t) := $ Jac$(\Id+t\thetaa)$ and $\JacB(t):=$ Jac$_{\Gamma(t)}(\Id+t\thetaa)$. 
To simplify the notations let $\uu_{\varepsilon,t} := \uu_\varepsilon(\Omega(t))$ and 
$\uu_{\varepsilon}^t := \uu_{\varepsilon,t}\circl(\Id+t\thetaa)$. Finally, we also introduce the map $\Phi_\varepsilon:\mathbb{R}_+ \to \Xx$ such that for each $t>0$, $\Phi_\varepsilon(t)=\uu_{\varepsilon}^t$.
\end{notation}
%

As differentiability of $\uu_\varepsilon$ with respect to the shape is directly linked to differentiability of $\Phi_\varepsilon$, we will focus on the latter. However, note that the direction $\thetaa$ is fixed in the definition of $\Phi_\varepsilon$, therefore every property of $\Phi_\varepsilon$ (continuity, differentiability) will be associated to a directional property for $\uu_\varepsilon$. 

\subsubsection{Continuity of $\Phi_\varepsilon$}

Before getting interested in differentiability, the first step is to prove continuity.

\begin{theo}
  If Assumption \ref{A0} holds, then for any $\thetaa \in \Cc^1_b(\mathbb{R}^d)$, $\Phi_\varepsilon$ is strongly continuous at $t=0^+$.
\end{theo}

\begin{prf} When $t$ is small enough, the transported potential contact zone verifies $\Gamma_C(t)\subset\partial\Omega_{rig}^h$, so that the regularities of $\gG_{\normalExt}$ and $\normalExt$ are preserved. 
When transported to $\Omega(t)$, problem \eqref{FV} becomes: find $\uu_{\varepsilon,t} \in \Hh^1_{\Gamma_D(t)}(\Omega(t))=:\Xx(t)$ such that, 
\begin{equation} 
    \begin{aligned}
        \int_{\Omega(t)}& \Aa : \epsilonn(\uu_{\varepsilon,t}) : \epsilonn(\vv_t) + \frac{1}{\varepsilon} \int_{\Gamma_C(t)} \maxx(\uu_{\varepsilon,t} \cdot \normalExt -\gG_{\normalExt}) (\vv_t )_{\normalExt} \\
        \: & + \frac{1}{\varepsilon} \int_{\Gamma_C(t)}  \qq\left(\varepsilon\mathfrak{F}s,(\uu_{\varepsilon,t})_{\tanExt}\right)(\vv_t)_{\tanExt} = \int_{\Omega(t)} \ff \: \vv_t + \int_{\Gamma_N(t)} \tauu \: \vv_t \quad \forall \vv_t \in \Xx(t).
  \end{aligned}
  \label{FVT0}
\end{equation}
We can transform \eqref{FVT0} as an expression on the reference domain $\Omega$. 
For the test function we use $\vv^t:=\vv_t\circl(\Id+t\thetaa)$. Moreover, to simplify the expressions, we introduce
\begin{align*}
    R_{\normalExt}(\vv) &:=\maxx(\vv_{\normalExt}-\gG_{\normalExt})\:, \quad R_{\normalExt}^{t}(\vv) :=\maxx(\vv_{\normalExt(t)}-\gG_{\normalExt}(t))\:,\\
    S_{\tanExt}(\vv) &:= \qq(\varepsilon\mathfrak{F}s,\vv_{\tanExt})\:, \quad S_{\tanExt}^t(\vv) := \qq(\varepsilon(\mathfrak{F}s)(t),\vv_{\tanExt(t)})\:,
\end{align*}
and finally, the transported strain tensor $\epsilonn^t$ is also introduced: for all $\vv\in\Xx$,
\begin{equation*}
    \epsilonn^t(\vv):= \frac{1}{2} \left( \gradd \vv{(\Ii+t\gradd\thetaa)}^{-1} + {(\Ii+t\gradd\thetaa^T)}^{-1}{\gradd \vv}^T  \right) \: .
\end{equation*}
With the notations introduced and the change of variables mentionned above we have
\begin{equation}
    \begin{aligned}
        \int_{\Omega} \Aa(t) : &\epsilonn^t(\uu_{\varepsilon}^{t}) : \epsilonn^t(\vv^t) \:\JacV(t) + \frac{1}{\varepsilon} \int_{\Gamma_{C}} R_{\normalExt}^{t}(\uu_{\varepsilon}^{t}) \vv^t_{\normalExt(t)}  \:\JacB(t) 
        \\ & 
        + \frac{1}{\varepsilon} \int_{\Gamma_{C}} S_{\tanExt}^t(\uu_{\varepsilon}^{t}) \vv^t_{\tanExt(t)} \JacB(t) 
        =  \int_{\Omega} \ff(t)\vv^t \JacV(t) + \int_{\Gamma_{N}} \tauu(t)\vv^t \JacB(t) \:.
    \end{aligned}
    \label{FVT}
\end{equation}
Note that for $t$ sufficiently small, $\norml t \thetaa \normr_{1,\infty} \! < 1$. Thus the application $(\Id+t\thetaa)$ is a $\pazocal{C}^1$-diffeomorphism, and so the map $\vv_t \mapsto \vv^t$ is an isomorphism from $\Xx(t)$ to $\Xx$. Thus, one deduces that $\uu_{\varepsilon}^{t}$ is the solution of the variational formulation obtained when replacing $\vv^t$ by $\vv$ in \eqref{FVT}, which holds for all $\vv$ in $\Xx$.

\paragraph{Uniform boundedness of $\uu_\varepsilon^t$ in $\Xx$} 
Let us show that $\uu_\varepsilon^t$ is uniformly bounded in $t$. To achieve this we use the first order Taylor expansions with respect to $t$ of all known terms in \eqref{FVT}. Such expansions are valid due to Assumption \ref{A0} and the regularity assumptions on $\Omega$, see \cite{henrot2006variation,sokolowski1992introduction}. We recall some of them: $\forall \vv \in \Xx$,
\begin{equation*}
    \begin{aligned}
        \left\lVert \: \epsilonn^t(\vv) - \epsilonn(\vv) + \frac{t}{2}\left( \gradd \vv \gradd \thetaa + {\gradd \thetaa}^T {\gradd \vv}^T \right)\right\lVert_{0,\Omega} &= O(t^2)\norml \vv \normr_{\Xx} \:,\\
        \norml \Aa(t) - \Aa - t\gradd \Aa :\thetaa \normr_{\infty,\Omega} &= O(t^2)\:, \\
        \norml \JacV(t)-1-t\divv\thetaa \normr_{\infty,\Omega} &= O(t^2)\:, \\
        \norml \JacB(t)-1-t\divv_{\Gamma}\thetaa \normr_{\infty,\partial\Omega} &= O(t^2)\:, \\
        \norml \vv_{\normalExt(t)}- \vv_{\normalExt} - t(\vv\cdot(\gradd \normalExt \thetaa))\normr_{0,\Gamma_C} &= O(t^2)\norml \vv \normr_{0,\Gamma_C}\:,\\
        \norml \vv_{\tanExt(t)}- \vv_{\tanExt} + t(\vv\cdot(\gradd \normalExt \thetaa))\normalExt + t(\vv\cdot\normalExt)(\gradd \normalExt \thetaa) \normr_{0,\Gamma_C} &= O(t^2)\norml \vv \normr_{0,\Gamma_C}\:.
    \end{aligned}
\end{equation*}

Making use of these expansions, the ellipticity of $a$ and taking $\uu_\varepsilon^t$ as test-function in \eqref{FVT}, one gets the following estimate:
\begin{equation*}
    (\alpha_0 + O(t))\norml \uu_\varepsilon^t \normr_{\Xx}^2 \:\leq \: O(t)\norml \uu_\varepsilon^t \normr_{\Xx} + \:O(t^2) \:.
\end{equation*}
Thus, for $t$ small enough, one gets that there exist some positive constants $C_1$ and $C_2$ such that the sequence $C_1\norml \uu_\varepsilon^t \normr_{\Xx}^2-C_2\norml \uu_\varepsilon^t \normr_{\Xx}$ is uniformly bounded in $t$, which proves uniform boundedness of $\{\uu_\varepsilon^{t_k}\}_k$ in $\Xx$, for any sequence $\{t_k\}_k$ decreasing to 0.

\paragraph{Continuity}
First, one needs to show that the limit (in some sense that will be specified) of $\uu_\varepsilon^t$ as $t\to 0$ is indeed $\uu_\varepsilon$. Let $\{t_k\}_k$ be a sequence decreasing to 0. Since the sequence $\{\uu_\varepsilon^{t_k}\}_k$ is bounded 
and $\Xx$ a reflexive Banach space, there exists a weakly convergent subsequence (still denoted $\{\uu_\varepsilon^{t_k}\}_k$), say $\uu_\varepsilon^{t_k} \rightharpoonup \hat{\uu}_\varepsilon \in \Xx$. 

Due to the Taylor expansions above, the weak convergence of $\{\uu_\varepsilon^{t_k}\}_k$, the compact embedding $\Hh^{\frac{1}{2}}(\Gamma_C) \hookrightarrow \Ll^2(\Gamma_C)$ and Lipschitz continuity of $\maxx$ and $\qq$, taking $t=t_k$ in \eqref{FVT} and passing to the limit $k\to +\infty$ leads to: for all $\vv \in \Xx$,
\begin{equation*}
        \int_{\Omega} \Aa : \epsilonn(\hat{\uu}_{\varepsilon}) : \epsilonn(\vv) + \frac{1}{\varepsilon} \int_{\Gamma_{C}} R_{\normalExt}(\hat{\uu}_{\varepsilon}) \vv_{\normalExt}
        + \frac{1}{\varepsilon} \int_{\Gamma_{C}} S_{\tanExt}(\hat{\uu}_{\varepsilon}) \vv_{\tanExt} =  \int_{\Omega} \ff \: \vv + \int_{\Gamma_{N}} \tauu\: \vv \:.
\end{equation*}
This precisely means that $\hat{\uu}_\varepsilon = \uu_\varepsilon$, since they are both solution of problem \eqref{FV}, which admits a unique solution. The uniqueness also proves that the whole sequence $\{\uu_\varepsilon^{t_k}\}_k$ tends to $\uu_\varepsilon$.

Now, strong continuity of the map $t \mapsto \uu_\varepsilon^t$ at $t=0^+$ in $\Xx$ may be proved using the difference $\boldsymbol{\delta}_{\uu,\varepsilon}^t:=\uu_\varepsilon^t-\uu_\varepsilon$, which appears when subtracting the formulations verified by $\uu_\varepsilon^t$ and $\uu_\varepsilon$, respectively. Note that $\boldsymbol{\delta}_{\uu,\varepsilon}^t$ is bounded in $\Xx$ and that it converges weakly to 0 in $\Xx$. 

For $t$ sufficiently small, let us consider
\begin{equation}
  \begin{aligned}
    \int_{\Omega} &\Aa(t) : \epsilonn^t(\uu_{\varepsilon}^{t}) : \epsilonn^t(\vv) \:\JacV(t) - \int_{\Omega} \Aa : \epsilonn(\uu_\varepsilon) : \epsilonn(\vv) \\
    & + \frac{1}{\varepsilon} \int_{\Gamma_{C}} R_{\normalExt}^{t}(\uu_{\varepsilon}^{t}) \vv_{\normalExt(t)} \: \JacB(t) - \frac{1}{\varepsilon} \int_{\Gamma_{C}} R_{\normalExt}(\uu_{\varepsilon})\vv_{\normalExt} \\
    & + \frac{1}{\varepsilon} \int_{\Gamma_{C}} S_{\tanExt}^t(\uu_{\varepsilon}^{t}) \vv_{\tanExt(t)} \JacB(t) - \frac{1}{\varepsilon} \int_{\Gamma_{C}} S_{\tanExt}(\uu_{\varepsilon}) \vv_{\tanExt} \\
    & = \int_{\Omega} \ff(t) \vv \: \JacV(t) - \int_{\Omega} \ff \: \vv + \int_{\Gamma_{N}} \tauu(t)\vv \: \JacB(t) - \int_{\Gamma_{N}} \tauu \: \vv \:.
  \end{aligned}
  \label{FVTMoinsFV}
\end{equation}
Let us introduce three groups of terms, for any $\vv \in \Xx$, say $T_1(\vv)$, $T_2(\vv)$, $T_3(\vv)$ and $T_4(\vv)$, each $T_i(\vv)$ corresponding to the $i$-th line in equation \eqref{FVTMoinsFV}. The terms $T_1$ and $T_4$ have already been treated in the literature as they appear in the classical elasticity problem. Especially, one gets from \cite[Section 3.5]{sokolowski1992introduction} that
\begin{equation*}
    \begin{aligned}
    T_1(\boldsymbol{\delta}_{\uu,\varepsilon}^t) & \: \geq \: \alpha_0 \norml \boldsymbol{\delta}_{\uu,\varepsilon}^t \normr_{\Xx}^2 - tC\norml \boldsymbol{\delta}_{\uu,\varepsilon}^t \normr_{\Xx} \:, \\
    | T_4(\boldsymbol{\delta}_{\uu,\varepsilon}^t) | & \: \leq \: tC \norml \boldsymbol{\delta}_{\uu,\varepsilon}^t \normr_{\Xx} = o(t) \:.
    \end{aligned}
\end{equation*}
As for $T_2$ and $T_3$, 
some boundedness results are needed which can be deduced from the properties of functions $\maxx$ and $\qq$, the trace theorem and continuity of $t\mapsto \normalExt(t)$, $t\mapsto \gG_{\normalExt}(t)$, $t\mapsto (\mathfrak{F}s)(t)$, $t\mapsto \vv_{\tanExt(t)}$. For any $\vv\in\Xx$, one has
\begin{equation*}
    \begin{aligned}
        \norml \vv_{\normalExt(t)}\normr_{0,\Gamma_C} \leq C \norml \vv \normr_{\Xx}\:, \hspace{2em} &\norml \vv_{\tanExt(t)}\normr_{0,\Gamma_C} \leq C \norml \vv \normr_{\Xx}\:, \\
        \norml R_{\normalExt}^{t}(\vv) \normr_{0,\Gamma_C} \leq C\left(1+\norml \vv \normr_{\Xx}\right)\:, \hspace{2em} &  \norml S_{\tanExt}^{t}(\vv) \normr_{0,\Gamma_C} \leq C \:,
    \end{aligned}
\end{equation*}
and the same inequalities applies to $\vv_{\normalExt}$, $\vv_{\tanExt}$, $R_{\normalExt}$ and $S_{\tanExt}$. 
Then, for  $T_2$
\begin{equation*}
    \begin{aligned}
        T_2(\vv) = & \: \frac{1}{\varepsilon} \int_{\Gamma_{C}} R_{\normalExt}^{t}(\uu_{\varepsilon}^{t})\vv_{\normalExt(t)} \: (\JacB(t)-1) + \frac{1}{\varepsilon} \int_{\Gamma_{C}} R_{\normalExt}^{t}(\uu_{\varepsilon}^{t}) (\vv_{\normalExt(t)}-\vv_{\normalExt}) \\
        & + \frac{1}{\varepsilon} \int_{\Gamma_{C}} \left( R_{\normalExt}^{t}(\uu_{\varepsilon}^{t}) - R_{\normalExt}(\uu_{\varepsilon})\right)\vv_{\normalExt} \:.
    \end{aligned}
\end{equation*}
since $ \left( R_{\normalExt}^{t}(\uu_{\varepsilon}^{t}) - R_{\normalExt}(\uu_{\varepsilon})\right)$ is bounded in $L^2(\Gamma_C)$ and $\boldsymbol{\delta}_{\uu,\varepsilon}^t\to 0$ strongly in $\Ll^2(\Gamma_C)$, 
$$
    \left| T_2(\boldsymbol{\delta}_{\uu,\varepsilon}^t) \right| \leq t\,C \norml \boldsymbol{\delta}_{\uu,\varepsilon}^t \normr_{\Xx} +  \frac{1}{\varepsilon} \int_{\Gamma_{C}} \left|R_{\normalExt}^{t}(\uu_{\varepsilon}^{t}) - R_{\normalExt}(\uu_{\varepsilon})\right| \left|\boldsymbol{\delta}_{\uu,\varepsilon}^t \right|  = o(1) \:.
$$
As for $T_3$, using the boundedness of $\left(S_{\tanExt}^t(\uu_{\varepsilon}^{t}) - S_{\tanExt}(\uu_{\varepsilon}) \right)$ in $\Ll^2(\Gamma_C)$ and the same decomposition as for $T_2$ we get
\begin{equation*}
    | T_3(\boldsymbol{\delta}_{\uu,\varepsilon}^t) | \leq t\,C 
    \norml \boldsymbol{\delta}_{\uu,\varepsilon}^t \normr_{\Xx} 
    +  \frac{1}{\varepsilon} \int_{\Gamma_{C}} \left| S_{\tanExt}^t(\uu_{\varepsilon}^{t}) - S_{\tanExt}(\uu_{\varepsilon}) \right| \left|\boldsymbol{\delta}_{\uu,\varepsilon}^t \right|  = o(1) \:,
\end{equation*}
Thus, choosing $\boldsymbol{\delta}_{\uu,\varepsilon}^t$ as a test-function in \eqref{FVTMoinsFV} yields: $\alpha_0 \norml \boldsymbol{\delta}_{\uu,\varepsilon}^t \normr_{\Xx}^2 \leq o(1)$, which proves strong continuity of $t \mapsto \uu_\varepsilon^t$ in $\Xx$ at $t=0^+$.
\end{prf}

This result means that $\uu_\varepsilon$ is strongly directionally continuous with respect to the shape. Now, it remains to prove that differentiability also holds.


\subsubsection{Directional differentiability of $\maxx$ and $\qq$}
\label{app:DirDiffQ}


In order to study differentiability of $\Phi_\varepsilon$, we need some preliminary results concerning the  directional differentiability of the non-Fréchet differentiable functions $\maxx$ and $\qq$.
Let us briefly recall the definition of a Nemytskij operator. 
\begin{defn}
  Let $S$ be a measurable subset of $\mathbb{R}^d$, let $X$ and $Y$ be two real Banach spaces of functions defined on $S$. Given a mapping $\psi : S\times X \to Y$, the associated Nemytskij operator $\Psi$ is defined by:
  \begin{displaymath}
    \Psi(v)(x) := \psi(x,v(x)) \:, \:\: \mbox{ for all } x \in S \: . 
  \end{displaymath}
\end{defn}
As explained in details in \cite{goldberg1992nemytskij} or \cite[Section 4.3]{troltzsch2010optimal}, the smoothness of $\psi$ does not guarantee the smoothness of $\Psi$. In our case, we are only interested in directional differentiability of the Nemytskij operators associated to $\maxx$ and $\qq$. Thus, directional differentiability and Lipschitz continuity of $\maxx$ and $\qq$ in $\mathbb{R}$ and $\mathbb{R}^d$, combined with Lebesgue's dominated convergence, will enable us to conclude directly, without using the more general results from \cite{goldberg1992nemytskij}.


\begin{lemma}
 The function $\maxx:\mathbb{R}\to\mathbb{R}$ is Lipschitz continuous and directionally differentiable, with directional derivative at $u$ in the direction $v\in\mathbb{R}$:
 $$
  \dmaxx(u;v) = \left\{
      \begin{array}{lr}
        0 & \mbox{ if } \:u<0 ,\\
        \maxx(v) & \mbox{ if } \:u=0, \\
        v & \mbox{ if } \:u>0.
      \end{array}
    \right.
$$
\label{LemDirDiffMax.2}
\end{lemma}
\begin{lemma}
    The Nemytskij operator $\maxx:L^2(\Gamma_C)\to L^2(\Gamma_C)$ is Lipschitz continuous and directionally differentiable.
\label{LemDirDiffNemytskijMax.2}
\end{lemma}
The reader is referred to \cite{susu2018optimal}, for example, for the proof of those results.

\begin{notation} 
Let us introduce three subsets of $\mathbb{R}_{+}^{*}\times\mathbb{R}^{d-1}$: 
$$\pazocal{J}^-:=\{ (\alpha,z) \: : \: |z|<\alpha\},
\quad\pazocal{J}^0:=\{ (\alpha,z) \: : \: |z|=\alpha\},\quad 
\pazocal{J}^+:=\{ (\alpha,z) \: : \: |z|>\alpha\},$$
and the functions $\partial_\alpha \qq$ and $\partial_z \qq$, from $\mathbb{R}_+^*\times\mathbb{R}^{d-1}\!\setminus\pazocal{J}^0$ to $\mathcal{L}(\mathbb{R};\mathbb{R}^{d-1})$ and $\mathcal{L}(\mathbb{R}^{d-1})$, respectively, such that:
\begin{equation*}
  \partial_\alpha \qq(\alpha,z) = \left\{
      \begin{array}{lr}
        0 & \mbox{ in } \pazocal{J}^- ,\\
        \frac{z}{|z|} & \mbox{ in } \pazocal{J}^+ ,
      \end{array}
    \right.
\qquad
  \partial_z \qq(\alpha,z) = \left\{
      \begin{array}{lr}
           I_{d-1} & \mbox{in } \pazocal{J}^-, \\
           \frac{\alpha}{|z|}\big(I_{d-1}  -  \frac{1}{|z|^2} z \otimes  z\big) & \mbox{in } \pazocal{J}^+.
      \end{array}
    \right.
\end{equation*}
\end{notation}
\begin{lemma}
 The function $\qq:\mathbb{R}_+^*\times\mathbb{R}^{d-1}\to\mathbb{R}^{d-1}$ is Lipschitz continuous and directionally differentiable, with 
 derivative at $(\alpha,z)$ in the direction $(\beta,h)\in\mathbb{R}\times\mathbb{R}^{d-1}$:
 $$
  \dqq\left((\alpha,z); (\beta,h)\right) = \left\{
      \begin{array}{lr}
        h & \mbox{ in } \pazocal{J}^- ,\\
        h - \maxx\left( h\cdot \frac{z}{|z|}-\beta\right)\frac{z}{|z|} & \mbox{ in } \pazocal{J}^0 , \\
        \frac{\alpha}{|z|}\big(h- \frac{1}{|z|^2} (z\cdot h) z\big) + \beta\frac{z}{|z|} & \mbox{ in } \pazocal{J}^+ .
      \end{array}
    \right.
$$
\label{LemDirDiffQ}
\end{lemma}
\begin{lemma}
    The Nemytskij operator $\qq:L^2(\Gamma_C;\mathbb{R}^*_+)\times\Ll^2(\Gamma_C)\to \Ll^2(\Gamma_C)$ is Lipschitz continuous and directionally differentiable.
\label{LemDirDiffNemytskijQ}
\end{lemma}
\begin{prf}
    First, Lipschitz continuity of this Nemytskij operator follows directly from Lipschitz continuity of $\qq:\mathbb{R}_+^*\times\mathbb{R}^{d-1}\to\mathbb{R}^{d-1}$.
    Then, from Lemma \ref{LemDirDiffQ}, it is clear that, for all $(\alpha,z)\in\mathbb{R}_+^*\times\mathbb{R}^{d-1}$ and $(\beta,h)\in\mathbb{R}\times\mathbb{R}^{d-1}$, one has:
    \begin{equation}
         \left| \dqq\left((\alpha,z);(\beta,h)\right)\right| \leq |\beta| + |h|  \:.
         \label{EstPartDerQ}
    \end{equation}
    Let $(\alpha,\zz)\in L^2(\Gamma_C;\mathbb{R}^*_+)\times\Ll^2(\Gamma_C)$ and $(\beta,\hh)\in L^2(\Gamma_C)\times\Ll^2(\Gamma_C)$, and $t>0$. Directional differentiability of $\qq:\mathbb{R}_+^*\times\mathbb{R}^{d-1}\to\mathbb{R}^{d-1}$ yields:
    \begin{equation*}
        \begin{aligned}
            \left|\: \frac{\qq(\alpha+t\beta,\zz+t\hh)-\qq(\alpha,\zz)}{t}- \dqq\left((\alpha,\zz);(\beta,\hh)\right) \:\right| \longrightarrow 0 \hspace{0.5em} \mbox{ a.e.$\!$ on } \Gamma_C.
        \end{aligned}
    \end{equation*}
    Moreover, from estimation \eqref{EstPartDerQ}, along with Lispchitz continuity of $\qq$, 
    one gets:
    \begin{equation*}
        \begin{aligned}
            \left|\: \frac{\qq(\alpha+t\beta,\zz+t\hh)-\qq(\alpha,\zz)}{t}- \dqq\left((\alpha,\zz);(\beta,\hh)\right) \:\right| \leq 2\left(|\beta|+|\hh|\right) \hspace{0.5em} \mbox{ a.e.$\!$ on } \Gamma_C.
        \end{aligned}
    \end{equation*} 
    Since $|\hh|$ and $|\beta|\in L^2(\Gamma_C)$, Lebesgue's dominated convergence theorem finishes the proof.
\end{prf}


\subsubsection{Differentiability of $\Phi_\varepsilon$}


We are now in mesure to state the main results of this work.
\begin{notation}
For any smooth function $f$ defined on $\mathbb{R}^d$, and that does not depend on $\Omega$, we denote $f'[\thetaa]$ or simply $f'$ the following directional derivative:
$$
  f'[\thetaa] := \lim_{t\searrow 0} \:\frac{1}{t}\left( f\circl(\Id+t\thetaa) - f\right) = (\grad f) \thetaa \:.
$$
Using this notation, $\normalExt':=(\gradd \normalExt) \thetaa\:$ and for any $\vv\in\Xx$ we define :
$$
    \vv_{\normalExt'}:=\vv\cdot{\normalExt}'\:,  \hspace{0.5em} \vv_{\tanExt'}:=-\vv\cdot((\gradd \normalExt) \thetaa)\normalExt - (\vv\cdot\normalExt)(\gradd \normalExt) \thetaa = -\vv_{ \normalExt'}\normalExt - \vv_{\normalExt} \normalExt'\:.
$$
For the gap $\,\gG_{\normalExt}':=(\grad \gG_{\normalExt}) \thetaa$ and since $\gG_{\normalExt}$ is the oriented distance function to the smooth boundary $\partial\Omega_{rig}$, $\grad \gG_{\normalExt} = -\normalExt$, which implies that $\gG_{\normalExt}'= -\thetaa\cdot\normalExt$. However, we will still use the notation $\gG_{\normalExt}'$ to emphasize that this term comes from differentiation of the gap. Finally we define 
$$\pazocal{I}_\varepsilon^0:=\{x \in \Gamma_C \: : \: \uu_{\varepsilon,\normalExt}-\gG_{\normalExt}=0\}\subset \Gamma_C\:, \quad \pazocal{J}_\varepsilon^0:=\{x \in \Gamma_C \: : \: |\uu_{\varepsilon,\tanExt}|=\varepsilon\mathfrak{F}s\}\subset \Gamma_C\:,$$ two sets of special interest in the rest of this work.

\end{notation}
\begin{theo} If Assumption \ref{A0} holds, then for any $\thetaa\in\Cc^1_b(\mathbb{R}^d)$, $\Phi_\varepsilon$ is strongly differentiable at $t=0^+$.
  \label{ThmExistMDer}
\end{theo}

\begin{prf}
    In order to prove the differentiability of this map, one has to study the difference $\ww_\varepsilon^t:=\frac{1}{t}(\uu_\varepsilon^t-\uu_\varepsilon)$, which appears when dividing \eqref{FVTMoinsFV} by $t$. Of course, this leads to the formulation: $\frac{1}{t}(T_1(\vv)+T_2(\vv)+T_3(\vv))=\frac{1}{t}T_4(\vv)$. 
Again, from \cite[Section 3.5]{sokolowski1992introduction}, taking $\vv=\ww_\varepsilon^t$ as a test-function, one gets the following estimates for the first and fourth groups of terms:
\begin{equation*}
	\begin{aligned}
    \frac{1}{t}\:T_1(\ww_\varepsilon^t) & \: \geq \: \alpha_0 \norml \ww_\varepsilon^t \normr_{\Xx}^2 - C\norml \ww_\varepsilon^t \normr_{\Xx} \:, \\
    \frac{1}{t}\:T_4(\ww_\varepsilon^t) &\: \leq \:  C\norml \ww_\varepsilon^t \normr_{\Xx} \:.
    \end{aligned}
\end{equation*}

Using the property $(\maxx(a)-\maxx(b))(a-b) \geq 0$, for all $a$, $b \in \mathbb{R}$, one gets for the second group of terms:
\begin{equation*}
    \begin{aligned}
        \frac{1}{t}\:T_2(\ww_\varepsilon^t)\:  & = \:\frac{1}{\varepsilon} \int_{\Gamma_{C}} R_{\normalExt}^{t}(\uu_{\varepsilon}^{t}) \ww_{\varepsilon,\normalExt(t)}^t  \: \frac{1}{t}(\JacB(t)-1) 
        \\
        &\quad + \frac{1}{\varepsilon} \int_{\Gamma_{C}} R_{\normalExt}^{t}(\uu_{\varepsilon}^{t}) 
        \frac{1}{t}(\ww_{\varepsilon,\normalExt(t)}^t - \ww_{\varepsilon,\normalExt}^t)
        + \frac{1}{\varepsilon} \int_{\Gamma_{C}} \frac{1}{t} \left( R_{\normalExt}^{t}(\uu_{\varepsilon}^{t}) - 
        R_{\normalExt}(\uu_{\varepsilon})\right) \ww_{\varepsilon,\normalExt}^t  
        \\
        & \geq \: - C\norml \ww_\varepsilon^t \normr_{\Xx} + \:\frac{1}{\varepsilon} \int_{\Gamma_{C}} \frac{1}{t} \left( R_{\normalExt}(\uu_{\varepsilon}^{t}) - 
        R_{\normalExt}(\uu_{\varepsilon})\right)  \ww_{\varepsilon,\normalExt}^t 
        \\
        & \geq \: - C \norml \ww_\varepsilon^t \normr_{\Xx} \:.
    \end{aligned}
\end{equation*}
One can estimate the third term in the same way, using this time the properties of $\qq$, and especially the property $(\qq(\alpha,z_1)-\qq(\alpha,z_2))(z_1-z_2)\geq 0$, for all $\alpha\in\mathbb{R}_+^*$, $z_1,z_2\in\mathbb{R}^{d-1}$.
\begin{equation*}
    \begin{aligned}
        \frac{1}{t}\:T_3(\ww_\varepsilon^t) & =  \: \frac{1}{\varepsilon} \int_{\Gamma_{C}} S_{\tanExt}^t(\uu_{\varepsilon}^{t}) \ww_{\varepsilon,\tanExt(t)}^t \frac{1}{t}(\JacB(t)-1) 
        + \frac{1}{\varepsilon} \int_{\Gamma_{C}} S_{\tanExt}^t(\uu_{\varepsilon}^{t}) \frac{1}{t}(\ww_{\varepsilon,\tanExt(t)}^t-\ww_{\varepsilon,\tanExt}^t) 
        \\
        &\quad + \frac{1}{\varepsilon} \int_{\Gamma_{C}} \frac{1}{t}\left( S_{\tanExt}^t(\uu_{\varepsilon}^{t}) - S_{\tanExt}(\uu_{\varepsilon}) \right) \ww_{\varepsilon,\tanExt}^t  
        \\
        & \geq \: - C\norml \ww_\varepsilon^t \normr_{\Xx} + \frac{1}{\varepsilon} \int_{\Gamma_{C}} \frac{1}{t}\left( S_{\tanExt}(\uu_{\varepsilon}^{t}) - S_{\tanExt}(\uu_{\varepsilon}) \right)\ww_{\varepsilon,\tanExt}^t  
        \\
        & \geq \: - C\norml \ww_\varepsilon^t \normr_{\Xx} \:.
    \end{aligned}
\end{equation*}
Combining these four estimates leads to boundedness of $\ww_\varepsilon^t$ in $\Xx$ (uniformly in $t$). Thus for any sequence $\{t_k\}_k$ decreasing to 0, there exists a weakly convergent subsequence of $\{ \ww_\varepsilon^{t_k} \}_k$ (still denoted $\{ \ww_\varepsilon^{t_k} \}_k$), say $\ww_\varepsilon^{t_k} \rightharpoonup \ww_\varepsilon \in \Xx$. 
\vspace{0.5em}

The next step is to characterize this weak limit as the solution of a variational formulation. This can be done by taking $t=t_k$, then passing to the limit $k\to+\infty$ in formulation \eqref{FVTMoinsFV} divided by $t$. Before doing that, the bilinear form $a'$ and the linear form $\epsilonn'$, which will be very useful, are introduced as in \cite[Section 3.5]{sokolowski1992introduction}: for any $\uu$, $\vv \in \Xx$,
\begin{equation*}
    \begin{aligned} 
        &a'(\uu,\vv) := \int_\Omega \big\{ \Aa:\epsilonn'(\uu):\epsilonn(\vv) + \Aa:\epsilon(\uu):\epsilonn'(\vv) 
        \\ &\hspace{0.25\textwidth} 
        + (\divv \thetaa \: \Aa + \gradd \Aa \: \thetaa):\epsilonn(\uu):\epsilonn(\vv) \big\} \:, \\
        &\epsilonn'(\vv) := -\frac{1}{2}\left( \gradd \vv \gradd \thetaa + {\gradd \thetaa}^T {\gradd \vv}^T \right).
    \end{aligned}
\end{equation*}

Now, passing to the limit $k\to+\infty$ in $T_1(\vv)$ is rather straightforward and gives:
\begin{equation}
    \frac{1}{t_k}\:T_1(\vv) \: \longrightarrow \: a(\ww_\varepsilon,\vv) + a'(\uu_\varepsilon,\vv) \:.
    \label{FVMDerT1}
\end{equation}
For the second group of terms, one gets:
\begin{equation*}
    \begin{aligned}
        \frac{1}{t_k}\:T_2(\vv) \: \longrightarrow \: & \: \frac{1}{\varepsilon} \int_{\Gamma_{C}} R_{\normalExt}(\uu_{\varepsilon}) \left(\vv \cdot (\divv_\Gamma \thetaa \normalExt + \normalExt'\right) \\
        & \qquad + \lim_k \frac{1}{\varepsilon} \int_{\Gamma_{C}} \frac{1}{t_k} \left( R_{\normalExt}^{t_k}(\uu_{\varepsilon}^{t_k})- R_{\normalExt}(\uu_{\varepsilon})\right) (\vv \cdot \normalExt) \:.
    \end{aligned}
\end{equation*}
The key ingredient to deal with the second limit is the directional differentiability of the function $\maxx$ from $L^2(\Gamma_C)$ to $L^2(\Gamma_C)$, see Section \ref{app:DirDiffQ}. The candidate function for the derivative of $t\mapsto R_{\normalExt}^{t}(\uu_{\varepsilon}^{t})$ at $t=0^+$ is:
$$
    R_{\normalExt}'(\uu_{\varepsilon}):=
    \dmaxx\left( \uu_{\varepsilon,\normalExt}-\gG_{\normalExt} ; \zz_{\varepsilon,\normalExt}^{\thetaa} \right)\:.
$$
where $\zz_{\varepsilon,\normalExt}^{\thetaa} := \ww_{\varepsilon,\normalExt}+\uu_{\varepsilon,\normalExt'}-\gG_{\normalExt}'$. Let us show strong convergence in $L^2(\Gamma_C)$ to this candidate function by estimating:
\begin{equation*}
    \begin{aligned}
        &\norml \frac{1}{t_k} \left( R_{\normalExt}^{t_k}(\uu_{\varepsilon}^{t_k}) - R_{\normalExt}(\uu_{\varepsilon})\right) - R_{\normalExt}'(\uu_{\varepsilon}) \normr_{0,\Gamma_C}  
        \\ & \: 
        \leq \: \left\lVert \frac{1}{t_k} \left( \maxx\left(\uu_{\varepsilon,\normalExt(t_k)}^{t_k}-\gG_{\normalExt}(t_k)\right) -  \maxx\left(\uu_{\varepsilon,\normalExt}-\gG_{\normalExt}+\, t_k\zz_{\varepsilon,\normalExt}^{\thetaa}\right)\right) \right\lVert_{0,\Gamma_C} 
        \\ & \qquad 
        +  \bigg\lVert \frac{1}{t_k} \left( \maxx\left(\uu_{\varepsilon,\normalExt}-\gG_{\normalExt}+\,t_k \zz_{\varepsilon,\normalExt}^{\thetaa} \right)- \maxx\left(\uu_{\varepsilon,\normalExt}-\gG_{\normalExt}\right)\right)
        - R_{\normalExt}'(\uu_{\varepsilon}) \bigg\lVert_{0,\Gamma_C} 
        \\ & \: 
        \leq \: \norml \ww_{\varepsilon}^{t_k} - \ww_{\varepsilon} \normr_{0,\Gamma_C} + t_k C\left( 1 + \norml \ww_\varepsilon \normr_{0,\Gamma_C} + \norml \uu_\varepsilon \normr_{0,\Gamma_C} \right)
        \\ & \quad 
        +  \bigg\lVert \frac{1}{t_k} \left( \maxx\left(\uu_{\varepsilon,\normalExt}-\gG_{\normalExt}+\,t_k \zz_{\varepsilon,\normalExt}^{\thetaa}\right)- \maxx\left(\uu_{\varepsilon,\normalExt}-\gG_{\normalExt}\right)\right)
        - \dmaxx\left( \uu_{\varepsilon,\normalExt}-\gG_{\normalExt} ; \zz_{\varepsilon,\normalExt}^{\thetaa} \right) \bigg\lVert_{0,\Gamma_C}
    \end{aligned}
\end{equation*}
The first term goes to 0 due to compact embedding, and the last one also goes to 0 using directional differentiability of the function $\maxx$ from $L^2(\Gamma_C)$ to $L^2(\Gamma_C)$.
This finally leads for the second group of terms:
\begin{equation}
        \frac{1}{t_k}\:T_2(\vv) \longrightarrow \frac{1}{\varepsilon} \int_{\Gamma_{C}} R_{\normalExt}(\uu_{\varepsilon}) 
        \left(\vv_{\normalExt} \divv_\Gamma \thetaa + \vv_{\normalExt'} \right)
        + \frac{1}{\varepsilon} \int_{\Gamma_C}  R_{\normalExt}'(\uu_{\varepsilon}) \vv_{\normalExt}.
    \label{FVMDerT2}
\end{equation}
From Lemma \ref{LemDirDiffMax.2}, 
$R_{\normalExt}'(\uu_{\varepsilon}) = \maxx(\zz_{\varepsilon,\normalExt}^{\thetaa})$ on $\pazocal{I}_\varepsilon^0$.
The function $\maxx:\mathbb{R}\to\mathbb{R}$ being non-linear, the limit formulation is non-linear in $\thetaa$ if $\pazocal{I}_\varepsilon^0$ is not of null measure.
For the third group of terms, one gets:
\begin{equation*}
    \begin{aligned}
        \frac{1}{t_k}\:T_3(\vv) \: \longrightarrow \: & \: \frac{1}{\varepsilon} \int_{\Gamma_{C}} S_{\tanExt}(\uu_{\varepsilon}) \left(\vv_{\tanExt} \divv_\Gamma \thetaa + \vv_{\tanExt'} \right) \\
        & \qquad + \lim_k \frac{1}{\varepsilon} \int_{\Gamma_{C}} \frac{1}{t_k} \left( S_{\tanExt,t_k}(\uu_{\varepsilon}^{t_k})- S_{\tanExt}(\uu_{\varepsilon})\right) \vv_{\tanExt} \:.
    \end{aligned}
\end{equation*}
The key ingredient is the directional differentiability of the Nemytskij operator associated to $\qq$, see Section \ref{app:DirDiffQ}. The candidate function for the derivative of $t\mapsto S_{\tanExt}^t(\uu_{\varepsilon}^{t})$ at $t=0^+$ is:
$$
    S_{\tanExt}'(\uu_{\varepsilon}):=\dqq\left( (\varepsilon\mathfrak{F}s,\uu_{\varepsilon,\tanExt}) ; \left(\varepsilon\grad(\mathfrak{F}s)\thetaa, \zz_{\varepsilon,\tanExt}^{\thetaa} \right) \right)\:,
$$
where $\zz_{\varepsilon,\tanExt}^{\thetaa}:=\ww_{\varepsilon,\tanExt}+\uu_{\varepsilon,\tanExt'}$. Another series of estimations gives strong convergence to this candidate function in $\Ll^2(\Gamma_C)$.
\begin{equation*}
    \begin{aligned}
        &\norml \frac{1}{t_k} \left( S_{\tanExt,t_k}(\uu_{\varepsilon}^{t_k}) - S_{\tanExt}(\uu_{\varepsilon})\right) - S_{\tanExt}'(\uu_{\varepsilon}) \normr_{0,\Gamma_C}  
        \\ & 
        \leq \: \left\lVert \frac{1}{t_k} \left( \qq\left( \varepsilon (\mathfrak{F}s)(t_k),\uu_{\varepsilon,\tanExt(t_k)}^{t_k} \right)
        - \qq\left( \varepsilon (\mathfrak{F}s + t_k\grad(\mathfrak{F}s)\thetaa) , \uu_{\varepsilon,\tanExt} + t_k \zz_{\varepsilon,\tanExt}^{\thetaa}\right) \right) \right\lVert_{0,\Gamma_C} 
        \\ & \ \
        +  \bigg\lVert \frac{1}{t_k} \left( \qq\left( \varepsilon (\mathfrak{F}s + t_k\grad(\mathfrak{F}s)\thetaa) , \uu_{\varepsilon,\tanExt} + t_k \zz_{\varepsilon,\tanExt}^{\thetaa}\right)
        - \qq\left( \varepsilon \mathfrak{F}s,\uu_{\varepsilon,\tanExt} \right) \right)
        - S_{\tanExt}'(\uu_{\varepsilon}) \bigg\lVert_{0,\Gamma_C} 
        \\ & 
        \leq \: \norml \ww_{\varepsilon}^{t_k} - \ww_{\varepsilon} \normr_{0,\Gamma_C} + C t_k \left( \varepsilon + \norml \ww_\varepsilon \normr_{0,\Gamma_C} + \norml \uu_\varepsilon \normr_{0,\Gamma_C} \right)
        \\ & \ \
        + \bigg\lVert \frac{1}{t_k} \left( \qq\left( \varepsilon (\mathfrak{F}s + t_k\grad(\mathfrak{F}s)\thetaa) , \uu_{\varepsilon,\tanExt} + t_k \zz_{\varepsilon,\tanExt}^{\thetaa}\right)
        - \qq\left( \varepsilon \mathfrak{F}s,\uu_{\varepsilon,\tanExt} \right) \right)
        - S_{\tanExt}'(\uu_{\varepsilon}) \bigg\lVert_{0,\Gamma_C} \!.
    \end{aligned}
\end{equation*}
Due to compact embedding and directional differentiability 
for $\qq$, all terms on the right hand side converge to 0.  Thus, 
\begin{equation}
    \begin{aligned}
        \frac{1}{t_k}\:T_3(\vv) \: \longrightarrow \: & \: \frac{1}{\varepsilon} \int_{\Gamma_{C}} S_{\tanExt}(\uu_{\varepsilon}) \left(\vv_{\tanExt} \divv_\Gamma \thetaa + \vv_{\tanExt'} \right) 
        + \frac{1}{\varepsilon} \int_{\Gamma_{C}} S_{\tanExt}'(\uu_{\varepsilon}) \vv_{\tanExt} \:.
    \end{aligned}
    \label{FVMDerT3}
\end{equation}
From Lemma \ref{LemDirDiffQ}, $S_{\tanExt}'(\uu_{\varepsilon})$ is non linear uniquely on $\pazocal{J}_\varepsilon^0$ where it uses the $\maxx$ function. Therefore the limit variational formulation is non linear only on $ \pazocal{J}_\varepsilon^0$.

Using once again the results from \cite[Section 3.5]{sokolowski1992introduction} gives 
\begin{equation}
    \frac{1}{t_k}\:T_4(\vv) \: \longrightarrow \: \int_\Omega \left(\divv \thetaa \: \ff + \gradd\ff \thetaa\right)\vv + \int_{\Gamma_N} \left(\divv_\Gamma \thetaa \: \tauu + \gradd\tauu \thetaa\right) \vv \:.
    \label{FVMDerT4}
\end{equation}
Combining \eqref{FVMDerT1}, \eqref{FVMDerT2}, \eqref{FVMDerT3} and \eqref{FVMDerT4}, and using the Heaviside function $H$ and $\partial_\alpha$ and $\partial_z$ defined  
in Section \ref{app:DirDiffQ}, one gets that $\ww_\varepsilon\in \Xx$ is the solution of
\begin{equation}
        b_\varepsilon(\ww_\varepsilon,\vv) + \frac{1}{\varepsilon} \prodL2{  R_{\normalExt}'(\uu_{\varepsilon}), \vv_{\normalExt} }{\pazocal{I}_\varepsilon^0} 
        + \frac{1}{\varepsilon} \prodL2{ S_{\tanExt}'(\uu_{\varepsilon}) , \vv_{\tanExt} }{\pazocal{J}_\varepsilon^0} = L_\varepsilon[\thetaa](\vv) \:, \hspace{1em} 
        \forall \vv \in \Xx ,
    \label{FVMDer}
\end{equation} 
where the bilinear form $b_\varepsilon$ and linear form $L_\varepsilon[\thetaa]$ are defined as, for any $\uu$, $\vv\in\Xx$:
\begin{equation*}
    \begin{aligned}
        b_\varepsilon(\uu,\vv):=a(\uu,\vv) &+ \frac{1}{\varepsilon} \prodL2{H(\uu_{\varepsilon,\normalExt}-\gG_{\normalExt}) \uu_{\normalExt} , \vv_{\normalExt}}{\Gamma_C\setminus\pazocal{I}_\varepsilon^0 } \\
        \:& + \frac{1}{\varepsilon} \prodL2{\partial_z \qq(\varepsilon\mathfrak{F}s,\uu_{\varepsilon,\tanExt}) \uu_{\tanExt} , \vv_{\tanExt}}{\Gamma_C\setminus\pazocal{J}_\varepsilon^0} ,
    \end{aligned}
\end{equation*}
\begin{equation*}
    \begin{aligned}
        L_\varepsilon[\thetaa](\vv) := &\int_\Omega (\divv \thetaa \: \ff + \gradd \ff  \thetaa) \vv + \int_{\Gamma_N} (\divv_\Gamma \thetaa \: \tauu+ \gradd \tauu \thetaa)\vv - \:a'(\uu_\varepsilon,\vv) \\
        & - \frac{1}{\varepsilon} \int_{\Gamma_C} R_{\normalExt}(\uu_{\varepsilon})\left(\vv \cdot (\divv_\Gamma \thetaa \normalExt + \normalExt')\right) \\
        & - \frac{1}{\varepsilon} \int_{\Gamma_C\setminus\pazocal{I}_\varepsilon^0} H(\uu_{\varepsilon,\normalExt}-\gG_{\normalExt}) \left( \uu_{\varepsilon,\normalExt'} - \gG_{\normalExt}' \right)\vv_{\normalExt} \\
        & - \frac{1}{\varepsilon} \int_{\Gamma_{C}} S_{\tanExt}(\uu_{\varepsilon}) \left(\vv_{\tanExt} \divv_\Gamma \thetaa + \vv_{\tanExt'} \right) \\
        & - \int_{\Gamma_{C}\setminus\pazocal{J}_\varepsilon^0} \left( \partial_\alpha \qq (\varepsilon\mathfrak{F}s,\uu_{\varepsilon,\tanExt}) \grad(\mathfrak{F}s)\thetaa 
        + 
        \frac{1}{\varepsilon} \partial_z \qq (\varepsilon\mathfrak{F}s,\uu_{\varepsilon,\tanExt}) \uu_{\varepsilon,\tanExt'}\right) \vv_{\tanExt}\:.
    \end{aligned}
\end{equation*}
Due to the regularities of $\normalExt$, $\gG_{\normalExt}$, $\thetaa$, $\ff$, $\tauu$, $\mathfrak{F}s$, $\uu_\varepsilon$, and uniform boundedness of both $\partial_\alpha\qq$, $\partial_z\qq$, it is clear that $L_\varepsilon[\thetaa] \in \Xx^*$ for any $\thetaa$. 
From the uniform boundedness and positivity of $H(\cdot)$ and $\partial_z \qq(\cdot,\cdot)$, one has, for all $\uu$, $\vv$, in $\Xx$
\begin{equation*}
  \begin{aligned}
   \left| \frac{1}{\varepsilon} \prodL2{H(\uu_{\varepsilon,\normalExt}-\gG_{\normalExt}) \uu_{\normalExt} , \vv_{\normalExt}}{\Gamma_C\setminus\pazocal{I}_\varepsilon^0} \right| &\leq \frac{K}{\varepsilon}\norml \uu \normr_{\Xx} \norml \vv \normr_{\Xx} \:, \\
   \frac{1}{\varepsilon} \prodL2{H(\uu_{\varepsilon,\normalExt}-\gG_{\normalExt}) \uu_{\normalExt} , \uu_{\normalExt}}{\Gamma_C\setminus\pazocal{I}_\varepsilon^0} &= \frac{1}{\varepsilon} \int_{\Gamma_C\setminus\pazocal{I}_\varepsilon^0} H(\uu_{\varepsilon,\normalExt}-\gG_{\normalExt}) \:(\uu_{\normalExt})^2 \geq 0 \: , 
   \\
   \left| \frac{1}{\varepsilon} \prodL2{\partial_z\qq(\varepsilon\mathfrak{F}s,\uu_{\varepsilon,\tanExt}) \uu_{\tanExt} , \vv_{\tanExt}}{\Gamma_C\setminus\pazocal{J}_\varepsilon^0} \right| 
   & \leq \frac{K}{\varepsilon}\norml \uu \normr_{\Xx} \norml \vv \normr_{\Xx} \:, 
   \\
   \frac{1}{\varepsilon} \prodL2{\partial_z\qq(\varepsilon\mathfrak{F}s,\uu_{\varepsilon,\tanExt}) \uu_{\tanExt} , \uu_{\tanExt}}{\Gamma_C\setminus\pazocal{J}_\varepsilon^0} &= \frac{1}{\varepsilon} \int_{\Gamma_C\setminus\pazocal{J}_\varepsilon^0}  \left(\partial_z\qq(\varepsilon\mathfrak{F}s,\uu_{\varepsilon,\tanExt}) \uu_{\tanExt}\right) \uu_{\tanExt} \geq 0 \: .
   \end{aligned}
\end{equation*} 
Thus  $b_\varepsilon$ is continuous and coercive over $\Xx \times \Xx$. Because of the non-linearities occuring on the sets $\pazocal{I}_\varepsilon^0$ and $\pazocal{J}_\varepsilon^0$, well-posedness of \eqref{FVMDer} is proved using an optimization argument. Let us introduce the following functionals, defined for any $\ww\in\Xx$:
$$
   \begin{aligned}
   \tilde{\phi}(\ww) := &\frac{1}{2} b_\varepsilon(\ww,\ww) - L_\varepsilon[\thetaa](\ww) + \phi_{\normalExt}(\ww) + \phi_{\tanExt}(\ww) \:,\\[1em]
        \phi_{\normalExt}(\ww) := \ & \frac{1}{2\varepsilon} \norml \maxx\left( \ww_{\normalExt}+\uu_{\varepsilon,\normalExt'}-\gG_{\normalExt}' \right) \normr^2_{0,\pazocal{I}_\varepsilon^0} , \\
        \phi_{\tanExt}(\ww) := \ & 
        \frac{1}{2\varepsilon} \norml \ww_{\tanExt}+\uu_{\varepsilon,\tanExt'} \normr^2_{0,\pazocal{J}_\varepsilon^0} \\ 
        & + \frac{1}{2\varepsilon} \norml 
        \maxx\left(
        -\varepsilon\grad(\mathfrak{F}s)\thetaa + \left(\ww_{\tanExt}+\uu_{\varepsilon,\tanExt'}\right)\cdot \frac{ \uu_{\varepsilon,\tanExt}}{| \uu_{\varepsilon,\tanExt}|}
        \right)
        \normr^2_{0,\pazocal{J}_\varepsilon^0} \:.
   \end{aligned}
$$
Obviously, solving \eqref{FVMDer} is equivalent to finding a minimum of $\tilde{\phi}$ over $\Xx$. Both $\phi_{\normalExt}$ and $\phi_{\tanExt}$ are convex, continuous and positive.   Due to the properties of $b_\varepsilon$ and $L_\varepsilon[\thetaa]$, $\tilde{\phi}$ is strictly convex, coercive and continuous, and one gets that problem \eqref{FVMDer} has a unique solution $\ww_\varepsilon$. Uniqueness also proves that the whole sequence $\{ \ww_\varepsilon^{t_k}\}_k$ converges weakly to $\ww_\varepsilon$.

\paragraph{Strong convergence} 
Strong convergence is proved taking $\boldsymbol{\delta}_{\ww,\varepsilon}^t:=\ww_\varepsilon^t-\ww_\varepsilon$ as test-function and subtracting: $\frac{1}{t}$\eqref{FVTMoinsFV} $-$ \eqref{FVMDer}. Gathering all terms properly enable to get the following estimation:
\begin{equation*}
    \begin{aligned}
        \alpha_0 \norml \boldsymbol{\delta}_{\ww,\varepsilon}^t \normr_{\Xx}^2 \ \leq & \ \left( tC + \norml \frac{1}{t} \left( R_{\normalExt}^{t}(\uu_{\varepsilon}^{t}) - R_{\normalExt}(\uu_{\varepsilon})\right) - R_{\normalExt}'(\uu_{\varepsilon}) \normr_{0,\Gamma_C} \right.\\
        & \ \ \left. + \norml \frac{1}{t} \left( S_{\tanExt}^t(\uu_{\varepsilon}^{t}) - S_{\tanExt}(\uu_{\varepsilon})\right) - S_{\tanExt}'(\uu_{\varepsilon}) \normr_{0,\Gamma_C}\right) \norml \boldsymbol{\delta}_{\ww,\varepsilon}^t \normr_{\Xx} \:.
    \end{aligned}
\end{equation*}
It has already been showed that all terms in parentheses go to $0$, which yields strong convergence of $\ww_\varepsilon^t$ to $\ww_\varepsilon$ in $\Xx$.
\end{prf}

Existence and uniqueness of the limit $\frac{1}{t}(\uu_\varepsilon^t-\uu_\varepsilon)$ have been established. In other words, it has been proved that $\uu_\varepsilon$ admits a strong material derivative in any direction $\thetaa$, namely $\ww_\varepsilon = \dot{\uu}_\varepsilon(\Omega)[\thetaa] \in \Xx$, or simply $\ww_\varepsilon = \dot{\uu}_\varepsilon\in \Xx$. Nevertheless, as mentionned in the previous proof, the map $\thetaa\mapsto \dot{\uu}_\varepsilon(\Omega)[\thetaa]$ fails to be linear  on $\pazocal{I}_\varepsilon^0 \cup \pazocal{J}_\varepsilon^0$ due to non-Gâteaux-differentiability of $\maxx$ and $\qq$. Thus, some additional assumptions are required.
%
%
A rather straightforward way to get around this is to assume that for a fixed value of $\varepsilon$, those sets are of measure zero: 

\begin{hypothesis}\label{A1.2}
    The sets $\pazocal{I}_\varepsilon^0$ and $\pazocal{J}_\varepsilon^0$ are of measure 0.
\end{hypothesis}

Note that, due to \eqref{CL2:1}, $x\in \pazocal{I}_\varepsilon^0$ implies that both $\uu_{\varepsilon,\normalExt}(x)-\gG_{\normalExt}(x)=0$ and $\sigmaa_{\normalInt\normalExt}(\uu_\varepsilon) (x)=0$, which means that $x$ is in contact but there is no contact pressure. On the other hand, by definition, a point $x\in \pazocal{J}_\varepsilon^0$ is at the same time in sliding contact and in sticking contact. In the case of the non-penalty formulation, $\pazocal{I}_\varepsilon^0$ is sometimes referred to as the \textit{weak contact set}, while $\pazocal{J}_\varepsilon^0$ is sometimes referred to as the \textit{weak sticking set} (see~\cite{beremlijski2014shape} for contact with Coulomb friction). Following these denominations, let us refer to the points of  $\pazocal{I}_\varepsilon^0$ and $\pazocal{J}_\varepsilon^0$ as \textit{weak contact points}, and \textit{weak sticking points}, respectively.

For example, Assumption \ref{A1.2} is satisfied when all weak contact points and all weak sticking points represent a finite number of points in 2D or a finite number of curves in 3D.

\begin{remark}
    Both sets can be gathered under the more general denomination of \textit{biactive sets}, borrowed from optimal control (see~\cite{wachsmuth2014strong} in the case of the obstacle problem). Moreover, in optimal control problems related to variational inequalities, Gâteaux differentiability of the solution with respect to the control parameter is obtained under the \textit{strict complementarity condition}, see for example \cite{bonnans2013perturbation}. This condition is actually quite difficult to explicit and to use in practice, see \cite[Lemma 2.6]{rauls2018generalized}, and \cite{wachsmuth2014strong} for a discussion. However, in our context, the variational inequality has  been regularized by the penalty approach. Therefore our conditions are simpler to express: the biactive sets are of zero measure.
\end{remark}

\begin{cor}  \label{CorExistSDer}
  If Assumption \ref{A0} and Assumption \ref{A1.2} hold, then $\uu_\varepsilon$ solution of \eqref{FV} is (strongly) shape differentiable in $\Ll^2(\Omega)$. For any $\thetaa\in\Cc^1_b(\mathbb{R}^d)$, its shape derivative in the direction $\thetaa$ writes $d\mathbf{u}_\varepsilon(\Omega)[\thetaa] :=  \dot{\uu}_\varepsilon(\Omega)[\thetaa] - \gradd \uu_\varepsilon \thetaa$, where $\dot{\uu}_\varepsilon(\Omega)[\thetaa]$ is the unique solution of 
    \begin{equation}
        b_\varepsilon(\dot{\uu}_\varepsilon,\vv) = L_\varepsilon[\thetaa](\vv) \:, \hspace{1em} \forall \vv \in \Xx .
        \label{FVMDerBis}
    \end{equation}
    Moreover, $\{\dot{\uu}_\varepsilon\}_\varepsilon$ and $\{d\mathbf{u}_\varepsilon\}_\varepsilon$ are uniformly bounded in $\Xx$ and $\Ll^2(\Omega)$, respectively.
\end{cor}

\begin{prf}
    When Assumption \ref{A1.2} holds, the variational formulation \eqref{FVMDer} solved by $\dot{\uu}_\varepsilon$ may be rewritten as \eqref{FVMDerBis}.
    Since the map $\thetaa\mapsto L_\varepsilon[\thetaa]$ is linear from $\Cc^1_b(\mathbb{R}^d)$ to $\Xx^*$, one gets that the map $\thetaa \mapsto \dot{\uu}_\varepsilon(\Omega)[\thetaa] \in \Xx$ is linear as well, which directly leads to the desired result. 
    
    Regarding boundedness of $\{\dot{\uu}_\varepsilon\}_\varepsilon$, the key ingredient is the choice of the right test function. Let us introduce
    $$
      \tilde{\uu}_\varepsilon := \left( \uu_{\varepsilon,\normalExt'}-\gG_{\normalExt}' \right)\normalExt - \uu_{\varepsilon,\normalExt}\normalExt'\:.
    $$
    It is clear that $\tilde{\uu}_\varepsilon\in\Xx$, and that one has the following estimation
    $$
      \norml \tilde{\uu}_\varepsilon \normr_{\Xx} \leq C\left( 1 + \norml \uu_\varepsilon \normr_{\Xx} \right) \:.
    $$
    Now, as $\normalExt' \perp \normalExt$, if $\vv\in\Xx$ is defined by $\vv = \dot{\uu}_\varepsilon + \tilde{\uu}_\varepsilon$, then 
    $$
        \vv_{\normalExt} =  \dot{\uu}_{\varepsilon,\normalExt} + \uu_{\varepsilon,\normalExt'}-\gG_{\normalExt}' \:, \quad\quad
        \vv_{\tanExt} =  \dot{\uu}_{\varepsilon,\tanExt} - \uu_{\varepsilon,\normalExt}\normalExt'\:.
    $$
    Therefore, due to positivity of $H$ and $\partial_z \qq$, combined with uniform boundedness of both $\frac{1}{\varepsilon}R_{\normalExt}(\uu_\varepsilon)$ in $L^2(\Gamma_C)$ and $\frac{1}{\varepsilon}S_{\tanExt}(\uu_\varepsilon)$ in $\Ll^2(\Gamma_C)$, taking such a $\vv$ as test-function in \eqref{FVMDerBis} enables to conclude.
\end{prf}

\begin{remark}
    Another approach to get around this non-differentiability issue is to modify the formulation by regularizing non-smooth functions: in this case, replacing $\maxx$ and $\qq$ by regularized versions $\text{p}_{c,+}$ and $\qq_c$, where $c$ stands for the regularization parameter, $c\to\infty$. This leads to a solution map that is Fréchet-differentiable. It can be proved, see \cite{chaudet2019phd}, that the solution of the regularized formulation $\uu_\varepsilon^c \to \uu_\varepsilon$ in $\Xx$, and that in addition, when Assumption \ref{A1.2} holds, the shape derivative $d\mathbf{u}_\varepsilon^c \to d\mathbf{u}_\varepsilon$ in $\Ll^2(\Omega)$.
\end{remark}

\begin{remark}
	Uniform boundedness of $\{d\mathbf{u}_\varepsilon\}_\varepsilon$ implies that the sequence converges weakly in $\Ll^2(\Omega)$ (up to a subsequence) when $\varepsilon\to 0$. However, it seems difficult to characterize this weak limit.
\end{remark}

\subsection{Computation of the shape derivative of a general criterion}

Now that shape sensitivity of the penalty formulation have been studied, one may go back to our initial shape optimization problem \eqref{ShapeOpt.2}. Let us focus on cost functionals of the rather general type:
\begin{equation}
  J_\varepsilon(\Omega) := \int_\Omega j(\uu_\varepsilon(\Omega)) + \int_{\partial\Omega} k(\uu_\varepsilon(\Omega))\:,
  \label{GeneralJType.2}
\end{equation}
where $\uu_\varepsilon(\Omega)$ is the solution of \eqref{FV} on $\Omega$. The functions $j,k$ are $\pazocal{C}^1(\mathbb{R}^d,\mathbb{R})$, and their derivatives with respect to $\uu_\varepsilon$, denoted $j'$, $k'$, are Lipschitz. It is also assumed that those functions and their derivatives satisfy, for all $u$, $v \in \mathbb{R}^d$,
\begin{equation}
    |j(u)| \leq C\left(1+|u|^2\right) \qquad
    |k(u)| \leq C\left(1+|u|^2\right)
  \label{Condjk.2}
\end{equation}
\begin{equation}
    |j'(u)\cdot v| \leq C |u\cdot v| \qquad
    |k'(u)\cdot v| \leq C|u\cdot v| 
  \label{Condj'k'.2}
\end{equation}
for some constants $C>0$.
From the shape differentiability of $\uu_\varepsilon$, one may deduce the following results, see for example \cite{henrot2006variation}.

\begin{theo}  \label{ThmDJVol.2}
    When Assumption \ref{A0} and Assumption \ref{A1.2} hold, 
    $J_\varepsilon$ is defined by \eqref{GeneralJType.2} and satisfy \eqref{Condjk.2} and \eqref{Condj'k'.2}
    and $\uu_\varepsilon$ is the solution of \eqref{FV}, then 
    $J_\varepsilon$ 
    is shape differentiable at $\Omega$ and its derivative in the direction $\thetaa\in\Cc^1_b(\mathbb{R}^d)$ writes:
    \begin{equation}
            dJ_\varepsilon(\Omega)[\thetaa] = \int_\Omega j'(\uu_\varepsilon)\cdot \dot{\uu}_\varepsilon + \: j(\uu_\varepsilon)\divv\thetaa 
            + \int_{\partial\Omega} k'(\uu_\varepsilon)\cdot \dot{\uu}_\varepsilon + k(\uu_\varepsilon) \divv_\Gamma\thetaa.
        \label{DJVol0.2}
    \end{equation}
    with $\dot{\uu}_\varepsilon$ solution of \eqref{FVMDerBis}.
\end{theo}

\begin{remark}
  From Corollary \ref{CorExistSDer}, one automatically gets that $dJ_\varepsilon$ is uniformly bounded in $\varepsilon$. Therefore, formula \eqref{DJVol0.2} produces usable shape derivatives, regardless how small $\varepsilon$ gets.    
\end{remark}


From a numerical point of view, this last expression contains a number of difficulties (mainly the right hand side of \eqref{FVMDerBis} and divergence of $\thetaa$) that can be circumvented through simple transformations. 
Introducing the adjoint state, it is possible to rewrite \eqref{DJVol0.2} avoiding the construction of the right hand side of \eqref{FVMDerBis} and resulting in an expression having only boundary integrals with integrand involving only $\thetaa$.

In the context of problem \eqref{FV} with the functional $J_\varepsilon$, the associated \textit{adjoint state} $\pp_\varepsilon \in \Xx$ is defined as the solution of: 
\begin{equation}
    b_\varepsilon(\pp_\varepsilon,\vv) = -\int_\Omega j'(\uu_\varepsilon)\cdot \vv - \int_{\partial\Omega} k'(\uu_\varepsilon)\cdot \vv \qquad \forall\vv \in \Xx\:.
    \label{FVA}
\end{equation}
Note that by application of Lax-Milgram lemma, existence and uniqueness of $\pp_\varepsilon$ are guaranteed. Using this adjoint state, one is able to get a boundary integral expression of the following form for $dJ_\varepsilon(\Omega)[\thetaa]$.
\begin{theo}
  Suppose $\Omega$ is of class $\pazocal{C}^2$. Then, under the hypothesis of Theorem \ref{ThmDJVol.2}, with $\uu_\varepsilon$, $\pp_\varepsilon\in \Hh^2(\Omega)\cap\Xx$ solutions of \eqref{FV} and \eqref{FVA} respectively, one has:
  \begin{equation}
      dJ_\varepsilon(\Omega)[\thetaa] = \int_{\partial \Omega} \mathfrak{A}_\varepsilon \: (\thetaa\cdot\normalInt) + \int_{\Gamma_N} \mathfrak{B}_\varepsilon \: (\thetaa\cdot\normalInt) + \int_{\Gamma_C} \mathfrak{C}_\varepsilon \: (\thetaa\cdot\normalInt) \:,
    \label{DJ.2}
  \end{equation} 
  where $\mathfrak{A}_\varepsilon$, $\mathfrak{B}_\varepsilon$ and $\mathfrak{C}_\varepsilon$ depend on $\uu_\varepsilon$, $\pp_\varepsilon$, their gradients, and the data.
  \label{ThmDJ}
\end{theo}
\begin{prf}
  Due to Theorem \ref{ThmDJVol.2}, when considering \eqref{FVA} with $\vv=\dot{\uu}_\varepsilon\in\Xx$ as test-function, one gets:
  $$
    dJ_\varepsilon(\Omega)[\thetaa] = -b_\varepsilon(\pp_\varepsilon,\dot{\uu}_\varepsilon) + \int_\Omega j(\uu_\varepsilon)\divv\thetaa + \int_{\partial\Omega} k(\uu_\varepsilon) \divv_\Gamma\thetaa \:.
  $$
  Now, noting that $b_\varepsilon$ is symmetric and taking $\vv=\pp_\varepsilon\in\Xx$ in \eqref{FVMDer} leads to
  \begin{equation}
    dJ_\varepsilon(\Omega)[\thetaa] = -L_\varepsilon[\thetaa](\pp_\varepsilon) + \int_\Omega j(\uu_\varepsilon)\divv\thetaa + \int_{\partial\Omega} k(\uu_\varepsilon) \divv_\Gamma\thetaa \:.
    \label{DJVol}
  \end{equation}
  From that point, due to the additional regularity assumption on $\uu_\varepsilon$ and $\pp_\varepsilon$, integrating by parts and using the variational formulations \eqref{FV} and \eqref{FVA} with well chosen test-functions yields the desired result, with
  \begin{equation}\label{EqABC.2}
  \left\{
      \begin{aligned}
      \mathfrak{A}_\varepsilon &= j(\uu_\varepsilon) + (\kappa +\partial_{\normalInt})k(\uu_\varepsilon) + \Aa:\epsilonn(\uu_\varepsilon):\epsilonn(\pp_\varepsilon) - \ff\pp_\varepsilon \:,\\[0.5em]
      \mathfrak{B}_\varepsilon &= -(\kappa +\partial_{\normalInt})\left( \tauu \pp_\varepsilon \right) \:,\\
      \mathfrak{C}_\varepsilon &=  \mathfrak{C}_{\varepsilon}^{\normalExt} + \mathfrak{C}_{\varepsilon}^{\tanExt} = \frac{1}{\varepsilon}(\kappa +\partial_{\normalInt})\left(R_{\normalExt}(\uu_\varepsilon)\pp_{\varepsilon, \normalExt}\right) +\frac{1}{\varepsilon}(\kappa +\partial_{\normalInt})\left(S_{\tanExt}(\uu_\varepsilon)\pp_{\varepsilon,\tanExt}\right) \:.
  \end{aligned}
  \right.
  \end{equation}
  In the previous formulae, $\kappa$ denotes the mean curvature on $\partial\Omega$, and $\partial_{\normalInt}$ stands for the normal derivative with respect to $\normalInt$.
\end{prf}

\begin{remark}\label{RemNormalDefField}
  Expression \eqref{DJVol} is often referred to as the \textit{distributed shape deri\-vative}, and it is always valid as it only requires $\uu_\varepsilon$, $\pp_\varepsilon\in\Xx$. The reader is referred to \cite{hiptmair2015comparison,laurain2016distributed} for more details about distributed shape derivatives. The additional regularity assumption enables to get an explicit expression that fits the Hadamard-Zolésio structure theorem. This structure of the shape derivative suggests to consider deformation fields $\thetaa$ of the form $\thetaa=\theta \normalInt$, where the normal $\normalInt$ has been extended to $\mathbb{R}^d$ (not necessarily using the oriented distance function to $\partial\Omega$), which is possible when $\partial\Omega$ is at least $\pazocal{C}^1$, see \cite{henrot2006variation}.
\end{remark}

\begin{remark}
  The first two terms in \eqref{DJ.2} are exactly the same as for the elasticity formulation without contact. There are two additional components coming from the contact conditions, namely $\mathfrak{C}_\varepsilon^{\normalExt}$, stemming from the normal constraint, and $\mathfrak{C}_\varepsilon^{\tanExt}$, stemming from the tangential constraint. Obviously, these are the only terms involving $\normalExt$. When considering problems without contact, those last two terms cancel, while in the case of pure sliding contact, only $\mathfrak{C}_\varepsilon^{\tanExt}\equiv 0$. As for contact problems without gap (see for instance \cite{maury2017shape}), in the expression of $R_{\normalExt}$ the gap $\gG_{\normalExt}$ is simply set to 0, and \eqref{DJ.2}, \eqref{EqABC.2} coincides with the derivative in \cite{maury2017shape}. Moreover, note that neglecting the term with $\mathfrak{C}_\varepsilon$ (imposing $\thetaa=0$ on $\Gamma_C$) is equivalent to excluding the contact zone from the optimization process. In other words, the derived expression \eqref{EqABC.2} is rather general and adapts to many situations: sliding or frictional contact, contact with or without gap, optimizing or not the contact zone, etc.
\end{remark}
%

  We decided to write the contact boundary conditions using the normal $\normalExt$ to the rigid foundation instead of the normal $\normalInt$ to $\partial\Omega$ because it leads to a simpler expression for $dJ_\varepsilon$. Indeed, when differentiating our formulation with respect to the shape, as $\normalExt$ and $\gG_{\normalExt}$ do not depend on $\Omega$, the contact boundary condition can be treated like any Neumann condition. Alternatively, when differentiating the classical formulation, based on $\normalInt$ for the contact boundary conditions, additional terms involving the shape derivatives of the gap, $d\mathbf{g}_{\normalInt}$, and the normal, $d\mathbf{n_o}$, appear in the shape derivative of $J_\varepsilon$ (see \cite{maury2017shape} in the case with no gap). It turns out that these shape derivatives are quite technical to handle in practice. As these two formulations solve the same mechanical problem, see Remark \ref{RemHypHPP}, the simplified expression for $dJ_\varepsilon$, \eqref{DJ.2} with \eqref{EqABC.2}, is valid for both formulations.

\section{Numerical results}

\subsection{Shape optimization algorithm}
Following the usual approach, 
\cite{allaire2004structural, maury2017shape},
the algorithm proposed here to minimize $J_\varepsilon(\Omega)$ is a descent method, based on the shape derivative. 
Starting from an initial shape $\Omega^0\subset D$, using the cost functional derivative \eqref{DJ.2}, the algorithm generates a sequence of shapes $\Omega^k\in \pazocal{U}_{ad}$ such that the real-valued sequence $\{J_\varepsilon(\Omega^k)\}_k$ decreases. Each shape $\Omega^k$ is represented explictly, as a meshed subdomain of $D$, as well as implicitly, as the zero level set of some function $\phi^k$. The explicit representation enables to apply all boundary conditions rigorously, while the implicit representation enables to make the shape evolve smoothly from an iteration to the next by solving the following Hamilton-Jacobi equation on $[0,T]\times \mathbb{R}^d$:
\begin{equation}
    \begin{aligned}
    &\frac{\partial\phi}{\partial t} + \theta |\grad\phi| = 0\:, \\
    &\phi(0,x) = \tilde{\phi}(x)\:,
    \label{HamJac}
    \end{aligned}
\end{equation}
where $T$ is strictly positive, $\tilde{\phi}$ is a given initial condition,  and $\theta$ is the norm of the normal deformation field (see Remark \ref{RemNormalDefField}). The reader is referred to the pioneer work \cite{allaire2004structural} for more details about shape optimization using the level set method.
\vspace{1em}

As mentioned earlier, the method to generate the sequence $\{\Omega^k\}_k$ is based on a gradient descent. It consists in several successive steps.
\begin{enumerate}[leftmargin=3em, topsep=4pt]
    \item Find $\uu_\varepsilon^{k}$ solution of \eqref{FV} on $\Omega^k$.
    \item Find the adjoint state $\pp_\varepsilon^{k}$ solution of \eqref{FVA} on $\Omega^k$.
    \item Find a descent direction with $\thetaa^k=\theta^k \normalInt^k$ using \eqref{DJ.2} and \eqref{EqABC.2}.
    \item Update the level set function $\phi^{k+1}$ by solving \eqref{HamJac} on some interval $[0,T^k]$ with $T^k>0$, taking $\theta=\theta^k$ as velocity field and $\tilde{\phi}=\phi^k$ as initial condition.
    \item Cut the mesh of $D$ around $\{\phi^{k+1}\!=\!0\}$ to get an explicit representation of $\Omega^{k+1}$.
\end{enumerate}

\begin{remark}
    In step 4, the real number $T^k$ is chosen such that the monotonicity of $\{J_\varepsilon(\Omega^k)\}_k$ is guaranteed at each iteration. This numerical trick tries to ensure a descent direction, even in situations where Assumption \ref{A1.2} is not verified. Indeed, in such situations, expression \eqref{DJ.2} will not be an accurate representation of the shape derivative, however it still can provide a valid descent direction.
\end{remark}

\begin{remark}
    Even though such algorithms prove themselves very efficient from the numerical point of view, there are a few limitations. First, note that, in general, problem $\eqref{ShapeOpt.2}$ is not well-posed and $J_\varepsilon$ is not convex. Thus, using a descent method to try and solve it necessarily leads to finding a local minimum that is highly dependant on the initial shape $\Omega^0$. The reader is referred to \cite{allaire2004structural} for a more detailed discussion on that matter. Second, the algorithm may generate shapes for which the expression of $dJ_\varepsilon$ is inaccurate (e.g. Assumption \ref{A1.2} is not verified) and, in the worst case scenario, from which no descent direction $\thetaa^k$ can be obtained. In such cases, which have not been encountered in practice, the algorithm will stop and no solution will be found.
\end{remark}


\paragraph{Some details about the implementation}

Although it is not the purpose of this work, we present summarily some aspects of the implementation. Numerical experiments are performed with the code MEF++, developped at the GIREF (Groupe Interdisciplinaire de Recherche en \'El\'ements Finis, Université Laval). Problems \eqref{FV} and \eqref{FVA} are solved using the finite element method using Lagrange $P^2$ finite elements. The Hamilton-Jacobi type equation is solved on a secondary grid, using the second order finite difference scheme presented in \cite{OshSet1988}, with Neumann boundary conditions on $\partial D$. The reader is referred to the rather recent work \cite{chouly2013convergence} for finite element resolution and error estimate of the penalty formulation of contact problems in linear elasticity, and to \cite{Set1996,QuaVal2008} for details about level set methods and their numerical treatment using finite differences.


\subsection{Specific context}

Even though the method could deal with any functional $J_\varepsilon$ of the general type \eqref{GeneralJType.2}, we focus here on the special case of a linear combination of the compliance and the volume (with some weight coefficients $\alpha_1$ and $\alpha_2$).
\begin{equation*}
    J_\varepsilon(\Omega) = \alpha_1 C(\Omega) + \alpha_2 \textit{Vol}(\Omega)= \int_\Omega (\alpha_1\ff \uu_\varepsilon(\Omega)+\alpha_2)+ \int_{\Gamma_N} \alpha_1\tauu \uu_\varepsilon(\Omega) \:.
\end{equation*}
Indeed, from the engineering point of view, minimizing such a $J_\varepsilon$ means finding the best compromise (in some sense) between weight and stiffness. 

The materials are assumed to be isotropic and obeying Hooke's law (linear elastic), that is:
$$
    \sigmaa(\uu) = \Aa:\epsilonn(\uu) = 2\mu \epsilonn(\uu)+\lambda \divv\uu \:, 
$$
where $\lambda$ and $\mu$ are the Lamé coefficients of the material, which can be expressed in terms of Young's modulus $E$ and Poisson's ratio $\nu$:
\begin{equation*}
    \lambda = \frac{E\nu}{(1+\nu)(1-2\nu)}\:, \hspace{1em} \mu = \frac{E}{2(1+\nu)}\:.
\end{equation*}
Here, those constants are set to the classical academic values $E=1$ and $\nu=0.3$ and the penalty parameter $\varepsilon$ is set to $10^{-6}$. Such a value for $\varepsilon$ ensures that the solution $\uu_\varepsilon$ is close enough to the solution $\uu$ of the original contact problem.  

Concerning $D$ and the admissible shapes, at each iteration $k$, the current domain $\Omega^k$ will be contained in $D$ and its boundary $\partial\Omega^k$ will be divided as follows (the colours refer to \ref{fig:InitCanti2dCercle}):

\begin{itemize}[topsep=4pt, parsep=1pt, leftmargin=3em]
    \item $\Gamma_N^k=\Gamma_N^0$ is fixed as the orange part of $\partial D$,
    \item $\Gamma_D^k$ is the intersection of $\partial\Omega^k$ and $\hat{\Gamma}_D$, the blue part of $\partial D$, 
    \item $\Gamma_C^k$ is the transformation, through $\thetaa^k$, of the green part of $\partial D$.
\end{itemize}

Working with a formulation with no gap is quite convenient in several cases. First, when an a priori potential contact zone $\hat{\Gamma}_C\subset \partial D$ is known, then defining $\Gamma_C=\partial\Omega \cap \hat{\Gamma}_C$ enables to enforce the contact boundary to be part of $\partial D$. In such situations, see for example \cite{maury2017shape}, the boundary $\Gamma_C^k$ is either treated like $\Gamma_D^k$ (the contact area cannot be empty) or $\Gamma_N^k$ (the contact area is fixed) during the optimization process. Second, those formulations are well suited for interface problems involving several materials, see \cite{lawry2015level}. 

However, introducing a gap in the formulation allows to extend the method to situations where we want to optimize the shape of a body in contact with a rigid foundation (known a priori). Especially, the potential contact zone is included into the shape optimization process : as the contact zone is not fixed, the shape can be modified along $\Gamma_C$.

\subsection{The cantilever}

We revisit one of the most frequently presented test in shape optimization: the design of a bidimensional cantilever beam. This test differs from the usual one by the added possibility of a support of the beam through contact (sliding or frictional) with a rigid body. 
For this benchmark, the domain $D$ is the rectangular box $[0,2]\times [0,1]$ meshed with triangles, with an average number of vertices equal to 1300. The rigid foundation is the circle of radius $R=8$ and center $x_C=(1,-8)$. 
External forces are chosen such that $\ff=0$, and $\tauu=(0,-0.01)$ is applied on $\Gamma_N$ (in orange in \ref{fig:InitCanti2dCercle}). In the frictional case, $s=10^{-2}$ and $\mathfrak{F}=0.2$. And the weight coefficients in $J$ are $\alpha_1=15$, $\alpha_2=0.01$. These choices are based on the generic behavior of the model for the given data, and can be reinterpreted as searching for a stiff structure under volume constraint. Besides, since $\tauu=(0,-0.01)$, the order of magnitude of $\uu_\varepsilon$ is also $10^{-2}$, hence the difference between the orders of magnitude of $\alpha_1$ and $\alpha_2$.
\begin{figure}[h]
	\begin{center}
	\resizebox{0.65\textwidth}{!}{
    \begin{tikzpicture}

	\draw[>=latex,->] (-2,0) -- (-1.5,0);
	\draw[>=latex,->] (-2,0) -- (-2,0.5);
	\node[black] at (-1.25,0) {$x$};
 	\node[black] at (-2,0.75) {$y$};
    \node[white] at (10,0) {$\:$};
   
	\draw[black] (8.66,-0.62) arc (75:105:18);
	\draw[dartmouthgreen, very thick] (0,0) -- (8,0);
	\draw[black] (8,0) -- (8,1.6);
	\draw[orange, very thick] (8,1.6) -- (8,2.4);
	\draw[black] (8,2.4) -- (8,4);
	\draw[black] (8,4) -- (0,4);
	\draw[blue, very thick] (0,4) -- (0,0);

	\node[white] at (4,-0.3) {$\:$};
	\node[black] at (4,2) {\large{$D$}};
	\node[blue] at (0.7,2) {\large{$\hat{\Gamma}_D$}};
	\node[orange] at (7.3,2) {\large{$\Gamma_N$}};
	\node[dartmouthgreen] at (4,0.5) {\large{$\Gamma_C$}};
	\node[black] at (4,-0.5) {\large{$\Omega_{rig}$}};
	\end{tikzpicture}
    }	
	\end{center}
  	\caption{Initial geometry for the 2d cantilever.}
    \label{fig:InitCanti2dCercle}
\end{figure}
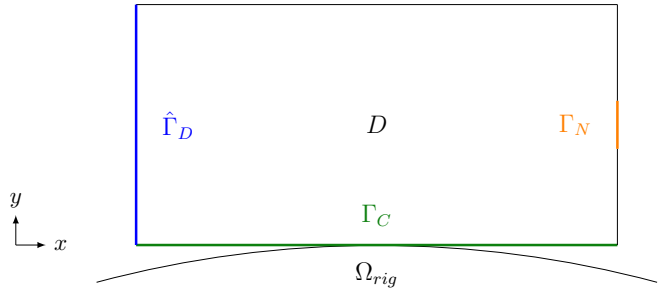

\begin{figure}[h]
  \begin{center}
    \subfloat[Initial design.]{
	\resizebox{0.46\textwidth}{!}{
	\begin{tikzpicture}
    	\node[anchor=south west,inner sep=0] at (0,0) {\includegraphics[width=\textwidth]{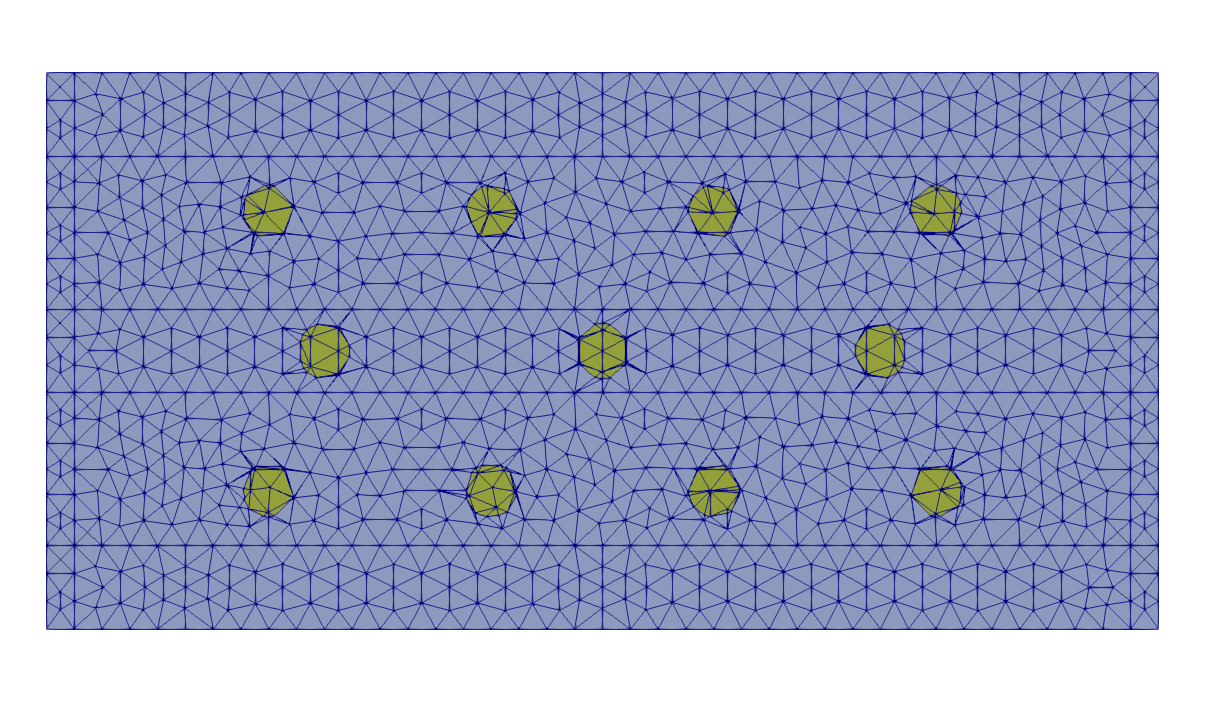}};
    	\draw[black] (13.2,0.08) arc (79:101:35);
    \end{tikzpicture}
    \label{sub:Canti2dCercleIt0}
    }
    }
    \hspace{.5em}
    \subfloat[Final design without contact.]{
	\resizebox{0.46\textwidth}{!}{
	\begin{tikzpicture}
    	\node[anchor=south west,inner sep=0] at (0,0) {\includegraphics[width=\textwidth]{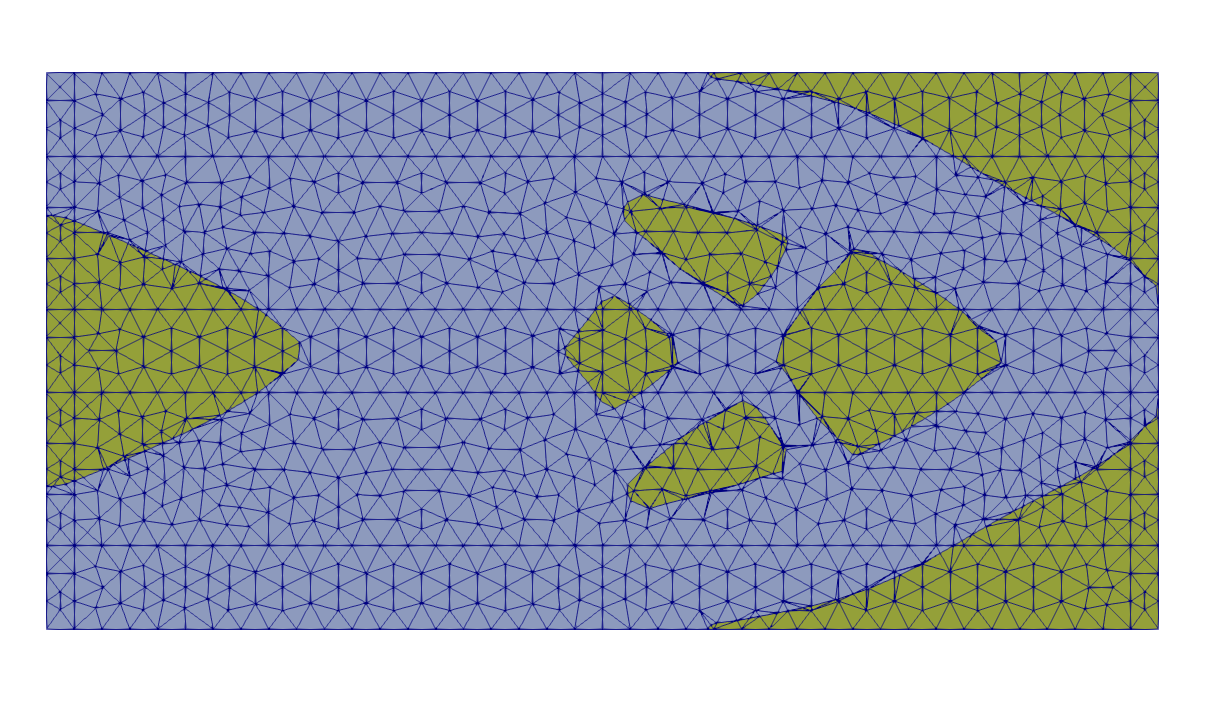}};
    	\draw[white] (13.2,0.08) arc (79:101:35);
    \end{tikzpicture}
    \label{sub:Canti2dBis}
    }
    }
    \hspace{.5em}
    \subfloat[Final design in pure sliding contact.]{
	\resizebox{0.46\textwidth}{!}{
	\begin{tikzpicture}
    	\node[anchor=south west,inner sep=0] at (0,0) {\includegraphics[width=\textwidth]{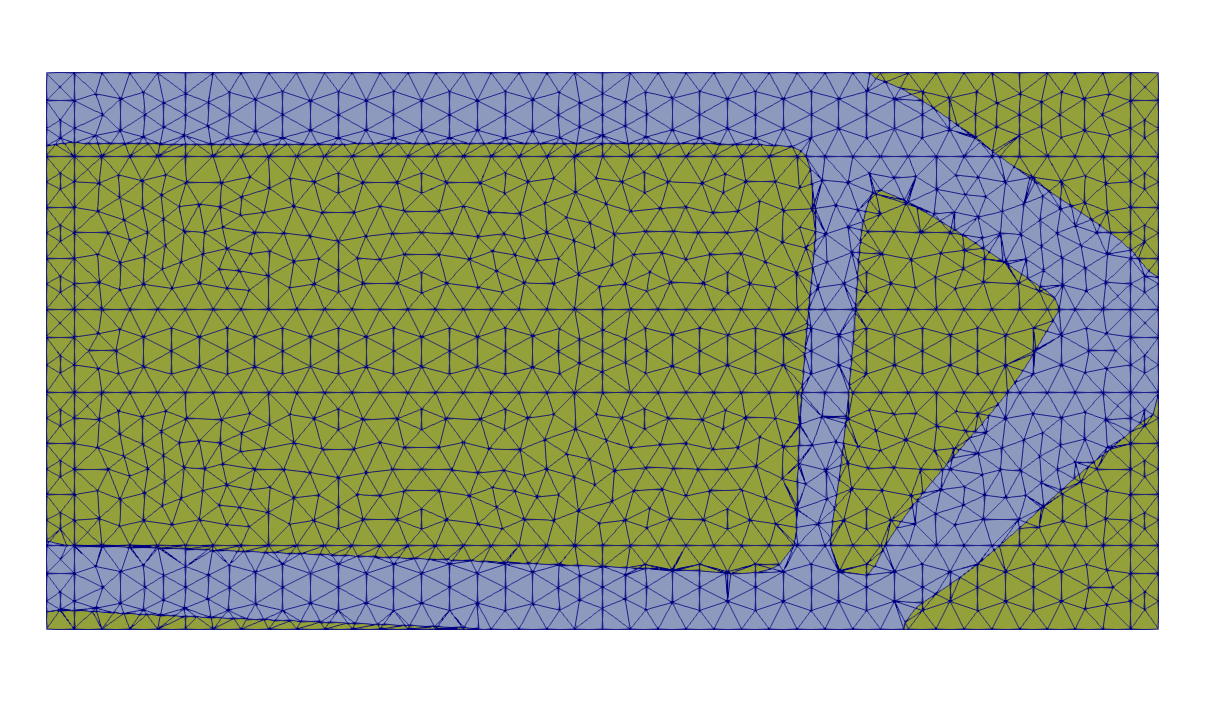}};
    	\draw[black] (13.2,0.08) arc (79:101:35);
    \end{tikzpicture}
    \label{sub:Canti2dCerclePena}
    }
    }
    \hspace{.5em}
    \subfloat[Final design in frictional contact.]{
	\resizebox{0.46\textwidth}{!}{
	\begin{tikzpicture}
    	\node[anchor=south west,inner sep=0] at (0,0) {\includegraphics[width=\textwidth]{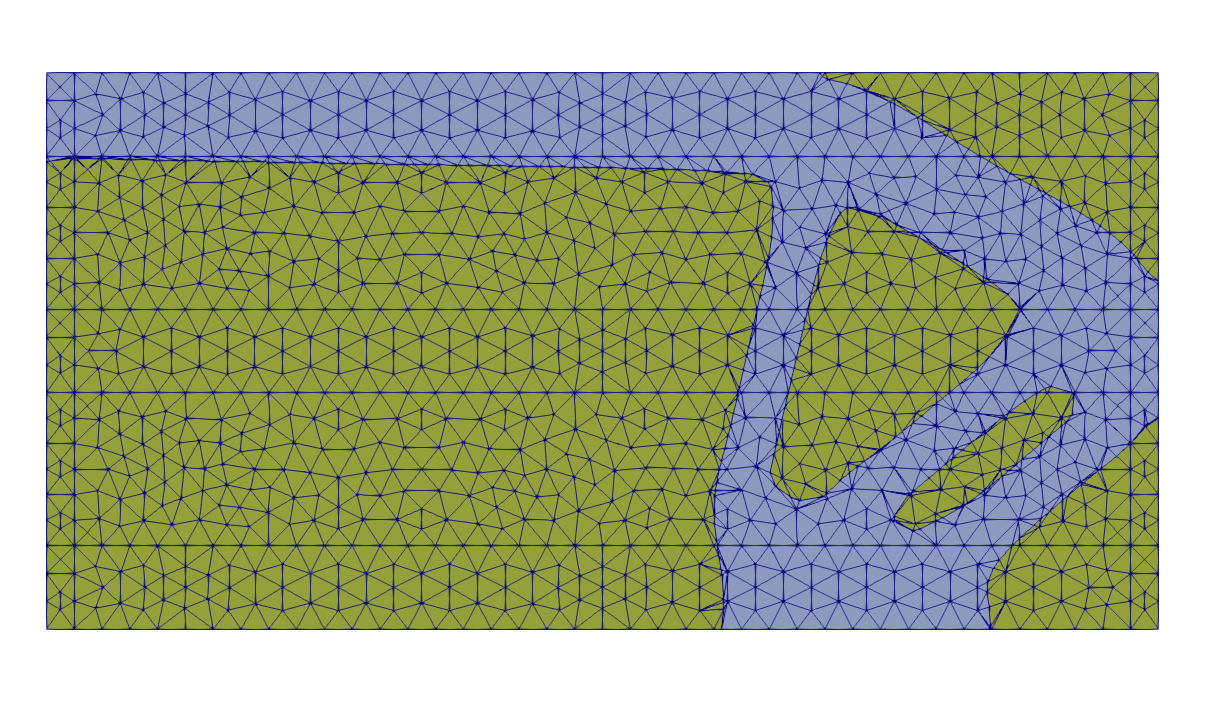}};
    	\draw[black] (13.2,0.08) arc (79:101:35);
    \end{tikzpicture}
    \label{sub:Canti2dCercleFrottPena}
    }
    }
    \end{center}
  \caption{Initial and final designs for the 2d cantilever in contact with a disk ($\Omega$ in blue, $D\setminus \Omega$ in yellow).}
  \label{fig:Canti2dCercle}
\end{figure}
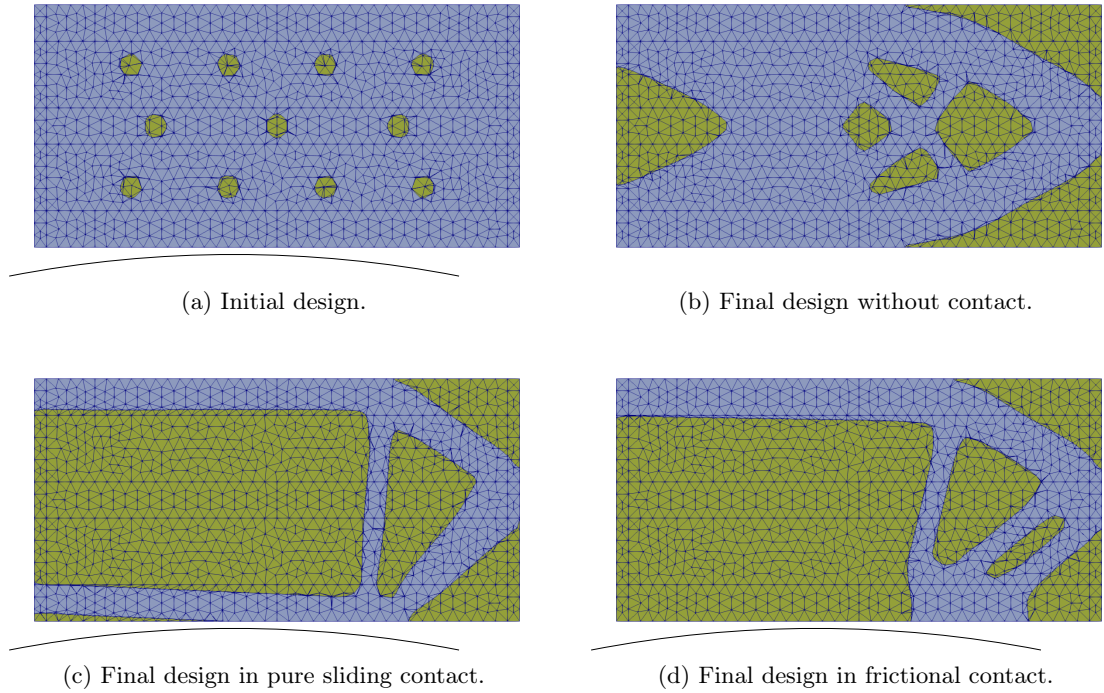

\begin{figure}
\begin{center}
    \subfloat[Without contact.]{   
	\resizebox{!}{0.27\textwidth}{
	\begin{tikzpicture}
		\begin{axis}[
    		xlabel={Number of iterations $l$},
    		xmin=0, xmax=70,
    		]
    		\addplot[color=red,mark=none] table {res_canti2d_bis.txt};
   			\legend{$J(\Omega^l)$}
   		\end{axis}
	\end{tikzpicture}
	}
	}
	%
    \subfloat[Pure sliding contact.]{   
	\resizebox{!}{0.27\textwidth}{
	\begin{tikzpicture}
		\begin{axis}[
    		xlabel={Number of iterations $l$},
    		xmin=0, xmax=310,
    		]
    		\addplot[color=gray,mark=none] table {res_canti2d_cercle_pena.txt};
   			\legend{$J(\Omega^l)$}
   		\end{axis}
	\end{tikzpicture}
	}
	}
	%
    \subfloat[Frictional contact.]{   
	\resizebox{!}{0.27\textwidth}{
	\begin{tikzpicture}
		\begin{axis}[
    		xlabel={Number of iterations $l$},
    		xmin=0, xmax=120,
    		]
    		\addplot[color=dartmouthgreen,mark=none] table {res_canti2d_cercle_frott_pena.txt};
   			\legend{$J(\Omega^l)$}
   		\end{axis}
	\end{tikzpicture}
	}
	}
\end{center}
\caption{Convergence history for the 2d cantilever.}
\label{fig:CvgCanti2d}
\end{figure}
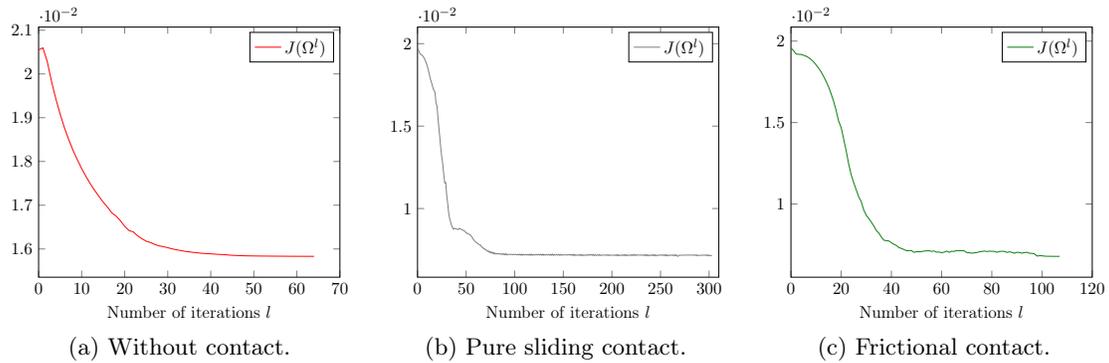

We tested our algorithm on three different physical models: the standard elasticity model without contact (as if the disk was not here), the pure sliding model which does not take into account potential friction (i.e.$\!$ $\mathfrak{F}=0$), and the model of contact with Tresca friction.
In the case without contact, we recover the classical result, although the cantilever obtained might seem a little heavy due to our choice of coefficients $\alpha_1$ and $\alpha_2$. 
In the cases with contact, as expected, the optimal design suggested by the algorithm uses the contact with the rigid foundation as well as the clamped region to gain stiffness. More specifically, it seems that the effective contact zone (active set) has been moved to the right during the process. This makes sense because the closer the contact zone is to the zone where the load is applied, the stiffer will be the structure. 
However, in the frictional case, the tangential stress associated to friction phenomena $\sigmaa_{\normalInt\!\tanExt}$ points to the left and slightly upwards, since it is parallel to $\tanExt$ and opposed to the tangential displacement. This helps the structure to be stiffer as it compensates part of the downward motion induced by the traction $\tauu$, which might explain why the optimal shape requires only one anchor point instead of two for the pure sliding case.

The convergence history for all cases is displayed figure \ref{fig:CvgCanti2d}. Note that the convergence is much faster in the case without contact, which was predictable since the mechanical problem is easier to solve, thus the shape derivatives should be more accurate. Moreover, the final value of $J$ is around $1.6$ in the case without contact whereas it is around $0.7$ in both cases with contact. Indeed, due to the possibility of laying onto a rigid foundation, the models with contact lead to better designs.

\section{Conclusion}

In this work, we expressed conditions (similar to strict complementarity conditions) that ensure shape differentiability of the solution to the penalty formulation of the contact problem with prescribed friction based on a Tresca model. In order to achieve this goal, we relied on Gâteaux differentiability, combined with an assumption on the measure of some subsets of the contact region where non-differentiabilities may occur. Under such assumptions, we derived an expression for the shape derivative of any general functional. 
Finally, this expression has been used in a gradient descent algorithm, which we tested on a revisited version of the classical cantilever benchmark. 

As far as future work is concerned, 
the idea of working with directional shape derivatives could be extended to other formulations where non-Gateaux-differentiable operators are involved: e.g.$\!$ the Augmented Lagrangian formulation or Nitsche-based formulations.


\chapter{Dérivées de forme pour la formulation pénalisée régularisée}     
\label{chap:2.2} 

\section*{Introduction}

Dans ce chapitre, nous étudions le même problème que dans le précédent sous un angle différent. En effet, plutôt que de considérer des dérivées directionnelles puis trouver des hypothèses qui garantissent la différentiabilité au sens de Gateaux, nous régularisons la formulation afin de la rendre Fréchet différentiable. Cette approche est très répandue, notamment dans le contexte de l'optimisation de formes des problèmes de contact. Nous citons \cite{haslinger1986shape} pour la première application au cas bidimensionnel, et plus récemment \cite{maury2017shape}, où les auteurs ont étendu la méthode au cas du contact frottant en trois dimensions.

Dans \cite{maury2017shape}, les auteurs s'intéressent à des problèmes de contact frottant (dont le frottement de Tresca), en écrivant le problème mécanique sous sa forme pénalisée, puis régularisée, de sorte que la formulation considérée ne contient que des termes différentiables (au sens classique). Après avoir montré la différentiabilité par rapport à la forme, ils donnent une expression pour la dérivée de forme d'un critère générique, qu'ils utilisent enfin dans un algorithme d'optimisation de formes de type gradient s'appuyant sur une représentation de la forme par des level-sets. Nous suivons ici les mêmes étapes, mais avec un choix de régularisation différent. Le but est dans un premier temps de retrouver des résultats similaires à ceux de \cite{maury2017shape}. Puis dans un second temps, nous souhaitons étudier le comportement des dérivées de forme lors du passage à la limite du paramètre de régularisation.

Plus précisément, nous commençons par introduire la formulation régularisée, et montrons sa consistance par rapport à la formulation pénalisée non régularisée du chapitre précédent. De plus, après avoir prouvé la (Fréchet) différentiabilité par rapport à la forme pour cette formulation, nous procédons à une analyse de convergence des dérivées matérielles et de forme lorsque le paramètre de régularisation tend vers $+\infty$. L'objectif est de retrouver, à la limite, les dérivées matérielles et de forme de la formulation non régularisée (obtenues au chapitre \ref{chap:2.1}).

\section{L'étape de régularisation}

\subsection{Principe}

Comme nous l'avons mentionné au chapitre précédent, la non différentiabilité de la formulation provient du fait que les fonctions $\maxx$ et $\qq$ ne sont pas Fréchet différentiables. Par conséquent, il suffit de régulariser ces fonctions non lisses pour obtenir une formulation que l'on peut dériver par rapport à la forme. Soient $\maxxc$ et $\qq_c$ des approximations respectives de $\maxx$ et $\qq$, qui sont de classe $\pazocal{C}^1$, de paramètre de régularisation $c$. Nous introduisons des primitives et les dérivées de ces fonctions régulières:
\begin{itemize}[topsep=2pt,leftmargin=*]
    \item $\psi_c$ est la primitive $\maxxc$ qui s'annule en $0$ (approximation $\pazocal{C}^2$ de $\frac{1}{2}\maxx^2$),
    \item $H_c$ est la dérivée de $\maxxc$ (approximation $\pazocal{C}^0$ de la fonction de Heaviside $H$),
    \item $h_c$ est la primitive de $\frac{1}{\varepsilon}\qq_c(\varepsilon\mathfrak{F}s,\cdot)$ qui s'annule en $0$ (aproximation $\pazocal{C}^2$ de $\mathfrak{F}s|\cdot|$),
    \item $\partial_\alpha \qq_c$est la dérivée partielle de $\qq_c$ par rapport à la première variable,
    \item $\partial_z \qq_c$ est la dérivée partielle de $\qq_c$ par rapport à la seconde variable.
\end{itemize}
Dans la section suivante, nous présentons les expressions analytiques de ces fonctions, ainsi que certaines de leurs propriétés.

Pour simplifier les raisonnements, comme dans \cite{maury2016shape}, nous supposerons dans ce chapitre que le coefficient de frottement $\mathfrak{F}$ et le seuil de Tresca $s$ sont des constantes strictement positives.

\subsection{Propriétés des fonctions régularisées}
\label{subsec:PropFctsReg}

Pour chaque fonction, nous commençons par donner sa définition, puis nous énonçons quelques propriétés utiles. 

\subsubsection{Terme normal}
Tout d'abord, $\maxxc$ est défini par:
\begin{equation*}
  \maxxc(x) := 
    \left\{
      \begin{array}{lr}
        x- \frac{1}{2c} & \mbox{si} \: x \geq \frac{1}{c}, \\
        c\frac{x^2}{2} & \mbox{si} \: 0 \leq x \leq \frac{1}{c}, \\
        0 & \mbox{sinon.}
      \end{array}
    \right.
\end{equation*}

\begin{figure}
\begin{center}\shorthandoff{;}
\begin{tikzpicture}[
	declare function={ func(\x)= (\x < 0)*(0) + and(\x >= 0, \x < 0.25)*(4*x^2/2) + (\x >= 0.25)*(\x-0.25/2) ;
					   funcbis(\x)= (\x < 0)*(0) + (\x >= 0)*\x ;
	}
  ]
	\begin{axis}[
		axis x line=middle,
		axis y line=middle,
    	xlabel={$x$},
    	ylabel={$\maxx(x)$, \color{blue}{$\mbox{p}_{c,+}(x)$}},
    	xmin=-1, xmax=1,
    	ymin=-0.5, ymax=1,
    	samples=1001
    	]
	\addplot[] {funcbis(x)};
	\addplot[blue,thick] {func(x)};
	\end{axis}    
\end{tikzpicture}
\caption{Graphes des fonctions $\maxx$ et $\maxxc$ ($c=4$).}
\label{fig:PlotMaxxc}
\end{center}
\end{figure}

Ce qui donne pour sa primitive $\psi_c$ (s'annulant en $0$) et sa dérivée $H_c$:
\begin{equation*}
  \psi_c(x) = 
    \left\{
      \begin{array}{lr}
        \frac{x^2}{2}-\frac{1}{2c}x+\frac{1}{6c^2} & \mbox{si} \: x \geq \frac{1}{c}, \\
        \frac{c}{6}x^3 & \mbox{si} \: 0 \leq x \leq \frac{1}{c}, \\
        0 & \mbox{sinon.}
      \end{array}
    \right.
  \hspace{1em}
  H_c(x) = 
    \left\{
      \begin{array}{lr}
       1 & \mbox{si} \: x \geq \frac{1}{c}, \\
        cx & \mbox{si} \: 0 \leq x \leq \frac{1}{c}, \\
        0 & \mbox{sinon.}
      \end{array}
    \right.
\end{equation*}
On peut montrer facilement que ces fonctions vérifient:
\begin{itemize}[topsep=2pt,parsep=0.05cm,itemsep=0.05cm]
  \item[(i)] $\maxxc$, $\psi_c$ et $H_c$ sont positives,
  \item[(ii)] pour tout $x \in \mathbb{R}$, $x \maxxc(x) \geq 0$,
  \item[(iii)] $\maxxc$ et $\psi_c$ sont convexes,
  \item[(iv)] pour tout $x \in \mathbb{R}$, $\maxxc(x) \leq \maxx(x) \leq |x|$,
  \item[(v)] $\norml \: \maxxc - \maxx \normr_\infty \leq \frac{1}{2c}$.
\end{itemize}

\subsubsection{Terme tangentiel}
Nous donnons maintenant la définition de la régularisation $\qq_c$ associée à la contrainte tangentielle:
\begin{equation*}
  \qq_c(\alpha,z) := \left\{
      \begin{array}{lr}
        z & \mbox{ si } \: |z| \leq \alpha -\frac{1}{c} ,\\
        \big[ \alpha + P_c(|z|-\alpha) \big] \frac{z}{|z|} & \mbox{ si } \: \alpha -\frac{1}{c} \leq |z| \leq \alpha +\frac{1}{c}, \\
        \alpha \frac{z}{|z|} & \mbox{ si } \: |z| \geq \alpha +\frac{1}{c}. 
        \end{array}
    \right.
\end{equation*}
Le polynôme $P_c$ est défini par $P_c(t) := -\frac{c}{4}t^2 + \frac{1}{2}t - \frac{1}{4c}$ et il vérifies donc $P_c\left(-\frac{1}{c}\right)=-\frac{1}{c}$,  $P_c'\left(-\frac{1}{c}\right)=1$, $P_c\left(\frac{1}{c}\right)=P_c'\left(\frac{1}{c}\right)=0$, et $-\frac{1}{c}\leq P_c\leq 0$ sur l'intervalle $\left[ -\frac{1}{c}, \frac{1}{c}\right]$.
\begin{figure}
\begin{center}\shorthandoff{;}
\begin{tikzpicture}[
	declare function={ func(\x)= (\x < -0.5)*(-0.5) + and(\x >= -0.5, \x < 0.5)*(\x) + (\x >= 0.5)*(0.5) ;
					   Pc(\t)= -\t^2 + 0.5*\t - 1./16;
					   funcbis(\x)= (\x < -0.75)*(-0.5) + and(\x >= -0.75, \x < -0.25)*(-(0.5 + Pc(-\x-0.5))) + and(\x >= -0.25, \x < 0.25)*(\x) + and(\x >= 0.25, \x < 0.75)*(0.5 + Pc(\x-0.5)) + (\x >= 0.75)*(0.5) ;
	}
  ]
	\begin{axis}[
		axis x line=middle,
		axis y line=middle,
    	xlabel={$x$},
    	ylabel={$\qq(0.5,x)$, \color{blue}{$\qq_c(0.5,x)$}},
    	xmin=-1, xmax=1,
    	ymin=-1, ymax=1,
    	samples=1001,
    	]
	\addplot[] {func(x)};
	\addplot[blue,thick] {funcbis(x)};
	\end{axis}    
\end{tikzpicture}
\label{fig:PlotQqc}
\caption{Graphes des fonctions $\qq(0.5,\cdot)$ et $\qq_c(0.5,\cdot)$ ($c=4$ et $d=2$).}
\end{center}
\end{figure}
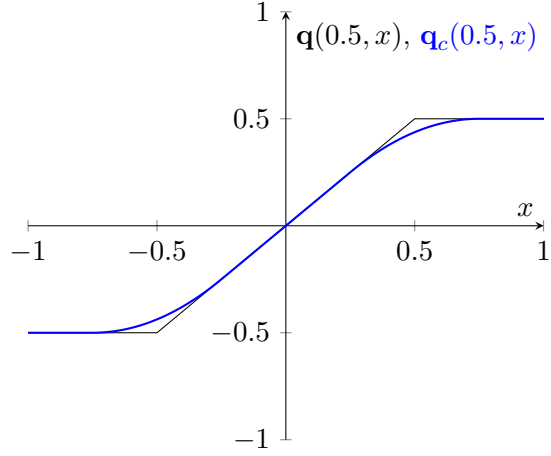

Passons maintenant à $h_c$, la primitive $\frac{1}{\varepsilon}\qq_c(\varepsilon\mathfrak{F}s,\cdot)$ qui s'annule en $0$. Cette fonction vérifie, pour tout $z\in\mathbb{R}^{d-1}$:
\begin{equation*}
  h_c(z) := \frac{1}{\varepsilon} \left\{
      \begin{array}{lr}
        \frac{1}{2}|z|^2 & \mbox{ si } \: |z| \leq \varepsilon\mathfrak{F}s -\frac{1}{c}\: ,\\
        \varepsilon\mathfrak{F}s |z| + Q_c(|z|-\varepsilon\mathfrak{F}s)- \frac{1}{2} (\varepsilon\mathfrak{F}s)^2  & \mbox{ si } \: \varepsilon\mathfrak{F}s -\frac{1}{c} \leq |z| \leq \varepsilon\mathfrak{F}s +\frac{1}{c} ,\\
        \varepsilon\mathfrak{F}s |z| - \frac{1}{2} (\varepsilon\mathfrak{F}s)^2 - \frac{1}{6c^2} & \mbox{ si } \:  |z| \geq \varepsilon\mathfrak{F}s +\frac{1}{c}\: ,
      \end{array}
    \right.
\end{equation*}
où $Q_c$ est la primitive de $P_c$ qui vaut $-\frac{1}{6c^2}$ en $0$, et $\mathfrak{F}$, $s$ sont des constantes indépendantes de $z$.

Par souci de lisibilité, nous introduisons les trois sous-ensembles de  $\mathbb{R}_+^*\times\mathbb{R}^{d-1}$ suivants: 
\begin{equation*}
	\begin{aligned}
		\pazocal{J}^{c,-}&:=\left\{ (\alpha,z) \mbox{ t.q. } |z| \leq \alpha -\frac{1}{c} \right\} \:,\\
		\pazocal{J}^{c,0}&:=\left\{(\alpha,z) \mbox{ t.q. } \alpha -\frac{1}{c} \leq |z| \leq \alpha +\frac{1}{c} \right\} \:,\\
		\pazocal{J}^{c,+}&:=\left\{(\alpha,z) \mbox{ t.q. } |z| \geq \alpha +\frac{1}{c} \right\}\:.
	\end{aligned}
\end{equation*}

On peut exprimer les dérivées partielles de $\qq_c$ par rapport à $\alpha$ et $z$, notées respectivement $\partial_\alpha\qq_c : (\frac{1}{c},+\infty)\times\mathbb{R}^{d-1} \to \mathcal{L}(\mathbb{R};\mathbb{R}^{d-1})$ et $\partial_z \qq_c :(\frac{1}{c},+\infty) \times\mathbb{R}^{d-1} \to \mathcal{L}(\mathbb{R}^{d-1})$, de la manière suivante:
\begin{equation*}
  \partial_\alpha \qq_c(\alpha,z) = \left\{
      \begin{array}{lr}
        0 & \mbox{ in } \pazocal{J}^{c,-} ,\\
        \big[ 1 - P_c'(|z|-\alpha) \big] \frac{z}{|z|} & \mbox{ in } \pazocal{J}^{c,0}\: , \\
        \frac{z}{|z|} & \mbox{ in } \pazocal{J}^{c,+} ,
      \end{array}
    \right.
\end{equation*}
\begin{equation*}
  \partial_z \qq_c(\alpha,z) = \left\{
      \begin{array}{lr}
           \Id & \mbox{in } \pazocal{J}^{c,-}, \\
           \frac{1}{|z|^2} P_c'(|z| - \alpha)z \otimes  z & \: \\
           \qquad + \frac{1}{|z|} \big[ \alpha + P_c(|z| - \alpha) \big] \big(\Id  -  \frac{1}{|z|^2} z \otimes  z\big)& \mbox{in } \pazocal{J}^{c,0}\:, \\
           \frac{\alpha}{|z|}\big(\Id  -  \frac{1}{|z|^2} z \otimes  z\big) & \mbox{in } \pazocal{J}^{c,+},
      \end{array}
    \right.
\end{equation*}
où $\Id$ correspond à la matrice identité de taille $d-1$.

Nous énonçons maintenant quelques propriétés utiles de ces fonctions.
\begin{prpstn}
Soient $\qq_c$, $h_c$, $\partial_\alpha \qq_c$, $\partial_z \qq_c$ telles que définies précédemment, alors on a: 
\begin{itemize}[topsep=2pt,parsep=0.05cm,itemsep=0.05cm]
    \item[(i)] $h_c$ est convexe,
    \item[(ii)] pour tous $\alpha\in\mathbb{R}_+^*$, $z\in\mathbb{R}^{d-1}$, $\qq_c(\alpha,z)z \geq 0$,
    \item[(iii)] pour tous $\alpha\in\mathbb{R}_+^*$, $z\in\mathbb{R}^{d-1}$, $|\qq_c(\alpha,z)| \leq |\qq(\alpha,z)|\leq\alpha$ et $|z|$,
    \item[(iv)] pour tous $\alpha\in\mathbb{R}_+^*$, $z,z'\in\mathbb{R}^{d-1}$, $(\qq(\alpha,z)-\qq(\alpha,z'))(z-z')\geq 0$,
    \item[(v)] pour tout $\alpha\in\mathbb{R}_+^*$, $\norml \qq_c(\alpha,\cdot)- \qq(\alpha,\cdot)\normr_\infty\leq \frac{2}{c}$,
    \item[(vi)] $\partial_\alpha \qq_c$ et $\partial_z \qq_c$ sont continues sur $(\frac{1}{c},+\infty)\times\mathbb{R}^{d-1}$,
    \item[(vii)] $\partial_\alpha \qq_c$ et $\partial_z \qq_c$ sont uniformément bornées,
    \item[(viii)] $\qq_c$ est Lispschitz sur $(\frac{1}{c},+\infty)\times\mathbb{R}^{d-1}$,
    \item[(ix)] $\partial_\alpha\qq_c $ est Lispschitz sur $(\frac{1}{c},+\infty)\times\mathbb{R}^{d-1}$,
    \item[(x)] il existe $K>0$ tel que, pour tous $\alpha, \alpha'\in(\frac{1}{c},+\infty)$, $z, z'\in\mathbb{R}^{d-1}$,
    \vspace{-0.5em}
    \begin{equation*}
        |\partial_z \qq_c(\alpha,z)-\partial_z \qq_c(\alpha',z')|\leq Kc\left( \left| (\alpha,z)-(\alpha',z')\right| + \left| (\alpha,z)-(\alpha',z')\right|^2 \right).
    \end{equation*}
\end{itemize}
\end{prpstn}

\begin{rmrk}
    Puisque $\qq_c$ a des dérivées partielles continues, elle est différentiable, et sa dérivée directionnelle en $(\alpha,z)$ dans n'importe quelle direction $(\beta,h)\in\mathbb{R}\times\mathbb{R}^{d-1}$ s'écrit:
    \begin{equation*}
        \qq_c'\left((\alpha,z);(\beta,h)\right) = \partial_\alpha\qq_c(\alpha,z)\beta+\partial_z\qq_c(\alpha,z)h\:.
    \end{equation*}
\end{rmrk}

\begin{proof}
     Les six premières propriétés sont immédiates. Nous nous concentrons sur les cinq suivantes.
    
    
    (vii): pour tout $(\alpha,z) \in \mathbb{R}_+^*\times\mathbb{R}^{d-1}$, puisque $\norml z\otimes z \normr \leq |z|^2$, on obtient:
    \begin{equation*}
        | \partial_\alpha\qq_c(\alpha,z) | \leq 1 \: , \hspace{3em} \norml \partial_z \qq_c(\alpha,z) \normr \leq 3 \: .
    \end{equation*}
   
    (viii): Soient $(\alpha,z)$ et $(\alpha',z') \in \mathbb{R}_+^*\times\mathbb{R}^{d-1}$, la propriété (vii) permet d'obtenir la majoration suivante:
    \begin{equation*}
        \begin{aligned}
        \left| \qq_c(\alpha',z') - \qq_c(\alpha,z) \right| &\leq \:\left| \qq_c(\alpha',z') - \qq_c(\alpha,z') \right| + \left| \qq_c(\alpha,z') - \qq_c(\alpha,z) \right| \: ,\\
        \: & \leq \: \sup_{\alpha} | \partial_\alpha\qq_c(\alpha,z) | \cdot |\alpha' - \alpha| + \sup_{z} \norml \partial_z \qq_c(\alpha,z) \normr \cdot |z'-z| \: ,\\
        \: & \leq \: 3 | (\alpha',z') - (\alpha,z) | \: .
    \end{aligned}
    \end{equation*}
    
    (ix) et (x): Pour démontrer ces deux dernières propriétés, il suffit de procéder par disjonction des cas, puis de développer les calculs. 
    
\end{proof}

\subsection{Formulation régularisée}

Avec ces nouvelles notations, nous pouvons énoncer une version régularisée de la formulation \eqref{FV}: trouver $\uu_\varepsilon^c \in \Xx$ tel que $\forall \vv \in \Xx$,
\begin{equation}
  a(\uu_\varepsilon^c,\vv) + \frac{1}{\varepsilon} \prodL2{\maxxc\left(  \uu_{\varepsilon,\normalExt}^c -\gG_{\normalExt} \right), \vv_{\normalExt}}{\Gamma_C} + \frac{1}{\varepsilon} \prodL2{\qq_c(\varepsilon\mathfrak{F}s,\uu^c_{\varepsilon,\tanExt}), \vv_{\tanExt}}{\Gamma_C} = L(\vv)\:.
  \label{FVR}
\end{equation}
\begin{rmrk}
    Pour que le terme $\qq_c(\varepsilon\mathfrak{F}s,\uu_{\varepsilon,\tanExt})$ soit régulier, il faut que $\varepsilon\mathfrak{F}s>\frac{1}{c}$ dans $\mathbb{R}^d$, ce qui est toujours le cas si $c$ vérifie $c>\frac{1}{\varepsilon\mathfrak{F}s}$. Nous supposerons que c'est le cas dans la suite. 
\end{rmrk}

On peut montrer que le problème \eqref{FVR} est équivalent au problème de minimisation suivant:
\begin{equation}
   \underset{\vv \in \Xx}{\inf} \left\{ \varphi(\vv) + \frac{1}{\varepsilon} \int_{\Gamma_C} \psi_c (\vv_{\normalExt}-\gG_{\normalExt}) + \int_{\Gamma_C} h_c (\vv_{\tanExt}) \right\}.
  \label{OPTR}
\end{equation}

\begin{thrm}\label{resu:existencePenalisation}
  Sous les hypothèses des sections~\ref{sec:geo.2}--\ref{sec:meca.2}, le problème \eqref{OPTR} est bien posé. Autrement dit, pour tout $\varepsilon > 0$ et $c >0$, $\uu_\varepsilon^c$ existe et est unique.
\end{thrm}

\begin{proof}
  $\varphi$ est semi-continue inférieurement (s.c.i.), strictement convexe et coercive sur $\Xx$, puisque $a$ est elliptique. Les applications $\vv\mapsto \vv_{\normalExt|_{\Gamma_C}}-\gG_{\normalExt}$ et $\vv\mapsto \vv_{\tanExt|_{\Gamma_C}}$ sont affines, elles sont donc convexes continues. D'où, par continuité, convexité et positivité de $\psi_c$ et $h_c$, les applications $\vv \mapsto \int_{\Gamma_C} \psi_c(\vv_{\normalExt}-\gG_{\normalExt})$ et $\vv \mapsto \int_{\Gamma_C} h_c (\vv_{\tanExt})$ sont s.c.i., positives et convexes. Par conséquent, la fonctionnelle à minimiser est strictement convexe et coercive sur $\Xx$ qui est un Hilbert, ce qui donne le résultat.
\end{proof}

\subsection{Convergence des solutions régularisées}

Maintenant que l'existence et l'unicité de la solution au problème régularisé a été démontrée, il reste à prouver que cette solution converge (dans un sens à préciser) vers celle de \eqref{FV} lorsque $c\rightarrow +\infty$. Cette sous-section est dédiée à la preuve du résultat suivant.

\begin{thrm}
    Sous les hypothèses du Théorème~\ref{resu:existencePenalisation}, pour tout $\varepsilon>0$, $\uu_\varepsilon^c \rightarrow \uu_\varepsilon$ fortement dans $\Xx$ lorsque $c \rightarrow +\infty$.
\end{thrm}

\begin{proof}
Pour alléger les notations, nous introduisons:
\begin{equation}\label{eq:notationInterpenetration}
    R_{\normalExt}(\vv) =\maxx\left(\vv_{\normalExt}-\gG_{\normalExt}\right)\:, \qquad  R_{\normalExt}^c(\vv) =\maxxc\left(\vv_{\normalExt}-\gG_{\normalExt}\right)\:,
\end{equation}
\begin{equation}\label{eq:notationFriction}
    S_{\tanExt}(\vv) = \qq(\varepsilon\mathfrak{F}s,\vv_{\tanExt})\:, \qquad S_{\tanExt}^c(\vv) = \qq_c(\varepsilon\mathfrak{F}s,\vv_{\tanExt}) \:.
\end{equation}
Commençons par montrer que $\{ \uu_\varepsilon^c \}_c$ admet une limite faible dans $\Xx$ (à une sous-suite près). En prenant $\uu_\varepsilon^c$ comme fonction-test dans \eqref{FVR}, on obtient
  \begin{equation*}
    a(\uu_\varepsilon^c,\uu_\varepsilon^c) + \frac{1}{\varepsilon} \prodL2{R_{\normalExt}^c(\uu^c_{\varepsilon}), \uu_{\varepsilon,\normalExt}^c}{\Gamma_C} + \frac{1}{\varepsilon} \prodL2{S_{\tanExt}^{c}(\uu^c_{\varepsilon}), \uu_{\varepsilon,\tanExt}^c}{\Gamma_C} = L(\uu_\varepsilon^c)\:.
  \end{equation*}
  Puis, en utilisant l'ellipticité de $a$ et la continuité de $L$,
  \begin{equation*}
    \begin{aligned}
        \alpha_0 \norml \uu_\varepsilon^c \normr_{\Xx}^2 + \frac{1}{\varepsilon} \int_{\Gamma_C} \underset{\geq 0}{\underbrace{R_{\normalExt}^c(\uu^c_{\varepsilon}) (\uu_{\varepsilon,\normalExt}^c-\gG_{\normalExt})}}& + \frac{1}{\varepsilon} \int_{\Gamma_C} \underset{\geq 0}{\underbrace{R_{\normalExt}^c(\uu^c_{\varepsilon})}} \underset{\geq 0}{\underbrace{\gG_{\normalExt}}} \\
        \: &  + \frac{1}{\varepsilon} \int_{\Gamma_C} \underset{\geq 0}{\underbrace{ S_{\tanExt}^{c}(\uu^c_{\varepsilon}) \uu_{\varepsilon,\tanExt}^c }} \:\leq\:  C \norml \uu_\varepsilon^c \normr_{\Xx} \:.
    \end{aligned}
  \end{equation*}  
 D'où $\{ \uu_\varepsilon^c \}_c$ est bornée dans $\Xx$, Banach réflexif. Ainsi, il existe une sous-suite, toujours notée $\{ \uu_\varepsilon^c \}_c$, qui converge faiblement dans $\Xx$, disons vers $\tilde{\uu}_\varepsilon$. 
 
 Clairement, $a(\uu_\varepsilon^c,\vv) \rightarrow a(\tilde{\uu}_\varepsilon,\vv)$ pour tout $\vv \in \Xx$. De plus, on a
  \begin{equation*}
    \begin{aligned}
      \prodL2{R_{\normalExt}^c(\uu^c_{\varepsilon}), \vv_{\normalExt}}{\Gamma_C} - \prodL2{R_{\normalExt}(\tilde{\uu}_{\varepsilon}), \vv_{\normalExt}}{\Gamma_C} 
      &= \prodL2{R_{\normalExt}^c(\uu^c_{\varepsilon})-R_{\normalExt}(\uu_{\varepsilon}^c), \vv_{\normalExt}}{\Gamma_C} 
      \\ & 
      \qquad\qquad + \prodL2{R_{\normalExt}(\uu_{\varepsilon}^c)-R_{\normalExt}(\tilde{\uu}_{\varepsilon}), \vv_{\normalExt}}{\Gamma_C} \:,
    \end{aligned}
  \end{equation*}
  \begin{equation*}
    \begin{aligned}
      \prodL2{S_{\tanExt}^{c}(\uu^c_{\varepsilon}), \vv_{\tanExt}}{\Gamma_C} - \prodL2{S_{\tanExt}(\tilde{\uu}_{\varepsilon}), \vv_{\tanExt}}{\Gamma_C} 
      &= \prodL2{S_{\tanExt}^{c}(\uu^c_{\varepsilon})-S_{\tanExt}(\uu_{\varepsilon}^c), \vv_{\tanExt}}{\Gamma_C} 
      \\ & 
      \qquad\qquad + \prodL2{S_{\tanExt}(\uu_{\varepsilon}^c)-S_{\tanExt}(\tilde{\uu}_{\varepsilon}), \vv_{\tanExt}}{\Gamma_C} \:.
    \end{aligned}
  \end{equation*}
  En utilisant les propriétés des fonctions régularisées, la continuité Lipschitz de $\maxx$ et $\qq$, l'inégalité de Cauchy-Schwarz, et la continuité de l'opérateur de trace, on obtient
 \begin{equation*}
    \begin{aligned}
        \left| \prodL2{R_{\normalExt}^c(\uu^c_{\varepsilon}), \vv_{\normalExt}}{\Gamma_C}
        -
        \prodL2{R_{\normalExt}(\tilde{\uu}_{\varepsilon}), \vv_{\normalExt}}{\Gamma_C} \right| 
        \leq &\;  \frac{K}{c}\norml \vv \normr_{\Xx} + \int_{\Gamma_C} | \uu_{\varepsilon,\normalExt}^c - \tilde{\uu}_{\varepsilon,\normalExt} | \: | \vv_{\normalExt} | 
        \\
        \leq &\;  \frac{K}{c}\norml \vv \normr_{\Xx} + K \norml \uu_{\varepsilon}^c-\tilde{\uu}_{\varepsilon} \normr_{0,\Gamma_C} \norml \vv \normr_{\Xx}\:,
   \end{aligned}
 \end{equation*}
  \begin{equation*}
    \begin{aligned}
        \left| \prodL2{S_{\tanExt}^{c}(\uu^c_{\varepsilon}), \vv_{\tanExt}}{\Gamma_C}
        -
        \prodL2{S_{\tanExt}(\tilde{\uu}_{\varepsilon}), \vv_{\tanExt}}{\Gamma_C} \right| 
        \leq &\;  \frac{K}{c}\norml \vv \normr_{\Xx} + \int_{\Gamma_C} | \uu_{\varepsilon,\tanExt}^c - \tilde{\uu}_{\varepsilon,\tanExt} | \: | \vv | 
        \\
        \leq &\;  \frac{K}{c}\norml \vv \normr_{\Xx} + K \norml \uu_{\varepsilon}^c-\tilde{\uu}_{\varepsilon} \normr_{0,\Gamma_C} \norml \vv \normr_{\Xx}\:.
   \end{aligned}
 \end{equation*}
 Puisque l'injection $\Hh^{1/2}(\Gamma_C) \hookrightarrow \Ll^2(\Gamma_C)$ est compacte, et que $\{ \uu_{\varepsilon}^c \}_c$ est bornée dans $ \Hh^{1/2}(\Gamma_C)$, il existe une sous-suite (toujours notée $\{ \uu_{\varepsilon}^c \}_c$) telle que $\uu_{\varepsilon}^c \rightarrow \tilde{\uu}_{\varepsilon}$ fortement dans $\Ll^2(\Gamma_C)$. Par conséquent, pour tout $\vv \in \Xx$, lorsque $c \rightarrow +\infty$,
 \begin{equation*}
    \begin{aligned}
        \prodL2{R_{\normalExt}^c(\uu^c_{\varepsilon}), \vv_{\normalExt}}{\Gamma_C} &\longrightarrow \: \prodL2{R_{\normalExt}(\tilde{\uu}_{\varepsilon}), \vv_{\normalExt}}{\Gamma_C}, \\
        \prodL2{S_{\tanExt}^{c}(\uu^c_{\varepsilon}), \vv_{\tanExt}}{\Gamma_C} &\longrightarrow \: \prodL2{S_{\tanExt}(\tilde{\uu}_{\varepsilon}), \vv_{\tanExt}}{\Gamma_C}.
    \end{aligned}
 \end{equation*}
 Le passage à la limite faible dans \eqref{FVR} donne:
 \begin{equation*}
   a(\tilde{\uu}_\varepsilon,\vv) + \frac{1}{\varepsilon} \prodL2{R_{\normalExt}(\tilde{\uu}_{\varepsilon}), \vv_{\normalExt}}{\Gamma_C} + \frac{1}{\varepsilon} \prodL2{S_{\tanExt}(\tilde{\uu}_{\varepsilon}), \vv_{\tanExt}}{\Gamma_C} = L(\vv), \:\:\:\:\: \forall \vv \in \Xx.
 \end{equation*}  
 On en déduit que $\tilde{\uu}_\varepsilon=\uu_\varepsilon$, par unicité de la solution à cette formulation variationnelle \eqref{FVPena}. L'unicité de $\uu_\varepsilon$, associée à la bornitude de $\{ \uu_\varepsilon^c\}_c$, permet de montrer que toute la suite converge (faiblement) vers $\uu_\varepsilon$.
 
 On s'intéresse maintenant à la convergence forte de $\{ \uu_\varepsilon^c \}_c$ vers $\uu_\varepsilon$ dans $\Xx$. La manière classique de procéder consiste à prendre $\boldsymbol{\delta}_{\uu,\varepsilon}^c:=\uu_\varepsilon^c-\uu_\varepsilon$ comme fonction-test dans \eqref{FV} et \eqref{FVR}, puis soustraire. Ici, cela donne:
 \begin{equation*}
    \begin{aligned}
        a(\boldsymbol{\delta}_{\uu,\varepsilon}^c,\boldsymbol{\delta}_{\uu,\varepsilon}^c) & + \frac{1}{\varepsilon} \prodL2{R_{\normalExt}^c(\uu^c_{\varepsilon})-R_{\normalExt}(\uu_{\varepsilon}), (\boldsymbol{\delta}_{\uu,\varepsilon}^c)_{\normalExt}}{\Gamma_C} \\
        \: & + \frac{1}{\varepsilon} \prodL2{S_{\tanExt}^{c}(\uu^c_{\varepsilon})-S_{\tanExt}(\uu_{\varepsilon}), (\boldsymbol{\delta}_{\uu,\varepsilon}^c)_{\tanExt}}{\Gamma_C} = 0\:.
   \end{aligned}
 \end{equation*}   
 Puis, en utilisant l'ellipticité de $a$ et l'inégalité triangulaire, il vient
 \begin{align*}
    \alpha_0 \left\| \boldsymbol{\delta}_{\uu,\varepsilon}^c \right\|_{\Xx}^2 \:
    \leq & \: \dfrac{1}{\varepsilon} \left( \norml \: R_{\normalExt}^c(\uu^c_{\varepsilon}) \normr_{0,\Gamma_C} + \norml \: R_{\normalExt}(\uu_{\varepsilon}) \normr_{0,\Gamma_C}  \right) \norml (\boldsymbol{\delta}_{\uu,\varepsilon}^c)_{\normalExt} \normr_{0,\Gamma_C} 
    \\ 
    \: & \: +  \dfrac{1}{\varepsilon} \left( \norml \: S_{\tanExt}^{c}(\uu^c_{\varepsilon}) \normr_{0,\Gamma_C} + \norml \: S_{\tanExt}(\uu_{\varepsilon}) \normr_{0,\Gamma_C}  \right) \norml (\boldsymbol{\delta}_{\uu,\varepsilon}^c)_{\tanExt} \normr_{0,\Gamma_C} \\ 
    \leq & \: \dfrac{2}{\varepsilon} \left( \norml \uu_{\varepsilon}^c \normr_{0,\Gamma_C}  + \norml \uu_{\varepsilon} \normr_{0,\Gamma_C} + \norml \gG_{\normalExt} \normr_{0,\Gamma_C}\right) \norml \boldsymbol{\delta}_{\uu,\varepsilon}^c \normr_{0,\Gamma_C} \:.
 \end{align*}
 Enfin, puisque $\gG_{\normalExt} \in \pazocal{C}^{2}(\mathbb{R}^d)\subset L^2(\Gamma_C)$, 
 \begin{equation*}
    \norml \boldsymbol{\delta}_{\uu,\varepsilon}^c \normr_{\Xx}^2 \: \leq \: \frac{K}{\varepsilon} \norml \boldsymbol{\delta}_{\uu,\varepsilon}^c \normr_{0,\Gamma_C} \:,
 \end{equation*}
 et la convergence forte $\uu_{\varepsilon}^c \rightarrow \uu_{\varepsilon}$ in $\Ll^2(\Gamma_C)$ permet de conclure. 
\end{proof}

\section{Dérivabilité par rapport à la forme pour la formulation régularisée}
\label{app:MSDiffFVReg}

Le but de cette section est de montrer que $\uu_\varepsilon^c$ est Fréchet différentiable par rapport à la forme en utilisant le théorème des fonctions implicites, selon la méthode présentée dans \cite{henrot2006variation}, et également appliquée dans un contexte similaire au nôtre dans \cite{maury2017shape}. Nous avons déjà suivi cette approche dans le chapitre \ref{chap:1.1}, section \ref{subsec:SDiffElas}, pour traiter le cas de l'élasticité sans contact. Puisque la seule différence entre cette-dernière formulation et la formulation \eqref{FVR} provient des termes issus du contact, ce sont ceux sur lesquels nous allons nous concentrer. Dans la suite, on suppose que l'hypothèse \ref{A1.2} du chapitre précédent est vérifiée. Nous commençons par énoncer le résultat, puis nous le démontrerons.

\begin{thrm}
    La solution $\uu_\varepsilon^c$ de \eqref{FVR} est (fortement) dérivable par rapport à la forme de $\Cc^1_b(\mathbb{R}^d)$ dans $\Ll^2(\Omega)$. 
    \label{ThmExistMSDerFVR}
\end{thrm}

\begin{proof}
    Tout d'abord, notons que puisque la fondation rigide $\Omega_{rig}$ est fixée, $\gG_{\normalExt}$ et $\normalExt$ ne dépendent pas de $\thetaa$. D'après les notations du chapitre précédent, la solution de la formulation \eqref{FVR} transportée à $\Omega(\thetaa)$ devrait s'écrire $\uu_\varepsilon^c(\Omega(\thetaa))$. Pour améliorer la lisibilité, nous la noterons simplement $\uu_{\varepsilon,\thetaa}^c$. Le problème \eqref{FVR} devient donc: trouver $\uu_{\varepsilon,\thetaa}^c \in \Xx_{\thetaa}$ tel que, pour tout $\vv_{\thetaa} \in \Xx_{\thetaa}$,
\begin{equation} 
    \begin{aligned}
        \int_{\Omega(\thetaa)}& \Aa : \epsilonn(\uu_{\varepsilon,\thetaa}^c) : \epsilonn(\vv_{\thetaa}) + \frac{1}{\varepsilon} \int_{\Gamma_C(\thetaa)} \maxxc\left(\uu_{\varepsilon,\thetaa}^c \cdot \normalExt -\gG_{\normalExt}\right) (\vv_{\thetaa} \cdot \normalExt) \\
        \: & + \frac{1}{\varepsilon} \int_{\Gamma_C(\thetaa)}  \qq_c\left(\varepsilon\mathfrak{F}s,(\uu_{\varepsilon,\thetaa}^c)_{\tanExt}\right)(\vv_{\thetaa})_{\tanExt} = \int_{\Omega(\thetaa)} \ff \: \vv_{\thetaa} + \int_{\Gamma_N(\thetaa)} \tauu \: \vv_{\thetaa} \:.
  \end{aligned}
  \label{FVRT}
\end{equation}
Comme pour le cas de l'élasticité, on peut transformer \eqref{FVRT} en une formulation posée sur le domaine de référence $\Omega$. Dans le même esprit que les notations déjà introduites au chapitre \ref{chap:1.1}, section \ref{subsec:SDiffElas}, la composante tangentielle associée à $\normalExt(\thetaa)$ d'un vecteur $v$ sera notée $v_{\tanExt(\thetaa)}$. En réutilisant le fait que $\Xx_{\thetaa}$ et $\Xx$ sont isomorphes, on obtient, pour tout $\vv\in\Xx$,
\begin{equation}
    \begin{aligned}
        \int_{\Omega} \Aa(\thetaa) & : \frac{1}{2} \left( \gradd\uu_{\varepsilon}^{c,\thetaa} {(\Ii+\gradd\thetaa)}^{-1}  + {(\Ii+\gradd\thetaa^T)}^{-1} {\gradd \uu_{\varepsilon}^{c,\thetaa}}^T \right) \\
        \: & : \frac{1}{2} \left( \gradd \vv{(\Ii+\gradd\thetaa)}^{-1} + {(\Ii+\gradd\thetaa^T)}^{-1}{\gradd \vv}^T  \right) \JacV(\thetaa) \\ 
        \: & \hspace{5em}+ \frac{1}{\varepsilon} \int_{\Gamma_{C}} \maxxc\left(\uu_{\varepsilon}^{c,\thetaa} \cdot \normalExt(\thetaa)-\gG_{\normalExt}(\thetaa)\right) (\vv \cdot \normalExt(\thetaa)) \: \JacB(\thetaa) \\
        \: & \hspace{5em} + \frac{1}{\varepsilon} \int_{\Gamma_C}  \qq_c\left(\varepsilon (\mathfrak{F}s)(\thetaa),(\uu_{\varepsilon}^{c,\thetaa})_{\tanExt(\thetaa)}\right)\vv_{\tanExt(\thetaa)} \JacB(\thetaa) \\
        \: & \hspace{5em}- \int_{\Omega} \ff(\thetaa) \: \vv \: \JacV(\thetaa) - \int_{\Gamma_{N}} \tauu(\thetaa) \: \vv \: \JacB(\thetaa) = 0 \:.
        \end{aligned}
  \label{FVRTR}
\end{equation}

On reprend l'opérateur $\pazocal{A}_{\thetaa}: \Xx \rightarrow \Xx^*$, et les deux éléments $\pazocal{B}_{\thetaa}$, $\pazocal{C}_{\thetaa} \in \Xx^*$ introduits au chapitre \ref{chap:1.1}, section \ref{subsec:SDiffElas}. Dans le même esprit, on introduit de plus les deux opérateurs $\pazocal{D}_{c,\thetaa}$, $\pazocal{E}_{c,\thetaa} : \Xx \rightarrow \Xx^*$ linéaires continus tels que: pour tous $\vv$, $\ww\in\Xx$,
\begin{equation*}
    \begin{aligned}
        \prodD{\pazocal{D}_{c,\thetaa} \ww , \vv }{} &:= \frac{1}{\varepsilon}\int_{\Gamma_C} \maxxc\left(\ww \cdot \normalExt(\thetaa)-\gG_{\normalExt}(\thetaa)\right) (\vv \cdot \normalExt(\thetaa)) \: \JacB(\thetaa) \:,\\
        \prodD{\pazocal{E}_{c,\thetaa} \ww , \vv }{} &:= \frac{1}{\varepsilon} \int_{\Gamma_C}  \qq_c\left(\varepsilon (\mathfrak{F}s)(\thetaa),\ww_{\tanExt(\thetaa)}\right)\vv_{\tanExt(\thetaa)} \JacB(\thetaa)\:.
    \end{aligned}
\end{equation*}
On veut prouver que $\Phi_\varepsilon^c : \thetaa \in \Cc^1_b(\mathbb{R}^d) \mapsto \uu_{\varepsilon}^{c,\thetaa} \in \Xx$ est Fréchet différentiable. De façon analogue à la démarche suivie au chapitre \ref{chap:1.1}, section \ref{subsec:SDiffElas}, on introduit l'opérateur:
 \begin{eqnarray*}
   \pazocal{F}_c : \: \:  \Cc^1_b(\mathbb{R}^d) \times \Xx & \longrightarrow & \Xx^* \\
   (\thetaa,\ww) & \longmapsto & \pazocal{A}_{\thetaa} \ww + \pazocal{D}_{c,\thetaa} \ww + \pazocal{E}_{c,\thetaa} \ww - \pazocal{B}_{\thetaa} - \pazocal{C}_{\thetaa}
 \end{eqnarray*}
 Ce qui nous permet de réécrire \eqref{FVRTR}: trouver $\uu_{\varepsilon,\thetaa}^c \in \Xx$ tel que,
 \begin{equation}
  \prodD{\pazocal{F}_c (\thetaa, \uu_{\varepsilon}^{c,\thetaa}), \vv} \:= 0, \hspace{2em} \forall \vv \in \Xx \:.
  \label{FVRTR2}
\end{equation}
Vérifions maintenant que $\pazocal{F}_c$ vérifie les hypothèses du théorème des fonctions implicites.
Même si les fonction $\maxx$ et $\qq$ ont été remplacées par des versions régularisées, ces termes constituent toujours la partie la plus technique dans l'analyse de \eqref{FVRTR2}.

\subsubsection*{$\pazocal{F}_c$ est $\pazocal{C}^1$ par rapport à $\thetaa$.}
\subsubsection*{$\bullet\,\thetaa \in \Cc^1_b \mapsto \pazocal{A}_{\thetaa} \ww, \pazocal{B}_{\thetaa},\,\pazocal{C}_{\thetaa} \in \Xx^*$.}
On a vu chapitre \ref{chap:1.1}, section \ref{subsec:SDiffElas}, que ces termes étaient $\pazocal{C}^1$.
\subsubsection*{$\bullet\,\thetaa \in \Cc^1_b \mapsto \pazocal{D}_{c,\thetaa} \ww \in \Xx^*$.}
Pour commencer, rappelons que $\normalExt\in\Cc^2(\partial\Omega_{rig}^h)$ et $\gG_{\normalExt}\in \pazocal{C}^2(\partial\Omega_{rig}^h)$. Donc, en supposant que $\thetaa$ soit suffisament petit pour que $\Gamma_C(\thetaa)\subset\partial\Omega_{rig}^h$, on peut effectuer un développement de Taylor avec reste intégral, ce qui donne, puisque $\Gamma_C$ est borné,
\begin{equation*}
    \begin{aligned}
        \norml \normalExt(\thetaa)-\normalExt - \gradd\normalExt\thetaa  \normr_{\infty,\Gamma_C} \leq K \norml \thetaa \normr_{\infty,\Gamma_C}^2 \:,\\
        \norml \gG_{\normalExt}(\thetaa)-\gG_{\normalExt} - \gradd\gG_{\normalExt} \thetaa  \normr_{\infty,\Gamma_C} \leq K \norml \thetaa \normr_{\infty,\Gamma_C}^2 \:.
    \end{aligned}
\end{equation*}
En d'autres termes, $\thetaa \mapsto \gG_{\normalExt}(\thetaa) \in L^\infty(\Gamma_C)$ et $\thetaa \mapsto \normalExt(\thetaa) \in \Ll^\infty(\Gamma_C)$ sont de classe $\pazocal{C}^1$. Il suit que $\thetaa \mapsto \ww \cdot \normalExt(\thetaa)-\gG_{\normalExt}(\thetaa) \in L^4(\Gamma_C)$ est $\pazocal{C}^1$ aussi, pour tout $\ww\in\Xx$. Il reste donc à prouver que $\maxxc$ est $\pazocal{C}^1$ de $L^4(\Gamma_C)$ dans $L^2(\Gamma_C)$. Pour ce faire, nous considérons l'opérateur de Nemytskij $\Psi$ (voir \cite{goldberg1992nemytskij}) associé à $\maxxc$:
\begin{equation*}
    \Psi(v)(x) = \psi(x,v(x)):= \maxxc\left(v(x)\right), \mbox{ pour } x \in \Gamma_C.    
\end{equation*}
On vérifie aisément que $\psi$ satisfait à la fois la condition de Carathéodory (CC) et la condition de croissante (GC) définies dans \cite{goldberg1992nemytskij}, ce qui assure que $\Psi$ est continu de $L^4(\Gamma_C)$ dans $L^2(\Gamma_C)$ (\cite[Théorème 4]{goldberg1992nemytskij}, avec $p=4$ et $q=2$). De plus, $\psi$ est Fréchet différentiable par rapport à $v$ dans $\mathbb{R}$, de dérivée $\psi_v(x,v)=H_c(v)$. Puisque $H_c$ est continue et bornée par $1$, $\psi_v$ satisfait aussi (CC) et (GC). Puis, toujours d'après \cite[Théorème 4]{goldberg1992nemytskij}, on a que $\Phi(v)(x) := \psi_v(x,v(x))$ est continu de $L^2(\Gamma_C)$ dans $L^r(\Gamma_C;\pazocal{L}(\mathbb{R}))$, pour n'importe quel $r\geq 1$. En prenant $r=\frac{pq}{p-q}=4$, on déduit de \cite[Théorème 7]{goldberg1992nemytskij} que $\Psi$ est continûment Fréchet différentiable de $L^4(\Gamma_C)$ dans $L^2(\Gamma_C)$. 
  
Ainsi, par composition, l'application $\thetaa \mapsto \maxxc\left(\ww \cdot \normalExt(\thetaa)-\gG_{\normalExt}(\thetaa)\right) \in L^2(\Gamma_C)$ est $\pazocal{C}^1$ pour tout $\ww\in \Xx$. Par ailleurs, $\thetaa \mapsto \JacB(\thetaa) \in L^\infty(\Gamma_C)$ est $\pazocal{C}^1$, et on a déjà vu que $\thetaa \mapsto \vv \cdot \normalExt(\thetaa)\in L^2(\Gamma_C)$ est aussi $\pazocal{C}^1$. On a donc que, pour tout $\ww \in \Xx$, pour tout $c>0$, l'application $\thetaa \in \Cc^1_b \mapsto \pazocal{D}_{c,\thetaa} \ww \in \Xx^*$ est de classe $\pazocal{C}^1$.
  
\subsubsection*{$\bullet\,\thetaa \in \Cc^1_b \mapsto \pazocal{E}_{c,\thetaa} \ww \in \Xx^*$.}
Étant donnée la régularité du produit $\mathfrak{F}s$, l'application $\thetaa\mapsto \mathfrak{F}(\thetaa) s(\thetaa) \in L^2_{\mathfrak{F}s}(\Gamma_C)$ est de classe $\pazocal{C}^1$, où $L^2_{\mathfrak{F}s}(\Gamma_C):=L^2 (\Gamma_C;(\mathfrak{F}s,+\infty))$. Concernant la trace tangentielle, on a par définition $\ww_{\tanExt(\thetaa)}=\ww-(\ww\cdot\normalExt(\thetaa))\normalExt(\thetaa)$, et on sait que $\normalExt(\thetaa)$ est continûment différentiable. On en déduit donc, pour tout $\ww\in\Xx$:
\begin{equation*}
    \norml \ww_{\tanExt(\thetaa)} -\ww_{\tanExt} - \left( \ww\cdot(\gradd\normalExt\thetaa)\right)\normalExt - (\ww\cdot\normalExt) (\gradd\normalExt\thetaa) \normr_{\Ll^4(\Gamma_C)} \leq K \norml \ww \normr_{\Ll^4(\Gamma_C)} \norml \thetaa \normr_{\infty,\Gamma_C}^2 \:.
\end{equation*}
Il reste encore à montrer que $\qq_c: L^2_{\mathfrak{F}s}(\Gamma_C)\times \Ll^4(\Gamma_C) \to \Ll^2(\Gamma_C)$ est $\pazocal{C}^1$. Malheureusement, il est impossible d'utiliser le concept d'opérateur de Nemytskij car $\qq_c$ est définie sur $\mathbb{R}_+^*\times \mathbb{R}^{d-1}$, qui n'est pas un espace vectoriel. On doit donc prouver cette propriété directement.

Premièrement, comme $\qq_c$ est Lipschitz, il est clair que l'application $\qq_c: L^2_{\mathfrak{F}s}(\Gamma_C)\times \Ll^4(\Gamma_C) \to \Ll^2(\Gamma_C)$ est au moins $\pazocal{C}^0$. Pour le reste, il s'agit de prouver que les deux applications
\begin{eqnarray*}
    \Psi_\alpha : & L^2_{\mathfrak{F}s}(\Gamma_C)\times \Ll^4(\Gamma_C) & \to \:\: L^2(\Gamma_C;\mathcal{L}(\mathbb{R};\mathbb{R}^{d-1})) \\
    \: & (\alpha,\zz) & \mapsto \:\: \partial_\alpha \qq_c (\alpha,\zz) \:,
\end{eqnarray*}
  \vspace{-0.8cm}
\begin{eqnarray*}
  \Psi_z : & L^2_{\mathfrak{F}s}(\Gamma_C)\times \Ll^4(\Gamma_C) & \to \:\: L^2(\Gamma_C;\mathcal{L}(\mathbb{R}^{d-1})) \\
    \: & (\alpha,\zz) & \mapsto \:\: \partial_z \qq_c (\alpha,\zz)
\end{eqnarray*}
sont continues pour obtenir la régularité souhaitée. D'après les propriétés de $\partial_\alpha \qq_c$ et $\partial_z \qq_c$ (voir section  \ref{subsec:PropFctsReg}) on a, pour tous $\alpha$, $\alpha'\in L^2_{\mathfrak{F}s}(\Gamma_C)$ et $\zz$, $\zz'\in\Ll^4(\Gamma_C)$:
\begin{equation*}
    \begin{aligned}
        \norml \Psi_\alpha(\alpha,\zz)-\Psi_\alpha(\alpha',\zz') \normr^2_{L^2(\Gamma_C;\mathcal{L}(\mathbb{R};\mathbb{R}^{d-1}))}  &\leq Kc\left(\int_{\Gamma_C}|\alpha-\alpha'|^2+\int_{\Gamma_C}|\zz-\zz'|^2\right) \\
        & \leq Kc\left( \norml \alpha-\alpha' \normr^2_{0,\Gamma_C} \!\!+ \norml \zz-\zz' \normr^2_{\Ll^4(\Gamma_C)} \right),
    \end{aligned}
\end{equation*}
\begin{equation*}
    \begin{aligned}
        \norml \Psi_z(\alpha,\zz)-\Psi_z(\alpha',\zz') \normr^2_{L^2(\Gamma_C;\mathcal{L}(\mathbb{R};\mathbb{R}^{d-1}))}  &\leq Kc\left(\int_{\Gamma_C}|\alpha-\alpha'|^2+\int_{\Gamma_C}|\zz-\zz'|^2\right. \\
        &  \hspace{4em} \left. + \int_{\Gamma_C}|\zz-\zz'|^4 \right) \\
        & \leq Kc\left( \norml \alpha-\alpha' \normr^2_{0,\Gamma_C} + \norml \zz-\zz' \normr^2_{\Ll^4(\Gamma_C)} \right. \\
        &  \hspace{4em} \left. + \norml \zz-\zz' \normr^4_{\Ll^4(\Gamma_C)} \right),
    \end{aligned}
\end{equation*}
ce qui montre que les applications $\Psi_\alpha$ et $\Psi_z$ sont bien continues.

Enfin, par composition, $\thetaa \mapsto \qq_c(\varepsilon(\mathfrak{F}s)(\thetaa),\ww_{\tanExt(\thetaa)}) \in \Ll^2(\Gamma_C)$ est $\pazocal{C}^1$ pour tout $\ww\in X$. Puisque $\theta \mapsto \JacB(\thetaa) \in L^\infty(\Gamma_C)$ est $\pazocal{C}^1$, et que $\theta \mapsto \vv_{\tanExt(\thetaa)} \in \Ll^2(\Gamma_C)$ l'est aussi, on obtient que pour tout $\ww \in \Xx$, pour tout $c>0$, l'application $\thetaa \in \Cc^1_b \mapsto \pazocal{E}_{c,\thetaa} \ww \in \Xx^*$ est de classe $\pazocal{C}^1$.

Tout ceci prouve que pour tout $c>0$ et tout $\ww \in \Xx$, $\thetaa \mapsto \pazocal{F}_c(\thetaa,\ww)$ est $\pazocal{C}^1$ sur $\Cc^1_b(\mathbb{R}^d)$, et en particulier autour de zéro.

\subsubsection*{$D_{\ww} \pazocal{F}_c(0,\uu_{\varepsilon}^c)$ est un isomorphisme.} 
On définit la forme bilinéaire suivante sur $\Xx\times\Xx$: pour tous $\ww$, $\vv \in \Xx$,
\begin{equation*}
    \begin{aligned}
        b_\varepsilon^c(\ww,\vv):=a(\ww,\vv) &+ \frac{1}{\varepsilon} \prodL2{H_c(\uu_{\varepsilon,\normalExt}^c-\gG_{\normalExt}) \ww_{\normalExt} , \vv_{\normalExt}}{\Gamma_C} \\
        \:& + \frac{1}{\varepsilon} \prodL2{\partial_z \qq_c(\varepsilon\mathfrak{F}s,\uu_{\varepsilon,\tanExt}^c) \ww_{\tanExt} , \vv_{\tanExt}}{\Gamma_C} .
    \end{aligned}
\end{equation*}
Avec ces notations, on peut exprimer la différentielle de $\pazocal{F}$ par rapport à $\ww$: pour tous $\ww$, $\vv \in \Xx$,
\begin{equation*}
  \prodD{D_{\ww} \pazocal{F}_c(0,\uu_{\varepsilon}^c) \ww, \vv} \:=\: b_\varepsilon^c(\ww,\vv)
\end{equation*}
Puisque $H_c(\cdot)$ et $\partial_z\qq_c$ sont toutes les deux uniformément bornées et positives, $b_\varepsilon^c$ est une forme bilinéaire continue et coercive sur $\Xx\times\Xx$. Ainsi, le lemme de Lax-Milgram nous permet de déduire que $D_{\ww} \pazocal{F}_c(0,\uu_{\varepsilon}^c) : \Xx \rightarrow \Xx^*$ est un isomorphisme.

On peut donc appliquer le théorème des fonctions implictes, et déduire l'existence d'une fonction $\thetaa \mapsto \vv_c(\thetaa) \in \Xx$ de classe $\pazocal{C}^1$ au voisinage de zéro, et telle que $\pazocal{F}_c(\thetaa,\vv_c(\thetaa)) = 0$ dans $\Xx^*$. D'où, par unicité, $\vv_c(\thetaa) = \uu_{\varepsilon}^{c,\thetaa}$, ce qui implique que $\Phi_\varepsilon^c$ est Fréchet différentiable en zéro.

\end{proof}

\section{Convergence des dérivées de forme régularisées}

D'une part, nous avons vu au chapitre précédent que, sous les bonnes hypothèses, $\uu_\varepsilon$ est dérivable par rapport à la forme. Et nous venons de voir d'autre part que $\uu_\varepsilon^c$ est toujours dérivable par rapport à la forme. Il paraît naturel de s'intéresser à la consistance de la régularisation proposée par rapport à la dérivation de forme. Autrement dit, en plus de la convergence $\uu_\varepsilon^c \to \uu_\varepsilon$, on aimerait, pour une direction $\thetaa$ donnée, avoir la convergence des dérivées de forme $\{{\uu_\varepsilon^c}\,'\}$ vers $\uu_\varepsilon'$. Pour cela, il faut commencer par étudier la convergence des dérivées matérielles.

\paragraph{Formulation variationnelle vérifiée par $\dot{\uu}_\varepsilon^c$}
Pour trouver cette formulation, comme nous l'avons vu au chapitre \ref{chap:1.1}, section \ref{subsec:SDiffElas}, il suffit de dériver \eqref{FVR} terme à terme par rapport à forme dans la direction $\thetaa$. 
Ceci revient à écrire la formulation vérifiée par $\uu_\varepsilon^{c,t}$, puis la dériver en $t=0$. Avec les notations de la section précédente, cette formulation s'écrit: $\forall \vv\in \Xx$,
\begin{equation*}
	\prodD{\pazocal{A}_{c,t\thetaa} \uu_\varepsilon^{c,t}, \vv }{} + \prodD{\pazocal{D}_{c,t\thetaa} \uu_\varepsilon^{c,t} , \vv }{} + \prodD{\pazocal{E}_{c,t\thetaa} \uu_\varepsilon^{c,t}, \vv }{} = \prodD{\pazocal{B}_{t\thetaa}, \vv }{} + \prodD{\pazocal{C}_{t\thetaa}, \vv }{}\:.
\end{equation*}
En dérivant chacun de ces termes par rapport à $t$ en $t=0$, on obtient alors: 

\begin{itemize}[leftmargin=*]
 
  \item $\left. \dfrac{\partial}{\partial t} \prodD{\pazocal{A}_{t\thetaa}\uu_\varepsilon^{c,t},\vv}{} \right|_{t=0} = a(\dot{\uu}_\varepsilon^c,\vv)+a'(\uu_\varepsilon^c,\vv)$ .
  
  \item $\displaystyle \left. \frac{\partial}{\partial t} \prodD{\pazocal{D}_{c,t\thetaa} \uu_\varepsilon^{c,t} , \vv }{} \right|_{t=0} = \frac{1}{\varepsilon} \int_{\Gamma_C} R_{\normalExt}^c(\uu_{\varepsilon}^c)\left(\vv \cdot (\divv_\Gamma \thetaa \normalExt + \normalExt'[\thetaa])\right)$ 
  
  \hspace{10em}  $\displaystyle + \: \frac{1}{\varepsilon} \int_{\Gamma_C}  H_c(\uu_{\varepsilon,\normalExt}^c-\gG_{\normalExt})\left( \dot{\uu}_{\varepsilon,\normalExt}^c + \uu^c_{\varepsilon,\normalExt'[\thetaa]} - \gG_{\normalExt}'[\thetaa] \right)\vv_{\normalExt}$ .
 
  \item $\displaystyle \left. \frac{\partial}{\partial t} \prodD{\pazocal{E}_{c,t\thetaa} \uu_\varepsilon^{c,t} , \vv }{} \right|_{t=0} = \frac{1}{\varepsilon} \int_{\Gamma_{C}} S_{\tanExt}^c(\uu_{\varepsilon}^c) \left(\vv_{\tanExt} \divv_\Gamma \thetaa + \vv_{\tanExt'[\thetaa]} \right)$
  
  \hspace{10em}  $\displaystyle + \: \frac{1}{\varepsilon}  \int_{\Gamma_{C}} \left( \partial_\alpha \qq_c \left(\varepsilon\mathfrak{F}s,\uu_{\varepsilon,\tanExt}^c\right) \gradd(\mathfrak{F}s)\thetaa \right) \vv_{\tanExt}$
  
  \hspace{10em}  $\displaystyle + \: \frac{1}{\varepsilon} \int_{\Gamma_C}  \left( \partial_z\qq_c\left(\varepsilon\mathfrak{F}s,\uu_{\varepsilon,\tanExt}^c\right)\left( \dot{\uu}_{\varepsilon,\tanExt}^c + \uu^c_{\varepsilon,\tanExt'[\thetaa]} \right)\right)\vv_{\tanExt}$ .  
    
  \item $\displaystyle \left. \frac{\partial}{\partial t} \prodD{\pazocal{B}_{t\thetaa},\vv}{} \right|_{t=0} = \int_\Omega (\divv \thetaa \: \ff + \gradd \ff \: \thetaa) \: \vv$ .
  
  \item $\displaystyle \left. \frac{\partial}{\partial t} \prodD{\pazocal{C}_{t\thetaa},\vv}{} \right|_{t=0} = \int_{\Gamma_N} (\divv_\Gamma \thetaa \: \tauu+ \gradd \tauu \: \thetaa) \: \vv$ .
  
\end{itemize}
On en déduit donc que $\dot{\uu}_\varepsilon^c$ est solution de:
\begin{equation}
        b_\varepsilon^c(\dot{\uu}_\varepsilon^c,\vv) = L_\varepsilon^c[\thetaa](\vv) \:, \hspace{1em} \forall \vv \in \Xx ,
    \label{FVRMDer}
\end{equation}
où la forme bilinéaire $b_\varepsilon^c$ a été définie plus haut, et la forme linéaire $L_\varepsilon^c[\thetaa]$ est telle que, pour tout $\vv\in\Xx$:
\begin{equation*}
    \begin{aligned}
        L_\varepsilon^c[\thetaa](\vv) := &\int_\Omega (\divv \thetaa \: \ff + \gradd \ff  \thetaa) \vv + \int_{\Gamma_N} (\divv_\Gamma \thetaa \: \tauu+ \gradd \tauu \thetaa)\vv - \:a'(\uu_\varepsilon^c,\vv) \\
        & - \frac{1}{\varepsilon} \int_{\Gamma_C} R_{\normalExt}^c(\uu_{\varepsilon}^c)\left(\vv \cdot (\divv_\Gamma \thetaa \normalExt + \normalExt'[\thetaa])\right) \\
        & - \frac{1}{\varepsilon} \int_{\Gamma_C} H_c(\uu_{\varepsilon,\normalExt}^c-\gG_{\normalExt}) \left( \uu_{\varepsilon,\normalExt'[\thetaa]}^c - \gG_{\normalExt}'[\thetaa] \right)\vv_{\normalExt} \\
        & - \frac{1}{\varepsilon} \int_{\Gamma_{C}} S_{\tanExt}^c(\uu_{\varepsilon}^c) \left(\vv_{\tanExt} \divv_\Gamma \thetaa + \vv_{\tanExt'[\thetaa]} \right) \\
        & - \int_{\Gamma_{C}} \left( \partial_\alpha \qq_c \left(\varepsilon\mathfrak{F}s,\uu_{\varepsilon,\tanExt}^c\right) \gradd(\mathfrak{F}s)\thetaa \right) \vv_{\tanExt} \\
        & - \frac{1}{\varepsilon} \int_{\Gamma_{C}} \left( \partial_z \qq_c \left(\varepsilon\mathfrak{F}s,\uu_{\varepsilon,\tanExt}^c\right) \uu_{\varepsilon,\tanExt'[\thetaa]}^c\right) \vv_{\tanExt}\:.
    \end{aligned}
\end{equation*}
On remarque que \eqref{FVRMDer} est une version régularisée de \eqref{FVMDer}. De plus, il est clair que $b_\varepsilon^c$ et $L_\varepsilon^c[\thetaa]$ satisfont toutes les conditions pour que le lemme de Lax-Milgram s'applique, ce qui assure que $\dot{\uu}_\varepsilon^c$ est bien définie comme l'unique solution de \eqref{FVRMDer}.

\begin{rmrk}
    D'après les propriétés de la section \ref{subsec:PropFctsReg}, on sait que les constantes d'ellipticité et de continuité de $b_\varepsilon^c$, tout comme la constante de continuité de $L_\varepsilon^c[\thetaa]$, ne dépendent pas de $c$.
\end{rmrk}

\begin{thrm}
  Sous l'hypothèse \ref{A1.2}, on a $\dot{\uu}_\varepsilon^c \rightarrow \dot{\uu}_\varepsilon $ fortement dans $\Xx$.
  \label{ThmCvgMDer}
\end{thrm}

De façon analogue à $\pazocal{I}_\varepsilon^{0}$ et $\pazocal{J}_\varepsilon^{0}$ (voir chapitre \ref{chap:2.1}, p.62), on introduit les sous-ensembles de $\Gamma_C$ suivants:
\begin{equation*}
    \pazocal{I}_\varepsilon^{+}:=\{\uu_{\varepsilon,\normalExt}-\gG_{\normalExt} > 0\}, \hspace{0.5em}  \pazocal{J}_\varepsilon^{-}:=\{|\uu_{\varepsilon,\normalExt}| < \varepsilon\mathfrak{F}s\}, \hspace{0.5em} \pazocal{J}_\varepsilon^{+}:=\{|\uu_{\varepsilon,\normalExt}| > \varepsilon\mathfrak{F}s\}.
\end{equation*}
Les sous-ensembles correspondants pour la solution régularisée $\uu_\varepsilon^c$ sont également introduits:
\begin{equation*}
    \pazocal{I}_\varepsilon^{c,0}:=\{\uu_{\varepsilon,\normalExt}^c-\gG_{\normalExt} = 0\}\:, \hspace{1em} \pazocal{I}_\varepsilon^{c,+}:=\{\uu_{\varepsilon,\normalExt}^c-\gG_{\normalExt} > 0\}\:,
\end{equation*}
\begin{equation*}
    \pazocal{J}_\varepsilon^{c,-}:=\{|\uu_{\varepsilon,\normalExt}^c| < \varepsilon\mathfrak{F}s\}, \hspace{0.5em} \pazocal{J}_\varepsilon^{c,0}:=\{|\uu_{\varepsilon,\normalExt}^c| = \varepsilon\mathfrak{F}s\}, \hspace{0.5em} \pazocal{J}_\varepsilon^{c,+}:=\{|\uu_{\varepsilon,\normalExt}^c| > \varepsilon\mathfrak{F}s\}.
\end{equation*}
La preuve du théorème ci-dessus repose sur le lemme préliminaire suivant. 
\begin{lmm}
    À une sous-suite près, on a convergence presque partout (p.p.$\!$) sur $\Gamma_C$ des produits de fonctions suivants:
\begin{itemize}[parsep=0.05cm,itemsep=0.05cm]
    \item[(i)] $\chi_{\pazocal{I}_\varepsilon^{c,+}\backslash\pazocal{I}_\varepsilon^{0}} H_c(\uu_{\varepsilon,\normalExt}^c-\gG_{\normalExt}) \longrightarrow \chi_{\pazocal{I}_\varepsilon^{+}}$ lorsque $c\to\infty$,
    \item[(ii)] $\chi_{\pazocal{J}_\varepsilon^{c,\pm}\backslash\pazocal{J}_\varepsilon^{0}} \partial_\alpha \qq_c\left(\varepsilon\mathfrak{F}s,\uu_{\varepsilon,\tanExt}^c\right) \longrightarrow \chi_{\pazocal{J}_\varepsilon^{\pm}} \partial_\alpha \qq \left(\varepsilon\mathfrak{F}s,\uu_{\varepsilon,\tanExt}\right)$ lorsque $c\to\infty$,
    \item[(iii)] $\chi_{\pazocal{J}_\varepsilon^{c,\pm}\backslash\pazocal{J}_\varepsilon^{0}} \partial_z \qq_c\left(\varepsilon\mathfrak{F}s,\uu_{\varepsilon,\tanExt}^c\right) \longrightarrow \chi_{\pazocal{J}_\varepsilon^{\pm}} \partial_z \qq \left(\varepsilon\mathfrak{F}s,\uu_{\varepsilon,\tanExt}\right)$ lorsque $c\to\infty$.
\end{itemize}
\end{lmm}

\begin{proof}
  Ces trois propriétés se démontrant de façon analogue, nous ne ferons la preuve que de la première.
  
  On montre dans un premier temps que $\pazocal{I}_\varepsilon^{+} = \underset{c\to\infty}{\lim} \left( \pazocal{I}_\varepsilon^{c,+}\backslash\pazocal{I}_\varepsilon^{0} \right)$ pour une certaine sous-suite.
  
  L'inclusion directe $\pazocal{I}_\varepsilon^{+} \subset \underset{c}{\lim\inf} \left( \pazocal{I}_\varepsilon^{c,+}\backslash\pazocal{I}_\varepsilon^{0} \right)$ est assez évidente. En effet, soit $x\in\pazocal{I}_\varepsilon^{+}$, alors $\uu_{\varepsilon,\normalExt}(x)-\gG_{\normalExt}(x)>0$, disons $\uu_{\varepsilon,\normalExt}(x)-\gG_{\normalExt}(x)>\delta$, avec $\delta>0$. Puisque $\uu_{\varepsilon}^c \rightarrow \uu_{\varepsilon}$ fortement dans $\Xx$, alors $\uu_{\varepsilon,\normalExt}^c \rightarrow \uu_{\varepsilon,\normalExt}$ fortement dans $L^2(\Gamma_C)$, donc il existe une sous-suite (indépendante de $x$) qui converge p.p.$\!$ sur $\Gamma_C$. En particulier, pour cette sous-suite, on a, $\exists c_0>0$, tel que $\forall c \geq c_0$, $| \uu_{\varepsilon,\normalExt}^c(x) - \uu_{\varepsilon,\normalExt}(x)| < \delta/2$. D'où $\forall c\geq c_0$, $\uu_{\varepsilon,\normalExt}^c(x)-\gG_{\normalExt}(x)>\delta/2>0$, c'est-à-dire que pour tout $c\geq c_0$, $x\in\pazocal{I}_\varepsilon^{c,+}\backslash\pazocal{I}_\varepsilon^{0}$.
  
  De plus, soit $x \in  \underset{c}{\lim\sup} \left( \pazocal{I}_\varepsilon^{c,+}\backslash\pazocal{I}_\varepsilon^{0} \right)$. Alors, par définition:
  \begin{equation*}
    \forall c>0, \: \exists c_1 \geq c  \:\:\mbox{tel que} \:\: \uu_{\varepsilon,\normalExt}^{c_1}(x)-\gG_{\normalExt}(x) > 0  \:\:\mbox{et} \:\:  \uu_{\varepsilon,\normalExt}(x)-\gG_{\normalExt}(x) \neq 0 \: .
  \end{equation*}
  Supposons que $x \notin \pazocal{I}_\varepsilon^{+}$, alors $\uu_{\varepsilon,\normalExt}(x)-\gG_{\normalExt}(x) <0$ (puisque ce terme est $\neq 0$), i.e.$\!$ il existe $\delta' >0$ tel que $\uu_{\varepsilon,\normalExt}(x)-\gG_{\normalExt}(x) < -\delta'$. Par ailleurs, la convergence p.p.$\!$ de $\{ \uu_{\varepsilon,\normalExt}^c \}$ donne, $\exists c_0'>0$, tel que $\forall c \geq c_0'$, $| u_{\varepsilon,\normalExt}^c(x) - u_{\varepsilon,\normalExt}(x)| < \delta'/2$. Il suit que $\forall c\geq c_0'$, $\uu_{\varepsilon,\normalExt}^c(x)-\gG_{\normalExt}(x) < -\delta'/2$, ce qui contredit la définition. Et donc, $x\in \pazocal{I}_\varepsilon^{+}$. \\
  Par conséquent, on a:
  \begin{equation*}
    \pazocal{I}_\varepsilon^{+} \subset \underset{c}{\lim\inf} \left( \pazocal{I}_\varepsilon^{c,+}\backslash\pazocal{I}_\varepsilon^{0} \right) \subset \underset{c}{\lim\sup} \left( \pazocal{I}_\varepsilon^{c,+}\backslash\pazocal{I}_\varepsilon^{0} \right) \subset \pazocal{I}_\varepsilon^{+} 
  \end{equation*}
  ce qui donne le résultat. Puisqu'on a égalité des ensembles, on peut en déduire la propriété de convergence suivante pour les fonctions caractéristiques:
  \begin{equation*}
    \chi_{\pazocal{I}_\varepsilon^{c,+}\backslash\pazocal{I}_\varepsilon^{0}} \longrightarrow \chi_{\pazocal{I}_\varepsilon^{+}} \hspace{1em} \mbox{p.p. sur} \:\: \Gamma_C . 
  \end{equation*}
  Maintenant, pour presque tout (p.p.t.$\!$) $x \in \Gamma_C$, si $\uu_{\varepsilon,\normalExt}(x)-\gG_{\normalExt}(x) =0$, alors $\chi_{\pazocal{I}_\varepsilon^{c,+}\backslash\pazocal{I}_\varepsilon^{0}}(x)H_c(\uu_{\varepsilon,\normalExt}^c(x)-\gG_{\normalExt}(x)) = 0 = \chi_{\pazocal{I}_\varepsilon^{+}}(x)$. 
  
  Sinon, $\uu_{\varepsilon,\normalExt}(x)-\gG_{\normalExt}(x)=\pm \delta$, avec $\delta>0$. D'où, en invoquant une fois de plus la convergence ponctuelle de $\{ \uu_{\varepsilon,\normalExt}^c \}$, on déduit que $\uu_{\varepsilon,\normalExt}^c(x)-\gG_{\normalExt}(x) \in \pm (\delta/2, +\infty)$ pour $c$ plus grand qu'un certain $c_0$ fixé. Ainsi, puisque $\delta$ ne dépend que de $x$, $H_c(\uu_{\varepsilon,\normalExt}^c(x)-\gG_{\normalExt}(x)) = \chi_{\pazocal{I}_\varepsilon^{+}}(x)$ pour $c$ suffisamment grand. En combinant ceci avec le résultat de convergence des fonctions caractéristiques, on obtient que la convergence p.p.$\!$ de
  \begin{equation*}
        \chi_{\pazocal{I}_\varepsilon^{c,+}\backslash\pazocal{I}_\varepsilon^{0}}(x)H_c(\uu_{\varepsilon,\normalExt}^c(x)-\gG_{\normalExt}(x)) \longrightarrow \chi_{\pazocal{I}_\varepsilon^{+}}(x)\:.  
  \end{equation*}

\end{proof}

\begin{proof}
  \textit{(Théorème \ref{ThmCvgMDer})} On peut montrer que $\left\{\dot{\uu}_\varepsilon^c \right\}$ est uniformément bornée dans $\Xx$ en partant de la formulation \eqref{FVRMDer}. En effet, prenons $\dot{\uu}_\varepsilon^c$ comme fonction-test. Puisque la forme bilinéaire est elliptique et que $L_\varepsilon^c$ est continue, on a:
  \begin{equation*}
    \norml \dot{\uu}_\varepsilon^c \normr_{\Xx} \leq \frac{K'}{\alpha_0} \: ,
  \end{equation*}
  où la constante $K'$ ne dépend pas de $c$. Donc $\left\{ \dot{\uu}_\varepsilon^c \right\}$ converge faiblement dans $\Xx$ (à une sous-suite près), disons vers $\tilde{\ww}_\varepsilon  \in \Xx$. On en déduit immédiatement que tous les termes de  $b_\varepsilon^c(\dot{\uu}_\varepsilon^c,\vv)$ et $L_\varepsilon^c[\thetaa](\vv)$ excluant ceux sur $\Gamma_C$ convergent vers les termes « non-régularisés »  correspondants de $b_\varepsilon(\tilde{\ww}_\varepsilon,\vv)$ et $L_\varepsilon[\thetaa](\vv)$. Il est en revanche nécessaire de détailler la preuve pour ceux qui proviennent des dérivées de $\maxxc$ et $\qq_c$. Nous allons traiter ces cinq termes (deux dans $b_\varepsilon^c(\dot{\uu}_\varepsilon^c,\vv)$ et trois dans  $L_\varepsilon^c[\thetaa](\vv)$) en utilisant l'hypothèse \ref{A1.2} ainsi que le lemme préliminaire.
 
  D'après la définition de $H_c$, l'hypothèse \ref{A1.2} (stipulant que $\pazocal{I}_\varepsilon^{0}$ et $\pazocal{J}_\varepsilon^{0}$ sont de mesure nulle) nous permet de réécrire les deux termes liés au contact de $b_\varepsilon^c(\dot{\uu}_\varepsilon^c,\vv)$:
  \begin{equation}
    \begin{aligned}
         \int_{\Gamma_C}\chi_{\pazocal{I}_\varepsilon^{c,+}\backslash\pazocal{I}_\varepsilon^{0}}H_c(\uu_{\varepsilon,\normalExt}^c-\gG_{\normalExt})\:\dot{\uu}^c_{\varepsilon,\normalExt} \:\vv_{\normalExt}\:, \quad 
         \int_{\Gamma_C}\chi_{\complement(\pazocal{J}_\varepsilon^{0})} \left( \partial_z \qq_c\left(\varepsilon\mathfrak{F}s,\uu_{\varepsilon,\tanExt}^c\right) \dot{\uu}^c_{\varepsilon,\tanExt}\right) \vv_{\tanExt} \:,
    \end{aligned}
    \label{TechnicalTermsBilForm}
  \end{equation}
  ansi que les trois liés au contact de $L_\varepsilon^c[\thetaa](\vv)$:
  \begin{equation}
    \begin{aligned}
         \int_{\Gamma_C}\chi_{\pazocal{I}_\varepsilon^{c,+}\backslash\pazocal{I}_\varepsilon^{0}}H_c(\uu_{\varepsilon,\normalExt}^c-\gG_{\normalExt})\left( \uu_{\varepsilon,\normalExt'[\thetaa]}^c - \gG_{\normalExt}'[\thetaa] \right)\vv_{\normalExt} \:, \\
         \int_{\Gamma_C}\chi_{\complement(\pazocal{J}_\varepsilon^{0})} \left( \partial_\alpha \qq_c \left(\varepsilon\mathfrak{F}s,\uu_{\varepsilon,\tanExt}^c\right) \gradd(\mathfrak{F}s)\thetaa \right) \vv_{\tanExt} \:, \\
         \int_{\Gamma_C}\chi_{\complement(\pazocal{J}_\varepsilon^{0})} \left( \partial_z \qq_c \left(\varepsilon\mathfrak{F}s,\uu_{\varepsilon,\tanExt}^c\right) \uu_{\varepsilon,\tanExt'[\thetaa]}^c\right) \vv_{\tanExt}\:.
    \end{aligned}
    \label{TechnicalTermsLinForm}
  \end{equation}  
  Par ailleurs, on sait d'après le lemme que les trois suites
  \begin{equation*}
    \begin{aligned}
      \left\{ \chi_{\pazocal{I}_\varepsilon^{c,+}\backslash\pazocal{I}_\varepsilon^{0}} H_c(\uu_{\varepsilon,\normalExt}^c-\gG_{\normalExt}) \right\}, \quad 
      \left\{ \chi_{\complement(\pazocal{J}_\varepsilon^{0})} \partial_\alpha \qq_c \left(\varepsilon\mathfrak{F}s,\uu_{\varepsilon,\tanExt}^c\right) \right\}, \quad 
      \left\{ \chi_{\complement(\pazocal{J}_\varepsilon^{0})} \partial_z \qq_c \left(\varepsilon\mathfrak{F}s,\uu_{\varepsilon,\tanExt}^c\right) \right\},
    \end{aligned}
  \end{equation*}
  convergent presque partout (à une sous-sous-suite près). Puisque ces suites sont également bornées dans $L^p(\Gamma_C)$, $L^p(\Gamma_C;\mathcal{L}(\mathbb{R}^{d-1};\mathbb{R}))$ et $L^p(\Gamma_C;\mathcal{L}(\mathbb{R}^{d-1}))$, respectivement, pour tout $1< p <+\infty$, alors elles convergent faiblement dans ces espaces.
  
  Commençons par les termes de \eqref{TechnicalTermsLinForm}. Nous les traitons en utilisant le résultat précédent de convergence faible avec $p=2$. En effet, étant données les régularités respectives de $\normalExt'[\thetaa]$, $\gG_{\normalExt}'[\thetaa]$, $\mathfrak{F}s$ et la convergence forte de $\{\uu_\varepsilon^c\}$ dans $\Hh^{\frac{1}{2}}(\Gamma_C)\hookrightarrow \Ll^4(\Gamma_C)$, on a que ces trois termes convergent vers les termes correspondants dans $L_\varepsilon[\thetaa](\vv)$, ce qui donne:
  \begin{equation*}
      L_\varepsilon^c[\thetaa](\vv) \longrightarrow L_\varepsilon[\thetaa](\vv)\:.
  \end{equation*}
  Concernant les termes de \eqref{TechnicalTermsBilForm}, on utilise toujours la même propriété mais avec $p=4$. Puisque l'injection $\Hh^{1/2}(\Gamma_C) \hookrightarrow \Ll^2(\Gamma_C)$ est compacte, $\dot{\uu}^c_{\varepsilon} \to \tilde{\ww}_\varepsilon$ fortement $\Ll^2(\Gamma_C)$ (à une sous-sous-sous-suite près...). Et donc, comme $\vv\in\Ll^4(\Gamma_C)$, on obtient bien la convergence des deux termes de \eqref{TechnicalTermsBilForm}, ce qui entraîne 
  \begin{equation*}
    b_\varepsilon^c(\dot{\uu}^c_{\varepsilon},\vv) \longrightarrow b_\varepsilon(\tilde{\ww}_\varepsilon,\vv) . 
  \end{equation*}
  En d'autres termes, la limite $\tilde{\ww}_\varepsilon$ est solution de \eqref{FVMDer}. Puisque la solution à ce problème est unique, on conclut que $\tilde{\ww}_\varepsilon=\dot{\uu}_\varepsilon$ et que toute la suite  $\left\{\dot{\uu}_\varepsilon^c \right\}$ converge faiblement dans $\Xx$ vers cette limite.
  
  Pour montrer la convergence forte, on prend comme à chaque fois $\dot{\boldsymbol{\delta}}_{\uu,\varepsilon}^c:=\dot{\uu}_\varepsilon^c-\dot{\uu}_\varepsilon $ comme fonction-test dans \eqref{FVRMDer} et \eqref{FVMDer}, puis on soustrait:
  \begin{equation}
    \begin{aligned}
      a\left(\dot{\boldsymbol{\delta}}_{\uu,\varepsilon}^c
      ,\dot{\boldsymbol{\delta}}_{\uu,\varepsilon}^c\right) 
      &+\frac{1}{\varepsilon}\prodL2{H_c(\uu_{\varepsilon,\normalExt}^c-\gG_{\normalExt})\:\dot{\uu}^c_{\varepsilon,\normalExt} -H(\uu_{\varepsilon,\normalExt}-\gG_{\normalExt})\:\dot{\uu}_{\varepsilon,\normalExt}, \left(\dot{\boldsymbol{\delta}}_{\uu,\varepsilon}^c\right)_{\normalExt}}{\Gamma_C}  \\
      & +\frac{1}{\varepsilon}\prodL2{\partial_z \qq_c\left(\varepsilon\mathfrak{F}s,\uu_{\varepsilon,\tanExt}^c\right) \dot{\uu}^c_{\varepsilon,\tanExt} - \partial_z \qq\left(\varepsilon\mathfrak{F}s,\uu_{\varepsilon,\tanExt}\right) \dot{\uu}_{\varepsilon,\tanExt} , \left(\dot{\boldsymbol{\delta}}_{\uu,\varepsilon}^c\right)_{\tanExt}}{\Gamma_C}  \\
      & = L_\varepsilon^c[\thetaa]\left(\dot{\boldsymbol{\delta}}_{\uu,\varepsilon}^c \right) - L_\varepsilon[\thetaa]\left(\dot{\boldsymbol{\delta}}_{\uu,\varepsilon}^c \right) \:.
  \end{aligned}
  \label{ConvDFV}
  \end{equation} 
  Les deuxième et troisième termes de l'expression précédente peuvent être majorés comme suit:
  \begin{equation*}
    \begin{aligned}
      \left| \prodL2{H_c(\uu_{\varepsilon,\normalExt}^c-\gG_{\normalExt})\:\dot{\uu}^c_{\varepsilon,\normalExt} -H(\uu_{\varepsilon,\normalExt}-\gG_{\normalExt})\:\dot{\uu}_{\varepsilon,\normalExt}, \left(\dot{\boldsymbol{\delta}}_{\uu,\varepsilon}^c \right)_{\normalExt}}{\Gamma_C} \right|  \hspace{6em} \\
      \leq \int_{\Gamma_C} \left( | \dot{\uu}_\varepsilon^c | + | \dot{\uu}_\varepsilon  | \right) \left| \dot{\boldsymbol{\delta}}_{\uu,\varepsilon}^c \right| \leq K\left( \norml  \dot{\uu}_\varepsilon^c \normr_{\Xx} + \norml \dot{\uu}_\varepsilon  \normr_{\Xx} \right) \norml \dot{\boldsymbol{\delta}}_{\uu,\varepsilon}^c  \normr_{0,\Gamma_C} \: ,
    \end{aligned}
  \end{equation*}
  \begin{equation*}
    \begin{aligned}
      \left| \prodL2{\partial_z \qq_c\left(\varepsilon\mathfrak{F}s,\uu_{\varepsilon,\tanExt}^c\right) \dot{\uu}^c_{\varepsilon,\tanExt} - \partial_z \qq\left(\varepsilon\mathfrak{F}s,\uu_{\varepsilon,\tanExt}\right) \dot{\uu}_{\varepsilon,\tanExt} , \left(\dot{\boldsymbol{\delta}}_{\uu,\varepsilon}^c\right)_{\tanExt}}{\Gamma_C} \right|  \hspace{6em} \\
      \leq \int_{\Gamma_C} \left( | \dot{\uu}_\varepsilon^c | + | \dot{\uu}_\varepsilon  | \right) \left| \dot{\boldsymbol{\delta}}_{\uu,\varepsilon}^c \right| \leq K\left( \norml  \dot{\uu}_\varepsilon^c \normr_{\Xx} + \norml \dot{\uu}_\varepsilon  \normr_{\Xx} \right) \norml \dot{\boldsymbol{\delta}}_{\uu,\varepsilon}^c  \normr_{0,\Gamma_C} \: .
    \end{aligned}
  \end{equation*}
   Comme on sait que $\{ \dot{\uu}_\varepsilon^c\}$ est bornée dans $\Xx$ et que toute la suite $\dot{\uu}_\varepsilon^c \rightarrow \dot{\uu}_\varepsilon $ fortement dans $\Ll^2(\Gamma_C)$, on en conclut que ces deux termes tendent vers $0$.
  
  Par ailleurs, les termes à droite de l'égalité dans \eqref{ConvDFV} se réécrivent:
  \begin{equation*}
    \begin{aligned}
       L_\varepsilon^c[\thetaa]\left(\dot{\boldsymbol{\delta}}_{\uu,\varepsilon}^c \right) - L_\varepsilon[\thetaa]\left(\dot{\boldsymbol{\delta}}_{\uu,\varepsilon}^c \right) 
       &= \int_\Omega (\divv \thetaa \: \ff + \gradd \ff  \thetaa) \dot{\boldsymbol{\delta}}_{\uu,\varepsilon}^c + \int_{\Gamma_N} (\divv_\Gamma \thetaa \: \tauu+ \gradd \tauu \thetaa)\dot{\boldsymbol{\delta}}_{\uu,\varepsilon}^c \\
       & \hspace{1.1em} - a'\left(\uu_\varepsilon^c-\uu_\varepsilon,\dot{\boldsymbol{\delta}}_{\uu,\varepsilon}^c \right) \\
       & \hspace{1.1em} - \frac{1}{\varepsilon} \int_{\Gamma_C} \left( R_{\normalExt}^c(\uu^c_{\varepsilon}) - R_{\normalExt}(\uu_{\varepsilon})\right) \left(\dot{\boldsymbol{\delta}}_{\uu,\varepsilon}^c \cdot (\divv_\Gamma \thetaa \normalExt + \normalExt'[\thetaa])\right) \\
       & \hspace{1.1em} - \frac{1}{\varepsilon} \int_{\Gamma_C} \left( H_c(\uu_{\varepsilon,\normalExt}^c-\gG_{\normalExt})\uu_{\varepsilon,\normalExt'[\thetaa]}^c  - H(\uu_{\varepsilon,\normalExt}-\gG_{\normalExt})\uu_{\varepsilon,\normalExt'[\thetaa]} \right)\left(\dot{\boldsymbol{\delta}}_{\uu,\varepsilon}^c \right)_{\normalExt}  \\
       & \hspace{1.1em} + \frac{1}{\varepsilon} \int_{\Gamma_C} \left( H_c(\uu_{\varepsilon,\normalExt}^c-\gG_{\normalExt}) - H(\uu_{\varepsilon,\normalExt}-\gG_{\normalExt})\right)  \gG_{\normalExt}'[\thetaa] \left(\dot{\boldsymbol{\delta}}_{\uu,\varepsilon}^c \right)_{\normalExt} \\
       & \hspace{1.1em} - \frac{1}{\varepsilon} \int_{\Gamma_{C}} \left( S_{\tanExt}^c(\uu_{\varepsilon}^c) -  S_{\tanExt}(\uu_{\varepsilon})\right) \left(\left(\dot{\boldsymbol{\delta}}_{\uu,\varepsilon}^c \right)_{\tanExt} \divv_\Gamma \thetaa + \left(\dot{\boldsymbol{\delta}}_{\uu,\varepsilon}^c \right)_{\tanExt'[\thetaa]} \right) \\
       & \hspace{1.1em} - \int_{\Gamma_{C}} \left( \partial_\alpha \qq_c \left(\varepsilon\mathfrak{F}s,\uu_{\varepsilon,\tanExt}^c\right) - \partial_\alpha \qq \left(\varepsilon\mathfrak{F}s,\uu_{\varepsilon,\tanExt}^c\right) \right) \left(\gradd(\mathfrak{F}s)\thetaa \right) \left(\dot{\boldsymbol{\delta}}_{\uu,\varepsilon}^c \right)_{\tanExt} \\
       & \hspace{1.1em} - \frac{1}{\varepsilon} \int_{\Gamma_{C}} \left( \partial_z \qq_c \left(\varepsilon\mathfrak{F}s,\uu_{\varepsilon,\tanExt}^c\right) \uu_{\varepsilon,\tanExt'[\thetaa]}^c - \partial_z \qq \left(\varepsilon\mathfrak{F}s,\uu_{\varepsilon,\tanExt}\right) \uu_{\varepsilon,\tanExt'[\thetaa]} \right) \left(\dot{\boldsymbol{\delta}}_{\uu,\varepsilon}^c \right)_{\tanExt} \:.
    \end{aligned}
  \end{equation*}
  On peut majorer ces termes un à un, et ensuite, en utilisant la convergence forte de $\{\uu_\varepsilon^c\}$ ainsi que la convergence faible de $\{\dot{\uu}_\varepsilon^c\}$, on montre que toute cette expression tend vers $0$.
  Puis l'ellipticité de $a$ termine la preuve.

\end{proof}

Évidemment, le résultat de convergence des dérivées matérielles donne directement le résultat de convergence suivant pour les dérivées de forme.
\begin{crllr}
  Sous l'hypothèse \ref{A1.2}, on a que $ \uu_\varepsilon^c\,' \rightarrow \uu_\varepsilon' $ fortement dans $\Ll^2(\Omega)$.
  \label{CorCvgSDer}
\end{crllr}

\chapter*{Transition}

Dans la partie précédente, nous avons proposé une approche par dérivées directionnelles permettant d'écrire les conditions d'optimalité d'ordre 1 associées au problème d'optimisation de formes pour la formulation pénalisée du contact. En particulier, l'idée était de s'appuyer sur les propriétés de différentiabilité connues des opérateurs de projection $\maxx$ et $\qq$ intervenant dans la formulation. Les résultats théoriques obtenus permettent de mettre en \oe uvre une méthode numérique de type gradient pour résoudre numériquement le problème d'optimisation de formes, dans le cas où le modèle de contact utilisé est la pénalisation.

Bien qu'elle soit la plus fréquemment utilisée dans les applications industrielles, la méthode de pénalisation présente deux inconvénients majeurs. Premièrement, elle est \textit{inconsistante}, c'est-à-dire que la solution $\uu_\varepsilon$ de la formulation pénalisée est différente de la solution $\uu$ de l'IV d'origine car elle dépend du paramètre de pénalisation $\varepsilon$. Deuxièmement, si on souhaite avoir un résultat précis, i.e.$\!$ si on veut que $\uu_\varepsilon$ soit proche de $\uu$, alors on doit choisir de très grandes valeurs pour $\varepsilon$, ce qui détériore fortement le conditionnement de la matrice du système algébrique associé, et a donc une influence néfaste sur la robustesse de la méthode.  

De ce point de vue, la formulation lagrangienne constitue une alternative intéressante. Le principe consiste à régulariser le Lagrangien du problème d'origine sans en modifier le point selle, ce qui en fait une approche consistante. Grâce à cette régularisation, on peut réécrire le problème comme une formulation variationnelle mixte (sans inégalités), qu'on résout habituellement en découplant les inconnues à l'aide d'une méthode itérative de type point fixe. De plus, la convergence de cette méthode itérative ne dépend d'aucun paramètre. Par ailleurs, comme cette formulation fait intervenir les mêmes opérateurs de projection $\maxx$ et $\qq$ que la formulation pénalisée, les idées de la partie précédente peuvent être adaptées puis réutilisées.

Dans la partie qui suit, nous nous intéressons précisément à l'optimisation de formes du problème de contact écrit sous sa formulation lagrangien augmenté. En pratique, la solution donnée par la méthode itérative correspond à la solution obtenue à la dernière itération. Nous écrivons donc, à l'aide des résultats de la partie précédente, les conditions d'optimalité relativement à la formulation résolue à chaque itération de la méthode de point-fixe, qui a une forme très semblable à la formulation pénalisée. À partir de la suite des solutions obtenues à chaque itération, nous construisons la suite des dérivées directionnelles de forme associées. Après avoir montré que ces dérivées directionnelles étaient en fait des dérivées coniques, nous donnons des conditions suffisantes pour que la solution à une itération donnée soit dérivable par rapport à la forme au sens classique, ce qui nous permet d'avoir un résultat similaire au Théorème \ref{ThmDJ}. Enfin, nous déterminons des conditions suffisantes pour que la suite des dérivées coniques de forme converge vers la dérivée conique de forme de la solution de l'IV d'origine. 

Nous avons fait le choix de nous concentrer ici sur le contact glissant, car les résultats de dérivabilité conique de la solution de l'IV sont plus difficiles à obtenir dans le cas du modèle avec frottement de Tresca, voir \cite{sokolowski1992introduction}. Cependant, tous les résultats portant sur la formulation de la méthode itérative s'étendent facilement au cas avec frottement de Tresca. De même, toujours pour cette formulation, on pourrait adapter les preuves du chapitre \ref{chap:2.2} pour montrer la consistance d'une étape de régularisation supplémentaire.

\part{Formulation lagrangien augmenté des problèmes de contact}

\chapter{Shape derivatives for an augmented Lagrangian formulation of elastic contact problems}     
\label{chap:3.1}                   

\section*{Résumé}
    Ce travail s'intéresse à l'optimisation de la forme d'un solide élastique en contact glissant (Signorini) avec une fondation rigide. Nous écrivons le problème mécanique sous sa formulation Lagrangien augmenté, que nous résolvons par l'approche itérative classique. 
    En pratique (numériquement), la méthode itérative n'est pas totalement convergé, on s'intéresse donc aux conditions d'optimalité associées à la formulation résolue à chaque itération.
    Comme cette formulation fait intervenir un opérateur de projection, elle n'est pas dérivable par rapport à la forme au sens classique. Cependant, nous arrivons à prouver que sa solution admet une dérivée de forme conique, ainsi qu'à exprimer des conditions suffisantes pour qu'elle soit de plus dérivable par rapport à la forme. 
    L'analyse de la suite des dérivées de forme coniques des itérées nous permet d'établir les conditions pour la convergence vers la dérivée de forme conique du problème de contact original. 
    Enfin, 
    après avoir établi des conditions de différentiabilité, nous donnons   
    une expression pour la dérivée de forme d'une fonction coût générale à l'aide d'un état adjoint.  

\section*{Abstract}
     This work deals with shape optimization of an elastic body in sliding contact (Signorini) with a rigid foundation. The mechanical problem is written under its augmented Lagrangian formulation, then solved using a classical iterative approach. For practical reasons we are interested in applying the optimization process with respect to an intermediate solution produced by the iterative method. Due to the projection operator involved 
     at each iteration, the iterate solution is not classically shape differentiable. However, using an approach based on directional derivatives, we are able to prove that it is conically differentiable with respect to the shape, and express sufficient conditions for shape differentiability. 
     From the analysis of the sequence of conical shape derivatives of the iterative process, conditions are established for the convergence to the conical derivative of the original contact problem. 
     Finally, differentiability conditions having been established, 
     we get an expression for the shape derivative of a generic functional thanks to an adjoint state. 

\section{Introduction}
Structural optimization has become an integral part of industrial conception, with applications in more and more challenging mechanical contexts. Those contexts often lead to complex mathematical formulations involving non-linearities and/or non-differentiabilities, which causes many difficulties when considering the associated shape optimization or optimal control problem. 

In this article, we study a shape optimization problem in the context of contact mechanics. Especially, the physical system models an elastic body (without restriction on its dimension) coming in sliding contact with a rigid foundation, which takes the mathematical form of an elliptic variational inequality (VI) of the first kind. As variational inequalities involve projection operators, differentiation with respect to the control parameter (in order to derive optimality conditions or use a gradient descent optimization method) is not an easy task. 


Using the terminology from \cite{heinemann2016shape}, let us gather the approaches to treat optimal control or shape optimization problems for variational inequalities in two families: the ones following the \textit{first optimize then discretize} paradigm, and the others following the \textit{first discretize then optimize} paradigm. In the first one, the idea is to work with a weaker notion of differentiability, namely conical differentiability (see Definition \ref{DefConDiff}), in order to get optimality conditions. Let us mention the pioneer work \cite{mignot1976controle}, as well as other works in the same direction \cite{mignot1984optimal}, the series of works \cite{sokolowski1982shape,sokolowski1985derivee,sokolowski1986sensitivity,soklowski1987shape, sokolowski1987shape,sokolowski1988shape,sokolowski1992introduction,sokolowski4differential} and also \cite{jaruvsek2003conical,heinemann2016shape}. In those works, the conical derivative of the solution with respect to the control parameter is expressed by means of a variational inequality. Thus the optimality conditions obtained might be difficult to use in practice. However, in \cite{neittaanmaki1988optimization}, for the two-dimensional Signorini problem, the authors give conditions for some specific functional to be shape differentiable. This allows them to get an explicit expression for the shape derivative of the functional using an adjoint state. In the second family of approaches, first the variational inequality is discretized, then tools from subdifferential calculus are used in the finite dimensional setting. The interested reader is referred to the series of papers \cite{kovccvara1994optimization,beremlijski2002shape,haslinger2012shape,beremlijski2014shape} dedicated to shape optimization for contact problems. There, the authors manage to characterize an outer approximation of the shape subdifferential of the functional to minimize, then use a bundle algorithm.


The approach proposed here follows 
the \textit{first optimize then discretize} paradigm, and aims at finding conditions that gives simple, explicit and usable expression for the shape derivative of some generic cost functional. It can somehow be understood as a continuation of \cite{neittaanmaki1988optimization}, targeting the development of the basic elements needed for a future efficient numerical implementation. In order to facilitate shape sensitivity of the contact problem, we write the sliding contact conditions using the normal to the rigid foundation instead of the normal to the body, as in \cite{chaudet2019shape}. 
Furthermore, the variational inequality arising from the basic formulation of sliding contact problems is transformed using an Augmented Lagrangian Formulation (ALF) which we plan on solving using a basic iterative approach, the Augmented Lagrangian method (ALM). 
Aiming at practical use of the ALM in a numerical shape optimization process, we study consistency of this iterative process with respect to shape differentiation. In other words, we aim at deriving conditions for the shape derivatives of the iterates generated by this method to converge to the shape derivative of the solution to the original problem.


This work is structured as follows. 
Section 2 presents the original problem, its different formulations 
as well as some notations and notions related to convex analysis and conical differentiability. 
We also introduce the augmented Lagrangian method (ALM) applied to this problem. 
Section 3 is dedicated to shape optimization and is divided in three parts. 
In the first one, we give a proof (adapted from the classical one) that the solution of this formulation is conically differentiable. 
In the second one, we prove the same property for each of the iterates generated by the iterative algorithm. 
Convergence analysis of those conical shape derivatives to the one obtained for the original formulation is studied in the third part. 
Finally, we give the expression of the derivative of a general cost functional $J$ for a sliding contact mechanical system.

\section{Problem formulation}

\subsection{Geometrical setting}\label{sec:geo}
Here, the body $\Omega \subset \mathbb{R}^d$, $d \in \{2,3\}$, is assumed to have $\pazocal{C}^1$ boundary, and to be in contact with a rigid foundation $\Omega_{rig}$, which has a $\pazocal{C}^3$ compact boundary $\partial\Omega_{rig}$, see Figure \ref{SchOpen}. Let $\Gamma_D$ be the non-empty part of the boundary where homogeneous Dirichlet conditions apply (green part), $\Gamma_N$ the part where a Neumann condition $\tauu$ applies, and $\Gamma_C$ the potential contact zone (blue part), such that $\Gamma_D$, $\Gamma_N$, $\Gamma_C$ are open and $\overline{\Gamma_D} \cup \overline{\Gamma_N} \cup \overline{\Gamma_C} = \partial \Omega$. In order to avoid technical difficulties, it is assumed that $\overline{\Gamma_C} \cap \overline{\Gamma_D} = \emptyset$.

The outward normal to $\partial\Omega$ is denoted $\normalInt$. Similarly, the inward normal vector to $\partial\Omega_{rig}$ is denoted $\normalExt$.

\subsection{Notation, function spaces and preliminaries}\label{sec:esp}
Throughout this article, for any $\pazocal{O}\subset \mathbb{R}^d$, $L^p(\pazocal{O})$ represents the usual set of $p$-th power measurable functions on $\pazocal{O}$, and $\left(L^p(\pazocal{O})\right)^d = \Ll^p(\pazocal{O})$. The scalar product defined on $L^2(\pazocal{O})$ or $\Ll^2(\pazocal{O})$ is denoted (without distinction) by  $\prodL2{\cdot,\cdot}{\pazocal{O}}$ and its norm $\|\cdot\|_{0,\pazocal{O}}$. 

The Sobolev spaces, denoted $W^{m,p}(\pazocal{O})$ with $p\in [1,+\infty]$, $p$ integer are defined as 
$$
W^{m,p}(\pazocal{O}) = \left\{u\in L^p(\pazocal{O})\: | \: D^{\alpha} u \in L^p(\pazocal{O})\ \forall |\alpha|\le m\right\},
$$
where $\alpha$ is a multi-index in $\mathbb{N}^d$ and $\Ww^{m,p}(\pazocal{O})= \left(W^{m,p}(\pazocal{O})\right)^d$. The spaces $W^{s,2}(\pazocal{O})$ and $\Ww^{s,2}(\pazocal{O})$, $s\in \mathbb{R}$, are denoted $H^s(\pazocal{O})$ and $\Hh^s(\pazocal{O})$ respectively. Their norm are denoted $\|\cdot\|_{s,\pazocal{O}}$.

Without distinction for the dimension, we denote the duality pairing between $H^{\frac{1}{2}}(\pazocal{O})$ and its dual $H^{-\frac{1}{2}}(\pazocal{O})$ (or between $\Hh^{\frac{1}{2}}(\pazocal{O})$ and $\Hh^{-\frac{1}{2}}(\pazocal{O})$)) by $\prodD{\cdot, \cdot}{\pazocal{O}}$. More generally, for a space $V$ and $V^*$ its dual, we denote the duality pairing by $\prodD{\cdot, \cdot}{V^*,V}$.

The subspace of functions in $H^s(\pazocal{O})$ and $\Hh^s(\pazocal{O})$ that vanish on a part of the boundary $\Gamma\subset\partial \pazocal{O}$ are denoted $H^s_\Gamma(\pazocal{O})$ and $\Hh^s_\Gamma(\pazocal{O})$. In particular, we denote the vector space of admissible displacements $\Xx:=\Hh^1_{\Gamma_D}(\Omega)$.

The sets $\mathbb{T}^2$ and $\mathbb{T}^4$ are the sets of real valued tensors of order 2 and 4 respectively. For any $v$ vector in $\mathbb{R}^d$ or second order tensor in $\mathbb{T}^2$, the product with the normal $v \cdot \normalInt$ (respectively with the normal to the rigid foundation $v \cdot \normalExt$) is denoted $v_{\normalInt}$ (respectively $v_{\normalExt}$). Similarly, the tangential part of  $v$ is denoted $v_{\tanInt} = v - v_{\normalInt} \normalInt$ (respectively $v_{\tanExt} = v - v_{\normalExt} \normalExt$). 

Finally, we introduce some notations from convex analysis. The indicator function and the characteristic function of an arbitrary set $S$ are denoted $I_S$ and $\chi_S$, respectively, that is
\begin{equation*}
    I_S(x) := \left\{
    \begin{array}{cc}
         0 & \mbox{ if } x\in S\:, \\
         +\infty & \mbox{ if } x\notin S\:,
    \end{array}
    \right.
    \hspace{2em}
    \chi_S(x) := \left\{
    \begin{array}{cc}
         1 & \mbox{ if } x\in S\:, \\
         0 & \mbox{ if } x\notin S\:.
    \end{array}
    \right.    
\end{equation*}
Let $H$ be a Hilbert space, and $C$ a non-empty closed convex subset of $H$, then for any $x\in H$, the unique projection of $x$ onto $C$ is denoted $\Proj_C(x)$. Moreover, if $b$ is a bilinear form inducing an inner product on $H$, the projection of $x$ onto $C$ with respect to this inner product is denoted  $\Proj^b_C(x)$.

Let us now recall some notions related to conical differentiability, see \cite{mignot1976controle}.

If $H$ denotes a Hilbert space and $b$ a coercive symmetric bilinear form on $H\times H$, then for any $K\subset H$ and $y\in K$, we define:
\begin{itemize}
    \item the \textit{polar cone} of $K$ with respect to $b$ as $\left[ K \right]_b^0 := \{ x\in H \: : \: \forall z\in K, \: b(x,z)\leq 0 \}$,
    \item the \textit{radial cone} of $K$ at $y$ as $C_y(K):= \{ w \in H \: : \: \exists t>0, \: y+tw\in K \}$,
    \item the \textit{tangent cone} of $K$ at $y$ as $S_y(K):= \overline{C_y(K)}$,
    \item the cone $S^y(K):= S_y(K)\cap \left[ \mathbb{R}(v-y)\right]^0_b$, where $v\in y + \left[ S_y(K) \right]^0_b$ and $\mathbb{R}(v-y) := \left\{\eta(v-y)\ |\ \eta \in \mathbb{R}\right\}$.
\end{itemize}
\begin{defn}
    Let $K\subset H$ be a closed convex set. $K$ is said to be \textit{polyhedric} at $v\in H$ if, denoting $y=\Proj^b_K(v)$, one has:
    $$
      S^y(K) = \overline{C_y(K)\cap \left[ \mathbb{R}(v-y)\right]^0_b} \:.
    $$
\end{defn}
\begin{defn}
    Let $K\subset H$ be a closed convex set. $K$ is said to be \textit{polyhedric} in $H$ if it is polyhedric at each $v\in H$.
\end{defn}
\begin{defn}\label{DefConDiff}
    Let $V_1$, $V_2$ be two Banach spaces. A continuous function $f:V_1\to V_2$ admits a conical derivative
at $x$ if there exists an operator $Q:V_1\to V_2$ positively homogeneous such that:
$$
    \forall h \in V_1, \:\forall t > 0,\quad f (x + th) = f (x) + tQ(h) + o(t)\:.
$$
\end{defn}
Using these notations, we may recall one of the main results on the differentiability of projection operators, namely \cite[Théorème 2.1]{mignot1976controle}.
\begin{theo}\label{ThmMignotPoly}
    Let $v\in H$ and $y=\Proj^b_K(v)$. If $K$ is polyhedric at $v$, then the projection $\Proj^b_K$ is conically differentiable at $v$, with conical
    derivative $\Proj^b_{S^y(K)}$. In other words, for all $w\in H$ and $t>0$:
    $$
        \Proj^b_K(v+tw) = y + t\Proj^b_{S^y(K)}(w) + o(t) \:.
    $$
\end{theo}

\subsection{Mechanical model}\label{sec:meca}

In this work the material is assumed to verify the linear elasticity hypothesis (small deformations and Hooke's law, see for example \cite{Cia1988a}), associated with the small displacements assumption (see \cite{KikOde1988}).
The physical displacement is denoted $\uu$, and belongs to $\Xx$.
The stress tensor is defined by $\sigmaa(\uu) = \Aa : \epsilonn(\uu)$, where $\boldsymbol{\epsilon}(\uu)=\frac{1}{2}(\gradd\uu+\gradd\uu^T)$ denotes the linearized strain tensor, and $\Aa$ is the elasticity tensor. This elasticity tensor is a fourth order tensor belonging to $L^\infty(\Omega,\mathbb{T}^4)$, and it is assumed to be elliptic (with constant $\alpha_0$). Regarding external forces, the body force $\ff \in \Ll^2(\Omega)$, and traction (or surface load) $\tauu \in \Hh^{\frac{1}{2}}(\Gamma_N)$. 

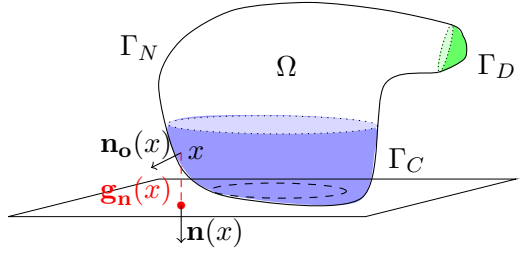
\begin{figure}
\begin{center}
\begin{tikzpicture}

\fill [green!60] plot [smooth cycle] 
coordinates {(3.9,1.8) (4.055,1.64) (4.095,1.5) (4.055,1.34) (3.73,1.175) (3.76,1.5)};

\draw[black,fill=green!15, thin,densely dotted,rotate=-15] (3.3,2.425) ellipse (0.05cm and .33cm);

\fill [blue!40] plot [smooth  cycle] 
coordinates {(2.795,.50) (2.87,.1) (2.65,-.55) (1.5, -.55) (.5,-.3) (0.15,.5) (1.5, .6)};

\draw[black, fill=blue!15, thin,dotted, rotate=-0.5] (1.51,0.485) ellipse (1.4cm and .12cm);

\draw [black] plot [smooth cycle] 
coordinates {(0,1) (0.3, 1.6) (1,2) (2, 2.1) (3,2) (3.9,1.8) (4.055,1.64) (4.095,1.5) (4.055,1.34) (3.73,1.175) (3,1) (2.7,-.5) (1,-.5) (0.3, -.1)};

\draw[black, fill=blue!40,thin,dashed,rotate=-.5] (1.6,-.39) ellipse (.9cm and .1cm);

\draw[black, ultra thin] (-2,-0.75) -- (2.75,-0.75) -- (4.75,-.25) -- (0, -.25) -- (-2,-0.75);

\node[] at (4.5,1.25) {$\Gamma_D$};
\node[] at (-.25,1.5) {$\Gamma_N$};
\node[] at (3.3,0.) {$\Gamma_C$};
\node[] at (1.7,1.25) {$\Omega$};
\draw[->] (0.3, 0.1) -- (-0.1, -0.1) ;
\node[] at (0.5, 0.1) {$x$};
\node[] at (-0.3, 0.2) {$\normalInt(x)$};
\draw[red, densely dashed] (0.3, 0.1) -- (0.3, -0.6) ;
\node[red] at (0.3,-0.6) {\tiny{$\bullet$}};
\node[red] at (-0.3, -0.4) {$\gG_{\normalExt}(x)$};
\draw[->] (0.3 , -0.6) -- (0.3, -1.1) ;
\node[] at (.75, -0.97) {$\normalExt(x)$};

\end{tikzpicture}
\end{center}
  \caption{Elastic body in contact with a rigid foundation.}
  \label{SchOpen}
\end{figure}

\subsection{Non-penetration condition} At each point $x$ of $\Gamma_C$, let us define the gap $\gG_{\normalExt}(x)$, as the oriented distance function to the rigid foundation at $x$, see Figure \ref{SchOpen}. 
Due to the regularity of the rigid foundation, there exists $h_0$ sufficiently small such that, for all $h<h_0$
\begin{equation*}
    \partial\Omega_{rig}^h := \{ x\in\mathbb{R}^d \ | \ \ |\gG_{\normalExt}(x)| < h \}  \:, 
\end{equation*}
is a neighbourhood of $\partial\Omega_{rig}$ where $\gG_{\normalExt}$ is of class $\pazocal{C}^3$, see \cite{delfour1995boundary}. In particular, this ensures that $\normalExt$ is well defined on $\partial\Omega_{rig}^h$, and that $\normalExt\in \pazocal{C}^2(\partial\Omega_{rig}^h,\mathbb{R}^d)$. Moreover, in the context of small displacements, it can be assumed that the potential contact zone $\Gamma_C$ is such that $\Gamma_C \subset \partial\Omega_{rig}^h$. Hence there exists a neighbourhood of $\Gamma_C$ such that $\gG_{\normalExt}$ and $\normalExt$ are of class $\pazocal{C}^3$ and $\pazocal{C}^2$, respectively. Especially, this implies that the function $\gG_{\normalExt}\normalExt \in \pazocal{C}^2(\partial\Omega_{rig}^h,\mathbb{R}^d)$. It can thus be extended to a function $\gG \in \pazocal{C}^2(\mathbb{R}^d)$ such that $\gG=0$ on $(\partial\Omega_{rig}^{h'})^c$ for some $h'>h$. Since $\overline{\Gamma_C}\cap \overline{\Gamma_D}=\emptyset$, one has that $\overline{\Gamma_D}\subset (\partial\Omega_{rig}^{h'})^c$ for $h$, $h'$ small enough. Consequently, $\gG\in\Xx$.

The non-penetration condition can be stated as follows: $\uu_{\normalExt}\leq \gG_{\normalExt}$ a.e.$\!$ on $\Gamma_C$. Thus, we introduce the closed convex set of admissible deformations that realize this condition, see \cite{eck2005unilateral}:
$$
	\Kk:=\{\vv \in \Xx \:|\: \vv_{\normalExt}\leq \gG_{\normalExt} \:\:\mbox{a.e.$\!$ on}\: \Gamma_C\} = \gG + \Kk_0 \:,
$$
where $\Kk_0$ is a closed convex cone defined as $\Kk_0:=\{\vv \in \Xx \:|\: \vv_{\normalExt}\leq 0 \:\:\mbox{a.e.$\!$ on}\: \Gamma_C\} $.

\begin{remark}
    Our definition of $\Kk$ differs from the usual one since we compute the gap in the direction of the normal $\normalExt$ to the rigid foundation instead of the normal $\normalInt$ to $\partial\Omega$. Actually, under the small displacements hypothesis, the normal vector $\normalExt$ and the gap $\gG_{\normalExt}$ to the rigid foundation can be replaced by $\normalInt$ and $\gG_{\normalInt}$ (we refer to \cite[Chapter 2]{kikuchi1988contact} for the details). This ensures that in our context, using $\normalInt$ or $\normalExt$ to write the formulation has no impact on the solution. However, it will be seen in the next sections that using $\normalExt$ will be really convenient when dealing with shape optimization.
    \label{RmkNN0}
\end{remark}

\subsection{Mathematical formulation of the problem}
Let us introduce the bilinear and linear forms $a : \Xx \times \Xx \rightarrow \mathbb{R}$ and $L : \Xx \rightarrow \mathbb{R}$, such that:
\begin{equation*}
  a(\uu,\vv) := \int_\Omega \Aa : \epsilonn(\uu) : \epsilonn(\vv) \:, \hspace{1.5em} L(\vv) := \int_\Omega \ff \vv + \int_{\Gamma_N} \tauu \vv \:.
\end{equation*}
According to the assumptions of the previous sections, one is able to show (see \cite{Cia1988a}) that $a$ is $\Xx$-elliptic with constant $\alpha_0$ (ellipticity of $\Aa$ and Korn's inequality), symmetric, continuous, and that $L$ is continuous (regularity of $\ff$ and $\tauu$). 

The unknown displacement $\uu$ of the frictionless contact problem is the minimizer of the total mechanical energy of the elastic body, which reads, in the case of pure sliding (unilateral) contact problems:
\begin{equation}
	\underset{\vv\in \Kk}{\inf} \:\: \varphi(\vv) \: := \underset{\vv\in \Kk}{\inf} \:\: \frac{1}{2}a(\vv,\vv)-L(\vv) \:.
 	\label{OPT0}
\end{equation}
It is clear that the space $\Xx$, equipped with the usual $\Hh^1$ norm, is a Hilbert space. Moreover, under the conditions of the previous section, since $\Kk$ is obviously non-empty and the energy functional is strictly convex, continuous and coercive, we are able to conclude that $\uu$ solution of \eqref{OPT0} exists and is unique, see for example \cite{ekeland1999convex}.

It is well known that \eqref{OPT0} may be rewritten as a variational inequality (of the first kind):
\begin{equation}
	a(\uu,\vv-\uu) \:\geq\: L(\vv-\uu)\:, \:\:\: \forall \vv \in \Kk\: .
    \label{IV0}
\end{equation}

Even if this variational inequality is very well known from the theoretical point of view such formulation is not well suited for computational purposes. One approach to get around this difficulty consists in rewriting the formulation so that the constraint is implicitly verified. One of the most frequently used formulation (in practical applications) is the penalty formulation. In \cite{chaudet2019shape}, we suggested a method based on directional derivatives in order to derive optimality conditions for the penalty formulation. 

The penalty formulation of \eqref{IV0} is numerically robust and relatively simple to implement, but its solutions depend on the penalty parameter (we say that the approach is not consistent). From a numerical point of view the simplest way to avoid this inconsistency while retaining the simplicity of the penalty approach is the augmented Lagrangian formulation. The rest of this work is related to the augmented Lagrangian formulation and the basic iterative process to solve this formulation, namely the augmented Lagrangian method.

It should be notice that both formulations write as non-linear non-differentiable variational equations (possibly mixed), and thus lead to the same kind of technical difficulties (related to regularity) when studying the associated shape optimization problem. Here, an approach similar to \cite{chaudet2019shape} will be followed.

\subsection{Lagrangian formulation}
The content of this paragraph is based on \cite{stadler2004infinite}. The reader is referred to the chapter 4 of this monograph for all proofs and technical details.
A mixed formulation of problem \eqref{OPT0} can be recovered in the framework of Fenchel duality theory by deriving the dual problem associated to this minimization problem. This result is formally stated in the following theorem.

\begin{theo} \label{ThmExistLM}
  If $\uu \in \Kk$ is the solution of \eqref{OPT0}, then there exists a unique dual variable $\lambda \in H^{-\frac{1}{2}}(\Gamma_C)$ such that:
  \begin{subequations} \label{LMF0:all}
    \begin{align}
    a(\uu,\vv) - L(\vv) + \prodD{\lambda, \vv_{\normalExt}}{\Gamma_C} &= 0\:, \hspace{1em} \forall \vv \in \Xx\:, \label{LMF0:1}\\
    \prodD{\lambda, \zeta }{\Gamma_C}  &\geq 0\:, \hspace{1em} \forall \:\zeta \in H^\frac{1}{2}(\Gamma_C)\:, \:\: \zeta \geq 0\:, \label{LMF0:2}\\
    \prodD{\lambda, \uu_{\normalExt}-\gG_{\normalExt}}{\Gamma_C} &= 0\:. \label{LMF0:3}
    \end{align}
  \end{subequations}
\end{theo}

\begin{remark}
  It is important to note that, in general, the regularity of $\lambda$ is only $H^{-\frac{1}{2}}(\Gamma_C)$.
\end{remark}

\subsection{Generalized Moreau-Yosida approximation}

The goal of this paragraph is to transform conditions \eqref{LMF0:2}, \eqref{LMF0:3}, into pointwise conditions using Moreau-Yosida regularizations.
Since \eqref{OPT0} is a non-smooth convex optimization problem, one may follow the approach from \cite[Chapter 4]{ito2008lagrange} and introduce a consistent Moreau-Yosida regularization of \eqref{LMF0:all}, provided that the Lagrange multiplier belongs to a Hilbert space. From Theorem \ref{ThmExistLM} 
this is not the case 
since $\lambda \in H^{-\frac{1}{2}}(\Gamma_C)$. However, it is known (see for example \cite[Chapter 4]{stadler2004infinite}) that the additional regularity $\lambda \in L^2(\Gamma_C)$ can be obtained if the set $\{ \uu_{\normalExt}-\gG_{\normalExt}=0\}$ is strictly contained in $\Gamma_C$, which can be formulated as follows: 

\begin{hypothesis}
     $\ \overline{\{ \uu_{\normalExt}-\gG_{\normalExt}=0\}} \subset \Gamma_C$ .
     \label{A1}
\end{hypothesis}

\begin{remark}
  From the mechanical point of view, this assumption means that the points at the boundary of the potential contact zone $\Gamma_C$ do not come in contact with the rigid foundation. Intuitively, in the context of small displacements, it should be the case when $\Gamma_C$ is chosen large enough. Moreover, in practice, this assumption can be checked easily a posteriori.
\end{remark}

From now on, it is assumed that Assumption \ref{A1} holds, thus, one may follow the steps of \cite[Chapter 4]{ito2008lagrange}, which directly leads to the desired consistency result for the regularization. 
Before stating this result, we introduce a last notation, the projection onto $\mathbb{R}_+$ in $\mathbb{R}$ (also called the positive part function) will be denoted $\maxx$ ($\maxx(y):=\max\{0,y\}$, for all $y\in\mathbb{R}$). We are now ready to introduce an Augmented Lagrangian Formulation (or characterization of the solution) of the sliding contact problem.

\begin{theo}
    Suppose Assumption \ref{A1} holds. If $(\uu,\lambda)\in \Xx\times L^2(\Gamma_C)$ denotes the solution of \eqref{LMF0:all}, it verifies, for any $\gamma >0$,
\begin{subequations}\label{ALF:all}
    \begin{align}
        a(\uu,\vv) - L(\vv) + \prodL2{\lambda, \vv_{\normalExt}}{\Gamma_C}  &= 0\:, \hspace{1em} \forall \vv \in \Xx\:, \label{ALF:1}\\
        \lambda - \maxx( \lambda + \gamma (\uu_{\normalExt}-\gG_{\normalExt}) ) &= 0 \:\: \mbox{ a.e. on } \Gamma_C \: . \label{ALF:2}
    \end{align}
\end{subequations}
Conversely, if a pair $(\uu,\lambda)\in \Xx\times L^2(\Gamma_C)$ satisfies \eqref{ALF:all} for some $\gamma >0$, then $\uu$ is the solution of \eqref{OPT0}.
\end{theo}

\begin{prf}
    Let us start by rewriting \eqref{OPT0} in a more suitable form. For this purpose, let $\Lambda \in \mathcal{L}(\Xx,L^2(\Gamma_C))$ denote the normal trace operator, such that for all $\vv\in\Xx$
\begin{equation*}
    \Lambda\!\vv = \vv_{\normalExt}\:.
\end{equation*}
We also introduce the convex closed set $K:=\{ \zeta \in L^2(\Gamma_C), \ \zeta\leq \gG_{\normalExt} \mbox{ a.e. on } \Gamma_C \}$. Using these notations, problem \eqref{OPT0} rewrites as an unconstrained non-smooth convex optimization problem:
\begin{equation}
	\underset{\vv\in \Xx}{\inf} \:\: \varphi(\vv)+I_K(\Lambda\!\vv) \:.
 	\label{OPT2}
\end{equation}
$I_K$ being non-smooth, we introduce $I_{K,\gamma}$ its generalized Moreau-Yosida regularization, with regularization parameter $\gamma >0$. Given $\zeta$, $\rho\in L^2(\Gamma_C)$, one has:
\begin{equation*}
        I_{K,\gamma}(\zeta,\rho) = \underset{\eta \in L^2(\Gamma_C)}{\inf} \left\{ I_K(\eta) + \prodL2{\rho,\zeta-\eta}{\Gamma_C} + \frac{\gamma}{2}\norml \zeta-\eta\normr_{0,\Gamma_C}^2 \right\} \:.
\end{equation*}
Now, due to the indicator function in the $\inf$, one gets:
\begin{equation*}
    \begin{aligned}
        I_{K,\gamma}(\zeta,\rho) & = \underset{\eta \in K}{\inf} \left\{ \frac{\gamma}{2}\norml \frac{\rho}{\gamma} +\zeta-\eta\normr_{0,\Gamma_C}^2 \right\} -\frac{1}{2\gamma}\norml \rho \normr_{0,\Gamma_C}^2 \\
        & = \frac{\gamma}{2}\norml (\Id-\Proj_K)\left(\frac{\rho}{\gamma} +\zeta\right)\normr_{0,\Gamma_C}^2 -\frac{1}{2\gamma}\norml \rho \normr_{0,\Gamma_C}^2 \:.
    \end{aligned}
\end{equation*}
Then, \cite[Theorem 4.45]{ito2008lagrange} ensures that the complementarity conditions \eqref{LMF0:2}, \eqref{LMF0:3} can equivalently be expressed as
\begin{equation}
        \lambda = \left.\frac{\partial }{\partial \zeta}I_{K,\gamma}(\zeta,\rho)\right|_{(\Lambda\!\uu,\lambda)}, 
    \label{ComplCondMYR}
\end{equation}
for any $\gamma>0$, where the pair $(\uu,\lambda)$ is the solution of \eqref{LMF0:all}. 
Moreover, due to definition of $K$, one has for any $\zeta \in L^2(\Gamma_C)$,
\begin{equation*}
    \Proj_{K}(\zeta) = \min\{\zeta,\gG_{\normalExt}\}\:.
\end{equation*}
Therefore, conditions \eqref{ComplCondMYR} directly lead to the desired result.
\end{prf}

\subsection{Augmented Lagrangian method}

The augmented Lagrangian method consists in an iterative algorithm to solve formulation \eqref{ALF:all} using an update strategy for the multiplier. The reader is referred to \cite{ForGlo1983a} for further details about the application of such methods to the numerical resolution of problems in mechanics.
We briefly recall the steps of the algorithm.

\begin{algorithm}
\caption*{\textbf{Algorithm:} ALM}
\begin{enumerate}
  \item Choose $\lambda^0\in L^2(\Gamma_C)$ and set $k=0$.
  \item Choose $\gamma^{k+1}>0$, then find $\uu^{k+1}\in\Xx$ the solution of, 
    \begin{equation}   
        a(\uu^{k+1},\vv) + \prodL2{\maxx\left( \lambda^k + \gamma^{k+1} (\uu^{k+1}_{\normalExt}-\gG_{\normalExt}) \right), \vv_{\normalExt}}{\Gamma_C} = L(\vv) \quad\forall \vv \in \Xx\:,
        \label{ItALM}
    \end{equation}
    \item Update the multiplier following the rule: 
    \begin{equation}
      \lambda^{k+1} = \maxx\left( \lambda^k + \gamma^{k+1} (\uu^{k+1}_{\normalExt}-\gG_{\normalExt}) \right) \:\: \mbox{ a.e. on } \Gamma_C \:.
      \label{UpdateMultALM}
    \end{equation}
    \item While a chosen convergence criterion is not satisfied, set $k=k+1$ and go back to step 2.
\end{enumerate}
\end{algorithm}

It has been shown (see for example in \cite{stadler2004infinite} for a detailed proof) that one gets the following convergence result for this algorithm.

\begin{theo}\label{ThmCvALM}
  Suppose Assumption \ref{A1} holds, then for any choice of parameters $0<\gamma^1\leq\gamma^2\leq \cdots$, the iterates $\uu^k$ converge to $\uu$ strongly in $\Xx$. Moreover, the iterates of the multiplier $\lambda^k$ converge to $\lambda$ weakly in $L^2(\Gamma_C)$.
\end{theo}

\begin{remark}
    When $\lambda^0\in H^{\frac{1}{2}}(\Gamma_C)$, the iterates generated by the ALM satisfy $\left( \uu^k,\lambda^k\right)\in \Xx\times H^{\frac{1}{2}}(\Gamma_C)$ for all $k\geq 1$.
\end{remark}

\section{Shape optimization}
\label{sec:shapeopt}

Given a cost functional $J(\Omega)$ depending explicitly on the domain $\Omega$, and also implicitly, through $y(\Omega)$ the solution of some variational problem on $\Omega$, the optimization of $J$ with respect to $\Omega$ or \textit{shape optimization problem} reads:
\begin{equation}
    \inf_{\Omega \in \pazocal{U}_{ad}} J(\Omega) \:,
    \label{ShapeOPT}
\end{equation}
where $\pazocal{U}_{ad}$ stands for the set of admissible domains. 

 Thanks to Theorem \ref{ThmCvALM}, from a analytical point of view, determining $\uu$ and $\lambda$ using the ALM or any other methods has limited impact on the computation of the shape sensitivity of \eqref{ALF:all} (or even \eqref{LMF0:all}). The shape sensitivities would simply be defined relatively to the converged results of the ALM process.
 However, in practice, given a convergence criterion, the ALM will stop at some iteration $k$, meaning that even though $\uu^{k}$ is assumed to be a good enough approximation of $\uu$, it is not the solution of \eqref{ALF:all}. Therefore the sensivities of \eqref{ALF:all} does not corresponds to the sensitivities of the equations defining $\uu^k$. 
  As mentioned in the introduction, the idea in this work is to perform shape optimization on the approximate formulation \eqref{ItALM} instead of the original formulation \eqref{OPT0}. 
Then, we look for a domain in $\pazocal{U}_{ad}$ that minimizes $J=J^k$, a cost functional depending on $y(\Omega)=\uu^k(\Omega)$ solution of \eqref{ItALM} defined on $\Omega$. Obviously we also want to study the consistency of the sensitivities of the sequence of solutions obtained from the ALM with respect to the sensitivities of the ALF. 

Let $D\subset \mathbb{R}^d$ be a fixed bounded smooth domain, and let $\hat{\Gamma}_D\subset\partial D$ be a part of its boundary which will be the "potential" Dirichlet boundary. This means that for any domain $\Omega\subset D$, the Dirichlet boundary associated to $\Omega$ will be defined as $\Gamma_D:=\partial\Omega \cap \hat{\Gamma}_D$. With these notations, we introduce the set $\pazocal{U}_{ad}$ of all admissible domains, which consists of all smooth open domains $\Omega$ such that the Dirichlet boundary $\Gamma_D\subset \partial D$ is of strictly positive measure, that is:
$$
   \pazocal{U}_{ad} := \{ \Omega \subset D \: | \: \Omega \mbox{ is smooth, and }  |\partial\Omega \cap \hat{\Gamma}_D| > 0 \}.
$$
\subsection{Derivatives}
\label{sec:deriv}
The shape optimization method followed in this work is the so-called \textit{perturbation of the identity}, as presented in~\cite{murat1975etude} and \cite{henrot2006variation}. Let us introduce $\Cc^1_b(\mathbb{R}^d) := {(\pazocal{C}^1(\mathbb{R}^d)\cap W^{1,\infty}(\mathbb{R}^d) )}^d$, equipped with the $d$-dimensional $W^{1,\infty}$ norm, denoted $\norml\cdot \normr_{1,\infty}$. In order to move the domain $\Omega$, let $\thetaa \in \Cc^1_b(\mathbb{R}^d)$ be a geometric deformation vector field. The associated perturbed or transported domain in the direction $\thetaa$ will be defined as: $\Omega(t) := (\Id+t\thetaa)(\Omega)$ for any $t>0$. It is known that for $t$ sufficiently small, more specifically for $t$ such that $t\norml \thetaa\normr_{1,\infty}<1$, $\Id+t\thetaa$ is a diffeomorphism, see for example \cite{henrot2006variation}. This enables to rely on the classical notion of differentiability in Banach spaces to define shape differentiability. To make things clear some basic notions of shape sensitivity analysis from \cite{sokolowski1992introduction} are briefly recalled. 

We denote again $y(\Omega)$ the solution, in some Sobolev space denoted $W(\Omega)$, of a variational formulation posed on $\Omega$. For any fixed $\thetaa$, for any small $t>0$, let $y(\Omega(t))$ be the solution of the same variational formulation posed on $\Omega(t)$. 
If the variational formulation is regular enough, which will be assumed to be true, it can be proved (see~\cite{sokolowski1992introduction}) that $y(\Omega(t))\circl(\Id+t\thetaa)$ also belongs to $W(\Omega)$.
\begin{itemize}[leftmargin=*]
    \item The \textit{Lagrangian derivative} or \textit{material derivative} of $y(\Omega)$ in the direction $\thetaa$ is the element $\dot{y}(\Omega)[\thetaa] \in W(\Omega)$ defined by:
    $$
        \dot{y}(\Omega)[\thetaa] := \lim_{t\searrow 0} \:\frac{1}{t}\left( y(\Omega(t))\circl(\Id+t\thetaa) - y(\Omega)\right) \: .
    $$
    If the limit is computed weakly in $W(\Omega)$ (respectively strongly), we talk about \textit{weak} material derivative (respectively \textit{strong} material derivative). Moreover, when the map $\thetaa \mapsto \dot{y}(\Omega)[\thetaa]$ is positively homogeneous, we will use the term \textit{conical material derivative}.
     \item If the additional condition $\gradd y(\Omega)\thetaa \in W(\Omega)$ holds for all $\thetaa \in \Cc^1_b(\mathbb{R}^d)$, then one may define a directional derivative called the \textit{Eulerian derivative} or \textit{shape derivative} of $y(\Omega)$ in the direction $\thetaa$ as the element $dy(\Omega)[\thetaa]$ of $W(\Omega)$ such that:
    $$
        dy(\Omega)[\thetaa] := \dot{y}(\Omega)[\thetaa] - \gradd y(\Omega)\thetaa \: .
    $$
    \item The solution $y(\Omega)$ is said to be \textit{directionally shape differentiable} if it admits a directional derivative for any admissible direction $\thetaa$. If, in addition, the map $\thetaa \mapsto dy(\Omega)[\thetaa]$ is positively homogeneous from $\Cc^1_b(\mathbb{R}^d)$ to $W(\Omega)$, $y(\Omega)$ is said to be \textit{conically shape differentiable}. Furthermore, if this map is linear continuous from $\Cc^1_b(\mathbb{R}^d)$ to $W(\Omega)$, $y(\Omega)$ is said to be \textit{shape differentiable}.
\end{itemize}
\begin{remark}
    Linearity and continuity of $\thetaa \mapsto \dot{y}(\Omega)[\thetaa]$ is actually equivalent to Gâteaux differentiability of the map $\thetaa \mapsto y(\Omega(\thetaa))\circl(\Id+\thetaa)$. The reader is referred to \cite[Chapter 8]{DelZol2001} for a complete review on the different notions of derivatives.
\end{remark}

\begin{notation}
Following the notations in \cite{sokolowski1992introduction}, when there is no ambiguity, $y(\Omega)$ will be replaced simply by $y$, and for some fixed direction $\thetaa$, we define $y_t:=y(\Omega(t))$, $y^t:=y_t\circl(\Id+t\thetaa)$. In the same way, the material and shape derivatives of $y$ at $\Omega$ in the direction $\thetaa$ will be denoted $\dot{y}$ and $dy$, respectively.
\end{notation}

\subsection{Shape sensitivity analysis of the original formulation}

In this section, we prove that the solution $\uu$ to problem \eqref{OPT0} admits conical material/shape derivatives in some specific directions $\thetaa$. It is well known that $\uu$ is not classically shape differentiable because the projection operator onto the closed convex $\Kk$ is not Fréchet-differentiable. However, it is known from \cite{mignot1976controle} that projection operators are conically differentiable. Using this result, it has been proved in \cite{sokolowski1988shape} that the Signorini problem admits conical material/shape derivatives in 2d and 3d. Here, since formulation \eqref{OPT0} is slightly different from the classical one due to our choice of $\normalExt$ instead of $\normalInt$ for the contact, the proof in \cite{sokolowski1988shape} needs to be adapted. Therefore, we will redo it, for some specific directions $\thetaa$. Then, we derive sufficient conditions for the conical material/shape derivative of \eqref{OPT0} to be solution of a more regular optimization problem. \\

\paragraph{\textbf{Restriction on the directions $\thetaa$.}}
In view of the Zolésio-Hadamard structure theorem, we choose to limit ourselves to geometric deformation fields $\thetaa\in \Cc^1_b(\mathbb{R}^d)$ along the direction of the normal $\normalInt$, i.e. vector fields of the form:
\begin{equation}
    \thetaa = \theta \normalInt\:, \mbox{ where } \theta\in \pazocal{C}^1_b(\mathbb{R}^d)\:,
    \label{FormTheta}
\end{equation}
where $\normalInt$ has been extended to $\pazocal{C}^1(\mathbb{R}^d)$, which is possible (see \cite{henrot2006variation}) since $\partial\Omega$ is assumed to have $\pazocal{C}^1$ regularity.
Moreover, this choice is well suited to our numerical algorithm, as it will be seen in the last section. The set of all these $\thetaa$ is denoted $\Thetaa$. Obviously, $\Thetaa$ is a closed subspace of $\Cc^1_b(\mathbb{R}^d)$, thus it is a Banach space.

Let $\thetaa\in\Thetaa$ be a fixed direction. In order to perform sensitivity analysis with respect to the shape, let us first characterize $\uu^t=\uu_t\circl(\Id+t\thetaa)$ for $t>0$. This will be done by writing the problem solved by $\uu_t$ on the transported domain $\Omega(t)$, then bring it back to $\Omega$ by a change of variables. Before that, some additional assumptions on the data are required. Indeed, since the domain is transported, the functions $\Aa$, $\ff$ and $\tauu$ have to be defined everywhere in $\mathbb{R}^d$. 
They also need to enjoy more regularity for usual differentiability results to hold. In particular, we make the following regularity assumptions :
\begin{hypothesis}
     $\Aa\in \pazocal{C}^1_b(\mathbb{R}^d,\mathbb{T}^4)$, $\ff \in \Hh^1(\mathbb{R}^d)$ and $\tauu \in \Hh^2(\mathbb{R}^d)$.
     \label{A2}
\end{hypothesis}
\begin{notation}
For the solution $\uu_t$ to the transported problem on $\Omega(t)$, we introduce $\Xx_t:=\Hh^1_{\Gamma_D(t)}(\Omega(t))$ and the convex subset of admissible displacements:
$$
    \Kk_t:=\{\vv \in \Xx_t \: | \: \vv_{\normalExt}\leq \gG_{\normalExt} \:\:\mbox{a.e.$\!$ on}\: \Gamma_C(t)\} \:.
$$
Composition with the operator $\circl (\Id+t\thetaa)$ will be denoted by $(t)$, for instance, $\Aa(t):=\Aa\circl(\Id+t\thetaa)$. Besides, the normal component associated to $\normalExt(t)$ of a vector $v$ is denoted $v_{\normalExt(t)}$. For integral expressions, the Jacobian and tangential Jacobian of the transformation give $\JacV(t) := $ Jac$(\Id+t\thetaa)$ and $\JacB(t):=$ Jac$_{\Gamma(t)}(\Id+t\thetaa)$. 
Finally, we introduce the transported strain tensor $\epsilonn^t$, the bilinear form $a^t$ on $\Xx\times\Xx$ and the linear form $L^t$ on $\Xx$, which are the versions of $\epsilonn$, $a$ and $L$ corresponding to the problem solved by $\uu^t$, and are defined as follows: 
\begin{equation*}
    \begin{aligned} 
        &a^t(\zz,\vv) :=  \int_{\Omega} \Aa(t) : \epsilonn^t(\zz) : \epsilonn^t(\vv) \:\JacV(t)\quad\forall \zz,\vv\in\Xx\:, \\
        &\epsilonn^t(\vv):= \frac{1}{2} \left( \gradd \vv{(\Ii+\gradd\thetaa)}^{-1} + {(\Ii+\gradd\thetaa^T)}^{-1}{\gradd \vv}^T  \right)\quad\forall \vv\in\Xx\:, \\
        &L^t(\vv) := \int_{\Omega} \ff(t) \: \vv \:\JacB(t) + \int_{\Gamma_{N}} \tauu(t) \: \vv \:\JacV(t) \quad\forall \vv\in\Xx\:.
    \end{aligned}
\end{equation*}
\end{notation}

\begin{lemma}
    Let $\vv\in\Xx_t$ and $\thetaa\in\Thetaa$, then $\vv \in \Kk_t$ if and only if $\vv(t)\in \gG(t)+\Kk_0$.
    \label{LemmaK0Kt}
\end{lemma}
\begin{prf}
   The result is a direct consequence of our choice, \eqref{FormTheta}, of direction $\Thetaa$. Indeed, $\normalExt=\normalInt$ on $\Gamma_C$ gives $\normalExt(t)=\normalExt$ on $\Gamma_C$. The rest follows from the definitions and the fact that $\Id+t\thetaa$ is an isomorphism:
   \begin{equation*}
       \begin{aligned}
          \vv\in\Kk_t & \quad \Longleftrightarrow \quad \left(\vv(y)-\gG(y) \right)\cdot\normalExt(y) \leq 0 \: \mbox{ for a.e. } y\in\Gamma_C(t)\:, \\
          & \quad \Longleftrightarrow \quad \left(\vv(x+t\thetaa(x))-\gG(x+t\thetaa(x))  \right)\cdot\normalExt(x+t\thetaa(x)) \leq 0 \: \mbox{ for a.e. } x\in\Gamma_C\:, \\
          & \quad \Longleftrightarrow \quad \vv(t) \in \gG(t)+\Kk_0 \:.
       \end{aligned}
   \end{equation*}
\end{prf}
\begin{remark}
   For $t$ sufficiently small, $\Gamma_D(t)\subset (\partial\Omega_{rig}^{h'})^c$ and thus $\gG \in \Xx_t$, which yields $\gG(t)\in\Xx$.
\end{remark}
From Lemma \ref{LemmaK0Kt}, it follows that $\uu^t$ solves: find $\uu^t\in \gG(t)+\Kk_0$ such that,
$$
    a^t(\uu^t,\vv-\uu^t) \:\geq\: L^t(\vv-\uu^t)\:, \:\:\: \forall \vv\in \gG(t)+\Kk_0\:.
$$
Let us introduce the auxiliary unknowns $\ww:=\uu-\gG$ and $\ww^t:=\uu^t-\gG(t)$ which both belong to $\Kk_0$. Those functions satisfy the following variational inequalities:
\begin{subequations}\label{IVAux}
    \begin{align}
        a(\ww,\vv-\ww) \:& \geq\: L(\vv-\ww) - a(\gG,\vv-\ww)\:, \:\:\: \forall \vv\in\Kk_0\:, \label{IVW}\\
        a^t(\ww^t,\vv-\ww^t) \:& \geq\: L^t(\vv-\ww^t) - a^t(\gG(t),\vv-\ww^t)\:, \:\:\: \forall \vv\in\Kk_0\:. \label{IVWt}
    \end{align}
\end{subequations}

Before stating the conical differentiability result for $\ww$ and $\uu$, some additional notations are required.

\begin{notation}
If $(\uu,\lambda)$ denotes the solution of problem \eqref{ALF:all}, let $\pazocal{A}$, $\pazocal{I}$ and $\pazocal{B}$ be the subsets of $\Gamma_C$ associated to the constraint $\uu\in \Kk$ in problem \eqref{OPT0}:
\begin{equation*}
    \begin{aligned}
        \pazocal{A} := & \: \{ x\in\Gamma_C  \: : \: \lambda(x)>0, \: \uu_{\normalExt}(x)-\gG_{\normalExt}(x)=0 \}\:, \\
        \pazocal{I} := & \: \{ x\in\Gamma_C  \: : \: \lambda(x)=0, \: \uu_{\normalExt}(x)-\gG_{\normalExt}(x)<0 \}\:, \\
        \pazocal{B} := & \: \{ x\in\Gamma_C  \: : \: \lambda(x)=0, \: \uu_{\normalExt}(x)-\gG_{\normalExt}(x)=0 \}\:.
    \end{aligned}
\end{equation*}
The subsets $\pazocal{A}$, $\pazocal{I}$ and $\pazocal{B}$ are usually referred to as the \textit{active}, \textit{inactive} and \textit{biactive} sets. Note that they are at least measurable due to the regularities of $\lambda$, $\uu$ and $\gG_{\normalExt}$. Moreover, we only consider cases where contact occurs, which means that $\pazocal{A}$ is non-empty.

The bilinear form $a'$ and the linear forms $\epsilonn'$, $L'$, which will be naturally appear when differentiating \eqref{IVWt} with respect to the shape, are introduced: for any $\uu$, $\vv \in \Xx$,
\begin{equation*}
    \begin{aligned} 
        &a'(\uu,\vv) := \int_\Omega \big\{ \Aa:\epsilonn'(\uu):\epsilonn(\vv) + \Aa:\epsilon(\uu):\epsilonn'(\vv) + (\divv \thetaa \: \Aa + \gradd \Aa \: \thetaa):\epsilonn(\uu):\epsilonn(\vv) \big\} \:, \\
        &\epsilonn'(\vv) := -\frac{1}{2}\left( \gradd \vv \gradd \thetaa + {\gradd \thetaa}^T {\gradd \vv}^T \right), \\
        & L'(\vv) = \int_\Omega (\divv \thetaa \: \ff + \gradd \ff  \thetaa) \vv + \int_{\Gamma_N} (\divv_\Gamma \thetaa \: \tauu+ \gradd \tauu \thetaa)\vv \:.
    \end{aligned}
\end{equation*}
Moreover, for any smooth function $f$ defined on $\mathbb{R}^d$, and that does not depend on $\Omega$, we denote $f'[\thetaa]$ or simply $f'$ the following directional derivative:
$$
  f'[\thetaa] := \lim_{t\searrow 0} \:\frac{1}{t}\left( f(t) - f\right) = (\grad f) \thetaa \:.
$$
Using this notation, $\normalExt':=(\gradd \normalExt) \thetaa\:$ and for any $\vv\in\Xx$ one has $\vv_{\normalExt'}:=\vv\cdot{\normalExt}'$.
For the gap $\,(\gG_{\normalExt})':=(\grad \gG_{\normalExt}) \thetaa$ and since $\gG_{\normalExt}$ is the oriented distance function to the smooth boundary $\partial\Omega_{rig}$, $\grad \gG_{\normalExt} = -\normalExt$, which implies that $(\gG_{\normalExt})'= -\thetaa\cdot\normalExt$. However, we will still use the notation $(\gG_{\normalExt})'$ to emphasize that this term comes from differentiation of the gap. 
\end{notation}
\begin{remark}
  From the definition of $\Thetaa$, one gets that any $\thetaa\in\Thetaa$ satisfies $\thetaa = \theta \normalExt$ on $\Gamma_C$. This implies that:
  $$
    \normalExt' := (\gradd\normalExt) \thetaa = \theta (\gradd\normalExt) \normalExt = 0 \ \mbox{ on } \Gamma_C
  $$
  since $\normalExt$ is a unit vector. Therefore, one automatically gets for the gap:
  $$
    (\gG_{\normalExt})' = (\gG\cdot\normalExt)' = \gG'\cdot \normalExt + \gG\cdot \normalExt' = \gG'\cdot \normalExt \ \mbox{ on } \Gamma_C \:.
  $$ 
  Thus, using the notations introduced earlier, it is possible to replace $(\gG_{\normalExt})'$ by $\gG_{\normalExt}'=\gG'\cdot \normalExt$ on $\Gamma_C$.
\end{remark}

We can now use some of the results on the conical differentiability of projections on closed convex sets.  In the next two results, material derivatives will be characterized using some concepts of capacity theory. More precisely the cones containing those derivatives will be defined up to a set of zero capacity (denoted \textit{q.e.} for \textit{quasi-everywhere}). The reader is referred to \cite{mignot1976controle} and \cite{deny1965theorie} for further details about capacity theory.

\begin{theo}
    Under Assumption \ref{A2}, the solution $\ww$ of \eqref{IVW} is conically shape differentiable on $\Thetaa$ and its conical material derivative in the direction $\thetaa\in\Thetaa$ is given by the solution of the problem: find $\dot{\ww}\in S^{\ww}(\Kk_0)$ such that,
    \begin{equation}
        a(\dot{\ww},\phii-\dot{\ww}) \: \geq\: L'(\phii-\dot{\ww}) - a'(\uu,\phii-\dot{\ww}) - a(\gG',\phii-\dot{\ww})\:, \:\:\: \forall \phii\in S^{\ww}(\Kk_0)\:,
        \label{IVConDerW}
    \end{equation}
    Moreover, one has the characterization $S^{\ww}(\Kk_0)=\left\{ \phii\in \Xx \: | \: \phii_{\normalExt} \leq 0 \mbox{ q.e.$\!$ on } \pazocal{A}\cup\pazocal{B} \mbox{ and } a(\uu,\phii)=L(\phii) \right\}$.
    \label{ThmConDerW}
\end{theo}

\begin{prf}
   We follow the same steps as in \cite[Section 5.2]{maury2016shape}, and adapt each of them to our specific formulation. The idea is to write $\ww^t$ as the projection onto $\Kk_0$ of some element in $\Xx$. Then, using the conical differentiability of the projection (see \cite{mignot1976controle}), one is able to perform a first order expansion of $\ww^t$ around $t=0^+$. 
   
   First, let us prove strong continuity of the map $t\mapsto \ww^t$ at $t=0^+$ in $\Xx$. Taking respectively $\vv=\ww$ and $\vv=\ww^t$ as test-functions in \eqref{IVWt} and \eqref{IVW}, then adding the two formulations, one obtains
   $$
     a(\ww-\ww^t,\ww-\ww^t) \: \leq\:  (L^t-L)(\ww-\ww^t) - a^t(\gG(t),\ww-\ww^t) + a(\gG,\ww-\ww^t)\:.
   $$
  Using ellipticity of $a$, one deduces the following estimation:
   $$
      \alpha_0 \norml \ww-\ww^t \normr_{\Xx} \: \leq \:  \norml L^t-L \normr_{\Xx^*} + C\left( \norml a^t-a\normr + \norml \gG(t)-\gG \normr_{\Xx} \right)\:,
   $$
   which yields the result since the right hand side goes to $0$, due to the properties of $a^t$, $L^t$, and the regularity of $\gG$. Next, using this result and the again the properties of $a^t$, $L^t$, we proceed to the expansion of each term in \eqref{IVWt}, which leads to:
   \begin{equation*}
        \begin{aligned}
            a^t(\ww^t,\vv-\ww^t) &= a(\ww^t,\vv-\ww^t) + ta'(\ww,\vv-\ww^t) + o(t)\:, \\
            L^t(\vv-\ww^t) &= L(\vv-\ww^t) + tL'(\vv-\ww^t) + o(t)\:, \\
            a^t(\gG(t),\vv-\ww^t) &= a(\gG,\vv-\ww^t) + ta'(\gG,\vv-\ww^t) + ta(\gG',\vv-\ww^t) + o(t)\:.
        \end{aligned}
   \end{equation*}
   Plugging these expansions in \eqref{IVWt} yields: for all $\vv\in\Kk_0$,
   \begin{equation}
        a(\ww^t,\vv-\ww^t) \: \geq\:  L(\vv-\ww^t) - a(\gG,\vv-\ww^t) + t\left( L'(\vv-\ww^t) - a'(\uu,\vv-\ww^t) - a(\gG',\vv-\ww^t)\right) + o(t) \:.
        \label{IVWtExp}
   \end{equation}
   Let $\mathfrak{l}$, $\mathfrak{l}'$, $\mathfrak{a}'_{\uu}\in\Xx$ such that for all $\zz\in\Xx$:
   $$
     a(\mathfrak{l},\zz)=L(\zz)\:, \hspace{2em} a(\mathfrak{l}',\zz)=L'(\zz)\:, \hspace{2em} a(\mathfrak{a}'_{\uu},\zz)= a'(\uu,\zz)\:.
   $$
   Then relation \eqref{IVWtExp} can be equivalently rewritten as:
   \begin{equation*}
     \begin{aligned}
       \ww^t &= \Proj_{\Kk_0}^a\left( \mathfrak{l}-\gG +t(\mathfrak{l}'-\mathfrak{a}'_{\uu}-\gG') + o(t) \right)\\
       &= \Proj_{\Kk_0}^a\left( \mathfrak{l}-\gG + t(\mathfrak{l}'-\mathfrak{a}'_{\uu}-\gG') \right) + o(t)\:.
     \end{aligned}
   \end{equation*}
   As seen in Remark \ref{RmkNN0}, one has the approximation $\normalExt=\normalInt$ on $\Gamma_C$ under the small displacements hypothesis. Thus we directly get from \cite[Lemma 5.2.9]{maury2016shape} that $\Kk_0$ is polyhedric. Therefore, from Theorem \ref{ThmMignotPoly}, it follows that $\Proj_{\Kk_0}^a$ is conically differentiable and that:
   $$
     \ww^t = \ww + t\Proj_{S^{\ww}(\Kk_0)}^a\left( \mathfrak{l}'-\mathfrak{a}'_{\uu}-\gG'\right) + o(t)\:,
   $$
   which yields \eqref{IVConDerW}. Finally, the characterization of $S^{\ww}(\Kk_0)$ is also given by \cite[Lemma 5.2.9]{maury2016shape}.
\end{prf}

\begin{cor}
    Under Assumption \ref{A2}, the solution $\uu$ of \eqref{OPT0} is conically shape differentiable on $\Thetaa$ and its conical material derivative in any direction $\thetaa\in\Thetaa$ is given by the solution of: find $\dot{\uu}\in \Ss$ such that,
    \begin{equation}
        a(\dot{\uu},\psii-\dot{\uu}) \: \geq\: L'(\psii-\dot{\uu}) - a'(\uu,\psii-\dot{\uu})\:, \:\:\: \forall \psii\in \Ss\:,
        \label{IVConDerU}
    \end{equation}
    where $\Ss:=\left\{ \psii\in \gG'+\Xx \: | \: \psii_{\normalExt} \leq \gG_{\normalExt}' \mbox{ q.e.$\!$ on } \pazocal{A}\cup\pazocal{B} \mbox{ and } a(\uu,\psii-\gG')=L(\psii-\gG') \right\}$.
\end{cor}

\begin{prf}
   This is a direct consequence of Theorem \ref{ThmConDerW} since $\uu^t=\ww^t+\gG(t)$ and therefore $\dot{\uu}=\dot{\ww}+\gG'$.
\end{prf}


As formulations \eqref{IVConDerW} and \eqref{IVConDerU} relies on zero capacity sets they are not easy to handle. However, under some mild additional regularity assumptions, it is possible to rewrite those formulations as standard optimization problems. 
In the same way as in \cite[Section 4.1]{hintermuller2011optimal} for the obstacle problem, we introduce the following regularity assumption:
\begin{hypothesis}
    $\pazocal{A}\cup \pazocal{B} = \overline{\inter (\pazocal{A}\cup \pazocal{B})}$.
    \label{A3}
\end{hypothesis}

\begin{remark}
   This assumption implies that not only $\pazocal{A}\cup \pazocal{B}$ is closed, but also it has a non-empty interior.
   Moreover, for every $\phii \in\Xx$ it is obvious, that
    $$
        \phii_{\normalExt}\leq 0 \mbox{ q.e. on } \inter(\pazocal{A}\cup\pazocal{B}) \ \Longrightarrow \  \phii_{\normalExt}\leq 0 \mbox{ a.e. on } \inter(\pazocal{A}\cup\pazocal{B})\:.
    $$
Regarding the converse implication, we invoke the result stated in \cite[Lemma 4.31]{sokolowski1992introduction}, from which we get that the subspace
    $$
        \pazocal{H}:= \left\{ \phii_{\normalExt} \in H^{\frac{1}{2}}(\Gamma_C) \:|\: \phii \in \Xx \right\}
    $$
    associated with the appropriate bilinear form (expressed in \cite[Formula (4.192) ]{sokolowski1992introduction}) is a Dirichlet space in the sense of \cite[Définition 3.1]{mignot1976controle}. Therefore $\phii_{\normalExt}$ admits a unique quasi-continuous representative in the equivalence class related to the "q.e. on $\Gamma_C$ equality". Considering this specific representative, one gets from \cite[Théorème 5]{deny1965theorie} that 
    $$
        \phii_{\normalExt}\leq 0 \mbox{ a.e. on } \inter(\pazocal{A}\cup\pazocal{B}) \ \Longrightarrow \  \phii_{\normalExt}\leq 0 \mbox{ q.e. on } \inter(\pazocal{A}\cup\pazocal{B})\:,
    $$
    since $\inter(\pazocal{A}\cup\pazocal{B})$ is an open subset of $\Gamma_C$. Quasi-continuity of $\phii_{\normalExt}$ on $\Gamma_C$ also implies that $ \phii_{\normalExt}\leq 0$ q.e. on $\overline{\inter(\pazocal{A}\cup\pazocal{B})}=\pazocal{A}\cup\pazocal{B}$.
    \label{RmkA3}
\end{remark}

Before stating the optimization problem that $\dot{\ww}$ solves when Assumption \ref{A3} is fulfilled, let us define the subspace $\Xx_{\pazocal{A}} \subset \Xx$, and the cone $\Kk_{\pazocal{A}}$ such that:
$$
  \Xx_{\pazocal{A}}:=\{ \phii \in \Xx \: | \: \phii_{\normalExt} = 0  \mbox{ a.e. on }  \pazocal{A}\}\:, \hspace{2em} \Kk_{\pazocal{A}}:=\{ \phii \in \Xx_{\pazocal{A}} \: | \: \phii_{\normalExt} \leq 0  \mbox{ a.e. on } \pazocal{B}\}\:.
$$
Clearly, $\Xx_{\pazocal{A}}$ is a closed subspace of $\Xx$, therefore it is a Hilbert space, and $\Kk_{\pazocal{A}}$ is a non-empty closed convex cone. 

\begin{theo}\label{ThmOptConDerW}
    If Assumption \ref{A2} and Assumption \ref{A3} hold, then $\dot{\ww}$ is the solution of \eqref{IVConDerW} if and only if it solves:
    \begin{equation}
    	\underset{\phii\in \Kk_{\pazocal{A}}}{\inf} \:\: \frac{1}{2}a(\phii,\phii)-L'(\phii)+a'(\uu,\phii)+a(\gG',\phii) \:.
    \label{OPTConDerW}
    \end{equation}
\end{theo}

\begin{prf}  
    It suffices to prove that $S^{\ww}(\Kk_0)=\Kk_{\pazocal{A}}$. First of all, let us point out that in the characterization of $S^{\ww}(\Kk_0)$ given in Theorem \ref{ThmConDerW}, since $\phii\in\Xx$ and $\uu$ solves \eqref{ALF:1}, condition $a(\uu,\phii)=L(\phii)$ is equivalent to $\prodL2{\lambda,\phii_{\normalExt}}{\Gamma_C}=0$. As $\lambda\in L^2(\Gamma_C)$, $\lambda\geq 0$ and $\supp \lambda = \pazocal{A}$, it follows that:
    $$
      \left( \phii\in \Xx \mbox{ s.t. } \phii_{\normalExt}\leq 0 \mbox{ a.e. on } \pazocal{A}\cup\pazocal{B} \mbox{ and } \prodL2{\lambda,\phii_{\normalExt}}{\Gamma_C}=0 \right) \ \Longleftrightarrow \ \phii\in\Kk_{\pazocal{A}} \:.
    $$
    And Remark \ref{RmkA3} gives us the equality of both sets. 
\end{prf}

Finally, some obvious results are obtained, introducing the assumption
\begin{hypothesis}
    The biactive set $\pazocal{B}$ is of measure zero.
    \label{A5}
\end{hypothesis}

\begin{cor}\label{CorConDerW}
If, in addition to the assumptions of Theorem \ref{ThmOptConDerW}, Assumption \ref{A5} holds, then $\dot{\ww}$ is the solution of \eqref{IVConDerW} if and only if it solves: find $\dot{\ww}\in\Xx_{\pazocal{A}}$ such that
  \begin{equation}
      a(\dot{\ww},\phii) = L'(\phii) - a'(\uu,\phii) - a(\gG',\phii)\:, \:\:\: \forall \phii\in \Xx_{\pazocal{A}}\:.
      \label{FVConDerW}
  \end{equation}  
  \end{cor}
\begin{prf}
   When $\pazocal{B}$ is of measure zero, $\Kk_{\pazocal{A}}=\Xx_{\pazocal{A}}$ and the result follows as problem \eqref{OPTConDerW} can be equivalently rewritten as a linear variational formulation.
\end{prf}
\begin{remark}\label{StrongShapeDerU}
 Under the assumption of Corollary \ref{CorConDerW}, the material derivative becomes linear with respect to $\thetaa$, which implies that $\uu$ is strongly shape differentiable in $L^2(\Omega)$, and its shape derivative in any direction $\thetaa \in \Thetaa$ is given by:
 $$
   d\mathbf{u} = \dot{\uu} - \gradd \uu \thetaa =  \dot{\ww} + \gG' - \gradd \uu \thetaa\:,
$$
 where $\dot{\ww}\in\Xx_{\pazocal{A}}$ is the unique solution of \eqref{FVConDerW}. Note that this conclusion is similar to the one in \cite[Remark 4.1]{neittaanmaki1988optimization}. 
\end{remark}
\subsection{Shape sensitivity analysis of the augmented Lagrangian formulation}

The goal of this section is to prove, at every iteration of the ALM algorithm, the differentiability of $\uu^{k}$ with respect to the shape. 
As $\maxx$ fails to be Fréchet differentiable it is not possible to rely on the implicit function theorem as in \cite{henrot2006variation}.
However, as it is a projection operator, it is conically differentiable, which enables us to show existence of conical material/shape derivatives for $\uu^{k}$ following the approach in \cite{sokolowski1992introduction}. 
Then, under assumptions on some specific subsets of $\Gamma_C$ (this will be presented and referred to as Assumption \ref{A4}), classical shape differentiability of $\uu^{k}$ is proved.


First of all, let us briefly recall some properties of function $\maxx$.
\begin{lemma}
 The function $\maxx:\mathbb{R}\to\mathbb{R}$ is Lipschitz continuous and conically differentiable, with conical derivative at $u$ in the direction $v\in\mathbb{R}$:
 $$
  \dmaxx(u;v) = \left\{
      \begin{array}{lr}
        0 & \mbox{ if } \:u<0 ,\\
        \maxx(v) & \mbox{ if } \:u=0, \\
        v & \mbox{ if } \:u>0.
      \end{array}
    \right.
$$
\label{LemDirDiffMax}
\end{lemma}
\begin{lemma}
    The Nemytskij operator $\maxx:L^2(\Gamma_C)\to L^2(\Gamma_C)$ is Lipschitz continuous and conically differentiable.
\label{LemDirDiffNemytskijMax}
\end{lemma}

\begin{remark}
   These results are well known if "conically" is replaced by "directionally", a proof can be found in \cite{susu2018optimal}, for example. Conical differentiability then follows directly because for any $u\in \mathbb{R}$, $\dmaxx(u;\cdot):\mathbb{R}\to\mathbb{R}$ is obviously positively homogeneous, and the same property holds for $\dmaxx(u;\cdot):L^2(\Gamma_C)\to L^2(\Gamma_C)$ for any $u\in L^2(\Gamma_C)$.
\end{remark}

\begin{theo}
  If Assumption \ref{A2} holds and $\lambda^0\in H^{\frac{1}{2}}(\Gamma_C)$, then for any $k\geq 0$, the pair $\left( \uu^k,\lambda^k\right)$ defined by \eqref{ItALM}--\eqref{UpdateMultALM} is conically shape differentiable on $\Thetaa$, strongly in $\Xx\times L^2(\Gamma_C)$.
  \label{ThmExistMDerk}
\end{theo}

\begin{prf}
  As the solution of \eqref{ItALM} at each step $k+1$ depends on the previous iterate, the result will be proved by induction.\\

  \paragraph{\textbf{Base case}}
  For $k=0$, given $\gamma^1>0$, the first iteration of the ALM gives $\uu^1\in \Xx$ the solution of 
  $$
    a(\uu^{1},\vv) + \prodL2{R^1_{\normalExt}(\uu^1), \vv_{\normalExt}}{\Gamma_C}  = L(\vv) \:,
  $$
    where we have introduced the non-linear map $R^1_{\normalExt}:\Xx\to H^{\frac{1}{2}}(\Gamma_C)$ defined as, 
  \begin{equation*}
      R^1_{\normalExt}(\vv) := \maxx \left(\lambda^0 + \gamma^{1} (\vv_{\normalExt}-\gG_{\normalExt})\right)\quad\forall\vv\in\Xx.
  \end{equation*}
  Obviously, $R^1_{\normalExt}(\uu^1)=\lambda^1$, but we prefer here to denote $R^1_{\normalExt}(\uu^1)$ to emphasize that $\lambda^1$ is defined explictly from $\uu^1$ and the data.
  Using the same arguments as in \cite{chaudet2019shape}, one gets that, due to the regularity of the data, especially $\lambda^0$, $\uu^1$ admits a strong material derivative $\dot{\uu}^1\in\Xx$ in each direction $\thetaa\in\Thetaa$, which is the unique solution of 
  \begin{equation*}
      a(\dot{\uu}^{1},\vv) + \prodL2{{R^1_{\normalExt}}'(\uu^1), \vv_{\normalExt}}{\Gamma_C} = L^1[\thetaa](\vv)\:,
  \end{equation*}
  where, 
  for any $\vv\in\Xx$, with material derivative $\dot{\vv}\in\Xx$, we have
  \begin{equation*}
      \begin{aligned}
      {R^1_{\normalExt}}'(\vv) := &\: \dmaxx \left(\lambda^0 + \gamma^{1} (\vv_{\normalExt}-\gG_{\normalExt}); (\lambda^0)'+\gamma^1\left( \dot{\vv}_{\normalExt}-\gG_{\normalExt}'\right)\right)\:, \\
        L^1[\thetaa](\vv) := & \:L'(\vv) - a'(\uu^1,\vv) - \int_{\Gamma_C} \lambda^1 \vv_{\normalExt} \divv_\Gamma \thetaa\:.
    \end{aligned}
\end{equation*}
Moreover, as $\lambda^1$ only depends on $\lambda^0$ and $\uu^1$, it is clear that it also admits a strong material derivative in $L^2(\Gamma_C)$ in the direction $\thetaa$, which is simply obtained by differentiating \eqref{UpdateMultALM}:
\begin{equation}
   \dot{\lambda}^1 = {R^1_{\normalExt}}'(\uu^1)\:,
   \label{EqDotLambda1}
\end{equation}
which finishes to prove that  $\left( \uu^1,\lambda^1\right)$ is directionally shape differentiable. Moreover, one has that $a$ is bilinear, the maps $\thetaa\mapsto L^1[\thetaa], (\lambda^0)', \gG_{\normalExt}'$ are linear and $\dmaxx\left(\lambda^0 + \gamma^{1} (\uu^1_{\normalExt}-\gG_{\normalExt}); \cdot \right)$ is positively homogeneous. Thus one deduces that the map $\thetaa \mapsto \dot{\uu}^1$ is positively homogeneous. The same property holds for $\thetaa\mapsto \dot{\lambda}^1$ due to \eqref{EqDotLambda1}.\\

\paragraph{\textbf{Inductive step}}
Let $k\geq 1$ and assume $\left( \uu^k,\lambda^k\right)$ is conically shape differentiable, strongly in $\Xx\times L^2(\Gamma_C)$. This implies that $\lambda^k$ admits a strong material derivative $\dot{\lambda}^k\in L^2(\Gamma_C)$. Hence, using the arguments of the case $k=0$, one may differentiate \eqref{ItALM}, which leads to existence of a unique material derivative $\dot{\uu}^{k+1}\in\Xx$ for $\uu^{k+1}$ that solves:
  \begin{equation}
      a(\dot{\uu}^{k+1},\vv) + \prodL2{{R^{k+1}_{\normalExt}}'(\uu^{k+1}), \vv_{\normalExt}}{\Gamma_C} = L^{k+1}[\thetaa](\vv)\:,
      \label{EqMDerUk}
  \end{equation}
where the same notation as in the case $k=0$ has been used for $L^{k+1}[\thetaa]$, while the notation for ${R^{k+1}_{\normalExt}}'$ has been slightly adapted:
  \begin{equation*}
      {R^{k+1}_{\normalExt}}'(\vv) :=  \: \dmaxx \left(\lambda^k + \gamma^{k+1} (\vv_{\normalExt}-\gG_{\normalExt}); \dot{\lambda}^k+\gamma^{k+1}\left( \dot{\vv}_{\normalExt}-\gG_{\normalExt}'\right)\right)\:.
\end{equation*}
Especially, this also proves that $\lambda^{k+1}$ admits a material derivative in $L^2(\Gamma_C)$ defined as
\begin{equation}
    \dot{\lambda}^{k+1} = {R^{k+1}_{\normalExt}}'(\uu^{k+1})\:.
    \label{EqMDerLambdak}
\end{equation}
Again, since $a$ is bilinear, $\thetaa\mapsto L^{k+1}[\thetaa], \gG_{\normalExt}'$ are linear, and $\dmaxx\left(\lambda^k + \gamma^{k+1} (\uu^{k+1}_{\normalExt}-\gG_{\normalExt});\cdot\right)$, $\thetaa \mapsto \dot{\lambda}^k$ are positively homogeneous, one deduces that $\thetaa \mapsto \left( \dot{\uu}^{k+1}, \dot{\lambda}^{k+1}\right)$ is positively homogeneous as well.

\end{prf}

\begin{remark}
Existence of directional (even conical) material/shape derivatives in all directions does not guarantee that the solution $\uu^{k+1}$ is shape differentiable. Indeed, as these derivative depends on $\maxx$, even when $\dot{\uu}^k$ and $\dot{\lambda}^k$ are linear with respect to $\thetaa$, the derivative $\dot{\uu}^{k+1}$ may fail to have that property if the set $\left\{ \lambda^k + \gamma^{k+1} (\uu^{k+1}_{\normalExt}-\gG_{\normalExt}) = 0 \right\}$ is not of null measure. 
\end{remark}

In light of the previous remark, we get interested in sufficient conditions for shape differentiability of $\uu^k$. 
We introduce the corresponding subsets for problem \eqref{ItALM}, for each $k\geq 0$,
\begin{equation*}
    \begin{aligned}
        \pazocal{A}^{k+1} := & \, \left\{ \: x\in\Gamma_C  \: : \: \left(\lambda^k + \gamma^{k+1}\left(\uu_{\normalExt}^{k+1}-\gG_{\normalExt}\right)\right)(x) > 0 \, \right\}\:, \\
        \pazocal{I}^{k+1} := & \: \left\{ \, x\in\Gamma_C  \: : \: \left(\lambda^k + \gamma^{k+1}\left(\uu_{\normalExt}^{k+1}-\gG_{\normalExt}\right)\right)(x) < 0 \,  \right\} \:, \\
        \pazocal{B}^{k+1} := & \: \left\{ \, x\in\Gamma_C  \: : \: \left(\lambda^k + \gamma^{k+1}\left(\uu_{\normalExt}^{k+1}-\gG_{\normalExt}\right)\right)(x) = 0 \, \right\} \:.
    \end{aligned}
\end{equation*}
Using these notations, for any $k\geq 1$, we are able to state conditions that guarantee shape differentiability of $\uu^k$, and will be referred to as Assumption \ref{A4} at rank $k$:
\begin{hypothesis}
    For each $j\in \{1, \dots, k \}$, the set $\pazocal{B}^j$ is of measure zero.
    \label{A4}
\end{hypothesis}

\begin{cor}\label{CorSDerk}
Let $k\geq 1$, $(\uu^k,\lambda^k)$ still denotes the solution of \eqref{ItALM}--\eqref{UpdateMultALM} and $(\dot{\uu}^k,\dot{\lambda}^k)$ its material derivative, defined by \eqref{EqMDerUk}--\eqref{EqMDerLambdak}. Under the assumptions of Theorem \ref{ThmExistMDerk}, if Assumption \ref{A4} holds at rank $k$, then the map $\thetaa\mapsto (\dot{\uu}^k,\dot{\lambda}^k)$ is linear continuous from $\Thetaa$ to $\Xx\times L^2(\Gamma_C)$. Therefore, $\uu^k$ is (strongly) shape differentiable in $\Ll^2(\Omega)$, and its shape derivative in any given direction $\thetaa$ writes:
    $$
        d\mathbf{u}^k =  \dot{\uu}^k - \gradd \uu^k \thetaa\:.
    $$
\end{cor}

\begin{prf}
  Here again, let us proceed by induction.\\
  
  \paragraph{\textbf{Base case}}
  Let $k=1$. From Assumption \ref{A4} at rank $1$, one gets that $|\pazocal{B}^1|=0$, which implies that 
  $$
    \dot{\lambda}^1 = {R^1_{\normalExt}}'(\uu^1) = \chi_{\pazocal{A}^1} \left( (\lambda^0)'+\gamma^1\left( \dot{\uu}^1_{\normalExt}-\gG_{\normalExt}'\right)\right) \: \mbox{ a.e. on } \Gamma_C\:.
  $$
  Therefore, $\dot{\uu}^1$ is the solution of a linear variational formulation. Moreover, due to the regularities of $\lambda^0$, $\gG_{\normalExt}$, $\normalExt$, and the expression of $L^1[\thetaa]$, the maps
  \begin{equation*}
      \begin{aligned}
        \Thetaa \ni \thetaa & \: \longmapsto \: (\lambda^0)' \in L^2(\Gamma_C)\:, \\
        \Thetaa \ni \thetaa & \: \longmapsto \: \gG_{\normalExt}' \in L^\infty(\Gamma_C)\:, \\
        \Thetaa \ni \thetaa & \: \longmapsto \: L^1[\thetaa] \in \Xx^*\:,
      \end{aligned}
  \end{equation*}
  are all linear continuous. Therefore, $\thetaa \mapsto (\dot{\uu}^1, \dot{\lambda}^1)$ is also linear continuous, from $\Thetaa$ to $\Xx\times L^2(\Gamma_C)$.\\
  
  \paragraph{\textbf{Inductive step}}
  Let $k\geq 1$ and assume the result of the corollary is true at iteration $k$. Now, suppose Assumption \ref{A4} holds at rank $k+1$. This exactly means that Assumption \ref{A4} holds at rank $k$ and $|\pazocal{B}^{k+1}|=0$. Thus, one has
  $$
    \dot{\lambda}^{k+1} = {R^{k+1}_{\normalExt}}'(\uu^{k+1}) = \chi_{\pazocal{A}^{k+1}} \left( \dot{\lambda}^{k}+\gamma^{k+1}\left( \dot{\uu}^{k+1}_{\normalExt}-\gG_{\normalExt}'\right)\right) \: \mbox{ a.e. on } \Gamma_C\:.
  $$  
  Hence, the variational formulation \eqref{EqMDerUk} is linear. Using the result at iteration $k$, one gets that $\thetaa\mapsto \dot{\lambda}^k$ is linear continuous from $\Thetaa$ to $L^2(\Gamma_C)$. This, combined with the properties of $\gG_{\normalExt}$, $\normalExt$, and linear continuity of $L^{k+1}[\cdot]$ from $\Thetaa$ to $\Xx^*$, enables us to conclude.
\end{prf}

\subsection{Convergence of the directional shape derivatives}

In this section, we aim at establishing the consitency of the augmented Lagrangian method with respect to shape differentiability. In other words, for any $\thetaa\in\Thetaa$ fixed, we get interested in convergence properties of the sequence $\left\{ \left(\dot{\uu}^{k}, \dot{\lambda}^{k}\right)\right\}_k$ as $k\to\infty$. Especially, we would like to derive sufficient conditions for $\left\{\dot{\uu}^{k}\right\}_k$ to converge to $\dot{\uu}$ the solution of \eqref{IVConDerU}.

\begin{remark}
  As we will manipulate in this section the material derivatives of $\uu^k$ and $\lambda^k$, it will be required that Assumption \ref{A2} holds (see Theorem \ref{ThmConDerW}).
\end{remark}

As for the original problem, let us define $\ww^{k+1}:=\uu^{k+1}-\gG\in\Xx$. From the previous section, we get that $\ww^{k+1}$ admits a strong material derivative in the direction $\thetaa$ given by $\dot{\ww}^{k+1}=\dot{\uu}^{k+1}-\gG'\in\Xx$. It follows from \eqref{EqMDerUk} and \eqref{EqMDerLambdak} that $\dot{\ww}^{k+1}$ solves:
\begin{subequations} \label{LMFMDerk:all}
  \begin{align}
    a(\dot{\ww}^{k+1},\vv) - L^{k+1}[\thetaa](\vv) + a(\gG',\vv) + \prodL2{\dot{\lambda}^{k+1},\vv_{\normalExt}}{\Gamma_C} &= 0\:, \hspace{1em} \forall \vv\in\Xx\:, \label{LMFMDerk:1} \\
    \dot{\lambda}^{k+1} - \dmaxx\left( \lambda^k + \gamma^{k+1}\ww_{\normalExt}^{k+1} ; \dot{\lambda}^k + \gamma^{k+1}\dot{\ww}^{k+1}_{\normalExt}\right)  &= 0 \:, \hspace{1em} \mbox{ a.e. on } \Gamma_C \:. \label{LMFMDerk:2}
  \end{align}
\end{subequations}
Note that \eqref{LMFMDerk:2} can be equivalently rewritten as:
\begin{equation}
    \dot{\lambda}^{k+1} = 
    \left\{
      \begin{array}{lr}
        \dot{\lambda}^k + \gamma^{k+1}\dot{\ww}^{k+1}_{\normalExt} &\mbox{ on } \pazocal{A}^{k+1}\:, \\
        0 &\mbox{ on } \pazocal{I}^{k+1}\:, \\
        \maxx\left(\dot{\lambda}^k + \gamma^{k+1}\dot{\ww}^{k+1}_{\normalExt}\right) &\mbox{ on } \pazocal{B}^{k+1}\:.
      \end{array}
    \right.
\end{equation}
In order to have a valid expression for all $k\geq 0$, the abuse of notation $\dot{\lambda}^0:={\lambda^0}'$ will be used.
Finally, we give some straightforward but very useful properties of $\dmaxx$:
\begin{itemize}
    \item for all $u$, $v\in \mathbb{R}$, one has $\dmaxx(u;v)v = |\dmaxx(u;v)|^2$,
    \item $\dmaxx$ is continuous on $\mathbb{R}^2\setminus (\{0\}\times\mathbb{R})$.
\end{itemize}

\begin{theo}
  For any increasing sequence of strictly positive parameters $\left\{\gamma^k\right\}_k$, there exists a subsequence of $\left\{ \left(\dot{\ww}^{k}, \dot{\lambda}^{k}\right) \right\}_k$ that is bounded in $\Xx\times H^{-\frac{1}{2}}(\Gamma_C)$.
  \label{ThmBoundMDer}
\end{theo}

\begin{prf}
  For now, let $\left\{\gamma^k\right\}_k \subset \mathbb{R}_+^*$ be any increasing sequence, from which the ALM algorithm generates iterates $(\uu^k,\lambda^k)$. From Theorem \ref{ThmCvALM}, one gets that $\left\{(\uu^k,\lambda^k)\right\}_k$ is bounded in $\Xx\times L^2(\Gamma_C)$, from which one deduces that $\left\{L^{k}[\thetaa]\right\}_k$ is bounded in $\Xx^*$. Thus, taking $\vv=\dot{\ww}^{k+1}$ as test-function in \eqref{LMFMDerk:1} yields
  \begin{equation}
    \alpha_0 \norml \dot{\ww}^{k+1} \normr_{\Xx}^2 + \prodL2{\dot{\lambda}^{k+1},\dot{\ww}^{k+1}_{\normalExt}}{\Gamma_C} \leq C \norml \dot{\ww}^{k+1} \normr_{\Xx} \:.
    \label{EstWk}
  \end{equation}

  Now, we rewrite the second term, then use the properties of $\dmaxx$ and Young's inequality:
  \begin{equation}
    \begin{aligned}
      \prodL2{\dot{\lambda}^{k+1},\dot{\ww}^{k+1}_{\normalExt}}{\Gamma_C} &= \frac{1}{\gamma^{k+1}} \prodL2{\dot{\lambda}^{k+1},\dot{\lambda}^{k}+\gamma^{k+1}\dot{\ww}^{k+1}_{\normalExt}}{\Gamma_C} - \frac{1}{\gamma^{k+1}} \prodL2{\dot{\lambda}^{k+1},\dot{\lambda}^{k}}{\Gamma_C} \\
      & \geq \frac{1}{\gamma^{k+1}} \norml \dot{\lambda}^{k+1} \normr_{0,\Gamma_C}^2 - \frac{1}{\gamma^{k+1}} \prodL2{\dot{\lambda}^{k+1},\dot{\lambda}^{k}}{\Gamma_C}\\
      & \geq \frac{1}{2\gamma^{k+1}} \left( \norml \dot{\lambda}^{k+1} \normr_{0,\Gamma_C}^2 - \norml \dot{\lambda}^{k} \normr_{0,\Gamma_C}^2 \right)\:.
    \end{aligned}
    \label{EstPSLambdakWk}
  \end{equation}
  Plugging this into inequality \eqref{EstWk} leads to
  \begin{equation}
    \alpha_0 \norml \dot{\ww}^{k+1} \normr_{\Xx}^2 - C \norml \dot{\ww}^{k+1} \normr_{\Xx} + \frac{1}{2\gamma^{k+1}} \norml \dot{\lambda}^{k+1} \normr_{0,\Gamma_C}^2 \leq \frac{1}{2\gamma^{k+1}} \norml \dot{\lambda}^{k} \normr_{0,\Gamma_C}^2 \:.
    \label{EstWkBis}
  \end{equation}
  From here, let us distinguish the two possible cases.
  \begin{enumerate}[(i),leftmargin=*]
      \item Case $\left\{ \frac{1}{\gamma^{k}}\norml \dot{\lambda}^{k} \normr_{0,\Gamma_C}^2 \right\}_k$ bounded. \\
      Boundedness of the whole sequence $\left\{\dot{\ww}^k\right\}_k$ in $\Xx$ follows directly from \eqref{EstWkBis} since $\gamma^{k+1}\geq \gamma^k$.
      \item Case $\left\{ \frac{1}{\gamma^{k}} \norml \dot{\lambda}^{k} \normr_{0,\Gamma_C}^2 \right\}_k$ unbounded. \\
      In that case, let us sum \eqref{EstWkBis} from $0$ to $k-1$, with $k>1$:
      $$
        \sum_{l=1}^{k} \left( \alpha_0 \norml \dot{\ww}^{l} \normr_{\Xx}^2 - C \norml \dot{\ww}^{l} \normr_{\Xx} \right) + \frac{1}{2\gamma^{k}} \norml \dot{\lambda}^{k} \normr_{0,\Gamma_C}^2 \leq \frac{1}{2\gamma^{1}} \norml \dot{\lambda}^{0} \normr_{0,\Gamma_C}^2 \:.
      $$
      If we denote $W^k$ the first term of the left hand side, it follows that the sequence $\left\{ W^k\right\}_k$ is unbounded below. Thus there exists a subsequence (denoted ${\varphi(k)}$) such that $\left\{ W^{\varphi(k)} \right\}_k$ is strictly monotone and tends to $-\infty$. Especially, strict monotonicity of this subsequence implies that, for all $k\geq 1$,
      $$
        W^{\varphi(k+1)}-W^{\varphi(k)} = \sum_{l=\varphi(k)}^{\varphi(k+1)} \left( \alpha_0 \norml \dot{\ww}^{l} \normr_{\Xx}^2 - C \norml \dot{\ww}^{l} \normr_{\Xx} \right) < 0\:.
      $$
      Hence, for all $k\geq 1$, there exists $\psi(k) \in [\![\varphi(k), \varphi(k+1) ]\!]$ such that $\alpha_0 \norml \dot{\ww}^{\psi(k)} \normr_{\Xx}^2 - C \norml \dot{\ww}^{\psi(k)} \normr_{\Xx} < 0$, which proves that $\left\{\dot{\ww}^{\psi(k)}\right\}_k$ in bounded in $\Xx$.
  \end{enumerate}
  Then, boundedness of $\left\{\dot{\lambda}^k\right\}_k$ in $H^{-\frac{1}{2}}(\Gamma_C)$ immediatly follows from \eqref{LMFMDerk:1} and the surjectivity of the trace operator from $\Xx$ to $\Hh^{\frac{1}{2}}(\Gamma_C)$. 
   
\end{prf}

\begin{cor}
  There exists an increasing sequence of strictly positive parameters $\left\{\gamma^k\right\}_k$ such that the whole sequence $\left\{ \left(\dot{\ww}^{k}, \dot{\lambda}^{k}\right) \right\}_k$ defined by \eqref{LMFMDerk:all} is bounded in $\Xx\times H^{-\frac{1}{2}}(\Gamma_C)$.
  \label{CorBoundMDer}
\end{cor}

\begin{prf}
  From estimation \eqref{EstWkBis}, it is clear that if the sequence of parameters $\left\{\gamma^k \right\}_k$ is built from some $\gamma^1>0$ following the rule
  \begin{equation}
      \gamma^{k+1} = \max \left\{ \gamma^k,  \norml \dot{\lambda}^{k} \normr_{0,\Gamma_C}^2 \right\}\:,
      \label{HypGammak}
  \end{equation}
  then one automatically gets boundedness of the whole sequence $\left\{\dot{\ww}^k\right\}_k$ in $\Xx$. In other words, such a choice of parameters ensures that we are in case (i) of the previous proof. Again, boundedness of $\left\{\dot{\lambda}^k\right\}_k$ in $H^{-\frac{1}{2}}(\Gamma_C)$ follows. 
\end{prf}

Of course, the previous result implies that for any sequence of parameters, the associated iterates $\left( \dot{\ww}^k, \dot{\lambda}^{k} \right)$ converge weakly in $\Xx\times H^{-\frac{1}{2}}(\Gamma_C)$ to some limit $(\hat{\ww},\hat{\lambda})$, up to a subsequence. In order to specify this limit, we need to make the following additional assumption.


\begin{hypothesis}
    The iterates $\lambda^k$ generated by the ALM converge to $\lambda$ a.e. on $\Gamma_C$.
    \label{A6}
\end{hypothesis}
\begin{remark}
Because of the non continuity of $\dmaxx$ on $\{0\}\times \mathbb{R}$ getting a convergence result without Assumption \ref{A5} seems rather difficult (maybe even impossible).
\end{remark}
\begin{lemma}
Let $\left\{\gamma^k\right\}_k$ be any increasing sequence in $\mathbb{R}_+^*$,
$(\uu^k,\lambda^k)$ still denotes the solution of \eqref{ItALM}--\eqref{UpdateMultALM}, and $(\uu,\lambda)$ the solution of \eqref{ALF:all}.
When Assumption \ref{A5} and Assumption \ref{A6} hold, the characteristic functions $\chi_{\pazocal{A}^k}$ (respectively $\chi_{\pazocal{B}^k}$) converge to $\chi_{\pazocal{A}}$ (respectively 0) strongly in $L^p(\Gamma_C)$, for each $1<p<+\infty$, up to a subsequence.
    \label{LemCvChik}
\end{lemma}

\begin{prf}
  First of all, note that since $\pazocal{A}^k$ and $\pazocal{A}$ are measurable, so are $\chi_{\pazocal{A}^k}$ and $\chi_{\pazocal{A}}$, for any $k\geq 0$. Since those functions are also bounded on $\Gamma_C$ which is of finite measure, they both belong to $L^1(\Gamma_C)$, and thus to any $L^p(\Gamma_C)$ with $1<p<+\infty$.
  Now, let us begin with proving pointwise convergence, from which we will deduce weak then strong convergence.
  
  Since $\uu^k \to \uu$ strongly in $L^2(\Gamma_C)$, there exists a subsequence, still denoted $\uu^k$, that converges a.e. on $\Gamma_C$. From now, we consider this subsequence. Then, from Assumption \ref{A5}, for a.e. $x\in\Gamma_C$, $x$ is either in $\pazocal{A}$, or in $\pazocal{I}$.
  \begin{itemize}[leftmargin=*]
      \item If $x\in\pazocal{A}$, then there is some $\delta>0$ such that $\lambda(x)>\delta$. From Assumption \ref{A6}, one gets that $\exists k_0>0$, $\forall k\geq k_0$, $|\lambda^k(x)-\lambda(x)|<\delta/2$. Especially, for all $k\geq k_0$,
      $$
        \lambda^{k+1}(x) = \maxx \left( \lambda^k(x)+\gamma^{k+1}\left( \uu^{k+1}_{\normalExt}(x)-\gG_{\normalExt}(x)\right) \right) >\delta/2 \:.
      $$
      Thus, for all $k\geq k_0+1$, $\chi_{\pazocal{A}^k}(x)=1=\chi_{\pazocal{A}}(x)$, and $\chi_{\pazocal{B}^k}(x)=0$.
      \item If $x\in\pazocal{I}$, then $\lambda(x)=0$ and there is some $\delta'>0$ such that $\uu_{\normalExt}(x)-\gG_{\normalExt}(x)<-\delta'/\gamma^1$. Moreover, $\exists k_0'>0$ such that $\forall k\geq k_0'$, $|\uu^k_{\normalExt}(x)-\uu_{\normalExt}(x)|<\delta'/2\gamma^1$ and $\lambda^k(x)<\delta'/4$. Consequently, for such values of $k$, one has
      $$
        \lambda^k(x)+\gamma^{k+1}\left( \uu^{k+1}_{\normalExt}(x)-\gG_{\normalExt}(x)\right) < \delta'/4-\gamma^{k+1}\delta'/2\gamma^1 \leq -\delta'/4\:.
      $$
      Hence, for any $k\geq k_0'+1$, one gets $\chi_{\pazocal{A}^k}(x)=0=\chi_{\pazocal{A}}(x)$, and $\chi_{\pazocal{B}^k}(x)=0$.
  \end{itemize}
  This proves that $\chi_{\pazocal{A}^k} \to \chi_{\pazocal{A}}$ and $\chi_{\pazocal{B}^k} \to 0$ a.e. on $\Gamma_C$. Let $1<p<+\infty$. For $\left\{\chi_{\pazocal{B}^k}\right\}_k$, strong convergence follows directly from Lebesgue's dominated convergence theorem and the fact that $|\chi_{\pazocal{B}^k}|^p\leq 1$ a.e. on $\Gamma_C$.  For $\left\{\chi_{\pazocal{A}^k}\right\}_k$, since the sequence is bounded in $L^p(\Gamma_C)$, weak convergence in $L^p(\Gamma_C)$ follows. Taking for example $p=2$, one gets that
  $$
    \int_{\Gamma_C} \chi_{\pazocal{A}^k} = \int_{\Gamma_C} \chi_{\pazocal{A}^k}\cdot 1 \longrightarrow \int_{\Gamma_C} \chi_{\pazocal{A}}\:.
  $$
  Obviously, this also proves that $\norml \chi_{\pazocal{A}^k} \normr_{L^p(\Gamma_C)} \to \norml \chi_{\pazocal{A}} \normr_{L^p(\Gamma_C)}$. For such values of $p$, $L^p(\Gamma_C)$ is uniformly convex, therefore weak convergence and convergence of the norms imply strong convergence.
\end{prf}

\begin{theo}
    Suppose Assumption \ref{A5} and Assumption \ref{A6} hold. Then, choosing the parameters $\gamma^k$ as in Corollary \ref{CorBoundMDer} leads to one of the two following cases:
    \begin{enumerate}[(i),leftmargin=*]
        \item $\left\{\gamma^k\right\}_k$ is bounded, then the whole sequence $\left\{ \dot{\ww}^k \right\}_k$ converges to $\dot{\ww}$, the solution of \eqref{FVConDerW}, strongly in $\Xx$,
        \item $\left\{\gamma^k\right\}_k$ is unbounded, then $\left\{ \dot{\ww}^k \right\}_k$ converges weakly to some limit $\hat{\ww}\in \Xx_{\pazocal{A}}$, up to a subsequence.
    \end{enumerate}
    \label{ThmCvgConDerALM}
\end{theo}

\begin{prf}
  As mentioned above, a direct consequence of Theorem \ref{ThmBoundMDer} is that, up to a subsequence, the iterates $\left( \dot{\ww}^k, \dot{\lambda}^{k} \right)$ converge weakly to some limit $(\hat{\ww},\hat{\lambda}) \in \Xx\times H^{-\frac{1}{2}}(\Gamma_C)$. In the following, let us consider that subsequence, which we still denote using the superscript $^k$. Due to strong convergence $\left( \uu^k, \lambda^{k} \right) \to \left( \uu,\lambda \right)$ in $\Xx\times H^{-\frac{1}{2}}(\Gamma_C)$, it is clear that $L^{k}[\thetaa] \to L[\thetaa]$ strongly in $\Xx^*$, where $L[\thetaa]$ is defined as: for all $\vv\in\Xx$,
  $$
    L[\thetaa](\vv) := L'(\vv) - a'(\uu,\vv) - \int_{\Gamma_C} \lambda \vv_{\normalExt} \divv_\Gamma \thetaa \:.
  $$
  Consequently, the weak limit satisfies:
  \begin{equation}
      a(\hat{\ww},\vv) - L[\thetaa](\vv) + a(\gG',\vv) + \prodD{\hat{\lambda},\vv_{\normalExt}}{\Gamma_C} = 0\:, \hspace{1em} \forall \vv\in\Xx\:.
  \end{equation}
  From now, we distinguish two possible cases.
  \begin{enumerate}[(i),leftmargin=*]
      \item Case $\left\{\gamma^k\right\}_k$ bounded. \\
      From the definition of $\left\{\gamma^k\right\}_k$ (see \eqref{HypGammak}), it follows that $\left\{\dot{\lambda}^k\right\}_k$ is bounded in $L^2(\Gamma_C)$. Therefore there exists a subsequence that converges weakly in $L^2(\Gamma_C)$ to $\hat{\lambda}$ (due to uniqueness of the weak limit), which also proves that $\hat{\lambda}\in L^2(\Gamma_C)$. As $\left\{\gamma^k\right\}_k$ is an increasing bounded sequence, it converges, say to $\hat{\gamma}_1>0$. Now, let $\eta\in L^4(\Gamma_C)$, one has:
      \begin{equation*}
        \begin{aligned} 
          \prodL2{\dot{\lambda}^{k+1},\eta}{\Gamma_C} &= \prodL2{\dot{\lambda}^{k}+\gamma^{k+1}\dot{\ww}^{k+1}_{\normalExt},\eta}{\pazocal{A}^{k+1}} + \prodL2{\dot{\lambda}^{k+1},\eta}{\pazocal{B}^{k+1}} \\
          &= \prodL2{\dot{\lambda}^{k},\chi_{\pazocal{A}^{k+1}}\eta}{\Gamma_C} + \gamma^{k+1}\prodL2{\dot{\ww}^{k+1}_{\normalExt},\chi_{\pazocal{A}^{k+1}}\eta}{\Gamma_C} + \prodL2{\dot{\lambda}^{k+1},\chi_{\pazocal{B}^{k+1}}\eta}{\Gamma_C}\:.
        \end{aligned}
      \end{equation*}
      We know that $\dot{\lambda}^{k}$, $\dot{\ww}^{k}_{\normalExt}$ converge weakly in $L^2(\Gamma_C)$, and using the results of Lemma \ref{LemCvChik} with $p=4$, it follows that $\chi_{\pazocal{A}^{k}}\eta$, $\chi_{\pazocal{B}^{k}}\eta$ converge strongly in $L^2(\Gamma_C)$. Thus, passing to the limit on both sides of the equality yields:
      $$
        \hat{\lambda} = \chi_{\pazocal{A}}\left( \hat{\lambda} + \hat{\gamma}_1 \hat{\ww}_{\normalExt}\right) \: \mbox{ in } L^{\frac{4}{3}}(\Gamma_C)\:.
      $$
      From this, one deduces that $\hat{\ww}\in\Xx_{\pazocal{A}}$ and
      $$
        a(\hat{\ww},\vv) - L'(\vv) + a'(\uu,\vv) + a(\gG',\vv) = 0\:, \hspace{1em} \forall \vv\in\Xx_{\pazocal{A}}\:,
      $$
      which means that $\hat{\ww}=\dot{\ww}$, since the solution of this problem is unique. Uniqueness also proves that the whole sequence converges weakly to $\dot{\ww}$.
      
      Finally, to get strong convergence, let us take $\vv=\dot{\ww}^{k+1}$ as test-function in \eqref{LMFMDerk:1}. Since the embedding $H^{\frac{1}{2}}(\Gamma_C)\hookrightarrow L^2(\Gamma_C)$ is compact, one obtains for the second term of \eqref{LMFMDerk:1}:
      $$
        \prodL2{\dot{\lambda}^{k+1},\dot{\ww}^{k+1}_{\normalExt}}{\Gamma_C} \longrightarrow \prodL2{\hat{\lambda},\dot{\ww}_{\normalExt}}{\Gamma_C} = 0\:,
      $$
      from which we get that $a(\dot{\ww}^{k+1},\dot{\ww}^{k+1}) \to a(\dot{\ww},\dot{\ww})$. The ellipticity of $a$ finishes the proof.
      \item Case $\left\{\gamma^k\right\}_k$ unbounded. \\
      Again, let us take $\vv=\dot{\ww}^{k+1}$ as test-function in \eqref{LMFMDerk:1}, and rewrite the second term of the formulation.
      \begin{equation*}
        \begin{aligned}
          \prodL2{\dot{\lambda}^{k+1},\dot{\ww}^{k+1}_{\normalExt}}{\Gamma_C} &= \prodL2{\dot{\lambda}^{k}+\gamma^{k+1}\dot{\ww}^{k+1}_{\normalExt},\dot{\ww}^{k+1}_{\normalExt}}{\pazocal{A}^{k+1}} + \prodL2{\dot{\lambda}^{k+1},\dot{\ww}^{k+1}_{\normalExt}}{\pazocal{B}^{k+1}} \\
          & \geq \gamma^{k+1} \norml \dot{\ww}^{k+1}_{\normalExt} \normr_{0,\pazocal{A}^{k+1}}^2 - C\left( \norml \dot{\lambda}^{k} \normr_{-1/2,\Gamma_C} + \norml \dot{\lambda}^{k+1} \normr_{-1/2,\Gamma_C} \right)\norml \dot{\ww}^{k+1} \normr_{\Xx}\:.
        \end{aligned}
      \end{equation*}
      Using this estimation in \eqref{LMFMDerk:1}, one obtains that $\left\{ \gamma^{k+1} \norml \dot{\ww}^{k+1}_{\normalExt} \normr_{0,\pazocal{A}^{k+1}}^2 \right\}_k$ is bounded. Using Lemma \ref{LemCvChik} with $p=2$ and the continuous embedding $H^{\frac{1}{2}}(\Gamma_C) \hookrightarrow L^4(\Gamma_C)$, one obtains, up to a subsequence,
      $$
        \norml \dot{\ww}^{k+1}_{\normalExt} \normr_{0,\pazocal{A}^{k+1}}^2 = \int_{\Gamma_C} \chi_{\pazocal{A}^{k+1}} \cdot | \dot{\ww}^{k+1}_{\normalExt} |^2 \: \longrightarrow \: \norml \hat{\ww}_{\normalExt} \normr_{0,\pazocal{A}}^2\:.
      $$
      On the other hand, since $\left\{\gamma^k\right\}_k$ is unbounded, there exists a subsequence (of the previous subsequence) that diverges to $+\infty$. Hence, boundedness of $\left\{ \gamma^{k+1} \norml \dot{\ww}^{k+1}_{\normalExt} \normr_{0,\pazocal{A}^{k+1}}^2 \right\}_k$ implies that $\norml \hat{\ww}_{\normalExt} \normr_{0,\pazocal{A}}^2=0$, that is $\hat{\ww}\in \Xx_{\pazocal{A}}$.
  \end{enumerate}
   
\end{prf}

\subsection{Shape derivative of a generic criterion}

As we aim at solving problem \eqref{ShapeOPT} using a gradient descent, we need to have usable and easy to evaluate expression for the shape derivatives of the functional $J$. From now, let us consider functionals which take the generic form:
\begin{equation}
  J^k(\Omega) := \int_\Omega l(\uu^k(\Omega)) + \int_{\partial\Omega} m(\uu^k(\Omega))\:.
  \label{GeneralJType}
\end{equation}
Let us make the usual regularity assumptions: the functions $l,m$ are $\pazocal{C}^1(\mathbb{R}^d,\mathbb{R})$, and their derivatives, denoted $l'$, $m'$, are Lipschitz. It is also assumed that those functions and their derivatives satisfy, for all $u$, $v \in \mathbb{R}^d$,
\begin{subequations}\label{CondReglm}
    \begin{align}
        |l(u)| \leq C\left(1+|u|^2\right)\:, \qquad
        &|m(u)| \leq C\left(1+|u|^2\right)\:,
        \label{Condlm} \\
        |l'(u)\cdot v| \leq C |u\cdot v|\:, \qquad
        &|m'(u)\cdot v| \leq C|u\cdot v| \:,
        \label{Condl'm'}
    \end{align}
\end{subequations}
for some constants $C>0$. Let us state the well known shape differentiability result for such functionals $J^k$ (see for example \cite{henrot2006variation}).

\begin{theo}  \label{ThmDJVol}
    Let $k\geq 1$. Under the assumptions of Corollary \ref{CorSDerk}, $J^k$ is shape differentiable at $\Omega$, and its derivative in the direction $\thetaa\in\Thetaa$ writes:
    \begin{equation}
        dJ^k(\Omega)[\thetaa] = \int_\Omega l'(\uu^k)\cdot \dot{\uu}^k + \: l(\uu^k)\divv\thetaa 
        + \int_{\partial\Omega} m'(\uu^k)\cdot \dot{\uu}^k + m(\uu^k) \divv_\Gamma\thetaa\:.
        \label{DJVol0}
    \end{equation}
\end{theo}

\begin{remark}
  It follows from Theorem \ref{ThmBoundMDer} that $dJ^k$ is bounded for suitable choices of parameters $\gamma^k$. Therefore, formula \eqref{DJVol0} produces usable shape derivatives, regardless how many iterations of the ALM algorithm are performed. Moreover, a sufficient condition to get convergence of $\left\{ dJ^k\right\}_k$ is that $\left\{\dot{\uu}^k\right\}_k$ converges weakly in $\Xx$.
\end{remark}

As usual, in order to rewrite \eqref{DJVol0} as a boundary integral, we introduce the associated adjoint state $\pp^k\in \Xx$ as the solution of:
\begin{equation}
    a(\pp^k,\vv) + \gamma^k\prodL2{\chi_{\pazocal{A}^k} \pp^k_{\normalExt}, \vv_{\normalExt}}{\Gamma_C} = -\int_\Omega l'(\uu^k)\vv - \int_{\partial\Omega} m'(\uu^k)\vv\:, \hspace{1em} \forall \vv \in \Xx\:.
    \label{FVAdj}
\end{equation}
The left hand side of this equation is obviously a continuous and coercive bilinear form on $\Xx\times\Xx$, while the right hand side belongs to $\Xx^*$ since $l$, $m$ satisfy conditions \eqref{CondReglm}. Thus, Lax-Milgram lemma ensures existence and uniqueness of $\pp^k$ for each $k$.

\begin{cor}
    Suppose $\Omega$ is of class $\pazocal{C}^2$. Then, if in addition to the assumptions of Theorem \ref{ThmDJVol}, $\uu^k$, $\pp^k\in \Hh^2(\Omega)$ and $\lambda^{k-1}\in H^{\frac{3}{2}}(\Gamma_C)$, then one has
    \begin{equation}
        dJ^k(\Omega)[\thetaa] = \int_{\partial\Omega} \mathfrak{A}^k \theta + \int_{\Gamma_N} \mathfrak{B}^k \theta + \int_{\Gamma_C} \mathfrak{C}^k \theta \:,
        \label{DJ}
    \end{equation}
  where $\mathfrak{A}^k$, $\mathfrak{B}^k$ and $\mathfrak{C}^k$ depend on $\uu^k$, $\pp^k$, $\lambda^{k}$, their gradients, and the data.
\end{cor}

\begin{prf}
   Suppose we are at iteration $k$. Then the solution at the previous iteration is known, and $\lambda^{k-1}$ can be considered as a fixed function. Due to its regularity, it will be seen as the trace of some function $\tilde{\lambda}\in H^2(\mathbb{R}^d)$, which leads to:
   $$
     \lambda^k = \maxx\left( \tilde{\lambda} + \gamma^k\left(\uu^k_{\normalExt}-\gG_{\normalExt} \right)\right) \in H^1(\Gamma_C)\:.
   $$
   Consequently, the technical difficulties related to this problem are similar to the ones in \cite{chaudet2019shape}, and the proof of \cite[Theorem 3.15]{chaudet2019shape} can be adapted, which enables to recover \eqref{DJ}, with:
   \begin{equation}\label{EqABC}
   \left\{\
      \begin{aligned}
        \mathfrak{A}^k &= l(\uu^k) + (\kappa +\partial_{\normalInt})m(\uu^k) + \Aa:\epsilonn(\uu^k):\epsilonn(\pp^k) - \ff\pp^k \:,\\
        \mathfrak{B}^k &= -(\kappa +\partial_{\normalInt})\left( \tauu \pp^k \right) \:,\\
        \mathfrak{C}^k &= -(\kappa +\partial_{\normalInt})\left( \lambda^k\pp^k_{\normalExt}\right)\:,
      \end{aligned}
  \right.
  \end{equation}
  where $\kappa$ is the mean curvature on $\partial\Omega$ and $\partial_{\normalInt}$ is the normal derivative.
  Of course, since $\thetaa \in \Thetaa$ here, $\thetaa\cdot\normalInt$ may simply be replaced by $\theta$.
\end{prf}


\section{Conclusion}

In this work, we have expressed sufficient conditions for shape differentiability of $\uu$, the solution to the original Signorini problem. We also got interested in shape differentiability properties for the sequence of the iterates $\uu^k$ (approaching $\uu$) obtained when applying the augmented Lagrangian method to this problem. After proving that these iterates were always conically shape differentiable, we have given sufficient conditions for the associated shape derivatives to converge to the shape derivative of $\uu$. On the other hand, we also found conditions that guarantee shape differentiability of the iterates. In this case, we were able to apply the usual method, that is to introduce an adjoint state, then get an explicit expression for the shape derivative of some generic functional.

A natural extension of the present work would be to consider the contact problem with Tresca friction. For this model, proving conical shape differentiabilty of $\uu^k$ and expressing sufficient conditions for $\uu^k$ to be shape differentiable is rather straightforward since the arguments from \cite{chaudet2019shape} apply. However, proving conical shape differentiability and deriving sufficient conditions for classical shape differentiability of the solution to the associated variational inequality seems more difficult. Indeed, for this problem, conical shape differntiability has been proved only in the two-dimensional case and for specific directions $\thetaa$, see \cite{sokolowski1988shape}.

Numerical analysis and comparison to other approaches is outside of the scope of this paper, and will be presented separately in a near future. From the numerical point of view, one could use formula \eqref{DJ} in a gradient based algorithm. After solving the last iteration $\bar{k}$ of the ALM, one may use the tangent matrix associated to the non-linear discrete system solved by $\uu^{\bar{k}}$ to get the adjoint and find a descent direction thanks to \eqref{DJ}. 

\part{Mise en \oe uvre numérique}

\chapter{Une méthode d'optimisation de formes avec des level-sets}     
\label{chap:4.1}

\section*{Introduction}

Nous revenons ici à l'algorithme d'optimisation de formes brièvement exposé au chapitre \ref{chap:1.1}. On rappelle que l'approche considérée s'appuie sur une représentation des formes par l'intermédiaire de level sets. Par ailleurs, pour une forme donnée, nous avons vu dans les chapitres précédents qu'on est capable de déterminer des directions de descente à partir de l'analyse de sensibilité par rapport à la forme. Ceci nous permet donc d'implémenter un algorithme d'optimisation de type descente de gradient. Cette approche a pour la première fois été introduite dans \cite{allaire2004structural}, dans le contexte de l'optimisation de structures. 

L'utilisation d'une fonction level set $\phi$ est très intéressante en pratique puisqu'elle permet de gérer les déformations géométriques du domaine de façon tout à fait naturelle en résolvant une équation d'advection ou de Hamilton-Jacobi pour $\phi$. Elle gère donc très bien les changements de topologie. En revanche, ce choix ne permet a priori pas d'avoir une représentation nette de l'interface $\partial\Omega=\{ \phi=0 \}$. Pour une revue complète des différentes variantes de l'approche par level-set, nous renvoyons le lecteur à \cite{van2013level}. À l'origine, dans \cite{allaire2004structural} et la majorité des travaux qui en découlent directement, les auteurs proposent de travailler sur un domaine de calcul maillé avec des quadrangles ou des hexaèdres réguliers. Ceci permet d'appliquer un schéma de type différences finies pour l'advection de $\phi$ et d'utiliser une formulation éléments finis de type $Q^1$ ou $Q^2$ pour la résolution mécanique, le tout sur un seul et même maillage. Dans ces cas-là, comme on ne dispose pas d'un sous-maillage pour $\Omega$, la résolution mécanique doit être faite sur tout $D$, en remplissant $D\setminus \Omega$ par un matériau très « mou », dit \textit{matériau Ersatz}. Quant aux conditions aux limites qui s'appliquent sur $\partial\Omega$, elles sont imposées de façon approchée, à l'aide d'une régularisation de la fonction Heaviside du bord. On peut également gérer ces conditions aux limites sans avoir besoin de remailler à l'aide de \textit{méthodes de frontières immergées} (en anglais IBTs pour \textit{Immersed Boundary Techniques}), parmi lesquelles nous citons la très populaire méthode XFEM, voir \cite{belytschko2003topology,kreissl2012levelset}. 

Dans ce travail, nous choisissons une approche différente qui consiste à coupler la méthode de level-set avec une \textit{discrétisation conforme}, c'est-à-dire qu'on cherche à avoir une description de $\Omega$ à la fois de manière implicite via $\phi$, et explicite via un sous-maillage de $D$. Dans cette direction, nous mentionnons entre autres les travaux \cite{ha2008level,yamasaki2011level,allaire2011topology,allaire2014shape}, ainsi que \cite{abe2007boundary} pour le cas d'une résolution du problème mécanique par la méthode des éléments de frontière. Bien que cette approche soit plus coûteuse puisqu'elle nécessite de procéder à un remaillage à chaque itération, elle permet de limiter la résolution au sous-domaine $\Omega$, et d'appliquer exactement les conditions aux limites sur $\partial\Omega$, ce qui est un critère important dans le cas des problèmes de contact, où la qualité et la précision de la solution obtenue dépendent fortement de la précision au voisinage de la zone de contact $\Gamma_C$.

L'objectif de ce chapitre est d'abord de présenter de façon plus détaillée les différentes étapes de l'algorithme proposé. Nous nous attardons en particulier sur certains choix d'implémentation, notamment la technique de découpage mise en \oe uvre pour disposer d'une discrétisation conforme. Ensuite, nous montrons des résultats numériques en deux et trois dimensions, pour des problèmes d'élasticité avec et sans contact.

\section{Détails d'implémentation}

On commence par rappeler le problème qu'on cherche à résoudre. Il s'agit d'un problème d'optimisation de la forme:
\begin{equation}
    \inf_{\Omega \in \pazocal{U}_{ad}} J(\Omega) \:,
    \label{ShapeOPT.4.1}
\end{equation}
où $J$ représente la fonctionnelle à minimiser, $\Omega$ est la forme ou le domaine sur lequel est posé notre problème mécanique, et $\pazocal{U}_{ad}$ est l'ensemble des formes admissibles.

Dans notre contexte, $J$ prend la forme d'une expression de type \eqref{ExpressionJ}, dépendant de $\uu$, la solution de l'équation d'état posée sur $\Omega$. En général, la définition de $\pazocal{U}_{ad}$ permet d'imposer des contraintes, d'une part sur les propriétés géométriques des domaines $\Omega$ considérés, et d'autre part sur certaines zones de ces domaines ou de leurs frontières relatives au problème physique sous-jacent. Avant de donner la définition que nous avons choisie pour $\pazocal{U}_{ad}$, nous introduisons quelques notations.

On se donne un domaine régulier $D\subset \mathbb{R}^d$ fixé, qui contiendra tous les domaines admissibles $\Omega$. Soit $\hat{\Gamma}_D\subset\partial D$ une partie de sa frontière représentant la zone de Dirichlet potentiel. C'est-à-dire que pour chaque domaine $\Omega\subset D$, le bord de Dirichlet associé à $\Omega$ sera donné par $\Gamma_D:=\partial\Omega \cap \hat{\Gamma}_D$. Si le bord de Dirichlet est « mobile » à l'intérieur de $\hat{\Gamma}_D$, il ne pourra pas être vide, car on veut assurer le caractère bien posé et la cohérence physique du problème. Quant au bord de Neumann (non-homogène) $\Gamma_N$, nous le supposons inclus dans $\partial D$ et fixé pendant tout le processus d'optimisation de formes. Ceci nous donne donc finalement:
$$
   \pazocal{U}_{ad} := \{ \Omega \subset D \: | \: \Omega \mbox{ ouvert régulier, } \Gamma_N\subset \partial D \mbox{ fixé, } \partial\Omega \cap \hat{\Gamma}_D \neq \emptyset \}.
$$

Pour résoudre \eqref{ShapeOPT.4.1}, nous proposons un algorithme de type gradient, qui génère une suite de formes $\{ \Omega^l\}_l$ de sorte que la monotonie de $J$ soit assurée, i.e.$\!$ la suite $\{ \Omega^l \}_l$ est telle que $J(\Omega^{l+1})<J(\Omega^l)$. Résumons l'algorithme présenté au chapitre \ref{chap:1.1}, section \ref{sec:algoODF}:
\begin{algorithm}
\caption*{\textbf{Algorithme:} Optimisation de formes}
\begin{enumerate}
	\item Choisir un domaine $\Omega^0$ et initialiser $l=0$. 
	\item Résoudre l'équation d'état \eqref{FVElash}, \eqref{FVPenah} ou \eqref{FVLagAugh}, posée sur $\Omega^l$.
	\item Résoudre la formulation adjointe \eqref{FVElasAdjh}, \eqref{FVPenaAdjh} ou \eqref{FVLagAugAdjh}, posée sur $\Omega^l$.
	\item Déterminer une direction de descente $\thetaa^l$ en résolvant \eqref{FVThetaScal}.
	\item Mettre à jour le domaine $\Omega^l \to \Omega^{l+1}$.
    \item Tant que le critère de convergence n'est pas respecté, mettre à jour $l=l+1$ et retourner à l'étape 2.
\end{enumerate}
\end{algorithm}

Nous nous intéressons maintenant au corps de l'algorithme, à savoir les étapes 2, 3, 4 et 5.

\begin{notation}
	Pour les descriptions des étapes 2, 3, 4 et 5, nous omettons les exposants $l$ relatifs à l'itération courante pour améliorer la lisibilité. Par exemple, les variables $\Omega^l$, $\thetaa^l$ seront remplacées par $\Omega$, $\thetaa$.  
\end{notation}

\subsection{Résolution de l'équation d'état}

Soit $D_h$ un maillage de $D$, $h$ étant le paramètre de discrétisation, correspondant ici à la taille de maille. On suppose qu'à chaque itération, on dispose d'une discrétisation $\Omega_h$ de $\Omega$ sous la forme d'un sous-maillage de $D_h$, sur lequel on cherche à résoudre une équation d'état. Dans notre cas, cette équation d'état est l'équation qui traduit l'équilibre mécanique (en élasticité linéaire ou en élasticité linéaire avec contact), et nous notons $\uu\in \Xx$ sa solution.

L'étape préliminaire est de choisir une discrétisation de cette équation sur le domaine $\Omega_h$. Ici, pour que l'étape de découpage / remaillage soit le plus efficace, on considère des maillages non structurés composés de triangles en 2D et de tetraèdres en 3D. On peut donc utiliser des éléments finis de Lagrange ($P^1$ ou $P^2$), ce qui nous donne une approximation $\uu_h$ de $\uu$. Pour une présentation de la méthode des éléments finis, nous renvoyons le lecteur à l'ouvrage de référence \cite{Cia2002a}.
 
\paragraph{Cas de l'élasticité linéaire.}
On choisit une discrétisation $\uu_h\in P^2$ de $\uu$. Le problème discret \eqref{FVElash} vérifié par $\uu_h$ étant linéaire, il se traduit en un système matrice-vecteur que nous résolvons à l'aide d'une méthode directe ($LU$). Pour l'analyse de convergence, nous renvoyons à \cite{Cia2002a}. Rappelons l'équation vérifiée par $\uu$ (introduite en \eqref{FVElas}):
\begin{equation}
	a(\uu,\vv) = L(\vv)\:, \hspace{1em} \forall \vv \in \Xx\:.
	\label{FVElash}
\end{equation}

\paragraph{Cas du contact pénalisé.}
On choisit toujours une discrétisation $\uu_{\varepsilon,h}\in P^2$ de $\uu_\varepsilon$. Comme le problème continu, non linéaire et non différentiable, est traité par une méthode de Newton semi-lisse, on obtient que la discrétisation de chaque itération de Newton donne un problème linéaire, que nous résolvons par $LU$. Pour l'analyse de convergence, voir \cite{chouly2013convergence}. On rappelle l'équation vérifiée par $\uu_{\varepsilon}$ (introduite en \eqref{FVPena}): $\forall \vv \in \Xx$,
\begin{equation}
	a(\uu_{\varepsilon},\vv) + \frac{1}{\varepsilon} \prodL2{\maxx(\uu_{\varepsilon,\normalExt}-\gG_{\normalExt}),\vv_{\normalExt}}{\Gamma_C} + \frac{1}{\varepsilon} \prodL2{\qq(\varepsilon\mathfrak{F}s,\uu_{\varepsilon,\tanExt}), \vv_{\tanExt}}{\Gamma_C} = L(\vv)\:.
	\label{FVPenah}
\end{equation}

\paragraph{Cas du contact par Lagrangien augmenté.}
Pour ce cas, nous prenons une discrétisation $\uu^{k}_h\in P^2$ de $\uu^{k}$ et $\left( \lambda^{k}_h, \muu^{k}_h\right)\in P^1\times P^1$ de $\left( \lambda^{k}, \muu^k\right)$. On donne plus bas la formulation \eqref{FVLagAugh} dans le cas frottant. Comme pour la pénalisation, après application de la méthode de Newton semi-lisse, la discrétisation mène à un système algébrique linéaire qu'on résout par $LU$. Pour l'analyse de convergence, voir \cite{burman2019augmented}. Nous mentionnons également \cite{chouly2013nitsche,chouly2014adaptation} pour la méthode de Nitsche, dont la méthode de Lagrangien augmenté est un cas particulier. Rappelons la formulation vérifiée par le triplet $\left(\uu^{k}, \lambda^{k}, \muu^{k}\right)$ (introduite en \eqref{IterLagAug})
\begin{equation}
	\begin{aligned}
	&a(\uu^k,\vv) + \prodL2{\lambda^k,\vv_{\normalExt}}{\Gamma_C}+ \prodL2{\muu^k,\vv_{\tanExt}}{\Gamma_C} = L(\vv)\:, \\
	&\lambda^k = \maxx\left( \lambda^{k-1} + \gamma_1^k (\uu^k_{\normalExt}-\gG_{\normalExt})\right) \:, \\
	&\muu^k = \qq\left( \mathfrak{F}s, \muu^{k-1} + \gamma_2^k\uu^k_{\tanExt} \right) \:.
	\end{aligned}
	\label{FVLagAugh}
\end{equation}

\subsection{Résolution de la formulation adjointe}

Puisque l'état adjoint $\pp$ vit dans le même espace que $\uu$, on choisit la même discrétisation, à savoir $\pp_h\in P^2$. Puis, quel que soit le problème mécanique considéré, la formulation vérifiée par l'état adjoint est linéaire et admet un second membre dépendant uniquement de $J$ et de $\uu$ donné par:
$$
	L_{adj}[\uu](\vv)= -\int_{\Omega} j'(\uu)\cdot\vv -\int_{\partial\Omega}k'(\uu)\cdot\vv\:, \ \ \forall \vv \in \Xx\:.
$$ 
Comme on connaît $j$, $k$, et $\uu_h$ (qu'on vient de calculer à l'étape précédente), on est capable d'assembler la version discrète de ce second membre. Pour ce qui est du membre de gauche de la formulation, nous distinguons les cas.

\paragraph{Cas de l'élasticité linéaire.}
Dans ce cas, comme la forme bilinéaire $a$ est symétrique, le membre de gauche de la formulation adjointe \eqref{FVElasAdjh} est le même que celui de la formulation primale \eqref{FVElas}, voir plus bas. On peut donc récupérer le système résolu à l'étape précédente, en changeant uniquement le second membre. 
\begin{equation}
	a(\pp,\vv) = L_{adj}[\uu](\vv)\:, \hspace{1em} \forall \vv \in \Xx\:.
	\label{FVElasAdjh}
\end{equation}

\paragraph{Cas du contact pénalisé.}
L'adjoint $\pp_\varepsilon$ dans ce cas-là est donné par la solution de \eqref{FVA}. En utilisant les notations introduites au chapitre \ref{chap:1.1} dans la présentation de la méthode de Newton semi-lisse, on peut réécrire cette formulation de la façon suivante:
\begin{equation}
	\prodD{G_\varepsilon(\uu_\varepsilon)\pp_\varepsilon,\vv}{\Xx^*,\Xx} = L_{adj}[\uu_\varepsilon](\vv)\:, \ \ \forall \vv \in \Xx\:, 
	\label{FVPenaAdjh}
\end{equation}
où $G_\varepsilon$ a été défini en \eqref{ExpDerGenFPena}. Cela signifie que pour construire le système algébrique vérifié par $\pp_{\varepsilon,h}$, il suffit de prendre le système résolu à la dernière itération de Newton de l'étape précédente, en d'en changer le second membre.

\paragraph{Cas du contact par Lagrangien augmenté.}
De même que pour la pénalisation, la formulation adjointe \eqref{FVAdj} vérfiée par $\pp^k$ se réécrit:
\begin{equation}
	\prodD{G^k(\uu^k)\pp^k,\vv}{\Xx^*,\Xx} = L_{adj}[\uu^k](\vv)\:, \ \ \forall \vv \in \Xx\:, 
	\label{FVLagAugAdjh}
\end{equation}
où $G^k$ a été défini en \eqref{ExpDerGenFLagAug}. Là encore, on récupère la dernière matrice tangente de l'étape précédente, à laquelle on associe notre second membre, ce qui nous donnera le système à résoudre pour trouver $\pp^k_h$.

\begin{rmrk}
	Le calcul de l'état adjoint est peu coûteux en pratique car il correspond toujours à la résolution d'un problème linéaire, dont la complexité est la même que celle d'une itération de Newton. Le coût de calcul est même moindre puisque l'assemblage et la décomposition $LU$ de la matrice ont déjà été faits.
\end{rmrk}

\subsection{Calcul d'une direction de descente}

Le processus choisi pour l'évolution du domaine exige de déterminer un champ de vecteurs $\thetaa=\theta \normalInt$ sur tout le domaine $D$. On voudra ici que ce champ vérifie $dJ(\Omega)[\thetaa]<0$. Pour chacune des formulations considérées, on a vu que sous les bonnes hypothèses, on disposait d'une expression de cette dérivée directionnelle sous la forme:
$$
	dJ(\Omega)[\thetaa] = \int_{\partial\Omega} \mathfrak{g}(\thetaa\cdot\normalInt) = \int_{\partial\Omega} \mathfrak{g} \,\theta\:,
$$
où $\mathfrak{g}$ dépend de $\uu$, $\pp$, leurs gradients, et les données du problème. Puisqu'on souhaite se limiter à des champs de vecteurs dirigés selon $\normalInt$, il nous suffit de déterminer un champ scalaire $\theta$ qui nous assure que $\thetaa=\theta\normalInt$ soit une direction de descente. Compte tenu de l'expression précédente, nous choisissons $\theta\in H^1(D)$ solution de:
\begin{equation}
	\int_{D} \gradd \theta\cdot\gradd v + \alpha\,\theta\, v = - \int_{\partial\Omega} \mathfrak{g} \,v \:, \ \ \forall v \in H^1(D)\:,
	\label{FVThetaScal}
\end{equation}
où la paramètre $\alpha$ est de l'ordre de la taille de maille $h$, voir \cite{dapogny2013shape}. Nous renvoyons au chapitre \ref{chap:1.1} pour les motivations de ce choix. On peut discrétiser cette formulation par éléments finis de type Lagrange sur $D_h$, en prenant par exemple $\theta_h\in P^1$ ou $P^2$. Le membre de gauche est standard, et puisqu'on est capable de calculer $\mathfrak{g}_h$ aux points d'intégration à partir de $\uu_h$ et $\pp_h$, on peut assembler le second membre en intégrant sur le bord maillé $\partial\Omega_h$. On obtient alors un système algébrique linéaire, qu'on résout encore par une méthode directe.

Nous rappelons ici les expressions de $\mathfrak{g}$ obtenues pour l'élasticité sans contact, pour le contact frottant pénalisé, et pour le contact frottant par Lagrangien augmenté (voir \eqref{ExpStructdJ}, \eqref{EqABC.2} et \eqref{EqABC}, respectivement):
\begin{equation*}
	\begin{aligned}
		\mathfrak{g} &= j(\uu)+\Aa:\epsilonn(\uu):\epsilonn(\pp)-\ff\pp + \chi_{\Gamma_N}(\kappa+\partial_{\normalInt})\left( k(\uu) - \tauu\pp\right)\:, \\
		\mathfrak{g}_\varepsilon &= j(\uu_\varepsilon)+\Aa:\epsilonn(\uu_\varepsilon):\epsilonn(\pp_\varepsilon)-\ff\pp_\varepsilon + \chi_{\Gamma_N}(\kappa+\partial_{\normalInt})\left( k(\uu_\varepsilon) - \tauu\pp_\varepsilon\right)\\
		\: & \hspace{1em} + \frac{1}{\varepsilon}\chi_{\Gamma_C}(\kappa+\partial_{\normalInt})\left( \maxx(\uu_{\varepsilon,\normalExt}-\gG_{\normalExt})\pp_{\varepsilon,\normalExt} + \qq(\varepsilon\mathfrak{F}s,\uu_{\varepsilon,\tanExt})\pp_{\varepsilon,\tanExt}\right) \:, \\
		\mathfrak{g}^k &= j(\uu^k)+\Aa:\epsilonn(\uu^k):\epsilonn(\pp^k)-\ff\pp^k + \chi_{\Gamma_N}(\kappa+\partial_{\normalInt})\left( k(\uu^k) - \tauu\pp^k\right)\\
		\: & \hspace{1em} + \chi_{\Gamma_C}(\kappa+\partial_{\normalInt})\left( \lambda^k\pp^k_{\normalExt} + \muu^k\pp^k_{\tanExt}\right) \:. 
	\end{aligned}
\end{equation*}

\begin{rmrk}
	Une façon de s'assurer que la partie $\Gamma_N$ du bord reste bien fixée est d'ajouter la condition de Dirichlet $\theta=0$ sur $\Gamma_N$ à la formulation \eqref{FVThetaScal}.
\end{rmrk}

\subsection{Évolution du domaine}

La fonction level-set $\phi$ avec laquelle nous travaillons est la fonction distance orientée à $\partial\Omega$. Si $\Omega$ est suffisamment régulier, alors cette fonction est bien définie, et elle admet de bonnes propriétés, voir \cite{DelZol2001}. L'évolution du domaine est alors régie par l'équation de Hamilton-Jacobi \eqref{HJEquation} sur un intervalle de temps $[0,T]$, où $\theta(t,x)=\theta(x)$ est le champ scalaire calculé à l'étape précédente.

Pour des questions de robustesse, nous choisissons de discrétiser cette équation sur un maillage cartésien auxiliaire, puis de la résoudre par différences finies (en espace et en temps). Notons justement qu'une partie de la robustesse, de l'efficacité et de la simplicité de l'approche proposée est dûe à la méthode de résolution (\textit{fast-marching}, voir \cite{Set1996}), et au fait que $D$ soit en général un parallelépipède. On a un maillage $D_\Delta$ de $D$, dont on note $(x_i,y_j,z_k)$ les différents n\oe uds. On subdivise également $[0,T]$ en $N$ sous-intervalles $[t^{n-1},t^n]$, $1\leq n \leq N$. Ce qui nous donne une version discrète $\phi_{ijk}^n = \phi\left(t^n,(x_i,y_j,z_k)\right)$ sur $[0,T]\times D_\Delta$. En réinterpolant $\theta_h$ sur la grille cartésienne, i.e.$\!$ en évaluant $\theta_h$ aux n\oe uds de la grille, on obtient une approximation $\theta_{ijk}$ de $\theta$, et on peut résoudre une version discrète de l'équation de Hamilton-Jacobi:
\begin{equation}
    \begin{aligned}
    &\frac{\partial\phi}{\partial t}(t,x) + \theta(t,x) \cdot|\grad\phi(t,x)| = 0\:, \\
    &\phi(0,x) = \phi_0(x)\:.
    \end{aligned}
	\label{HJEquationh}
\end{equation}

\begin{rmrk}
	En pratique, l'efficacité de l'algorithme nous permet d'utiliser une grille $D_\Delta$ très fine, afin de gagner un peu en précision. 
\end{rmrk}

Nous associons à cette équation la condition initiale $\phi^0_{ijk}=\phi_0(x_i,y_j,z_k)$ où $\phi_0$ désigne la fonction distance signée à la frontière $\partial\Omega$ de la forme courante, ainsi que des conditions aux limites de Neumann sur $\partial D_\Delta$. Pour la discrétisation en espace, afin de limiter la diffusion numérique, nous choisissons un schéma explicite d'ordre deux, voir \cite{harten1987uniformly,osher1991high}. Pour la discrétisation en temps, nous utilisons un schéma d'Euler explicite standard, ce qui impose une condition de type CFL sur le pas de temps. Après résolution, on se retrouve avec $\phi^N_{ijk}$ qui nous donne une fonction level-set associée au domaine après évolution, qu'on peut projeter $L^2$ sur le maillage éléments finis. Ceci nous donne $\tilde{\phi}_h\in P^m$, $m\geq 1$, représentation implicite du nouveau domaine $\tilde{\Omega}_h$ sur $D_h$.

\begin{rmrk}
	Même si la condition initiale $\phi_0$ est très régulière et correspond à une fonction distance, rien ne garantit que la fonction $\phi$ conserve ces bonnes propriétés au cours de la résolution de \eqref{HJEquationh}. En pratique, la solution peut devenir très irrégulière, notamment au voisinage de $\{ \phi=0 \}$. Elle peut aussi s'éloigner considérablement d'une fonction distance, i.e.$\!$ $|\grad \phi| \ll 1$ ou $\gg 1$. Pour éviter ces problèmes, il est possible de \textit{réinitialiser} périodiquement au cours de la résolution de \eqref{HJEquation}, voir \cite{SusSmeOsh1994}. Plus précisément, à $\bar{t}\in [0,T]$ fixé, réinitialiser $\bar{\phi}=\phi(\bar{t},\cdot)$ signifie remplacer $\bar{\phi}$ par la solution $\psi$ de:
	\begin{equation}
		\left\{ \
		\begin{array}{cr}
			\partial_t \psi + \mbox{sgn} (\bar{\phi}) \left( |\grad\psi|-1 \right) = 0  & \mbox{ sur } [0,+\infty)\times D,\\
			\psi(0,x) = \bar{\phi}(x) & \mbox{ sur } D.	
		\end{array}
		\right.
		\label{EqRedistPhi}
	\end{equation}
	Puisque \eqref{EqRedistPhi} est une équation de Hamilton-Jacobi, on peut encore utiliser le même schéma numérique pour la résoudre.
\end{rmrk}

\paragraph{Maillage de $\tilde{\Omega}_h$.} À cette étape de l'itération, on aimerait pouvoir décider si on accepte cette nouvelle forme. Pour cela, on doit pouvoir vérifier qu'elle respecte bien les critères géométriques pour être dans $\pazocal{U}_{ad}$, et on doit pouvoir calculer $J(\tilde{\Omega}_h)$. Pour ces deux raisons, dans le cadre de l'approche par discrétisation conforme, on souhaite disposer d'un maillage de $\tilde{\Omega}_h$. Comme on connaît $\tilde{\phi}_h$ partout sur $D_h$, il est possible de déterminer \textit{exactement} l'intersection  entre $\{ \tilde{\phi}_h=0\}$ et les arêtes du maillage $D_h$, peu importe le degré d'interpolation choisi pour $\tilde{\phi}_h$. Il suffit ensuite d'ajouter des n\oe uds au maillage à toutes ces intersections, en ajoutant les composantes de maillages permettant d'assurer la validité de ce-dernier à chaque étape du processus. On obtient alors un nouveau maillage $\tilde{D}_h$ de $D$, contenant un sous-maillage $\tilde{\Omega}_h$ représentant le nouveau domaine.

Cette idée très simple qui consiste à \textit{découper} le maillage autour de la ligne ou surface $\{ \tilde{\phi}_h=0\}$ est bien connue dans le cas d'un champ $\tilde{\phi}_h\in P^1$, voir par exemple \cite{lorensen1987marching,frey1996texel}. Dans le contexte de l'optimisation de formes, nous mentionnons la série de papiers \cite{allaire2011topology,allaire2013mesh,allaire2014shape}, où les auteurs utilisent cette technique de découpage comme l'étape préliminaire d'un remaillage plus sophistiqué. Elle présente deux avantages principaux: sa facilité d'implémentation, et sa robustesse. En revanche, elle génère des maillages dont les éléments ne sont pas de bonne qualité au voisinage de l'interface découpée (éléments étirés, pouvant être très petits), et de plus cette interface maillée peut s'avérer irrégulière, surtout en 3D. Nous renvoyons à \cite{dapogny2013shape} pour une discussion à ce sujet. 

\begin{figure}
\begin{center}

\subfloat[Points d'intersection trouvés.]{
\begin{tikzpicture}

\draw[scale=0.5,domain=-1.9:2.1,smooth,variable=\y,red, densely dashed]  plot ({\y*\y},{\y});     

\node[] at (1.3,1.4) {$\tilde{\phi}_h>0$};
\node[] at (2.4,0.4) {$\tilde{\phi}_h<0$};

\draw[black] (0.2,-1) -- (0.2,1);
\draw[black] (0.2,-1) -- (2.4,-0.3);
\draw[black] (0.2,1) -- (2.4,-0.3);
\draw[black] (0.2,-1) -- (-1.2,0.4);
\draw[black] (0.2,1) -- (-1.2,0.4);

\node[red] at (0.82,0.64) {\tiny{$\bullet$}};
\node[red] at (1.08,-0.74) {\tiny{$\bullet$}};

\end{tikzpicture}
}
\hspace{4em}
\subfloat[Maillage après découpage.]{
\begin{tikzpicture}
    
\draw[black] (0.2,-1) -- (0.2,1);
\draw[black] (0.2,-1) -- (2.4,-0.3);
\draw[black] (0.2,1) -- (2.4,-0.3);
\draw[black] (0.2,-1) -- (-1.2,0.4);
\draw[black] (0.2,1) -- (-1.2,0.4);

\draw[red, thick] (1.08,-0.74) -- (0.82,0.64);
\draw[black, ultra thin] (1.08,-0.74) -- (0.2,1);

\node[red] at (0.82,0.64) {\tiny{$\bullet$}};
\node[red] at (1.08,-0.74) {\tiny{$\bullet$}};

\end{tikzpicture}
}

\end{center}
  \caption{Schéma du découpage de $D_h$ pour $\tilde{\phi}_h\in P^1$.}
  \label{fig:SchDecP1}
\end{figure}
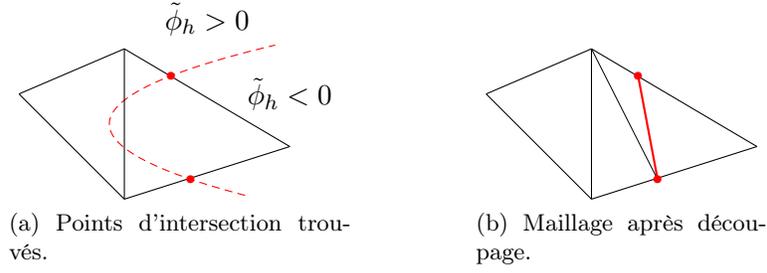

\begin{figure}
\begin{center}

\subfloat[Points d'intersection trouvés.]{
\begin{tikzpicture}

\draw[scale=0.5,domain=-1.9:2.1,smooth,variable=\y,red, densely dashed]  plot ({\y*\y},{\y});     

\node[] at (1.3,1.4) {$\tilde{\phi}_h>0$};
\node[] at (2.4,0.4) {$\tilde{\phi}_h<0$};

\draw[black] (0.2,-1) -- (0.2,1);
\draw[black] (0.2,-1) -- (2.4,-0.3);
\draw[black] (0.2,1) -- (2.4,-0.3);
\draw[black] (0.2,-1) -- (-1.2,0.4);
\draw[black] (0.2,1) -- (-1.2,0.4);

\node[red] at (0.2,0.3) {\tiny{$\bullet$}};
\node[red] at (0.2,-0.32) {\tiny{$\bullet$}};
\node[red] at (0.82,0.64) {\tiny{$\bullet$}};
\node[red] at (1.08,-0.74) {\tiny{$\bullet$}};

\end{tikzpicture}
}
\hspace{4em}
\subfloat[Maillage après découpage.]{
\begin{tikzpicture}
    
\draw[black] (0.2,-1) -- (0.2,1);
\draw[black] (0.2,-1) -- (2.4,-0.3);
\draw[black] (0.2,1) -- (2.4,-0.3);
\draw[black] (0.2,-1) -- (-1.2,0.4);
\draw[black] (0.2,1) -- (-1.2,0.4);

\draw[red, thick] (0.2,0.3) -- (0.2,-0.32);
\draw[red, thick] (0.2,0.3) -- (0.82,0.64);
\draw[red, thick] (0.2,-0.32) -- (1.08,-0.74);
\draw[black, ultra thin] (0.2,0.3) -- (-1.2,0.4);
\draw[black, ultra thin] (0.2,0.3) -- (2.4,-0.3);
\draw[black, ultra thin] (0.2,-0.32) -- (-1.2,0.4);
\draw[black, ultra thin] (0.2,-0.32) -- (2.4,-0.3);

\node[red] at (0.2,0.3) {\tiny{$\bullet$}};
\node[red] at (0.2,-0.32) {\tiny{$\bullet$}};
\node[red] at (0.82,0.64) {\tiny{$\bullet$}};
\node[red] at (1.08,-0.74) {\tiny{$\bullet$}};

\end{tikzpicture}
}

\end{center}
  \caption{Schéma du découpage de $D_h$ pour $\tilde{\phi}_h\in P^2$.}
  \label{fig:SchDecP2}
\end{figure}
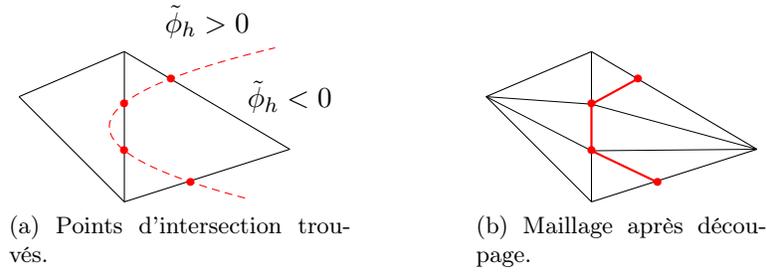

Dans ce travail, nous privilégions cette technique malgré ses inconvénients, principalement pour sa facilité d'implémentation. Cependant, nous proposons d'en améliorer la précision en considérant un champ $\tilde{\phi}_h\in P^m$ avec $m>1$. Cela permet de repérer un plus grand nombre de points d'intersection entre $\{ \tilde{\phi}_h=0\}$ et $D_h$ (jusqu'à deux par arête pour le $P^2$, trois pour le $P^3$, etc), et de donner une représentation plus précise de la ligne ou surface régulière $\{ \phi=0\}$ sous-jacente. Comme on peut le voir sur l'exemple \ref{ex:decoupage} (figures \ref{fig:SchDecP1} et \ref{fig:SchDecP2}), passer du découpage $P^1$ au découpage $P^2$ semble apporter un gain en précision, même dans des cas 2D relativement élémentaires.

\begin{exmpl} \label{ex:decoupage}
	Pour illustrer l'influence du degré d'interpolation de $\tilde{\phi}_h$ sur $\tilde{\Omega}_h$, prenons un exemple tridimensionnel. On considère un maillage régulier $D_h$ du cube $[0,1]\times[0,1]\times[0,1]\subset \mathbb{R}^3$ composé de 4913 sommets, voir figure \ref{sub:MailDh}. On définit sur ce domaine la fonction scalaire polynômiale de degré 4
$$
\tilde{\phi}(x,y,z):= 16\left(x-\frac{1}{2}\right)^4+\left(y-\frac{1}{2}\right)^2+\left(z-\frac{1}{2}\right)^2-\frac{1}{4}\:.
$$
On note toujours $\tilde{\phi}_h$ le champ éléments finis associé à $\tilde{\phi}$ sur $D_h$. Lorsqu'on utilise la technique de découpage présentée plus haut dans les trois cas $\tilde{\phi}_h\in P^1$, $P^2$ et $P^3$, on obtient les surfaces et domaines présentés figure \ref{fig:ResDecSphere}.

\begin{figure}[h]
  \begin{center}
    \subfloat[Domaine $\tilde{\Omega}_h$ pour $\tilde{\phi}_h\in P^1$.]{
    \includegraphics[width=0.28\textwidth]{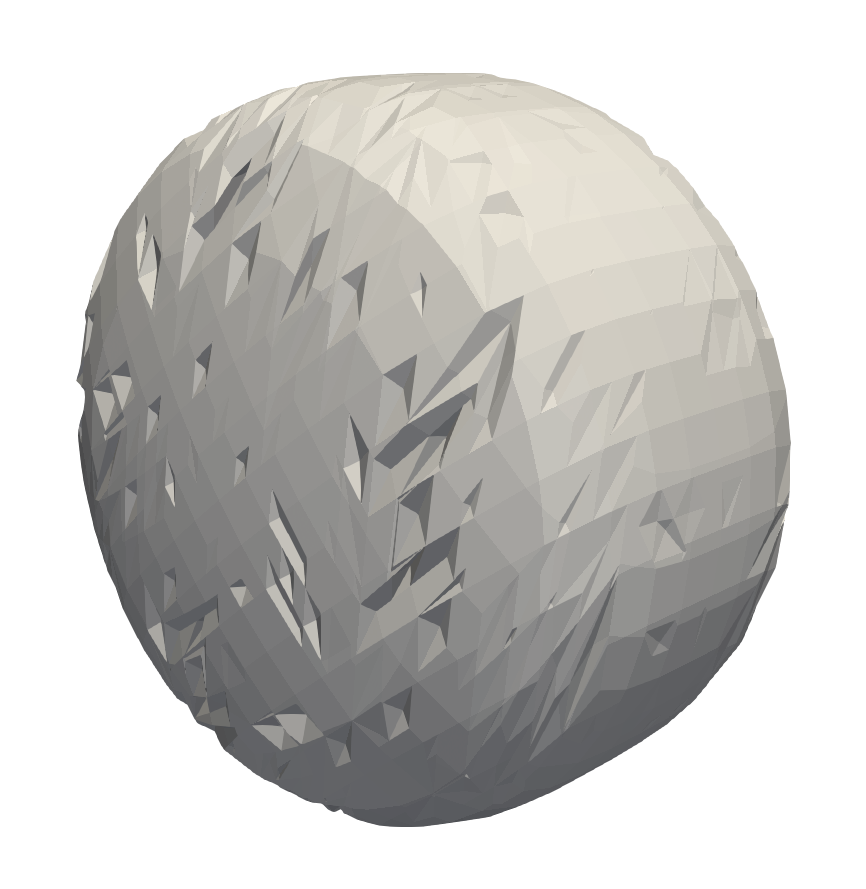}
    \label{sub:OmegahP1}
    }
    \hspace{1em}
    \subfloat[Domaine $\tilde{\Omega}_h$ pour $\tilde{\phi}_h\in P^2$.]{
    \includegraphics[width=0.28\textwidth]{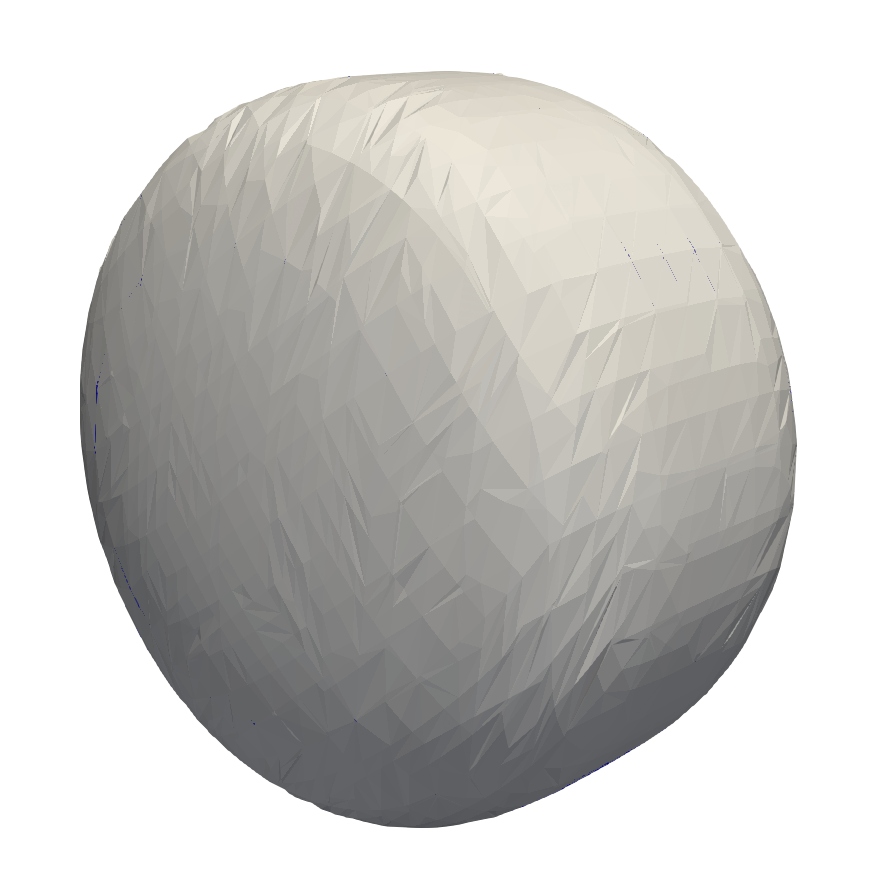}
    \label{sub:OmegahP2}
    }
    \hspace{1em}
    \subfloat[Domaine $\tilde{\Omega}_h$ pour $\tilde{\phi}_h\in P^3$.]{
    \includegraphics[width=0.28\textwidth]{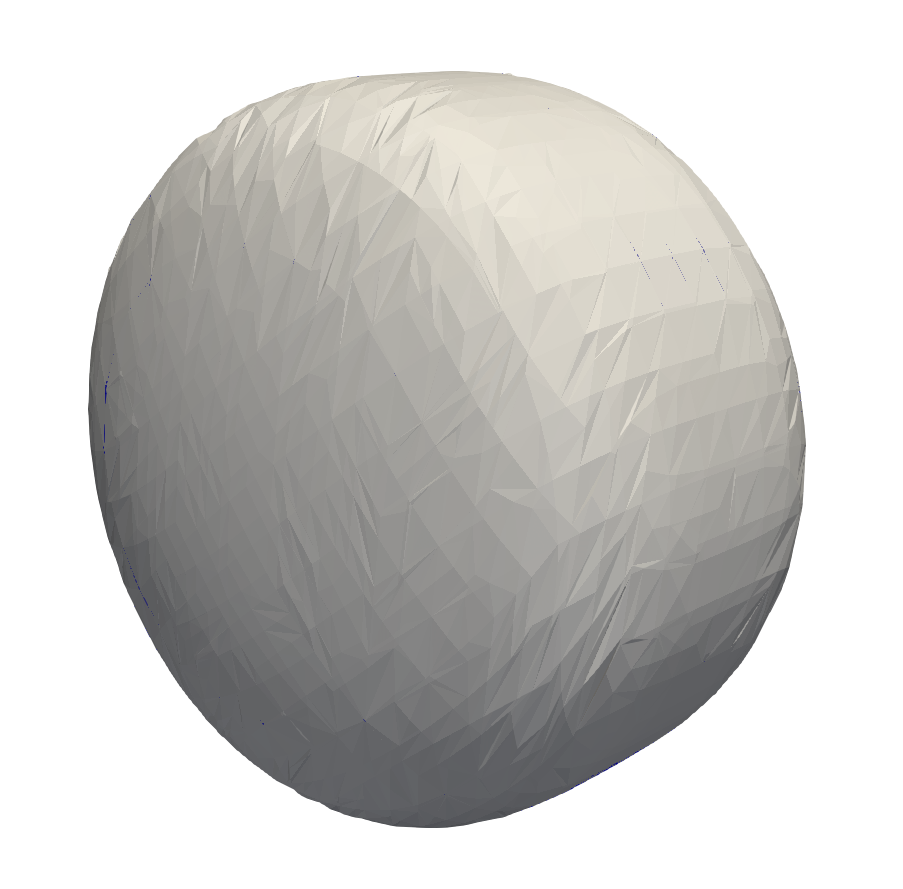}
    \label{sub:OmegahP3}
    }
    \end{center}
  \caption{Domaines $\tilde{\Omega}_h$ obtenus après découpage autour de la surface $\{ \tilde{\phi}_h=0 \}$.}
  \label{fig:ResDecSphere}
\end{figure}

	Comme attendu, dans les cas $P^2$ et $P^3$, on obtient une surface de niveau $\{ \tilde{\phi}_h=0 \}$ plus régulière, pour un nombre de sommets ajoutés du même ordre de grandeur: 3563 pour le $P^1$ contre 3574 pour le $P^2$ et le $P^3$. En revanche, le passage du $P^2$ au $P^3$ n'améliore pas le résultat de façon significative. En pratique, les observations faites pour cet exemple restent vraies dans la plupart des cas. Nous choisissons donc de travailler avec $\tilde{\phi}_h\in P^2$, qui constitue selon nous le meilleur compromis. Par ailleurs, comme on peut le voir sur la figure \ref{fig:MailDecSphere}, le choix du degré d'interpolation ne semble pas avoir d'influence sur la qualité du maillage découpé.
	
\begin{figure}[h]
  \begin{center}
    \subfloat[Maillage $D_h$.]{
    \includegraphics[width=0.23\textwidth]{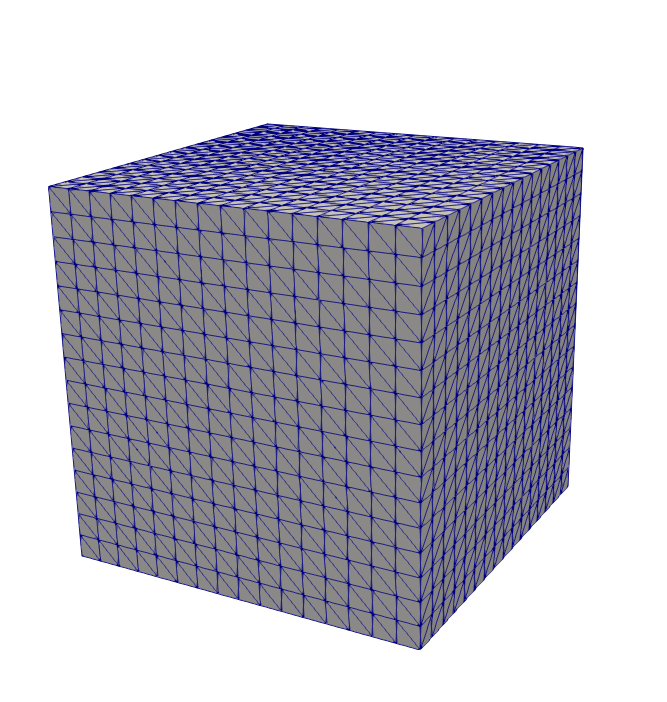}
    \label{sub:MailDh}
    }
    \subfloat[Maillage pour $\tilde{\phi}_h\in P^1$.]{
    \includegraphics[width=0.23\textwidth]{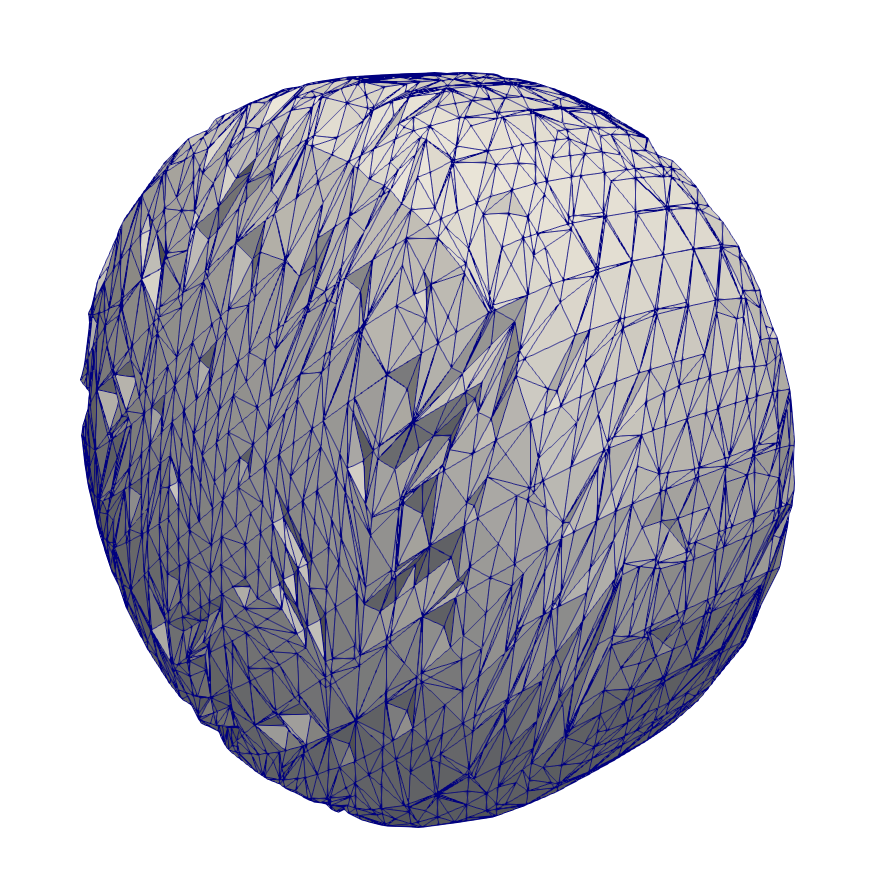}
    \label{sub:MailOmegahP1}
    }
    \subfloat[Maillage pour $\tilde{\phi}_h\in P^2$.]{    \includegraphics[width=0.23\textwidth]{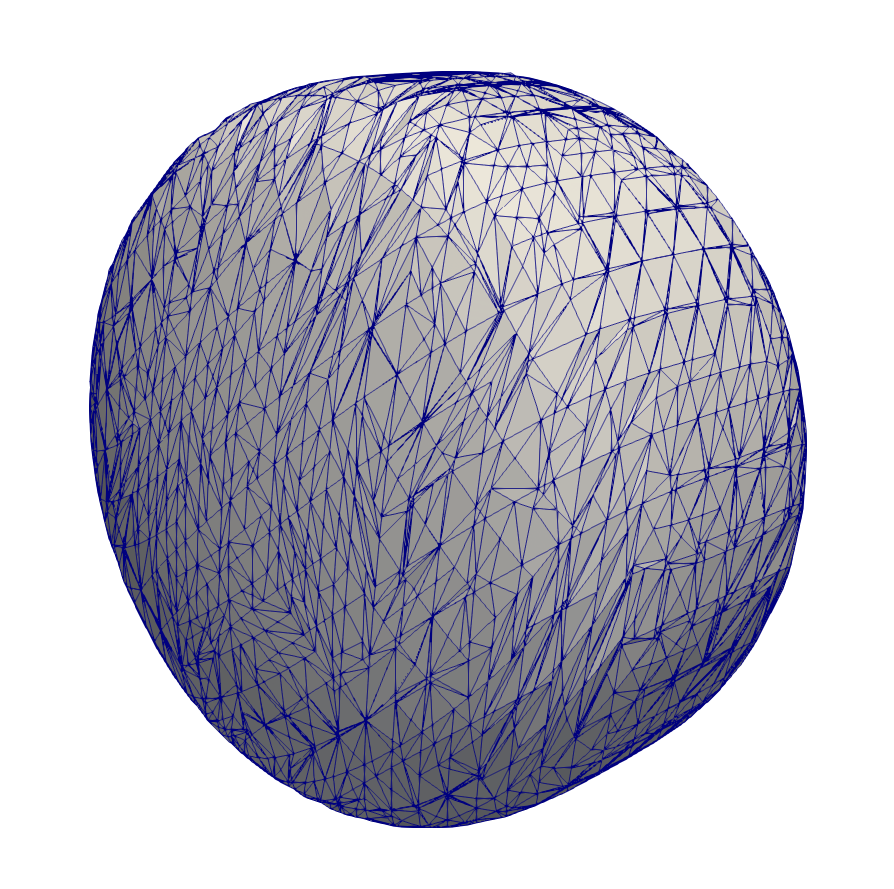}
    \label{sub:MailOmegahP2}
    }
    \subfloat[Maillage pour $\tilde{\phi}_h\in P^3$.]{    \includegraphics[width=0.23\textwidth]{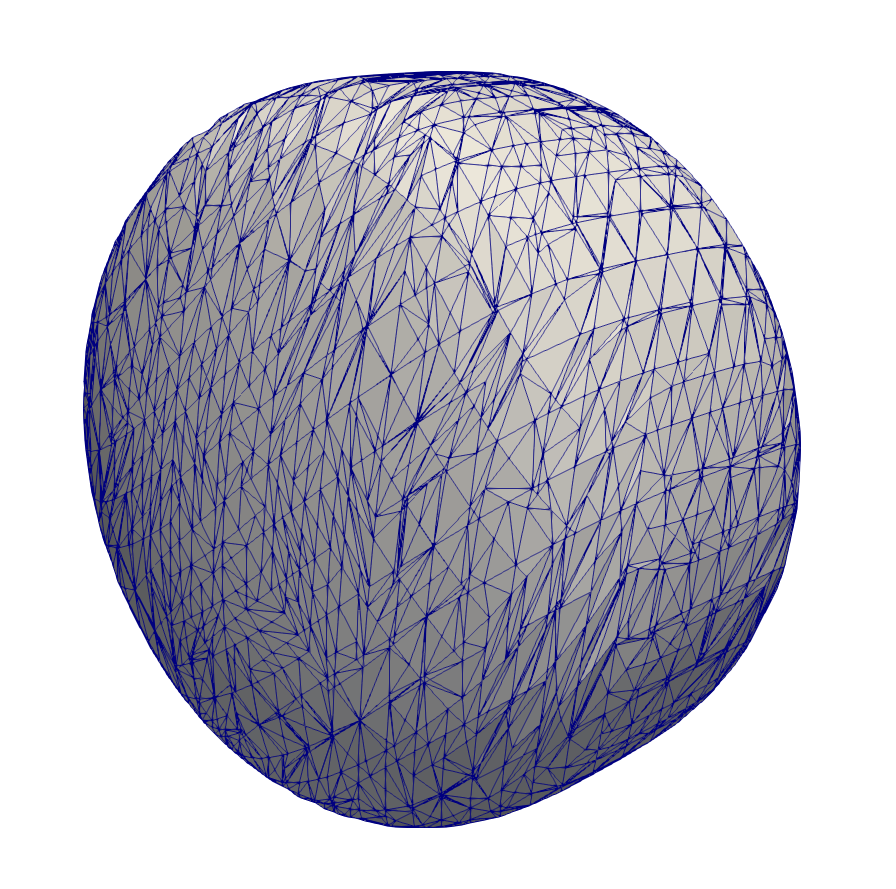}
    \label{sub:MailOmegahP3}
    }
    \end{center}
  \caption{Maillages pour le découpage autour de la surface $\{ \tilde{\phi}_h=0 \}$.}
  \label{fig:MailDecSphere}
\end{figure}

\end{exmpl}

\begin{rmrk}
	On voit sur sur l'exemple précédent que le maillage surfacique de $\{ \tilde{\phi}_h=0 \}$ obtenu a des éléments de mauvaise qualité. Cela peut mettre en péril la résolution lorsque les problèmes considérés sont difficiles à résoudre  (grand nombre d'inconnues, fortes non-linéarités, etc). Dans ce travail, nous nous limitons au cas du contact en élasticité linéaire, sur des maillages de 5000 sommets ou moins, nous ne rencontrons donc pas de problème lors de la résolution sur les maillages découpés. De plus, même dans le cas où un « problème » apparaîssait à une itération donnée, le processus d'optimisation de formes ne serait pas mis en péril grâce à la robustesse du solveur d'optimisation. En revanche, pour étendre notre méthode d'optimisation aux cas du contact avec frottement de Coulomb en élasticité non-linéaire sur des cas-tests indutriels (maillages fins et géométries complexes), il sera nécessaire d'ajouter une étape pour améliorer la qualité du maillage, comme par exemple dans \cite{dapogny2013shape}. L'idée est de contrôler la qualité du maillage obtenu, puis éventuellement de faire appel à une étape supplémentaire d'adaptation. 
\end{rmrk}

\begin{rmrk}
	Notons que nos expériences numériques n'ont pas démontré la nécessité d'une géométrie d'ordre supérieure à 1. Ainsi dans notre procédure, même si on découpe le maillage à partir d'un champ scalaire $P^2$, on conserve un champ de transformation géométrique $P^1$. Autrement dit, la transformation qui génère l'élément courant à partir de l'élément de référence est affine. En particulier, les maillages obtenus n'ont pas d'arêtes courbes, comme on peut le voir sur les figures \ref{fig:ResDecSphere} et \ref{fig:MailDecSphere}.
\end{rmrk}

\paragraph{Forme acceptée ou rejetée.}
À la fin de cette étape, on a une forme maillée $\tilde{\Omega}_h$. Deux cas de figure peuvent alors se produire.
\begin{itemize}[topsep=0pt, parsep=0pt]
	\item Si cette nouvelle forme est bien dans l'ensemble admissble $\pazocal{U}_{ad}$ et si elle est \textit{meilleure} que la précédente au sens où $J(\tilde{\Omega}_h)< J(\Omega_h)$, alors la forme est acceptée. Elle devient le point de départ de l'itération suivante.
	\item Sinon, c'est qu'on est allé \textit{trop loin} dans cette direction de descente, $i.e.\!$ qu'on a advecté $\phi$ dans la direction de $\theta\normalInt$ sur un intervalle $[0,T]$ trop grand. On rejette la forme, puis on retourne au début de cette étape 5, en advectant cette fois sur l'intervalle $[0,T/2]$. Une manière simple d'éviter d'avoir à résoudre à nouveau l'advection sur cet intervalle est de garder en mémoire le champ $\phi$ au temps $T/2$. 
\end{itemize}

\begin{rmrk}
	Cette procédure peut être vue comme une version simplifiée d'une méthode de recherche linéaire. Théoriquement, comme on est certain d'avoir une direction de descente, on devrait avoir $J(\tilde{\Omega}_h)< J(\Omega_h)$ si $T$ est suffisamment petit. À l'inverse, il peut être intéressant en pratique d'avoir de plus grandes valeurs de $T$ car cela permet à l'algorithme de converger en un plus petit nombre d'itérations.
\end{rmrk}

\begin{rmrk}
	En début de processus, afin de permettre des changements de topologie et d'éviter de tomber trop vite dans le voisinage d'un minimum local, nous acceptons des formes $\tilde{\Omega}_h$ telles que $J(\tilde{\Omega}_h)< \beta J(\Omega_h)$, où $\beta>1$.
\end{rmrk}

\section{Résultats numériques}

Dans cette section, nous présentons quelques résultats numériques obtenus pour des cas-tests en deux et trois dimensions. Dans un premier temps, pour tester la validité de la méthode d'optimisation de formes proposée, nous nous concentrons sur des cas académiques en élasticité linéaire sans contact, pour lesquels nous pouvons comparer nos résultats avec ceux de la littérature. Puis dans un second temps, nous nous intéressons à des cas en élasticité linéaire avec contact, dont certains sont tirés de la littérature, et d'autres complètement nouveaux.

L'algorithme proposé a été implémenté dans le logiciel de résolution numérique par éléments finis MEF++ du GIREF (Groupe Interdisciplinaire de Recherche en Éléments Finis de l'Université Laval). Il s'agit d'un code de recherche multiphysique, écrit en C++,  qui est également utilisé pour des applications industrielles. Nous renvoyons le lecteur au site web \url{https://giref.ulaval.ca/} pour plus d'informations et quelques exemples d'application. 

Dans la suite, nous considérons toujours des matériaux linéaires élastiques isotropes (voir chapitre \ref{chap:1.1}), de module d'Young $E=1$ et de coefficient de Poisson $\nu=0.3$. De plus, sauf indication contraire, la fonctionnelle qu'on cherche à minimiser $J$ est définie par une combinaison linéaire du volume $Vol$ et de l'énergie de déformation $C$ (compliance), de sorte que, si $\uu$ désigne le déplacement mécanique:
$$
    J(\Omega) = \alpha_1 C(\Omega) + \alpha_2 Vol(\Omega)= \int_\Omega (\alpha_1\ff \uu(\Omega)+\alpha_2)+ \int_{\Gamma_N} \alpha_1\tauu \uu(\Omega) \:,
$$
où $\alpha_1$ et $\alpha_2$ sont des coefficients réels positifs.

Pour ce qui est du critère de convergence de l'algorithme d'optimisation de formes, nous choisissons le critère classique de l'écart relatif entre   $J(\Omega^l)$ et $J(\Omega^{l+1})$.

\begin{rmrk}
	Comme nous l'avons indiqué au chapitre \ref{chap:1.1}, la méthode de level-set se prête bien aux changements de topologie. Plus précisément, on peut faire disparaître des zones isolées de matière ($\phi<0$) ou de vide ($\phi>0$). Par conséquent, en deux dimensions, si on part d'une forme ayant $n$ « trous » (nombre de composantes connexes du complémentaire), on aura une forme optimale ayant au plus $n$ trous. Autrement dit, si la méthode d'optimisation géométrique à l'aide de level-sets permet de fermer des trous au cours du processus, elle ne permet pas d'en ajouter. De fait, pour ces cas-tests là, il est habituel de partir d'une forme initiale ayant un certain nombre de trous afin d'agrandir l'ensemble admissible, voir \cite{allaire2004structural,allaire2005structural}. En trois dimensions, cela reste vrai, mais la notion de trou s'interprète comme une cavité à l'intérieur de la structure, ce qui n'est en général pas souhaitable pour une pièce mécanique (car difficile à concevoir d'un point de vue pratique). On peut malgré tout vouloir partir d'une forme avec des trous, car cela pourra faciliter d'autres types de changements topologiques, et éventuellement accelérer la convergence de l'algorithme.
\end{rmrk}

\subsection{Élasticité linéaire sans contact}

\subsubsection{Exemples en deux dimensions}

\paragraph{Cantilever.}
Nous commençons par le cas-test le plus fréquemment présenté en optimisation de structures: le cantilever en deux dimensions. Le domaine $D$ est ici le rectangle $[0,2]\times[0,1]$, et les zones de Neumann et Dirichlet sont définies comme indiqué sur la figure \ref{sub:InitCLCanti2d}. Pour ce qui est du problème mécanique, on ne considère pas de force volumique, i.e.$\!$ $\ff=0$, et les contraintes surfaciques sont imposées à $\tauu=(0,-0.01)$. Le maillage $D_h$ est composé en moyenne de 1400 sommets, et les coefficients associés à $J$ sont pris tels que $\alpha_1=10$, $\alpha_2=0.01$. L'écart d'ordre de grandeur entre ces deux coefficients vient du fait que $\tau\cdot\uu$ est de l'ordre de $10^{-4}$ alors que le volume est de l'ordre de l'unité.

\begin{figure}[h]
  \begin{center}
    \subfloat[Géométrie initiale.]{
	\resizebox{0.44\textwidth}{!}{
    \begin{tikzpicture}
    
	\draw[black] (0,0) -- (8,0);
	\draw[black] (8,0) -- (8,1.6);
	\draw[orange, very thick] (8,1.6) -- (8,2.4);
	\draw[black] (8,2.4) -- (8,4);
	\draw[black] (8,4) -- (0,4);
	\draw[blue, very thick] (0,4) -- (0,0);

	\node[white] at (4,-0.3) {$\:$};
	\node[black] at (4,2) {\large{$D$}};
	\node[blue] at (0.7,2) {\large{$\hat{\Gamma}_D$}};
	\node[orange] at (7.3,2) {\large{$\Gamma_N$}};
	\end{tikzpicture}
    \label{sub:InitCLCanti2d}
    }
    }
    \hspace{0.6em}
    \subfloat[Itération 0 pour $h$.]{
    \includegraphics[width=0.45\textwidth]{canti2d_bis_it0.png}
    \label{sub:ResCanti2dIt0}
    }
    \hspace{.5em}
    \subfloat[Itération 24 pour $h$.]{
    \includegraphics[width=0.45\textwidth]{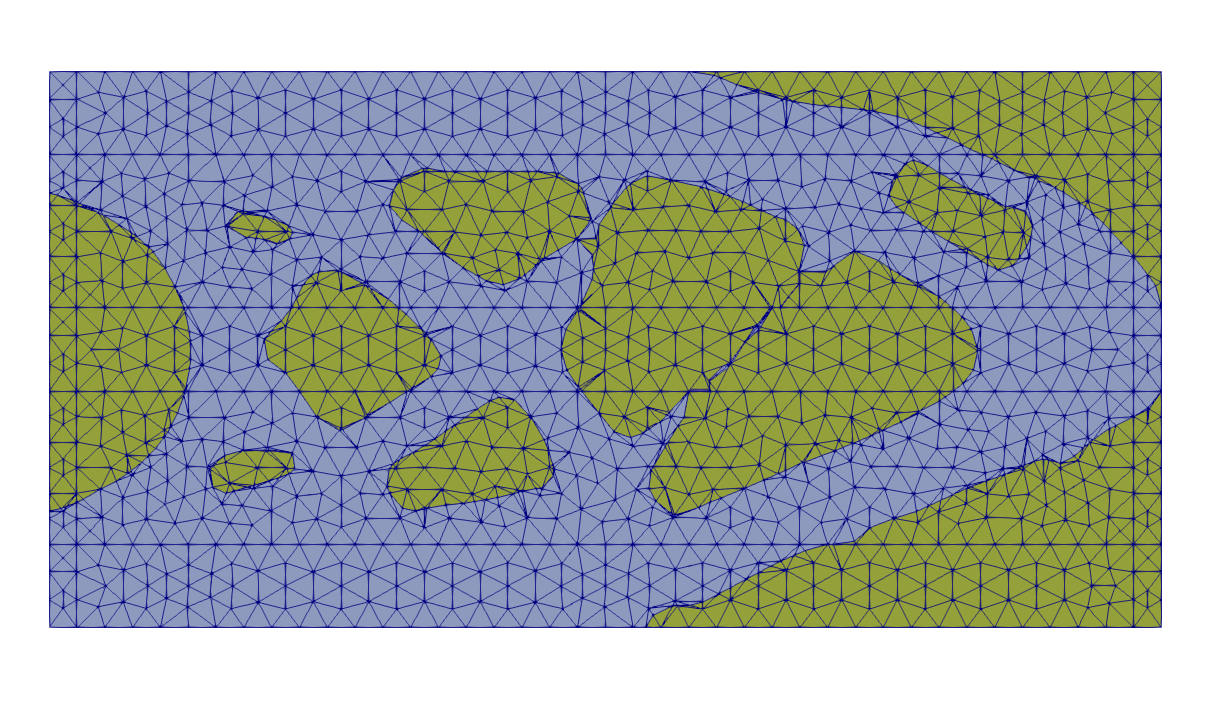}
    \label{sub:ResCanti2dIt24}
    }
    \hspace{.5em}
    \subfloat[Itération 91 pour $h$.]{   
    \includegraphics[width=0.45\textwidth]{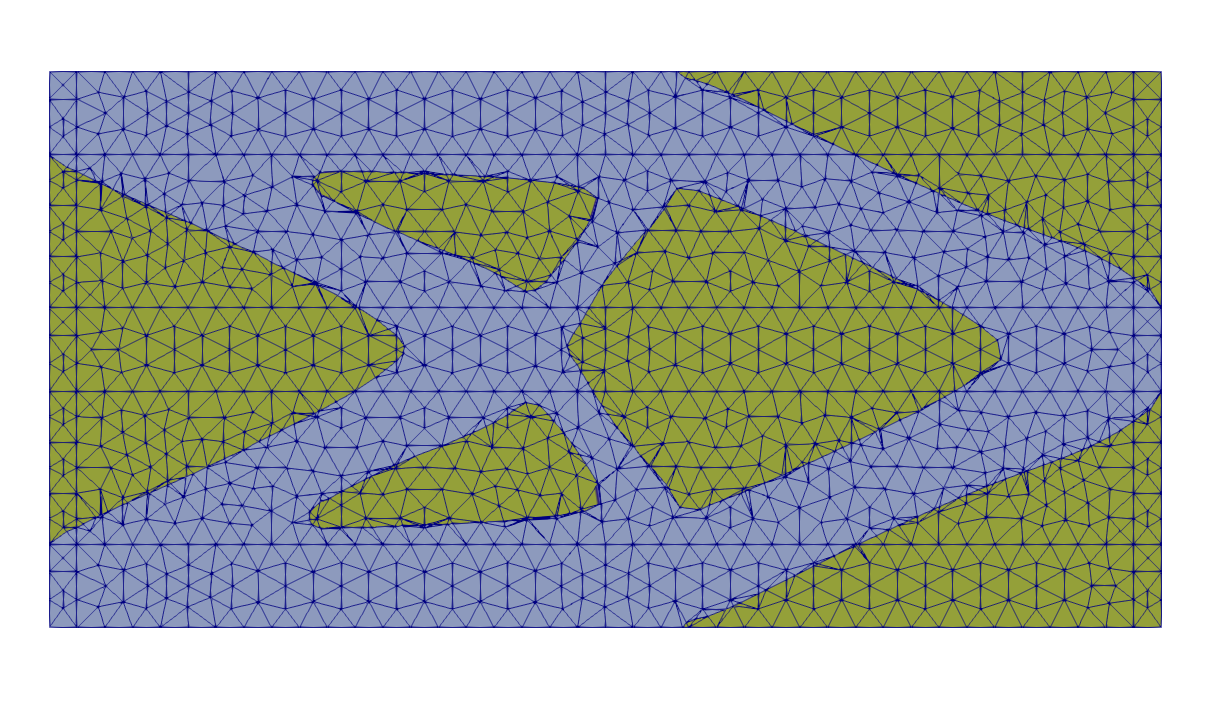}
    \label{sub:ResCanti2dIt91}
    }
    \hspace{.5em}
    \subfloat[Itération 41 pour $h/2$.]{
    \includegraphics[width=0.45\textwidth]{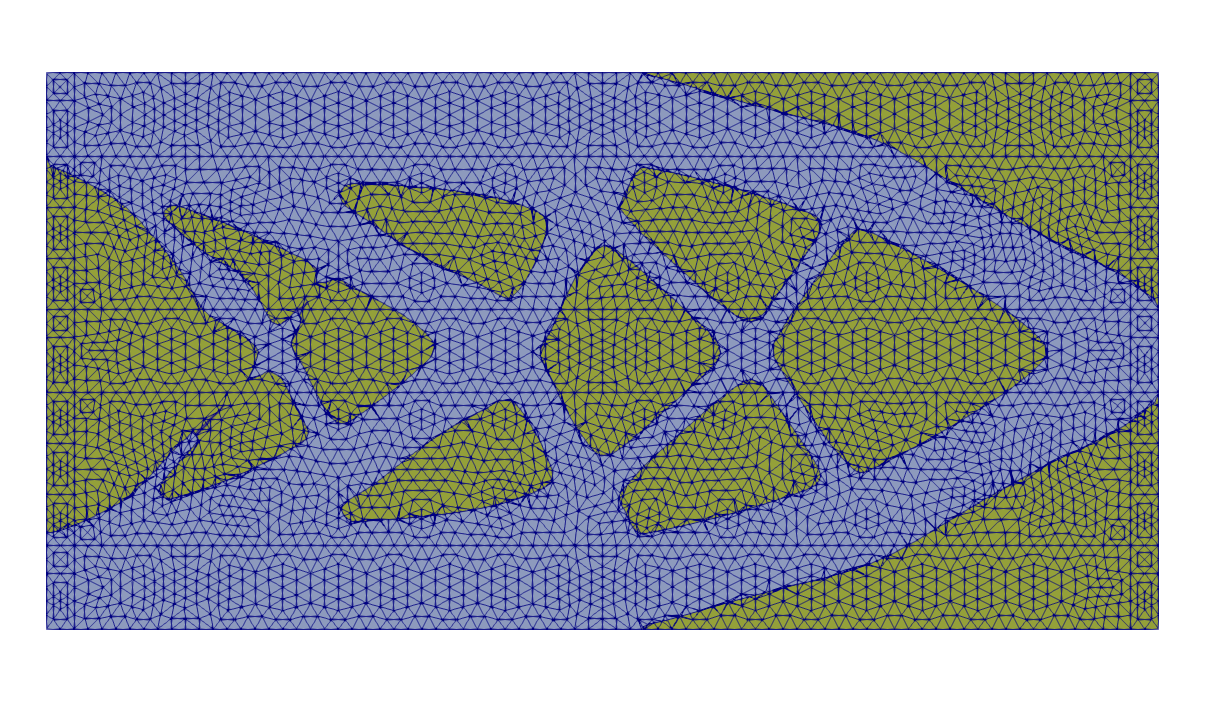}
    \label{sub:ResCanti2dFinIt41}
    }
    \hspace{.5em}
    \subfloat[Itération 190 pour $h/2$.]{   
    \includegraphics[width=0.45\textwidth]{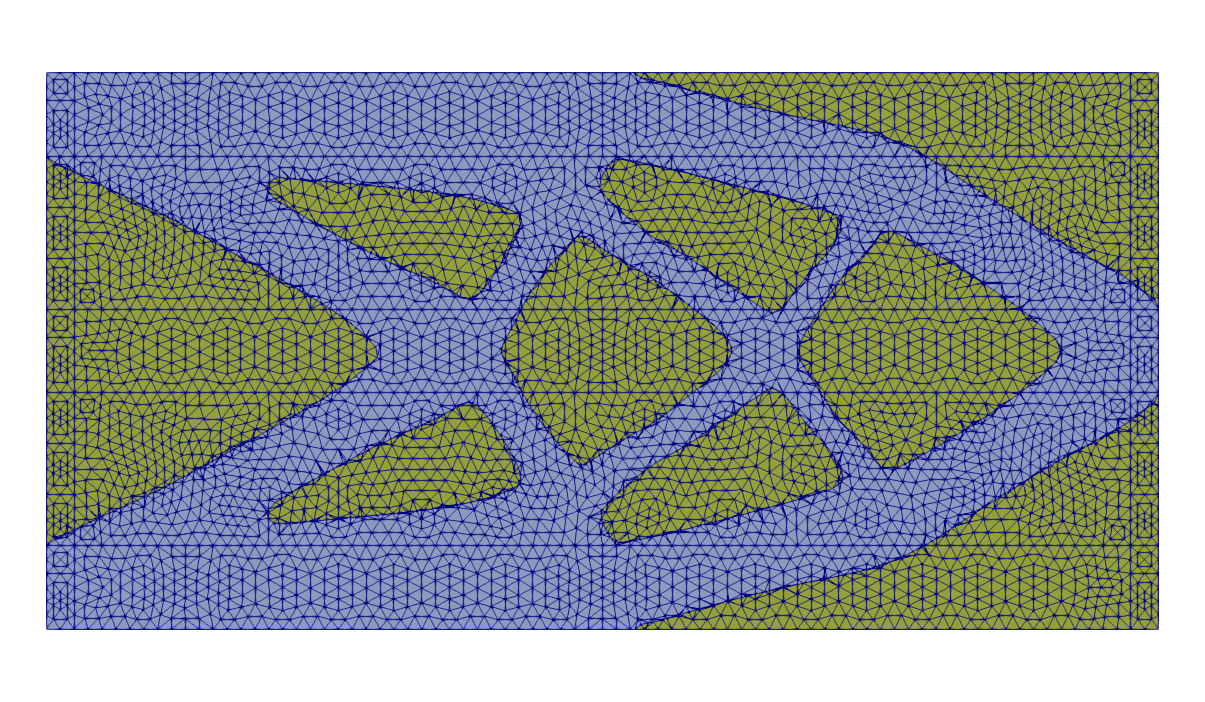}
    \label{sub:ResCanti2dFinIt190}
    }
    \end{center}
  \caption{Géométrie initiale et différentes itérations pour le cantilever 2D ($\Omega_h$ en bleu, $D_h\setminus \Omega_h$ en jaune).}
  \label{fig:ResCanti2d}
\end{figure}

L'algorithme converge en 91 itérations, et effectue au total 92 évaluations de la fonctionnelle $J$, ce qui représente un peu moins de 4 minutes. Nous avons représenté sur les figures \ref{sub:ResCanti2dIt0}, \ref{sub:ResCanti2dIt24}, \ref{sub:ResCanti2dIt91} les formes obtenues aux itérations 0, 24, et 91. On retrouve bien la forme optimale attendue pour ce genre de mise en données, voir par exemple \cite{allaire2004structural} ou \cite{dapogny2013shape}. De plus, comme attendu, la méthode de level-set permet de gérer des changements de topologie: le nombre de composantes connexes de $D_h\setminus \Omega_h$ varie au cours du processus. On a également représenté l'historique de convergence en figure \ref{fig:CvgIterCasTests2d}, à gauche. On remarque des sauts dans les valeurs de $J(\Omega^l)$ correspondant aux itérations $l$ où des changements de topologie surviennent (par exemple $l=24$). 

Pour illustrer l'influence du choix du maillage sur la forme optimale obtenue, nous relançons exactement le même cas test après avoir divisé en deux toutes les arêtes de $D_h$, on a donc remplacé le pas du maillage $h$ par $h/2$. Le maillage comporte cette fois environ 5400 sommets en moyenne, et l'algorithme converge en 190 itérations en 30 minutes, effectuant un total de 205 évaluations de $J$. Nous avons représenté les formes aux itérations 41 et 190 figures \ref{sub:ResCanti2dFinIt41} et \ref{sub:ResCanti2dFinIt190}. Avec ce maillage raffiné, on peut capter des détails géométriques plus petits (e.g.$\!$ des raidisseurs plus fins), et donc avoir une meilleure forme optimale, comme le montre le graphique de la figure \ref{fig:CvgIterCasTests2d}. Cette dépendance au maillage provient du fait que, dans le cas de l'élasticité linéaire, lorsqu'on cherche à minimiser la compliance, la solution optimale prend la forme d'une microstructure (voir \cite{allaire2012shape}).

\paragraph{Mât électrique.}
Nous passons maintenant à un autre exemple dans la littérature: le mât électrique optimal, voir \cite{dapogny2013shape}. Le domaine $D$ est en forme de T de 120 de hauteur. La barre du T a une largeur de 80, et le pied a une largeur de 40. La définition des zones géométriques est donnée figure \ref{sub:InitCLMat2d}. On prend encore ici $\ff=0$ et $\tauu=(0,-1)$. $D$ comporte en moyenne 3000 sommets, et on a pris $\alpha_1=\alpha_2=1$.

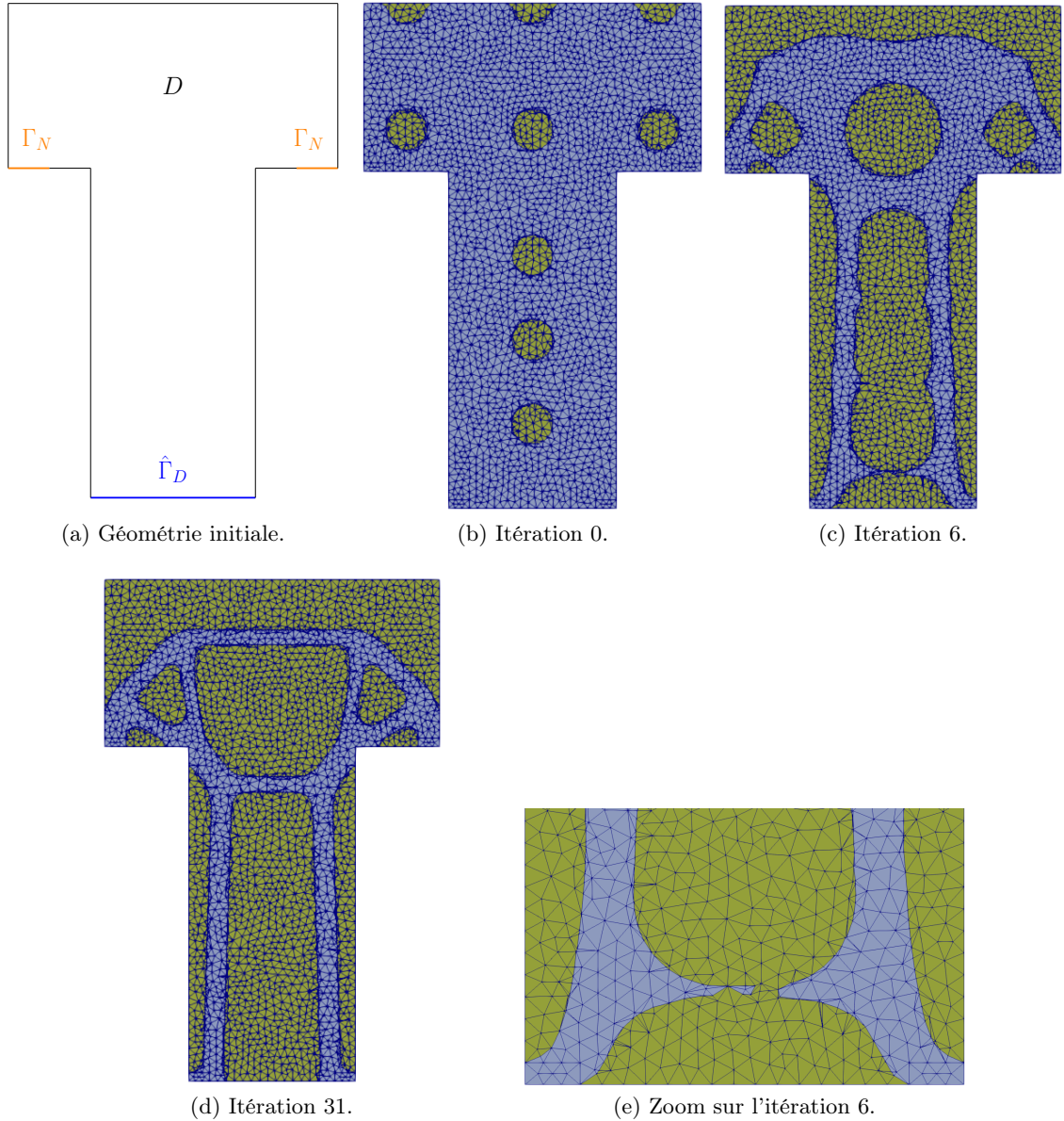
\begin{figure}[h]
  \begin{center}
    \subfloat[Géométrie initiale.]{
	\resizebox{0.31\textwidth}{!}{
    \begin{tikzpicture}
    
	\draw[blue, very thick] (2,0) -- (6,0);
	\draw[black] (6,0) -- (6,8);
	\draw[black] (6,8) -- (7,8);
	\draw[orange, very thick] (7,8) -- (8,8);
	\draw[black] (8,8) -- (8,12);
	\draw[black] (8,12) -- (0,12);
	\draw[black] (0,12) -- (0,8);
	\draw[orange, very thick] (0,8) -- (1,8);
	\draw[black] (1,8) -- (2,8);
	\draw[black] (2,8) -- (2,0);

	\node[white] at (4,-0.2) {$\:$};
	\node[black] at (4,10) {\LARGE{$D$}};
	\node[blue] at (4,0.7) {\LARGE{$\hat{\Gamma}_D$}};
	\node[orange] at (7.3,8.7) {\LARGE{$\Gamma_N$}};
	\node[orange] at (0.7,8.7) {\LARGE{$\Gamma_N$}};
	\end{tikzpicture}
    }
    \label{sub:InitCLMat2d}
    }
    \subfloat[Itération 0.]{
    \includegraphics[width=0.31\textwidth]{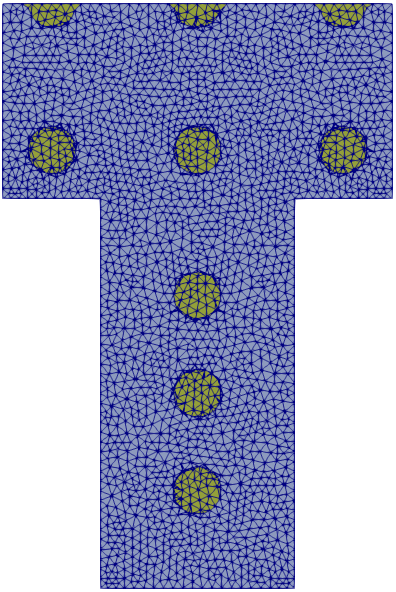}
    \label{sub:ResMat2dIt0}
    }
    \subfloat[Itération 6.]{
    \includegraphics[width=0.31\textwidth]{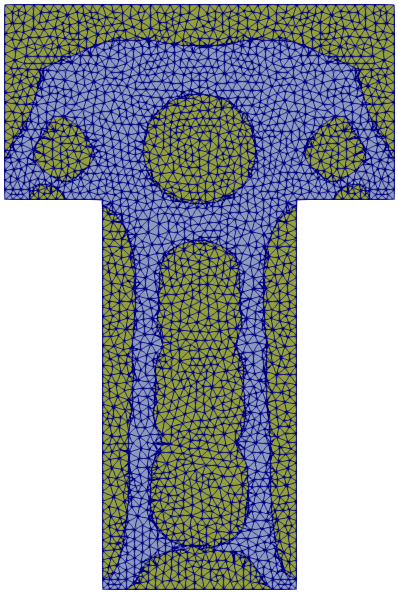}
    \label{sub:ResMat2dIt10}
    }
    \hspace{2em}
    \subfloat[Itération 31.]{   
    \includegraphics[width=0.31\textwidth]{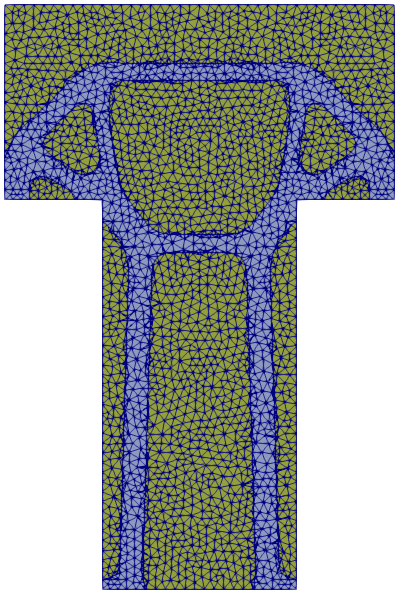}
    \label{sub:ResMat2dIt40}
    }
    \hspace{2em}
    \subfloat[Zoom sur l'itération 6.]{   
    \includegraphics[width=0.4\textwidth]{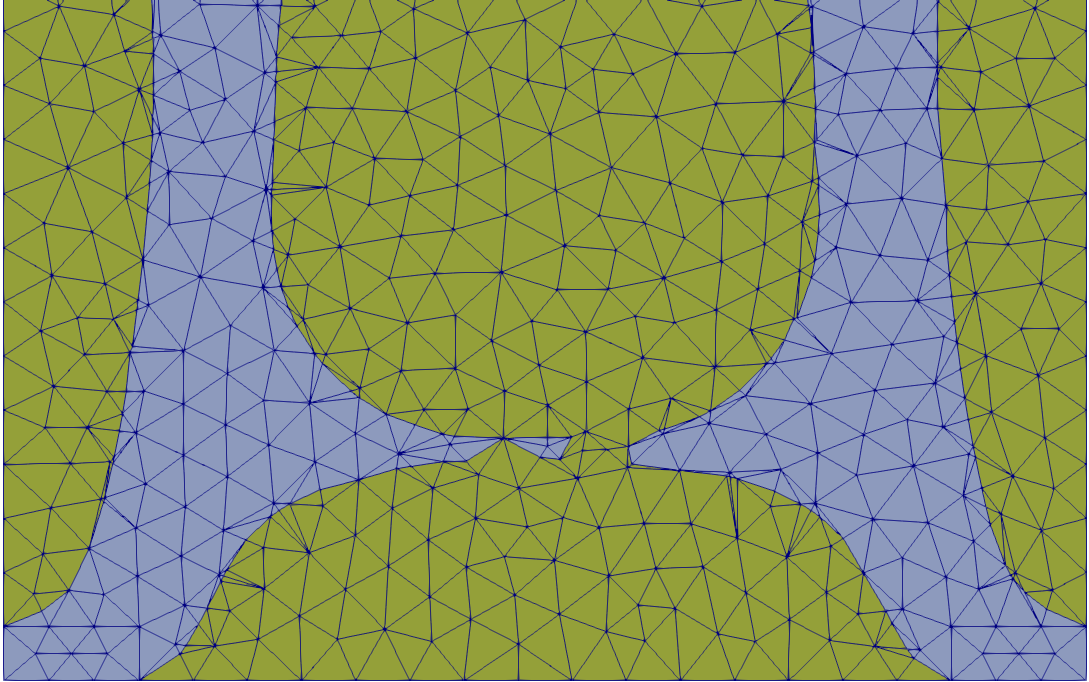}
    \label{sub:ZoomMat2dIt6}
    }
    \end{center}
  \caption{Différentes itérations pour le mât 2D ($\Omega_h$ en bleu, $D_h\setminus \Omega_h$ en jaune).}
  \label{fig:ResMat2d}
\end{figure}

L'algorithme converge en 31 itérations, avec 35 évaluations de $J$, en environ 2 minutes. Trois itérations (0, 6 et 31) sont données figure \ref{fig:ResMat2d}. Là encore, le résultat obtenu est semblable à ce qui est présenté dans la littérature: on retrouve les deux pieds verticaux, ainsi que les deux barres horizontales qui stabilisent la structure. La suite des valeurs de $J$ est montrée figure \ref{fig:CvgIterCasTests2d}. Ici, comme les changements de topologie ont lieu en tout début de processus, lorsque la fonctionnelle décroît le plus vite, on ne voit aucun saut dans les valeurs de $J$: la suite reste strictement monotone du début à la fin. 

Nous proposons également un zoom sur la zone de $D$ où un changement de topologie se produit à l'itération 6, voir \ref{sub:ZoomMat2dIt6}. Grâce à la level-set, on peut gérer le phénomène de façon tout à fait transparente, tout en disposant d'un maillage de $\Omega_h$ à chaque itération. On note par ailleurs que le découpage en $P^2$ donne un bord très régulier, qui respecte bien la courbure, même si le maillage n'est pas courbe.

\begin{figure}
\begin{center}
	\resizebox{!}{0.34\textwidth}{
	\begin{tikzpicture}
		\begin{axis}[
    		xlabel={Nombre d'itérations $l$},
    		xmin=0, xmax=200,
    		]
    		\addplot[color=red,mark=none] table {res_canti2d_gros.txt};
    		\addplot[color=blue,mark=none] table {res_canti2d_fin.txt};
   			\legend{$J(\Omega^l)$ pour $h$, $J(\Omega^l)$ pour $h/2$}
   		\end{axis}
	\end{tikzpicture}
	}
	\hspace{3em}
	\resizebox{!}{0.33\textwidth}{
	\begin{tikzpicture}
		\begin{axis}[
    		xlabel={Nombre d'itérations $l$},
    		xmin=0, xmax=32,
    		ymin=3000, ymax=6500,
    		ytick={3000,3500,4000,4500,5000,5500,6000,6500}
    		]
    		\addplot[color=red,mark=none] table {res_mat2d.txt};
   			\legend{$J(\Omega^l)$}
   		\end{axis}
	\end{tikzpicture}
	}
\end{center}
\caption{Convergence des itérations pour le cantilever (à gauche) et pour le mât électrique (à droite).}
\label{fig:CvgIterCasTests2d}
\end{figure}
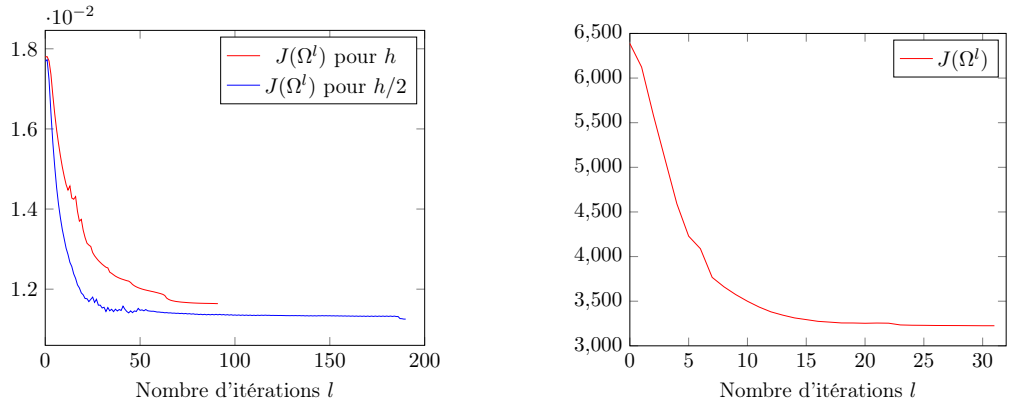

\subsubsection{Exemples en trois dimensions}

\paragraph{Cantilever.} Ce cas-test est la version 3d du celui présenté précédemment. Les dimensions de la boîte $D$ sont $5\times 3\times 2.4$. Et le maillage comporte en moyenne 4500 sommets. On a représenté le maillage initial et la topologie initiale avec les zones de conditions aux limites figures \ref{sub:MailCanti3d} et \ref{sub:ResCanti3dIt0}. Pour illustrer la topologie initiale, on a affiché $\complement \Omega_h$ plutôt que $\Omega_h$. Les forces extérieures sont telles que $\ff=0$ et $\tauu=(0,0,-0.01)$, et les coefficients de la fonctionnelle sont fixés à $\alpha_1=20$, $\alpha_2=0.01$. L'algorithme converge en 137 itérations, après avoir évalué la fonctionnelle 183 fois, en environ 1h50. L'historique de convergence est tracé figure \ref{sub:CvgIterCanti3d}. Comme le montrent les résultats présentés figure \ref{fig:ResCanti3d}, on obtient une forme optimale (figure \ref{sub:ResCanti3dIt137}) similaire aux résultats de la littérature. En 3d, comme en 2d, la méthode de level-set permet de gérer automatiquement les changements de topologie au cours du processus, cf figure \ref{sub:ResCanti3dIt14}. 

\begin{figure}[h]
  \begin{center}
    \subfloat[Maillage $D_h$.]{
    \includegraphics[width=0.31\textwidth]{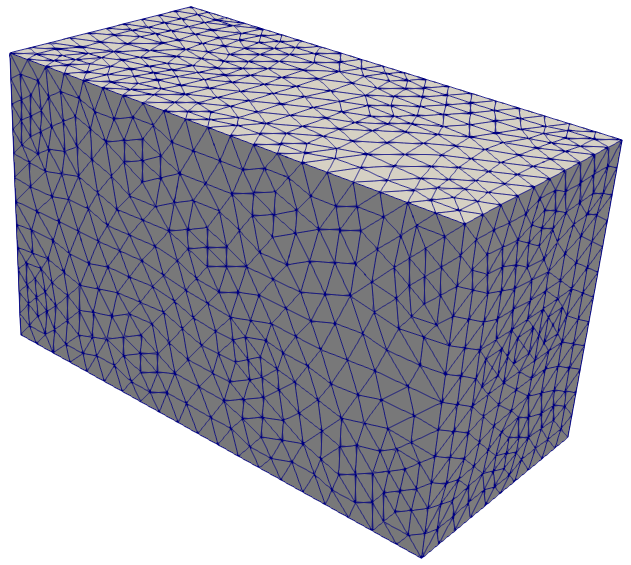}
    \label{sub:MailCanti3d}
    }
    \hspace{0.1em}
    \subfloat[Itération 0.]{
    \includegraphics[width=0.31\textwidth]{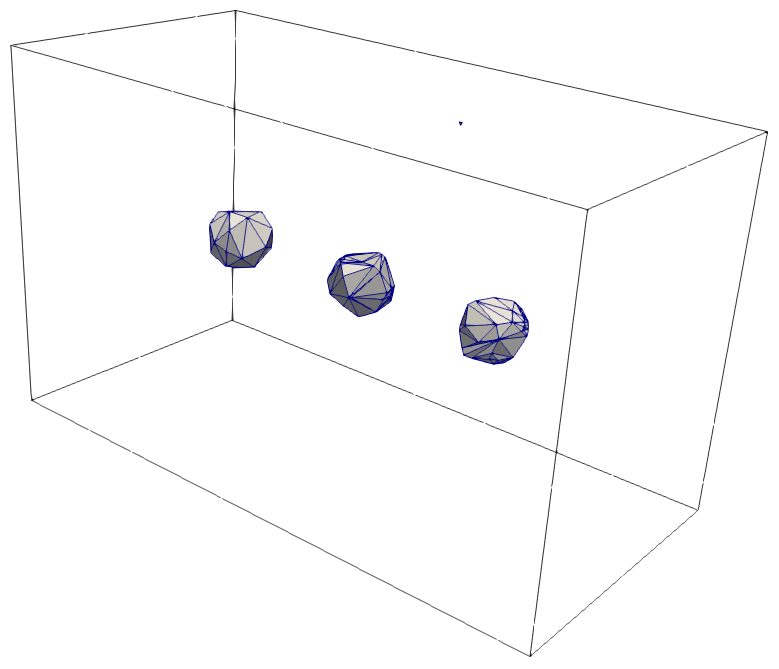}
    \label{sub:ResCanti3dIt0}
    }
    \hspace{.1em}
    \subfloat[Itération 14.]{
    \includegraphics[width=0.31\textwidth]{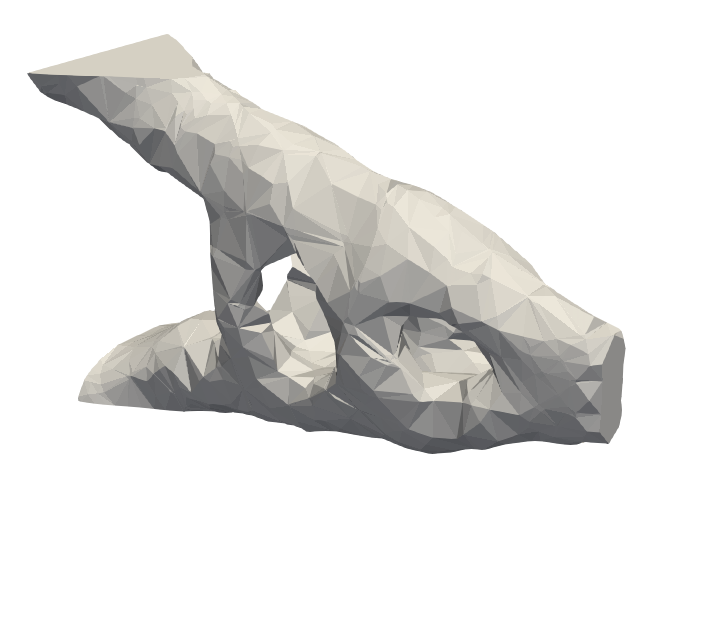}
    \label{sub:ResCanti3dIt14}
    }
    \hspace{.1em}
    \subfloat[Itération 137.]{   
    \includegraphics[width=0.31\textwidth]{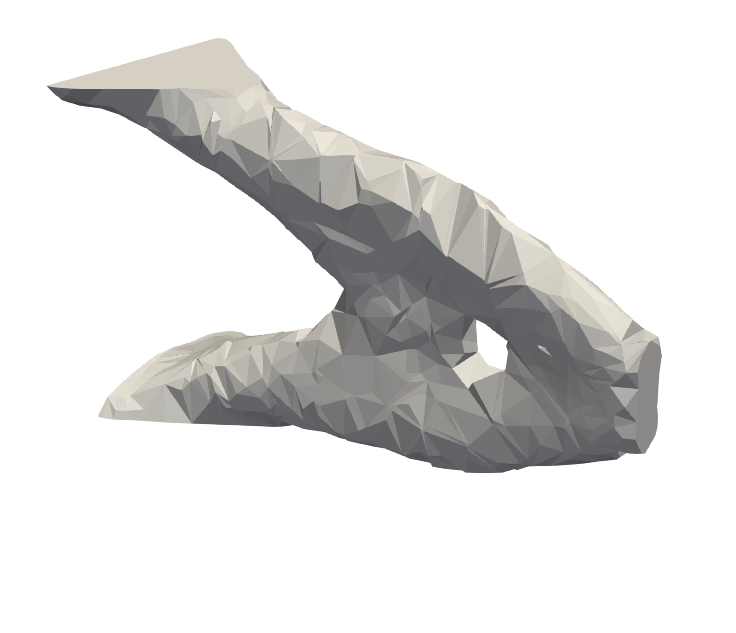}
    \label{sub:ResCanti3dIt137}
    }
    \hspace{1em}
    \subfloat[Historique de convergence.]{
    \resizebox{!}{0.31\textwidth}{
	\begin{tikzpicture}
		\begin{axis}[
    		xlabel={Nombre d'itérations $l$},
    		xmin=0, xmax=145,
    		]
    		\addplot[color=red,mark=none] table {res_canti3d.txt};
   			\legend{$J(\Omega^l)$}
   		\end{axis}
	\end{tikzpicture}
	}
	\label{sub:CvgIterCanti3d}
	}
    \end{center}
  \caption{Géométrie initiale, différentes itérations et historique de convergence pour le cantilever 3D.}
  \label{fig:ResCanti3d}
\end{figure}

\paragraph{Domaine en L.}
Nous terminons les cas-tests en élasticité linéaire par le domaine en forme de L, parfois appelé poutre en L, voir par exemple \cite{dapogny2013shape}. Les grands côtés du L sont de longueur 2, les petits 1.2, et la profondeur est de 0.8. Le maillage initial est donné figure \ref{sub:MailDomL3d}. La topologie de la configuration initiale, ainsi que les zones pour les conditions aux limites, sont représentées figure \ref{sub:InitCLDomL3d}. Ici, $\ff=0$ et $\tauu=(0,0,-1)$, et nous prenons cette fois $\alpha_1=4$, $\alpha_2=2$.
\begin{figure}[h]
  \begin{center}
    \subfloat[Maillage $D_h$.]{
    \includegraphics[width=0.31\textwidth]{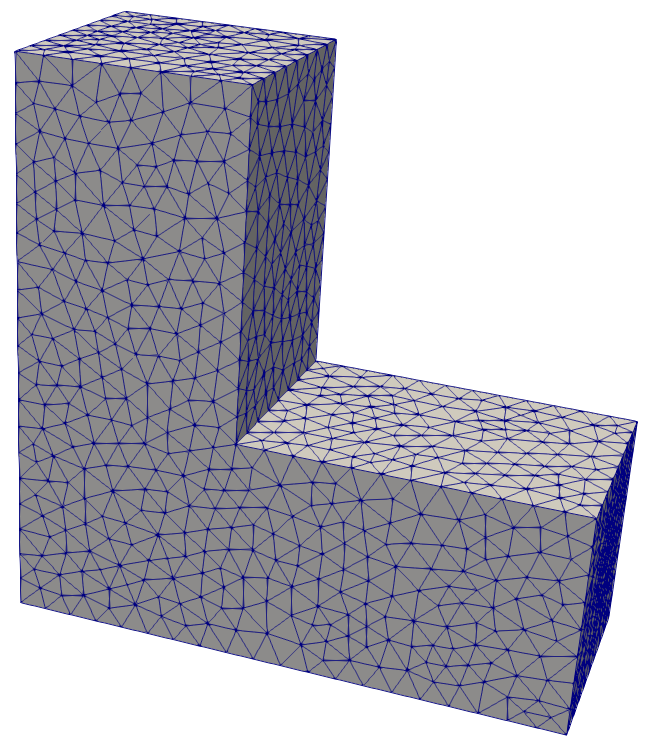}
    \label{sub:MailDomL3d}
    }
    \hspace{0.1em}
    \subfloat[Itération 0.]{
	\resizebox{0.31\textwidth}{!}{
	\begin{tikzpicture}
    	\node[anchor=south west,inner sep=0] at (0,0) {\includegraphics[width=\textwidth]{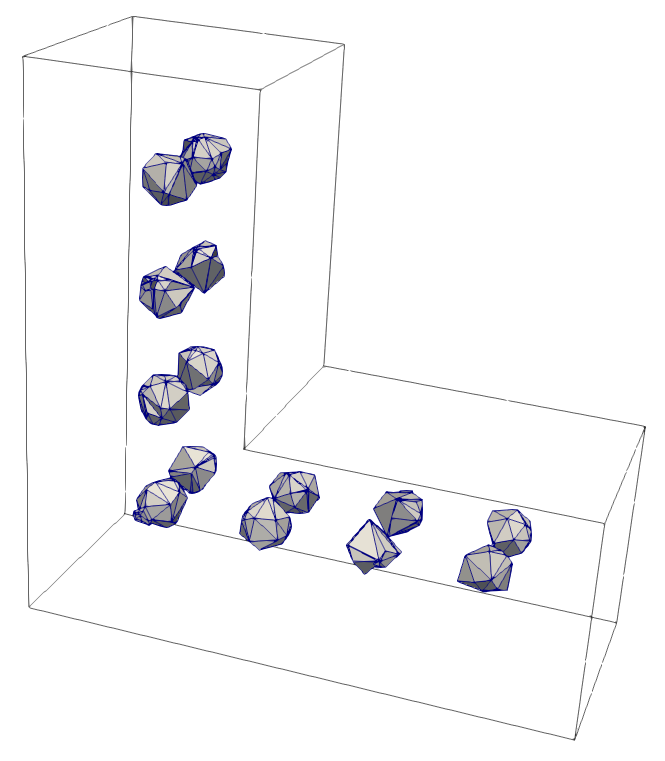}};
    	\draw[fill=blue, opacity=0.4]  (0.5,16.5) -- (6.15,15.7) -- (8,16.8) -- (3,17.45) -- cycle;
    	\node[blue] at (4.5,17.5) {\Huge{$\hat{\Gamma}_D$}};
    	\draw[fill=orange, opacity=0.4]  (14,2.6) -- (14.3,3.5) -- (14.6,5.5) -- (14.25,4.5) -- cycle;
    	\node[orange] at (13.5,7) {\Huge{$\Gamma_N$}};
    \end{tikzpicture}
    \label{sub:InitCLDomL3d}
    }
    }
    \hspace{.1em}
    \subfloat[Itération 8.]{
    \includegraphics[width=0.31\textwidth]{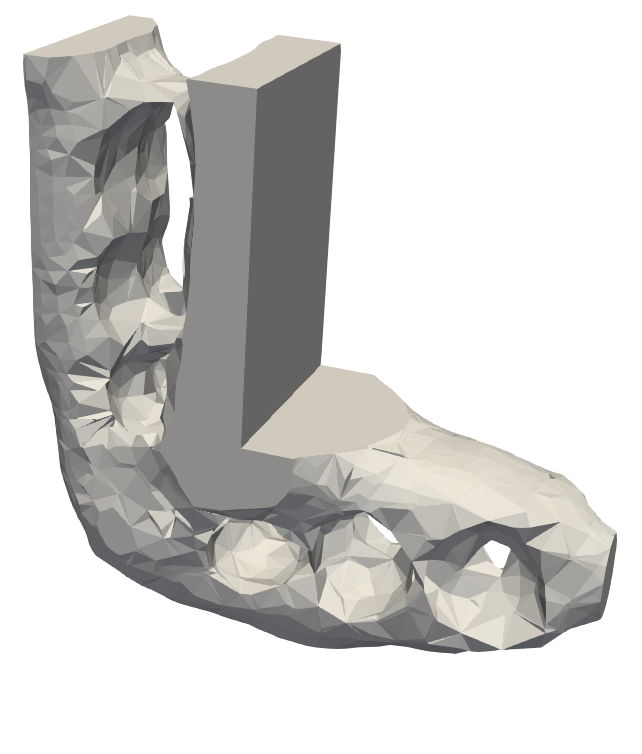}
    \label{sub:ResDomL3dIt8}
    }
    \hspace{.1em}
    \subfloat[Itération 45.]{   
    \includegraphics[width=0.31\textwidth]{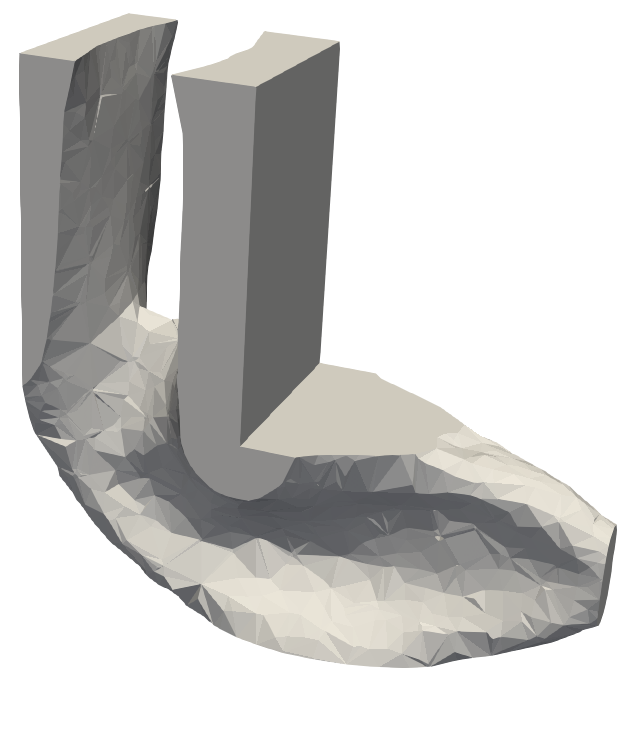}
    \label{sub:ResDomL3dIt45}
    }
    \hspace{1em}
    \subfloat[Historique de convergence.]{
    \resizebox{!}{0.35\textwidth}{
	\begin{tikzpicture}
		\begin{axis}[
    		xlabel={Nombre d'itérations $l$},
    		xmin=0, xmax=46,
    		ymin=2.2, ymax=4.5
    		]
    		\addplot[color=red,mark=none] table {res_domaineL3d.txt};
   			\legend{$J(\Omega^l)$}
   		\end{axis}
	\end{tikzpicture}
	}
	\label{sub:CvgIterDomL3d}
	}
    \end{center}
  \caption{Géométrie initiale, différentes itérations et historique de convergence pour le domaine en L 3D.}
  \label{fig:ResDomL3d}
\end{figure}
On obtient une forme convergée après 45 itérations, pour 72 évaluations de la fonctionnelle. Le processus prend environ 45 minutes. On a représenté les itérations 8 et 45 (figures \ref{sub:ResDomL3dIt8} et \ref{sub:ResDomL3dIt45}), ainsi que la courbe de convergence de $J$. Malgré les « trous » dans la configuration initiale, on a une forme optimale pleine. Comme on s'y attend, l'algorithme cherche à utiliser la zone de Dirichlet pour gagner en rigidité, tout en renforçant le creux du L (zone où les contraintes sont les plus élevées).

\begin{rmrk}
	En comparant le nombre d'itérations de l'algorithme et le nombre d'évaluations de $J$, on voit dans ces deux exemples qu'environ une forme sur trois est rejetée. Ceci suggère que la valeur du temps final $T$ choisie pour l'advection de $\phi$ est un peu trop grande.
\end{rmrk}

\begin{rmrk}
	Pour ces deux cas 3D, nous avons choisi des formes intiales avec des trous. En pratique, cela n'a pas d'influence sur la forme optimale obtenue (même résultat qu'avec une initialisation sans trou). Cependant, ces choix de formes initiales permettent de voir apparaître plus vite des changements de topologie (voir figures \ref{sub:ResCanti3dIt14} et \ref{sub:ResDomL3dIt8}).
\end{rmrk}

\begin{rmrk}
	Notons enfin qu'à chaque fois, nous avons utilisé des éléments finis $P^2$ pour le déplacement $\uu$. Cependant, dans les cas où les géométries ont des coins entrants (mât en forme de T et domaine en L), on sait que les solutions présentent des singularités. Il n'est donc pas forcément intéressant de monter en ordre dans la discrétisation car lorsque la solution est singulière, on n'a plus de contrôle sur l'erreur de discrétisation. En revanche, en pratique, on peut quand même espérer qu'utiliser du $P^2$ plutôt que du $P^1$ permet de gagner en précision loin des singularités. 
\end{rmrk}

\subsection{Élasticité linéaire avec contact}

Dans cette section, nous nous intéressons à deux types de configurations. Premièrement, afin de valider notre méthode, nous nous plaçons dans le même contexte que \cite{maury2016shape} ou \cite{stromberg2010topology}, c'est-à-dire que nous fixons dès le départ un ensemble de zones $\hat{\Gamma}_C\subset \partial D$ où des phénomènes de contact peuvent survenir. Au cours du processus, on définira la zone de contact potentiel $\Gamma_C$ associée à une forme $\Omega$ par $\Gamma_C:=\hat{\Gamma}_C\cap \partial\Omega$. Cela signifie que le traitement de $\Gamma_C$ est le même que celui de $\Gamma_D$: si la forme « s'accroche » à une partie de $\hat{\Gamma}_C$, alors il y aura contact à cet endroit-là, sinon le bord est libre de contrainte. En particulier dans cet exemple, on ne calcule pas de dérivées de forme sur $\Gamma_C$.

Dans un deuxième temps, nous considérons des cas où la zone de contact n'est pas connue a priori. On définit seulement un objet rigide $\Omega_{rig}$ avec lequel notre corps déformable est susceptible de rentrer en contact. Puis, pour une forme $\Omega$ donnée, pour l'ensemble des points de $\partial\Omega \setminus (\Gamma_D\cup \Gamma_N)$, c'est le calcul du gap qui déterminera les points entrant en contact avec l'objet rigide. Cette approche plus générale permet d'utiliser le calcul des dérivées de forme sur $\Gamma_C$, et d'optimiser cette partie du bord sans lui imposer de contrainte a priori. Comme il n'existe à notre connaissance pas de cas-tests 3d dans la littérature pour ce genre de situations dans le cas d'un objet rigide non plan, nous en proposons un qui est une généralisation d'un cas 2d tiré de \cite{fancello1995numerical}.

\subsubsection{Exemple où la zone de contact est connue a priori}

\paragraph{Anchrage 2d.} Ce premier cas-test est tiré de \cite{maury2016shape}. Le domaine $D$ est le carré unité $[0,1]\times[0,1]$, représenté par un maillage composé de 2600 sommets en moyenne. On a imposé deux zones disjointes de contact potentiel à gauche et à droite, on impose une condition de Neumann au milieu en bas, voir figure \ref{sub:InitCLAnchrage2d}. Pour ce qui est des forces extérieures, nous prenons $\ff=0$ et $\tauu=(0,-0.01)$. Les coefficients de $J$ vérifient $\alpha_1=25$, $\alpha_2=0.01$. Dans le cas du contact pénalisé, nous avons fixé $\varepsilon = 10^8$, et pour la Lagrangien augmenté,  $\gamma_1^k=\gamma_2^k=1000$ pour tout $k$. Et pour le cas frottant, nous avons pris un seuil de Tresca $s=10^{-2}$ (ordre de grandeur de la contrainte normale $\sigmaa_{\normalInt\!\normalExt}$), et un coefficient de frottement $\mathfrak{F}=0.2$. Les formes optimales obtenues dans chaque cas sont représentées figure \ref{fig:ResAnchrage2d}.

\begin{figure}[h]
  \begin{center}
    \subfloat[Géométrie intiale.]{
	\resizebox{0.29\textwidth}{!}{
    \begin{tikzpicture}
	\draw[black] (0,0) -- (0.5,0);
	\draw[dartmouthgreen, very thick] (0.25,0) -- (0.75,0);
	\draw[black] (0.75,0) -- (1.75,0);
	\draw[orange, very thick] (1.75,0) -- (2.25,0);
	\draw[black] (2.25,0) -- (3.25,0);
	\draw[dartmouthgreen, very thick] (3.25,0) -- (3.75,0);
	\draw[black] (3.75,0) -- (4,0);
	\draw[black] (4,0) -- (4,4);
	\draw[black] (4,4) -- (0,4);
	\draw[black] (0,4) -- (0,0);

	\node[white] at (2,-0.25) {$\:$};
	\node[white] at (4.2,2) {$\:$};
	\node[black] at (2,2) {$D$};
	\node[orange] at (2,0.4) {$\Gamma_N$};
	\node[dartmouthgreen] at (0.5,0.4) {$\Gamma_C$};
	\node[dartmouthgreen] at (3.5,0.4) {$\Gamma_C$};
	\end{tikzpicture}
    \label{sub:InitCLAnchrage2d}
    }
    }
   \hspace{.3em}
    \subfloat[Itération 0.]{
    \includegraphics[width=0.3\textwidth]{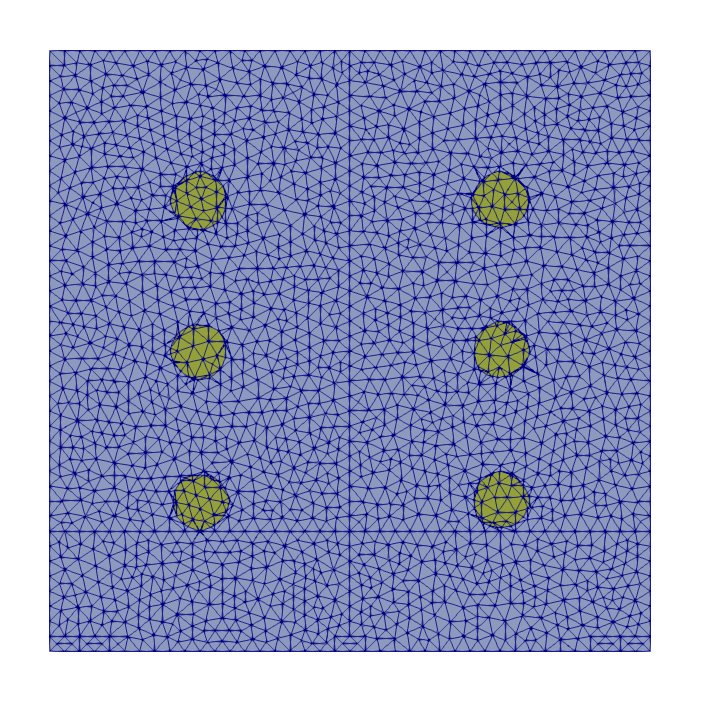}
    \label{sub:ResAnchrage2dInit}
    }
   \hspace{.3em}
    \subfloat[Itération finale contact glissant (pénalisation).]{
    \includegraphics[width=0.3\textwidth]{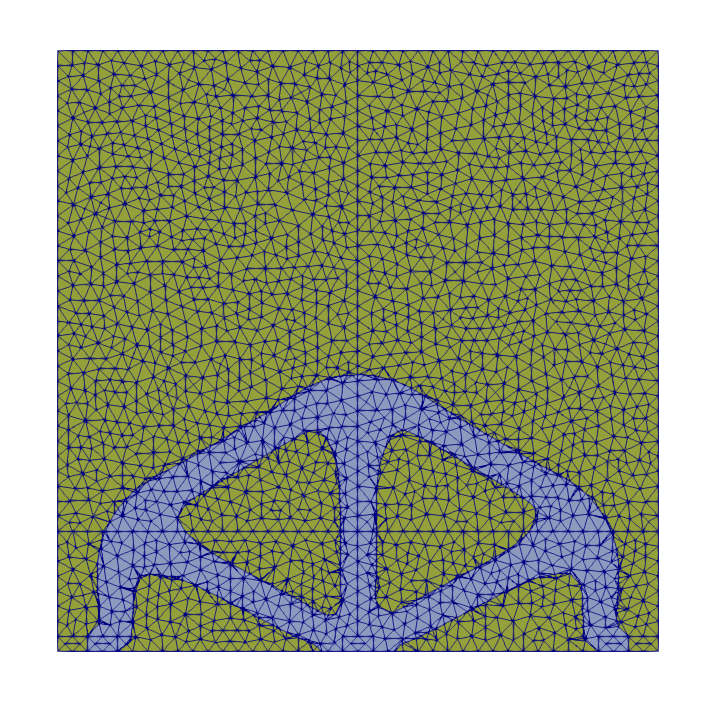}
    \label{sub:ResAnchrage2dPena}
    }
    
    \subfloat[Itération finale contact glissant (Lagrangien augmenté).]{
    \includegraphics[width=0.3\textwidth]{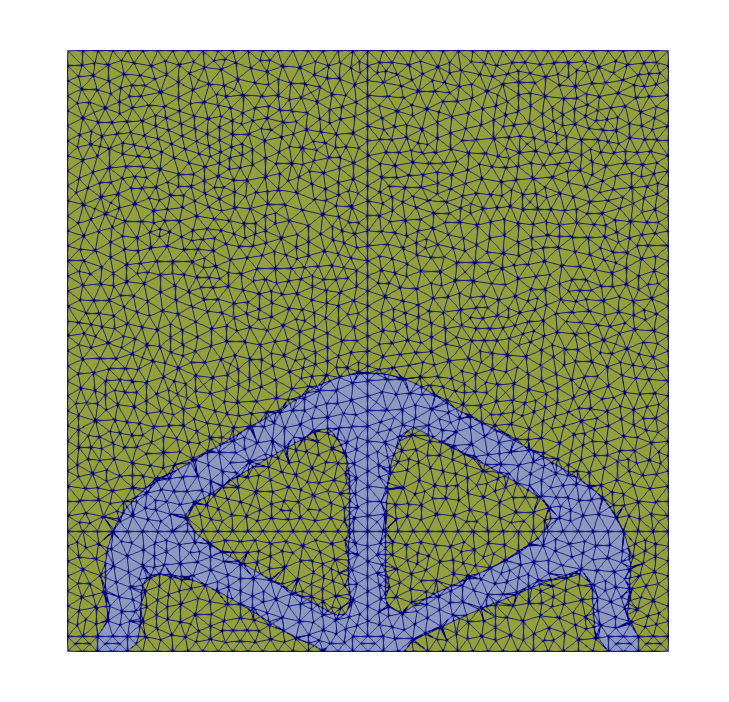}
    \label{sub:ResAnchrage2dLagAug}
    }
    \hspace{.3em}
    \subfloat[Itération finale contact frottant (pénalisation).]{
    \includegraphics[width=0.29\textwidth]{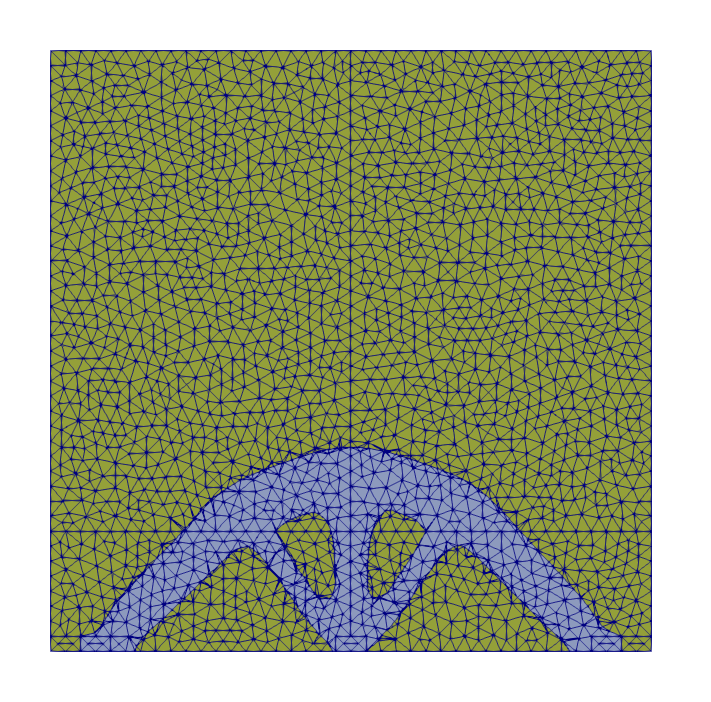}
    \label{sub:ResAnchrage2dFrottPena}
    }
    \hspace{.3em}
    \subfloat[Itération finale contact frottant (Lagrangien augmenté).]{
    \includegraphics[width=0.29\textwidth]{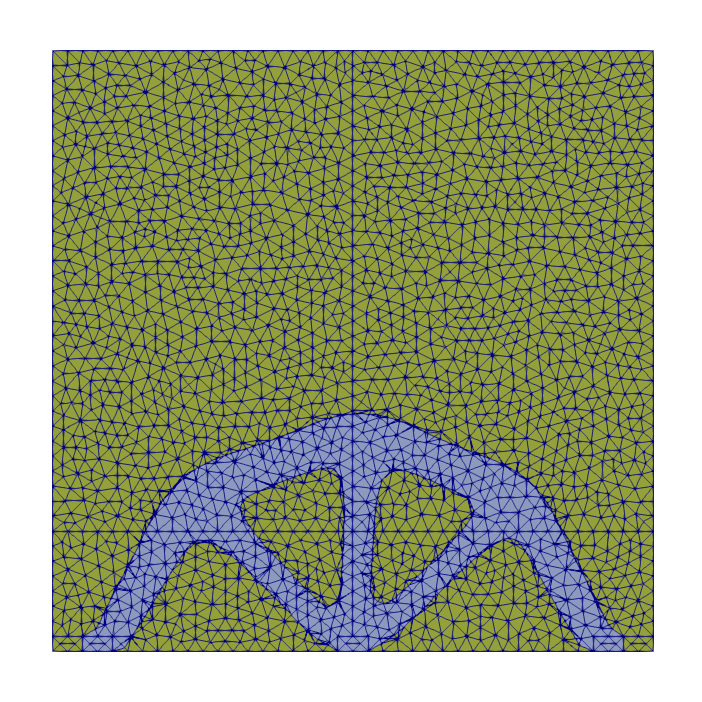}
    \label{sub:ResAnchrage2dFrottLagAug}
     }
    \end{center}
  \caption{Forme initiale et résultats pour l'anchrage 2d ($\Omega_h$ en bleu, $D_h\setminus \Omega_h$ en jaune).}
  \label{fig:ResAnchrage2d}
\end{figure}
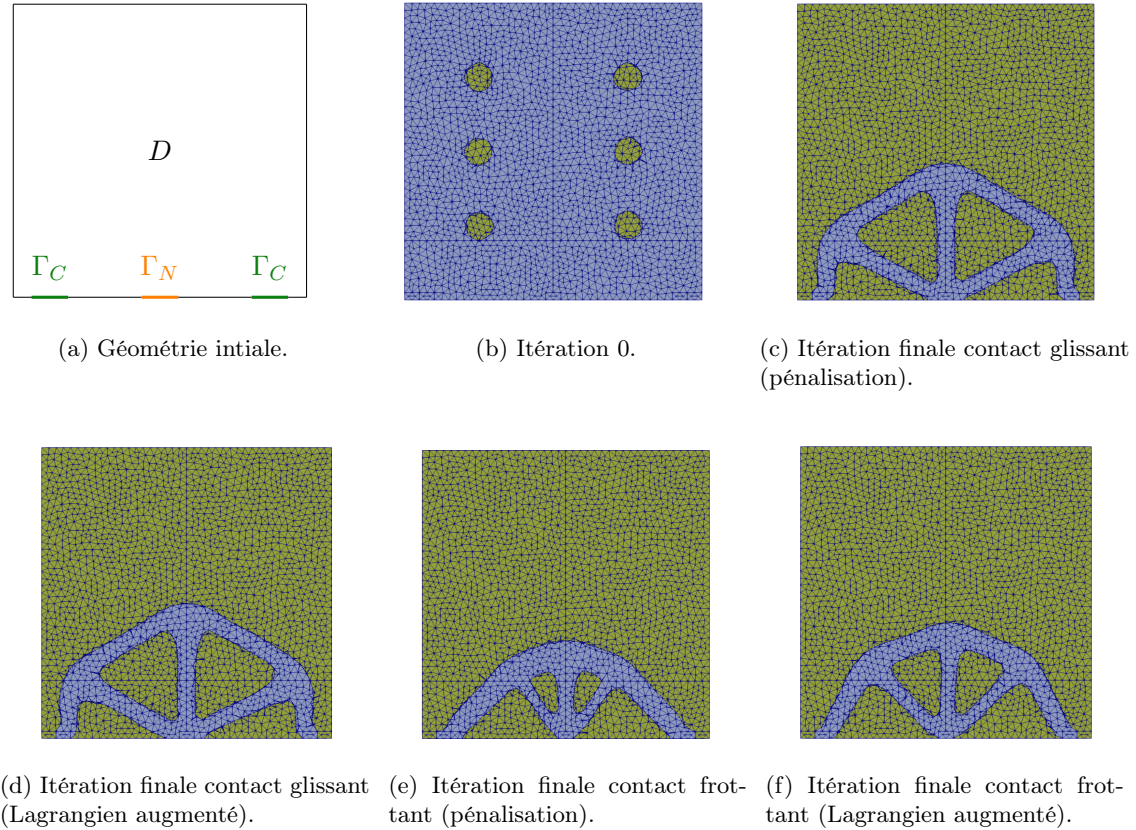

Dans le cas glissant, on obtient quasiment la même forme optimale avec les deux formulations, et ce design est similaire au résultat obtenu dans \cite{maury2016shape}. On note en particulier que les pattes de l'anchrage sont bien verticales au-dessus des zones de contact, ce qui permet de limiter le glissement et donc d'améliorer la rigidité. Les courbes de convergence de $J(\Omega^l)$, tracées figure \ref{fig:CvgIterAnchrage2d}, ont également le même comportement, et semblent converger vers la même valeur, en presque le même nombre d'itérations: 55 pour la pénalisation contre 52 pour le Lagrangien augmenté. Dans le cas frottant, on observe des formes optimales différentes, mais les valeurs de $J$ finales semblent très proches, voir figure \ref{fig:CvgIterAnchrage2d}. Comme prévu, dans les deux cas, grâce au frottement, les pattes peuvent être inclinées sans altérer la rigidité de la structure. 

\begin{figure}
\begin{center}
    \subfloat[Contact glissant.]{   
	\resizebox{!}{0.35\textwidth}{
	\begin{tikzpicture}
		\begin{axis}[
    		xlabel={Nombre d'itérations $l$},
    		xmin=0, xmax=60,
    		]
    		\addplot[color=gray,mark=none] table {res_anchrage2d_pena.txt};
    		\addplot[color=dartmouthgreen,mark=none] table {res_anchrage2d_lagaug.txt};
   			\legend{$J(\Omega^l)$ (pénalisation), $J(\Omega^l)$ (Lagrangien augmenté)}
   		\end{axis}
	\end{tikzpicture}
	}
	}
	\hspace{2em}
    \subfloat[Contact frottant.]{   
	\resizebox{!}{0.35\textwidth}{
	\begin{tikzpicture}
		\begin{axis}[
    		xlabel={Nombre d'itérations $l$},
    		xmin=0, xmax=290,
    		]
    		\addplot[color=gray,mark=none] table {res_anchrage2d_frott_pena_bis.txt};
    		\addplot[color=dartmouthgreen,mark=none] table {res_anchrage2d_frott_lagaug.txt};
   			\legend{$J(\Omega^l)$ (pénalisation), $J(\Omega^l)$ (Lagrangien augmenté)}
   		\end{axis}
	\end{tikzpicture}
	}
	}
\end{center}
\caption{Convergence des itérations pour l'anchrage 2d.}
\label{fig:CvgIterAnchrage2d}
\end{figure}
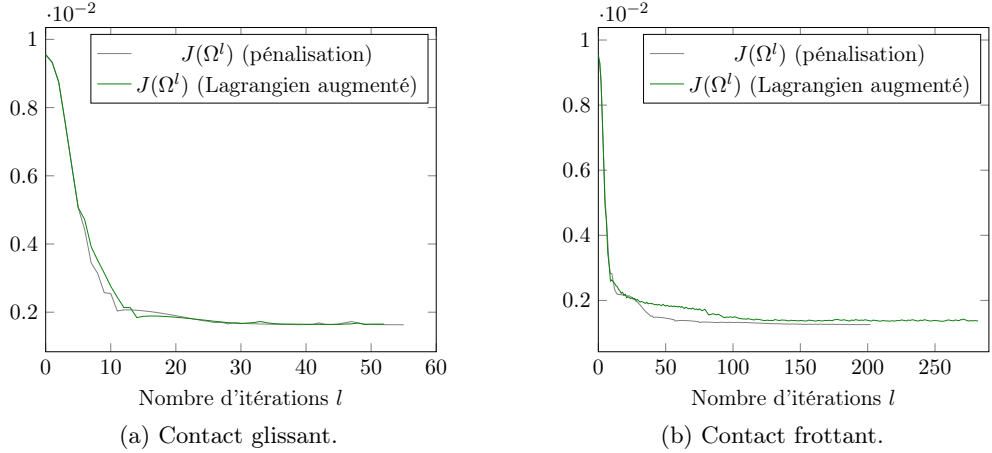

\subsubsection{Exemples où la zone de contact n'est pas connue a priori}

Afin de pouvoir valider notre approche, et de mettre en pratique les résultats théoriques des chapitres \ref{chap:2.1} et \ref{chap:3.1}, nous étudions dans ce qui suit un cas qui teste les dérivées de forme sur la zone de contact.

\paragraph{Cantilever 2d en contact avec un disque.}
En s'inspirant à la fois des cas-tests classiques d'optimisation de formes en élasticité sans contact et des cas-tests de résolution des problèmes de contact, on peut penser pour le cas 2d au cantilever en contact avec disque (voir par exemple \cite{stadler2004infinite}). Ce cas, ou des variantes proches, ont déjà été traités en optimisation de formes, nous citons par exemple \cite{fancello1995numerical}, ou encore \cite{haslinger1987shape,haslinger2003introduction} pour le cas d'un objet rigide plan. Nous préférons le cas plus général d'un objet non plan, car pour un plan on ne voit pas l'influence de la normale sur la dérivée de forme puisque $\normalExt'[\thetaa]=0$. Dans les travaux que nous venons de citer, seule la zone de contact, représentée à l'aide des valeurs d'un champ éléments finis sur un maillage donné (paradigme \textit{discrétiser-puis-optimiser}), est optimisée. Ici, en plus de nous placer dans le paradigme \textit{optimiser-puis-discrétiser}, nous proposons une optimisation de l'ensemble des frontières de $\Omega$ (incluant $\Gamma_C$), tout en ayant la possibilité de générer des changements de topologie (grâce à la représentation par level-set).

On considère donc le même domaine $D$ et les mêmes zones géométriques pour les conditions aux limites de Neumann et Dirichlet que dans le cas du cantilever 2d sans contact (figure \ref{sub:InitCLCanti2d}). On ajoute simplement un disque rigide de rayon $R$ de centre $(1,-R)$ avec lequel la poutre pourrait entrer en contact sous l'effet des contraintes extérieures. Ceci nous donne la géométrie figure \ref{fig:InitCLCanti2dDisque}. On impose à la structure les contraintes $\ff=0$ et $\tauu=(0,-0.01)$.
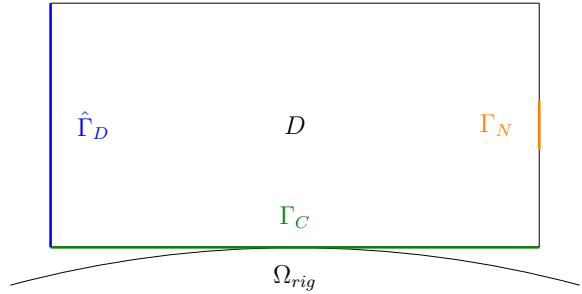
\begin{figure}[h]
	\begin{center}
	\resizebox{0.5\textwidth}{!}{
    \begin{tikzpicture}
    
	\draw[black] (8.66,-0.62) arc (75:105:18);
	\draw[dartmouthgreen, very thick] (0,0) -- (8,0);
	\draw[black] (8,0) -- (8,1.6);
	\draw[orange, very thick] (8,1.6) -- (8,2.4);
	\draw[black] (8,2.4) -- (8,4);
	\draw[black] (8,4) -- (0,4);
	\draw[blue, very thick] (0,4) -- (0,0);

	\node[white] at (4,-0.3) {$\:$};
	\node[black] at (4,2) {\large{$D$}};
	\node[blue] at (0.7,2) {\large{$\hat{\Gamma}_D$}};
	\node[orange] at (7.3,2) {\large{$\Gamma_N$}};
	\node[dartmouthgreen] at (4,0.5) {\large{$\Gamma_C$}};
	\node[black] at (4,-0.5) {\large{$\Omega_{rig}$}};
	\end{tikzpicture}
    }	
	\end{center}
  	\caption{Géométrie initiale pour le cantilever 2d en contact avec un disque.}
    \label{fig:InitCLCanti2dDisque}
\end{figure}

Ce benchmark présente plusieurs avantages. Tout d'abord, la géométrie est assez simple, et comme $\Omega_{rig}$ est un disque, on dispose d'expressions analytiques pour toutes les quantités (notamment $\normalExt$ et $\gG_{\normalExt}$), ce qui limite le biais introduit par la méthode numérique. De plus, comme on connaît d'une part le résultat du problème de contact seul et d'autre part celui du problème d'optimisation de formes sans contact, on peut avoir une intuition de la forme optimale obtenue en combinant les deux.

En effet, si on cherche à minimiser $J$, combinaison linéaire de la compliance et du volume, alors on s'attend à voir apparaître des raidisseurs au cours du processus d'optimisation de formes, comme pour la cas sans contact. En revanche, à cause du contact, on devrait perdre la symétrie de la forme optimale. On s'attend plutôt à ce que l'algorithme renforce les zones autour des points en contact, la structure pourra ainsi s'appuyer sur l'objet rigide pour gagner en rigidité, et limiter le mouvement vers le bas induit par la force de traction.

Dans cet exemple, on choisit un tel $J$, avec les coefficients $\alpha_1=15$ et $\alpha_2=0.01$. Pour la formulation pénalisé, on prend $\varepsilon=10^5$, et pour la formulation Lagrangien augmenté $\gamma_1^k=\gamma_2^k=100$ pour tout $k$. Le rayon $R$ du disque est fixé à 8. Enfin, dans le cas frottant, le seuil de Tresca est tel que $s=10^{-2}$ (ordre de grandeur de la contrainte normale), et le coefficient de frottement $\mathfrak{F}=0.2$. Avec ce jeu de données, on obtient les résultats présentés figure \ref{fig:ResCanti2dCercle}.

\begin{figure}[h]
  \begin{center}
    \subfloat[Itération 0.]{
	\resizebox{0.46\textwidth}{!}{
	\begin{tikzpicture}
    	\node[anchor=south west,inner sep=0] at (0,0) {\includegraphics[width=\textwidth]{canti2d_bis_it0.png}};
    	\draw[black] (15.6,0.01) arc (77:103:35);
    \end{tikzpicture}
    \label{sub:ResCanti2dCercleIt0}
    }
    }
    \hspace{.5em}
    \subfloat[Itération finale sans contact.]{
	\resizebox{0.46\textwidth}{!}{
	\begin{tikzpicture}
    	\node[anchor=south west,inner sep=0] at (0,0) {\includegraphics[width=\textwidth]{canti2d_bis_it64.png}};
    	\draw[white] (15.6,0.01) arc (77:103:35);
    \end{tikzpicture}
    \label{sub:ResCanti2dBis}
    }
    }
    \hspace{.5em}
    \subfloat[Itération finale contact glissant (pénalisation).]{
	\resizebox{0.46\textwidth}{!}{
	\begin{tikzpicture}
    	\node[anchor=south west,inner sep=0] at (0,0) {\includegraphics[width=\textwidth]{canti2d_cercle_pena_it303.png}};
    	\draw[black] (15.6,0.01) arc (77:103:35);
    \end{tikzpicture}
    \label{sub:ResCanti2dCerclePena}
    }
    }
    \hspace{.5em}
    \subfloat[Itération finale contact glissant (Lagrangien augmenté).]{
	\resizebox{0.46\textwidth}{!}{
	\begin{tikzpicture}
    	\node[anchor=south west,inner sep=0] at (0,0) {\includegraphics[width=\textwidth]{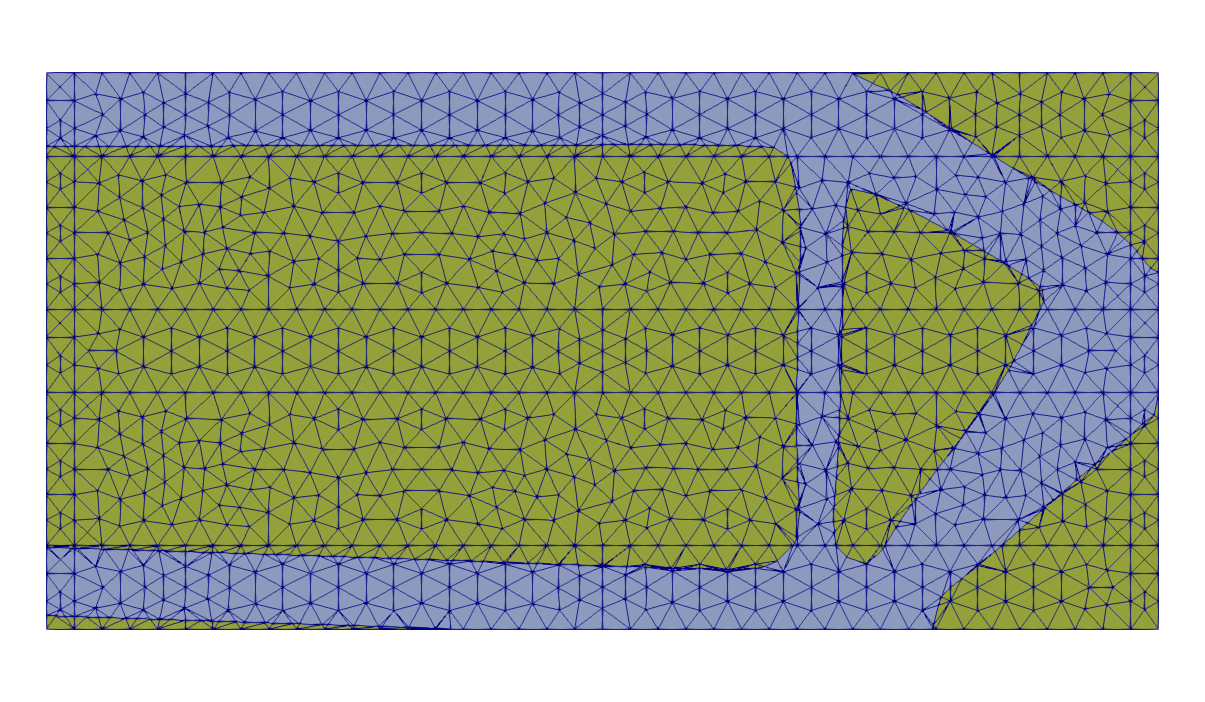}};
    	\draw[black] (15.6,0.01) arc (77:103:35);
    \end{tikzpicture}
    \label{sub:ResCanti2dCercleLagAug}
    }
    }
    \hspace{.5em}
    \subfloat[Itération finale contact frottant (pénalisation).]{
	\resizebox{0.46\textwidth}{!}{
	\begin{tikzpicture}
    	\node[anchor=south west,inner sep=0] at (0,0) {\includegraphics[width=\textwidth]{canti2d_cercle_frott_pena_it107.png}};
    	\draw[black] (15.6,0.01) arc (77:103:35);
    \end{tikzpicture}
    \label{sub:ResCanti2dCercleFrottPena}
    }
    }
    \hspace{.5em}
    \subfloat[Itération finale contact frottant (Lagrangien augmenté).]{
	\resizebox{0.46\textwidth}{!}{
	\begin{tikzpicture}
    	\node[anchor=south west,inner sep=0] at (0,0) {\includegraphics[width=\textwidth]{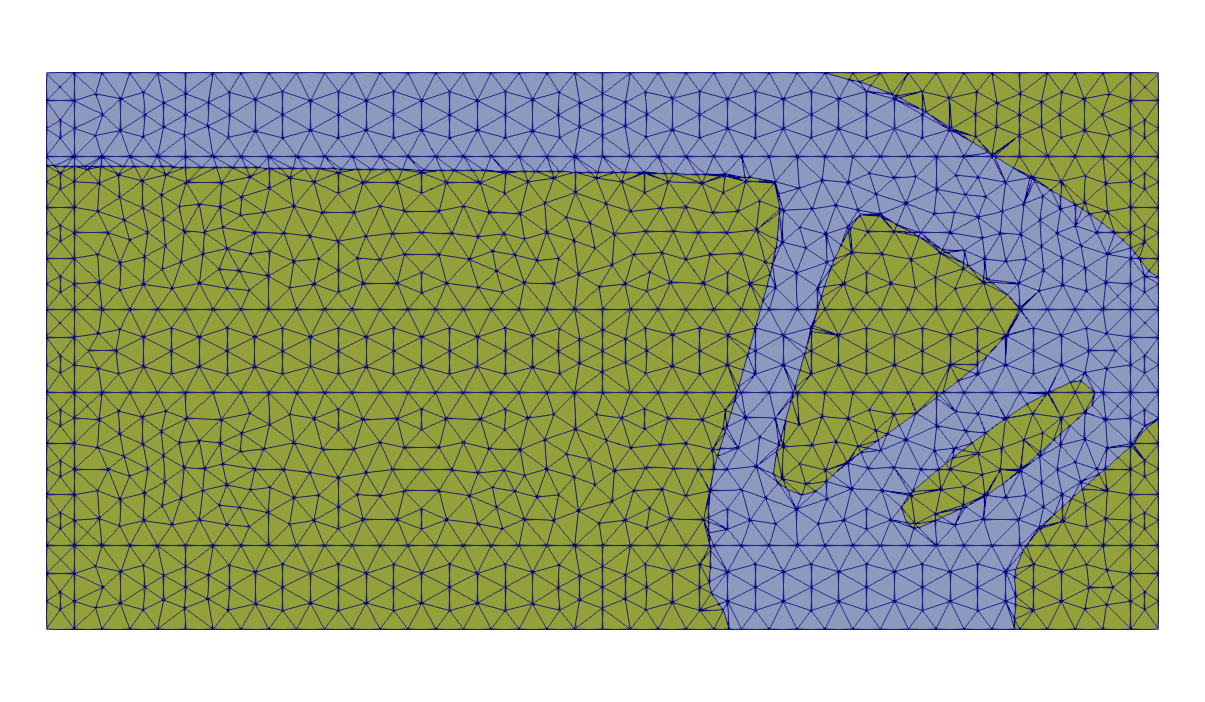}};
    	\draw[black] (15.6,0.01) arc (77:103:35);
    \end{tikzpicture}
    \label{sub:ResCanti2dCercleFrottLagAug}
    }
    }
    \end{center}
  \caption{Forme initiale et résultats pour le cantilever en contact avec un disque ($\Omega_h$ en bleu, $D_h\setminus \Omega_h$ en jaune).}
  \label{fig:ResCanti2dCercle}
\end{figure}

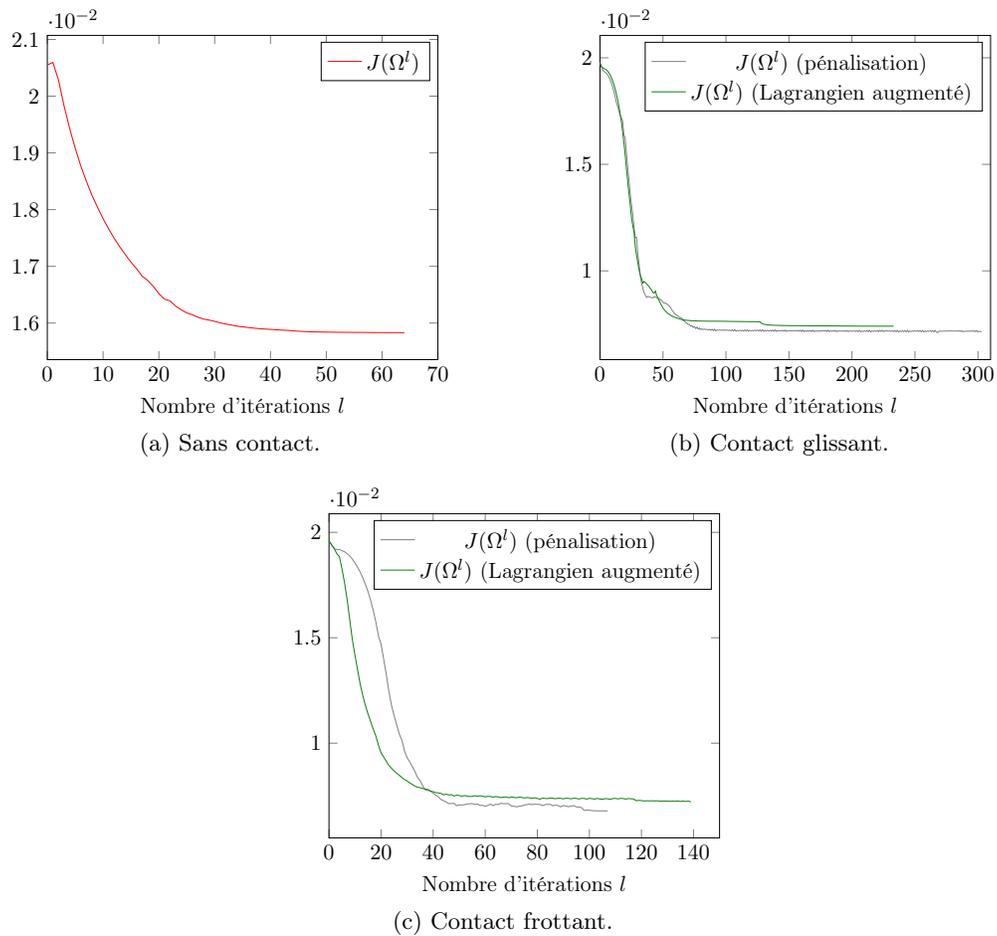
\begin{figure}
\begin{center}
    \subfloat[Sans contact.]{   
	\resizebox{!}{0.35\textwidth}{
	\begin{tikzpicture}
		\begin{axis}[
    		xlabel={Nombre d'itérations $l$},
    		xmin=0, xmax=70,
    		]
    		\addplot[color=red,mark=none] table {res_canti2d_bis.txt};
   			\legend{$J(\Omega^l)$}
   		\end{axis}
	\end{tikzpicture}
	}
	}
	\hspace{2em}
    \subfloat[Contact glissant.]{   
	\resizebox{!}{0.35\textwidth}{
	\begin{tikzpicture}
		\begin{axis}[
    		xlabel={Nombre d'itérations $l$},
    		xmin=0, xmax=310,
    		]
    		\addplot[color=gray,mark=none] table {res_canti2d_cercle_pena.txt};
    		\addplot[color=dartmouthgreen,mark=none] table {res_canti2d_cercle_lagaug.txt};
   			\legend{$J(\Omega^l)$ (pénalisation), $J(\Omega^l)$ (Lagrangien augmenté)}
   		\end{axis}
	\end{tikzpicture}
	}
	}
	\hspace{2em}
    \subfloat[Contact frottant.]{   
	\resizebox{!}{0.35\textwidth}{
	\begin{tikzpicture}
		\begin{axis}[
    		xlabel={Nombre d'itérations $l$},
    		xmin=0, xmax=150,
    		]
    		\addplot[color=gray,mark=none] table {res_canti2d_cercle_frott_pena.txt};
    		\addplot[color=dartmouthgreen,mark=none] table {res_canti2d_cercle_frott_lagaug.txt};
   			\legend{$J(\Omega^l)$ (pénalisation), $J(\Omega^l)$ (Lagrangien augmenté)}
   		\end{axis}
	\end{tikzpicture}
	}
	}
\end{center}
\caption{Convergence des itérations pour le cantilever en contact avec un disque.}
\label{fig:CvgIterCanti2dCercle}
\end{figure}

On remarque encore que pour le même jeu de paramètres, le Lagrangien augmenté et la pénalisation donnent des formes optimales très similaires. Dans tous les cas, on remarque que le lieu des points effectivement en contact (initialement au milieu) a été déplacé vers la droite au cours du processus d'optimisation. En effet, plus l'appui sur l'objet rigide sera proche de la zone où on applique la contrainte, plus la forme sera rigide. Pour le modèle de contact glissant, la forme optimale a deux points d'anchrage: un en haut, qui permet d'assurer une certaine rigidité, et l'autre en bas, pour empêcher l'objet de glisser vers le bas et vers la droite. Pour le modèle frottant, on n'a plus qu'un point d'anchrage, car le frottement sur l'objet rigide fait qu'on aura une plus grande stabilité au niveau de la zone d'appui.

Par ailleurs, comme on s'y attend, si on compare les designs optimaux obtenus avec contact et celui sans contact, on voit clairement que pour ce cas-test, la possibilité de s'appuyer sur un corps rigide permet d'avoir des formes tout aussi rigides mais bien plus légères. En effet, si on regarde les valeurs de $J$ pour ces designs optimaux, rassemblées dans les tableaux figure \ref{fig:TabValeursJCanti2d}, on voit que le cas avec contact donne des formes bien plus performantes.

Si les allures des tracés de l'historique de $J$ (voir figure \ref{fig:CvgIterCanti2dCercle}) sont comparables, on obtient des valeurs finales assez différentes pour la pénalisation et le Lagrangien augmenté. Plus précisément, les valeurs du premier tableau de la figure \ref{fig:TabValeursJCanti2d} semblent indiquer que le design obtenu dans le cas pénalisé est plus performant. En réalité, il paraît plus probable que cet écart soit dû à l'inconsistence de la méthode de pénalisation, d'autant qu'ici $\epsilon = 10^5$, ce qui n'est pas très élevé.

\begin{figure}
\begin{center}
    \begin{tabular}{|c|c|c|}
    \hline
    \textbf{Avec contact} & glissant & frottant \\
    \hline
    Pénalisation & 0.00714989 & 0.00678781  \\
    \hline
    Lagrangien augmenté & 0.00740285 & 0.00722132 \\
    \hline
    \end{tabular}
    \hspace{1.5em}
    \begin{tabular}{|c|c|}
    \hline
    \textbf{Sans contact} & 0.0158314 \\
    \hline
    \end{tabular}
    \caption{Valeurs de $J$ à l'itération finale pour le cantilever.}
    \label{fig:TabValeursJCanti2d}
\end{center}
\end{figure}

\paragraph{Cantilever 3d en contact avec une boule.}
Pour la dimension trois, nous mentionnons \cite{beremlijski2009shape} pour le cas d'un cantilever 3d en contact avec un plan, où la zone de contact est optimisée. Comme pour les travaux cités dans le cas 2d, les auteurs représentent $\Gamma_C$ par une fonction définie comme la distance au plan rigide (gap), qu'ils discrétisent par éléments finis avant d'écrire les conditions d'optimalité associées au système discrétisé. Pour les mêmes raisons qu'en dimension deux, nous proposons à la place un cantilever en contact avec un corps rigide $\Omega_{rig}$ sphérique. On peut trouver ce cas-test en mécanique du contact, bien que la variante la plus populaire soit celle où la boule est remplacée par un cylindre d'axe orthogonal à l'axe de la poutre.

On considère donc cette fois le même domaine $D$ et les mêmes $\hat{\Gamma}_D$ et $\Gamma_N$ que pour le cantilever 3d sans contact (voir figure \ref{sub:ResCanti3dIt0}). Puis on ajoute $\Omega_{rig}$, la boule de centre $(2.5,-R,1.2)$ et de rayon $R$. On applique toujours des forces extérieures $\ff=0$ et $\tau=(0,0,-0.01)$. Les paramètres de résolutions sont tels que $\varepsilon=10^7$ et $\gamma_1^k=\gamma_2^k=1000$ pour tout $k$, et les coefficients devant $J$ sont mis à $\alpha_1=30$, $\alpha_2=0.01$.

\begin{figure}[h]
  \begin{center}
    \subfloat[Maillage $D_h$.]{
	\includegraphics[width=0.35\textwidth]{canti3d_maillage.png}
    \label{sub:MailCanti3dBoule}
    }
    \hspace{.5em}
    \subfloat[Itération finale sans contact.]{
	\includegraphics[width=0.35\textwidth]{canti3d_it0.png}
    \label{sub:ResCanti3dBouleIt0}
    }
    \hspace{.5em}
    \subfloat[Itération finale contact glissant (pénalisation).]{
	\includegraphics[width=0.35\textwidth]{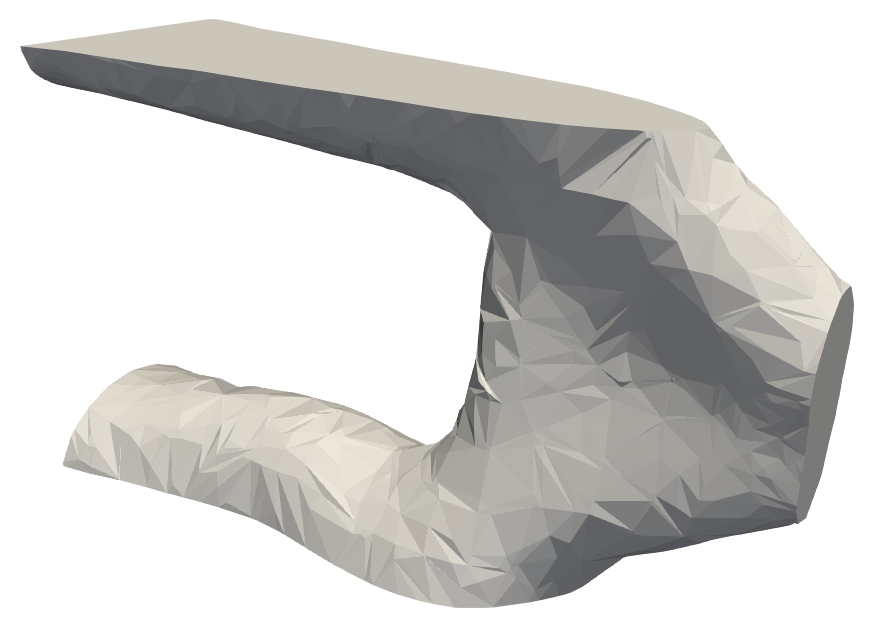}
    \label{sub:ResCanti3dBoulePena}
    }
    \hspace{.5em}
    \subfloat[Itération finale contact glissant (Lagrangien augmenté).]{
	\includegraphics[width=0.35\textwidth]{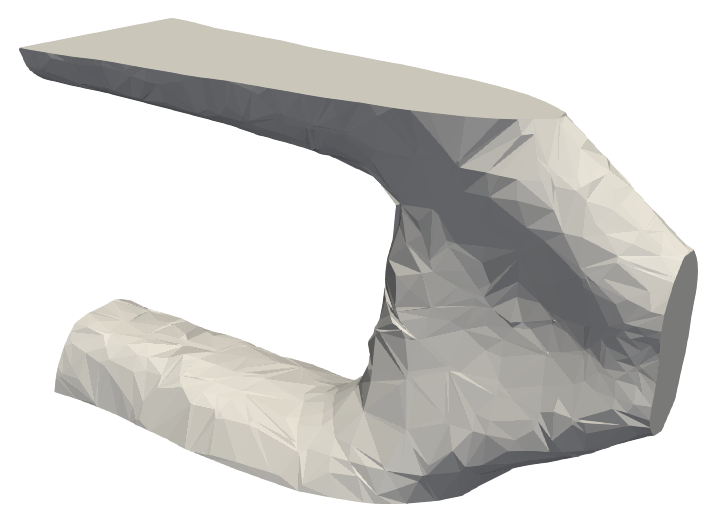}
    \label{sub:ResCanti3dBouleLagAug}
    }
    \hspace{.5em}
    \subfloat[Itération finale contact frottant (pénalisation).]{
	\includegraphics[width=0.35\textwidth]{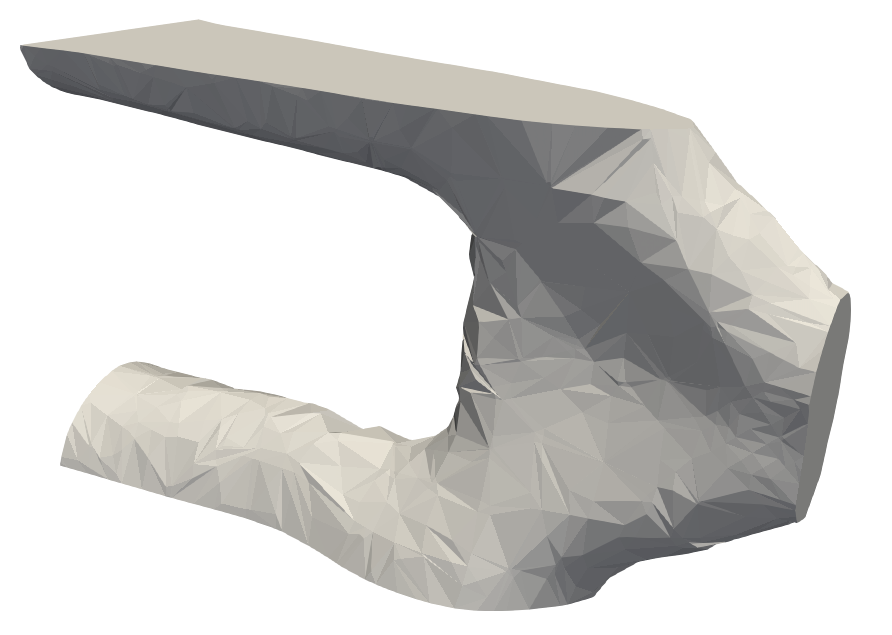}
    \label{sub:ResCanti3dBouleFrottPena}
    }
    \hspace{.5em}
    \subfloat[Itération finale contact frottant (Lagrangien augmenté).]{
	\includegraphics[width=0.35\textwidth]{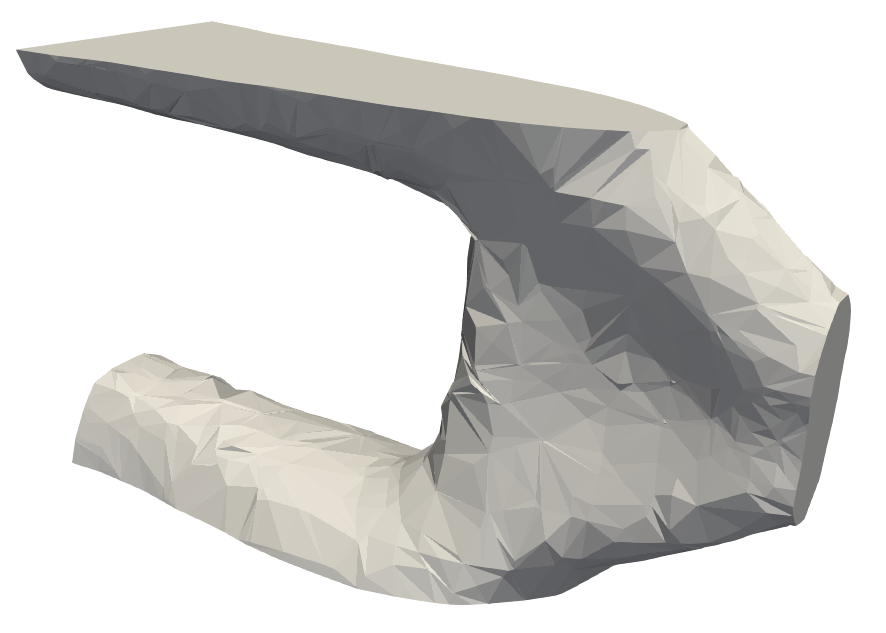}
    \label{sub:ResCanti3dBouleFrottLagAug}
    }
    \end{center}
  \caption{Forme initiale et finales pour le cantilever en contact avec une boule.}
  \label{fig:ResCanti3dBoule}
\end{figure}

On obtient cette fois encore des formes optimales similaires dans les cas du contact pénalisé et par Lagrangien augmenté, que ce soit dans le modèle de contact glissant ou frottant, voir figure \ref{fig:ResCanti3dBoule}. Du point de vue géométrique, on peut faire les mêmes remarques que pour le cantilever 2d en contact avec un disque: les formes optimales s'appuient sur l'objet rigide au niveau d'une zone de contact (effectif) ayant été déplacée plus près du bord de Neumann. Comme dans le cas du domaine en L en 3d, on note également que la forme cherche à s'affiner dans la direction orhtogonale au mouvement (sans pour autant créer de trou). Le coefficient de frottement étant relativement faible, on obtient des designs optimaux semblables dans les cas avec ou sans frottement.

De façon plus quantitative, si on regarde les tracés des valeurs de $J$ sur le figure \ref{fig:CvgIterCanti3dBoule}, on remarque un comportement inattendu de la fonctionnelle dans le cas du contact frottant pénalisé autour de l'itération 30. Dans le processus, cela correspond à des itérations où la zone de contact est fortement déplacée et où le Newton de la résolution du problème mécanique a des difficultés à converger. C'est d'ailleurs ce qui nous limite dans le choix du coefficient de frottement: si on le prend plus grand, l'algorithme d'optimisation ne parvient pas à converger pour le modèle pénalisé à cause des difficultés à résoudre le problème mécanique.

\begin{figure}
\begin{center}
    \subfloat[Contact glissant.]{   
	\resizebox{!}{0.35\textwidth}{
	\begin{tikzpicture}
		\begin{axis}[
    		xlabel={Nombre d'itérations $l$},
    		xmin=0, xmax=170,
    		]
    		\addplot[color=gray,mark=none] table {res_canti3d_sphere_pena.txt};
    		\addplot[color=dartmouthgreen,mark=none] table {res_canti3d_sphere_lagaug_bis.txt};
   			\legend{$J(\Omega^l)$ (pénalisation), $J(\Omega^l)$ (Lagrangien augmenté)}
   		\end{axis}
	\end{tikzpicture}
	}
	}
	\hspace{2em}
    \subfloat[Contact frottant.]{   
	\resizebox{!}{0.35\textwidth}{
	\begin{tikzpicture}
		\begin{axis}[
    		xlabel={Nombre d'itérations $l$},
    		xmin=0, xmax=160,
    		]
    		\addplot[color=gray,mark=none] table {res_canti3d_sphere_frott_pena.txt};
    		\addplot[color=dartmouthgreen,mark=none] table {res_canti3d_sphere_frott_lagaug.txt};
   			\legend{$J(\Omega^l)$ (pénalisation), $J(\Omega^l)$ (Lagrangien augmenté)}
   		\end{axis}
	\end{tikzpicture}
	}
	}
\end{center}
\caption{Convergence des itérations pour le cantilever en contact avec une boule.}
\label{fig:CvgIterCanti3dBoule}
\end{figure}
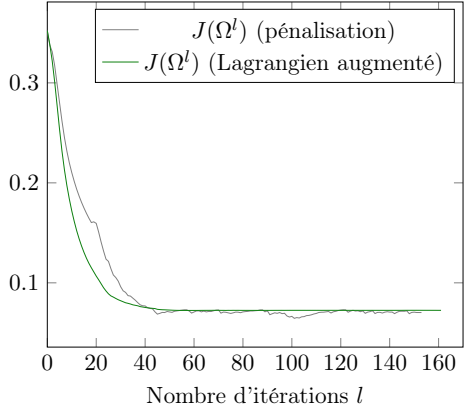
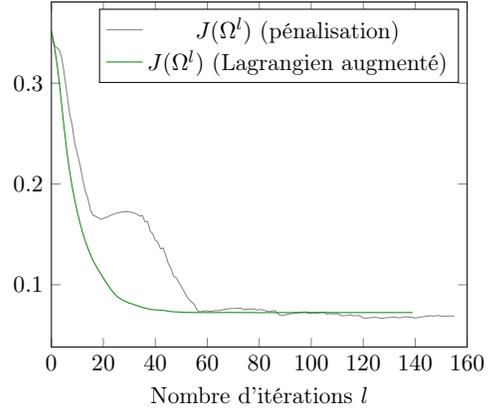
\chapter*{Conclusion}           
\label{chap:conclusion}         
\phantomsection\addcontentsline{toc}{chapter}{\nameref{chap:conclusion}} 

Dans ce travail, nous avons présenté une méthode d'optimisation de formes topologique appliquée à la mécanique des structures, plus précisément à l'élasticité linéaire avec contact, glissant et frottant (Tresca). Cette méthode s'appuie sur une représentation des formes à l'aide de level-sets, et sur un algorithme de type gradient utilisant les dérivées de formes calculées par l'analyse de sensibilité en dimension infinie. Afin de contourner les difficultés techniques liées à la formulation du problème de contact, nous avons proposé de nous concentrer sur deux versions régularisées de cette-dernière: la formulation pénalisée et la formulation Lagrangien augmenté. Ces deux formulations admettant des termes non-différentiables, nous avons introduit une approche par dérivées directionnelles permettant d'obtenir des conditions d'optimalité sans avoir à ajouter une étape de régularisation supplémentaire.

Après avoir présenté notre cadre d'étude, à la fois côté optimisation de formes et mécanique du contact, puis avoir fait l'analyse mathématique des formulations considérées, nous nous sommes consacrés à l'analyse de sensibilité par rapport à la forme dans le cas du contact pénalisé.
En particulier, nous avons montré que la solution de la formulation associée admettait toujours des dérivées de formes directionnelles. De plus, lorsque la zone de contact affleurant est de mesure nulle, cette solution est même dérivable (au sens de Gâteaux) par rapport à la forme, ce qui autorise l'application des méthodes classiques de la théorie du contrôle optimal.
Nous avons ensuite adapté cette approche par dérivées directionnelles au cas de la formulation Lagrangien augmenté du problème de contact, en s'appuyant sur les similarités entre la formulation résolue à chaque itération de la méthode de Lagrangien augmenté et la formulation pénalisée. De plus, sous certaines hypothèses, nous avons prouvé un résultat de convergence pour la suite de dérivées directionnelles de formes des itérées construite.
Enfin, nous avons proposé un algorithme d'optimisation de formes qui permet d'avoir à chaque itération non seulement une représentation implicite de la forme par l'intermédiaire d'une fonction level-set sur une grille cartésienne, mais aussi une représentation explicite grâce à un sous-maillage composé de triangle ou de tetraèdres, obtenu par découpage. Nous avons validé la méthode numérique en la testant sur des cas-tests classiques d'élasticité linéaire sans contact, en deux et trois dimensions. Puis nous avons présenté le cas-test du cantilever en contact (avec un disque ou une boule selon la dimension), dans un contexte où la zone de contact est complètement optimisée et où le gap est non nul. Des résultats ont été obtenus pour la formulation pénalisée et pour la formulation Lagrangien augmenté, ce qui a permis de mettre en évidence les différences entre les deux modèles.

Du côté théorique, plusieurs pistes pourraient être explorées dans la continuité des travaux présentés dans ce manuscrit. Premièrement, il serait intéressant d'étendre les résultats du chapitre \ref{chap:3.1} au cas du frottement de Tresca. Bien que la preuve de la dérivabilité directionnelle à chaque itération de Lagrangien augmenté soit directe, la dérivabilité conique de l'inéquation variationnelle d'origine est loin d'être aussi simple (voir \cite{sokolowski1992introduction}), et il en va de même de l'analyse de convergence des dérivées de forme. Ensuite, il conviendrait d'essayer de généraliser ce résultat de convergence (Théorème \ref{ThmCvgConDerALM}), dans le cas glissant comme dans le cas frottant, car les hypothèses qui ont été faites sont très restrictives.

On pourrait également penser à explorer le cas de matériaux non-linéaires en grandes déformations. Évidemment, la première condition nécessaire pour pouvoir espérer tirer des conclusions du point de vue de l'optimisation de formes est l'existence et l'unicité de la solution. Sous ces hypothèses, alors les méthodes présentées de ce travail devraient pouvoir s'adapter. Une différence importante avec l'élasticité linéaire en petites déformations est que la normale (que ce soit $\normalExt$ ou $\normalInt$) dépend du déplacement $\uu$. Les formules présentées ici ne sont donc plus valides, à moins qu'on considère le cas particulier d'un objet rigide plan, auquel cas $\normalExt$ est de toute façon constant. Il conviendrait donc d'essayer de les généraliser.

Du côté numérique, une première idée serait d'appliquer l'algorithme présenté au cas du contact avec frottement de Coulomb. En effet, on peut montrer que lorsque le coefficient de frottement est suffisamment petit, le problème de Coulomb peut s'interpréter comme la limite d'une suite de problèmes de Tresca, voir \cite{eck2005unilateral}. Ce résultat théorique suggère donc l'utilisation d'un algorithme itératif de type point fixe résolvant un problème de Tresca à chaque itération. On pourrait alors utiliser les expressions des dérivées de formes données aux chapitres \ref{chap:2.1} ou \ref{chap:3.1} pour la dernière itération de Tresca, ce qui donnerait une approximation de la dérivée de forme du problème de Coulomb.

Pour améliorer la qualité du maillage afin de faciliter la résolution mécanique, il pourrait être utile d'implémenter une stratégie de remaillage (similaire à celle présentée dans \cite{dapogny2013shape}) en plus du découpage. On pourrait réfléchir à un quantité permettant de mesurer la qualité du maillage surfacique, ce qui indiquerait si l'étape de remaillage est nécessaire. Pour la partie concernant le remaillage en tant que tel, il serait intéressant d'utiliser les nombreuses fonctionnalités d'adaptation de maillages disponibles dans MEF++.

Comme la méthode de résolution présentée se veut relativement générale, elle peut facilement être étendue aux cas de l'élasticité mixte (résolution de la mécanique en déplacements-pression $(\uu,p)$) ou encore de plusieurs matériaux élastiques, on parle parfois de \textit{distribution de matériaux}. Il faut seulement déterminer les expressions des dérivées de forme pour ces cas-là, ce qui permettra de trouver des directions de descente.

Une dernière perspective serait d'adapter la méthode d'optimisation de formes proposée au cas de matériaux élastiques en grandes déformations après en avoir fait l'étude théorique. Évidemment cela nécessite avant tout un solveur robuste pour le problème mécanique seul, par exemple s'appuyant sur une méthode de continuation pour gérer la non unicité de la solution. 

\bibliographystyle{plain}
\bibliography{bibliography,references}



\end{document}